\documentclass{article}

\usepackage[vmargin=1.18in,hmargin=1.1in]{geometry}
\usepackage{amsmath,amssymb,amsfonts,graphicx,amsthm,nicefrac,mathtools,bm}
\usepackage{hyperref}
\usepackage[numbers]{natbib}
\usepackage[english]{babel}
\usepackage[T1]{fontenc}
\usepackage{bbm,mdframed}
\usepackage[dvipsnames]{xcolor}
\usepackage{xspace}
\usepackage{rotating,multirow}
\usepackage{enumerate,graphics,enumitem}
\usepackage[algo2e]{algorithm2e}

\setlist[itemize]{noitemsep, topsep=0pt}
\setlist[enumerate]{itemsep=5pt, topsep=5pt, leftmargin=25pt}

\newtheorem{theorem}{Theorem}

\definecolor{verylightblue}{rgb}{0.7,0.8,1}
  {\begin{mdframed}[backgroundcolor=verylightblue]\begin{theorem}}%
  {\end{theorem}\end{mdframed}}

\definecolor{verylightgray}{gray}{0.95}
  {\begin{mdframed}[backgroundcolor=verylightgray]\begin{proof}}%
  {\end{proof}\end{mdframed}}

\newtheorem{claim}{Claim}
\newtheorem{lemma}{Lemma}
\definecolor{verylightred}{rgb}{1,0.8,0.8}
  {\begin{mdframed}[backgroundcolor=verylightred]\begin{lemma}}%
  {\end{lemma}\end{mdframed}}

\newtheorem{proposition}{Proposition}
  {\begin{mdframed}[backgroundcolor=verylightblue]\begin{proposition}}%
  {\end{proposition}\end{mdframed}}

\newtheorem{corollary}{Corollary}

\theoremstyle{definition}
\newtheorem{definition}{Definition}
\theoremstyle{remark}

\newtheorem{example}{Example}
\theoremstyle{empty}

\makeatletter

\newtheorem*{rep@theorem}{\rep@title}
\newcommand{\newreptheorem}[2]
{\newenvironment{rep#1}[1]
{\def\rep@title{#2 \ref{##1}} \begin{rep@theorem}}%
 {\end{rep@theorem}}}
\makeatother
\newreptheorem{theorem}{Theorem}
\newreptheorem{lemma}{Lemma}
\newreptheorem{corollary}{Corollary}
\newreptheorem{proposition}{Proposition}

\DeclareMathOperator*{\argmax}{arg\,max}
\DeclareMathOperator*{\argmin}{arg\,min}

\newcommand{\one}[1]{{\mathbf{1}}\{{#1}\}}

\newcommand{\PP}[1]{\mathbb{P}\!\left\{{#1}\right\}} 
\newcommand{\E}{\mathbb{E}}

\newcommand{\PPst}[2]{\mathbb{P}\!\left\{{#1}\ \middle| \ {#2}\right\}} 

\def\R{\mathbb{R}}

\def\O{\mathcal{O}}

\newcommand{\ignore}[1]{}

\let\emptyset\varnothing

\usepackage{tikz}
\usetikzlibrary{arrows}
\usetikzlibrary{shapes}

\newcommand{\A}{\mathcal{A}}
\newcommand{\B}{\mathcal{B}}

\newcommand{\thedate}{\today}
\newcommand{\theauthor}{Tijana Zrnic\\\texttt{tijana.zrnic@berkeley.edu}\\University of California, Berkeley \and Michael I. Jordan\\\texttt{jordan@stat.berkeley.edu}\\
University of California, Berkeley
}
\newcommand{\thetitle}{Post-Selection Inference via Algorithmic Stability}

\date{\thedate}
\author{\theauthor}
\title{\thetitle}

\usepackage[mmddyy]{datetime}

\def\P{\mathcal{P}}

\def\M{\mathcal{M}}

\def\F{\mathcal{F}}
\def\D{\mathcal{D}}

\def\G{\mathcal{G}}

\def\thetalasso{\hat\theta_{\mathrm{LASSO}}}
\def\ci{\mathrm{CI}}

\def\dsel{d_{\mathrm{sel}}}

\makeatletter
\newcommand{\dotfrac}[2]{
\mathchoice
{\ooalign{$\genfrac{}{}{0pt}{0}{#1}{#2}$\cr\leavevmode\cleaders\hb@xt@ .22em{\hss $\displaystyle\cdot$\hss}\hfill\kern\z@\cr}}
{\ooalign{$\genfrac{}{}{0pt}{1}{#1}{#2}$\cr\leavevmode\cleaders\hb@xt@ .22em{\hss $\textstyle\cdot$\hss}\hfill\kern\z@\cr}}
{\ooalign{$\genfrac{}{}{0pt}{2}{#1}{#2}$\cr\leavevmode\cleaders\hb@xt@ .22em{\hss $\scriptstyle\cdot$\hss}\hfill\kern\z@\cr}}
{\ooalign{$\genfrac{}{}{0pt}{3}{#1}{#2}$\cr\leavevmode\cleaders\hb@xt@ .22em{\hss $\scriptscriptstyle\cdot$\hss}\hfill\kern\z@\cr}}
}
\makeatother

\usepackage{algorithm,algorithmic}

\makeatletter
\long\def\@makecaption#1#2{
        \vskip 0.8ex
        \setbox\@tempboxa\hbox{\small {\bf #1:} #2}
        \parindent 1.5em  
        \dimen0=\hsize
        \advance\dimen0 by -3em
        \ifdim \wd\@tempboxa >\dimen0
                \hbox to \hsize{
                        \parindent 0em
                        \hfil 
                        \parbox{\dimen0}{\def\baselinestretch{0.96}\small
                                {\bf #1.} #2
                                } 
                        \hfil}
        \else \hbox to \hsize{\hfil \box\@tempboxa \hfil}
        \fi
        }
\makeatother

\usepackage[algo2e]{algorithm2e}

\begin{document}

 \maketitle


\begin{abstract}
When the target of statistical inference is chosen in a data-driven manner, the guarantees provided by classical theories vanish. We propose a solution to the problem of inference after selection by building on the framework of
\emph{algorithmic stability}, in particular its branch with origins in the
field of differential privacy. Stability is achieved via \emph{randomization} of selection and it serves as a quantitative measure that is sufficient to obtain non-trivial post-selection
corrections for classical confidence intervals.  
Importantly, the underpinnings of algorithmic stability translate directly
into computational efficiency---our method computes simple
corrections for selective inference without recourse to Markov chain
Monte Carlo sampling.
\end{abstract}

\section{Introduction}

Classical statistical theory provides tools for valid inference under the assumption that the statistical question is determined before observing any data. In practice, however, the choice of question is typically guided by exploring the same data that is used for inference. This coupling between the statistical question and the data used for inference induces dependencies that invalidate guarantees derived from classical theories.

While traditional wisdom might deem this coupling unacceptable, recent literature embraces this coupled approach to statistical investigation and grants novel ways of thinking about validity. Indeed, data-driven model selection is widely taught and practiced, and even stands as a research area of its own. Sometimes model selection is even unavoidable; in the canonical setting of linear regression, the statistician often starts with a pool of candidate variables large enough that it makes the solution unidentifiable without additional constraints, and when those constraints are data-dependent the solution depends on the data in two ways.

This coupling of the problem formulation and inference stages of statistical analysis has been thoroughly studied in a line of work called \emph{selective}, or \emph{post-selection}, \emph{inference} \cite{berk2013valid,taylor2015statistical, benjamini2010simultaneous}. To this day, however, there are few general principles that enable both statistically powerful and computationally tractable inference after selection. Most existing solutions are either tailored to specific selection strategies \citep[e.g.,][]{lee2016exact, tibshirani2016exact, lee2014exact}, and as such do not generalize to all popular selection methods, or are valid for arbitrary selections at the cost of increased conservativeness \citep[e.g.,][]{berk2013valid, bachoc2020uniformly}.

In the current paper, we build on concepts from the field of differential privacy \cite{dwork2006calibrating, dwork2014algorithmic} to derive selective confidence intervals that are both tractable computationally and powerful statistically.  Our theoretical framework delivers intervals of \emph{tunable width}, a useful consequence of the fact that our confidence intervals derive from a quantitative measure of the \emph{algorithmic stability} of the selection procedure. More precisely, we provide a valid correction to classical, non-selective confidence intervals simultaneously for all procedures that have the same level of algorithmic stability. Informally, a selection being stable means that it is not too sensitive to the particular realization of the data, and the more stable the selection is, the smaller the resulting intervals are. In particular, if the selection is ``perfectly stable'' in the sense that the inferential target is fixed up front and does not depend on the data at hand, the confidence intervals resulting from our approach smoothly recover classical confidence intervals.

We sketch our main result. Let $\hat S$ denote a data-dependent outcome involving selection. For example, $\hat S$ could be subset of $\{1,\dots,d\}$ corresponding to the variables selected for inclusion in a linear regression model. For every possible selection $S$, let $\beta_{S}$ denote the resulting inferential target. In the linear regression context, $\beta_S$ could be the population-level least-squares solution within the model determined by $S$.


Imagine that there is an oracle that guesses $\hat S$, only knowing the method used to arrive at the selection together with the distribution of the data, but not its realization. Denote by $\hat S_0$ the oracle's guess. We say that a selection procedure is $\eta$-stable for some $\eta>0$ if there exists an oracle such that, with high probability over the distribution of the data, the likelihood of any selection under $\hat S$ and the likelihood of the same selection under $\hat S_0$ can differ by at most a multiplicative factor of $e^\eta$. Intuitively, $\eta$ quantifies how much the selection can vary across different realizations of the data; $\eta=0$ essentially means that the selection cannot depend on the data and hence $\hat S$ is fixed, while as $\eta$ grows the selection is allowed to be increasingly data-adaptive. Note that the magnitude of stability depends not only on the selection method, but also on the distribution of the data.

Our main result provides a post-selection-valid correction to classical, non-selective confidence intervals for stable selection procedures. We state an informal version of our key inference tool.

\begin{theorem}[Informal]
\label{thm:informal}
For every \emph{fixed} selection $S$, suppose that $\ci_{S}^{(\alpha)}$ are confidence intervals with valid coverage,
$$\PP{\beta_{S}\not\in \ci_{S}^{(\alpha)}} \leq \alpha.$$
Let $\hat S$ be an $\eta$-stable selection. Then,
$$\PP{\beta_{\hat S}\not\in \ci_{\hat S}^{(\alpha e^{-\eta})}} \leq \alpha.$$
\end{theorem}

Theorem \ref{thm:informal} is valid simultaneously across \emph{all} possible selection methods which are $\eta$-stable. In other words, under the computational notion of stability we consider, the stability parameter of a selection method alone is sufficient to correct for selective inferences.

Our stability designs are based on explicit randomization schemes which calibrate the level of randomization to a pre-specified algorithmic stability requirement. Together with Theorem~\ref{thm:informal}, this allows the statistician to \emph{choose} the confidence interval width, obtaining a perturbation of a selection algorithm (e.g., the LASSO), to obtain a target interval width and a guarantee of valid coverage. Since the derived perturbation is an explicit function of the target interval width, this provides a way to understand the loss in utility due to randomization; for example, expressing how ``far'' the perturbed LASSO solution is from the standard, non-randomized LASSO solution, in some appropriate sense. With this methodology in hand, one can explicitly analyze the inherent tradeoff between the post-selection correction and loss in utility due to randomization for any stable procedure.

We note that the use of randomization in selective inference is by no means a new idea~\citep[see, e.g.,][]{tian2018selective, tian2016magic, tian2016selective, kivaranovic2020tight, panigrahi2017mcmc, panigrahi2019approximate, bi2020inferactive}. The main difference between our work and previous work is the use of stability as an analysis tool, which, on the one hand, leads to a computationally efficient, sampling-free approach to constructing selective confidence intervals with strict coverage, and on the other hand, explicitly connects the level of randomization to the resulting interval width. 
We elaborate on the comparisons to related work in Section \ref{sec:related_work}.

\paragraph{Organization.}
In the following section we present two motivating vignettes together with our solutions based on stability. In Section~\ref{sec:related_work} we discuss related work. In Section \ref{sec:stability} we introduce the notion of algorithmic stability at the focus of our study and in Section~\ref{sec:cis_via_stability} we give theory for statistical inference under this definition. Then, in Section \ref{sec:linear_reg} we instantiate our theory in the context of model selection in linear regression. In Section \ref{sec:conditional} we draw connections to conditional post-selection inference. In Section \ref{sec:algorithms} we discuss the design of stable algorithms and give stable versions of the LASSO and marginal screening. In Section~\ref{sec:experiments} we study the performance of our procedures empirically. We end with a brief discussion in Section \ref{sec:discussion}.

\section{Motivating vignettes}
\label{sec:vignettes}

To illustrate our framework, we present two motivating examples together with solutions implied by our theory, deferring the proofs of validity of the solutions to later sections. 
In addition, we compare our correction to some relevant baselines.

\paragraph{Vignette 1: Winner's curse.}
 The first vignette considers the problem of selecting the largest observed effect. Suppose that we observe an $n$-dimensional vector $y\sim \mathcal{N}(\mu,\sigma^2 I)$; for example, each entry in this vector could be an observed treatment effect for a separate treatment.
  We are interested in doing inference on the most significant effect. More formally, denoting $i_* = \argmax_i y_i$, we want to construct a confidence interval for $\mu_{i_*}$. Note that this is a \emph{random} inferential target because $i_*$ is a function of the data.

One simple way of providing valid inference for $\mu_{i_*}$ is to apply the Bonferroni correction:
$$\PP{\mu_{i_*} \in (y_{i_*} \pm z_{1-\alpha/(2n)}\sigma) }\geq 1-\alpha,$$
where $z_q$ is the $q$ quantile of the standard normal distribution.

Benjamini et al.~\cite{benjamini2019confidence} show that a tighter correction is valid, namely
$$\PP{\mu_{i_*} \in (y_{i_*} \pm z_{1-\alpha/(n+1)}\sigma) }\geq 1-\alpha.$$
We show that, if we \emph{randomize} the selection step, rather than select $i_*$ exactly, the intervals can be made even tighter. Furthermore, the reduction in interval width is directly related to the amount of randomization.

\begin{claim}
\label{claim:max_effect}
Suppose that we select $\hat i_* = \argmax_i (y_i + \xi_i)$, where $\xi_i\stackrel{\mathrm{i.i.d.}}{\sim} \mathrm{Lap}\left(\frac{2z_{1-\alpha\delta/(2n)}}{\eta}\right)$, for user-chosen parameters $\eta>0, \delta\in(0,1)$. Then, 
$$\PP{\mu_{\hat i_*} \in (y_{\hat i_*} \pm z_{1-\alpha(1-\delta) e^{-\eta}/2}\sigma) }\geq 1-\alpha.$$
\end{claim}
The proof of validity of this construction relies on our notion of stability, introduced in later sections; we defer the analysis of Claim \ref{claim:max_effect} to the Supplement. Note that, as $\eta$ and $\delta$ decrease toward zero, the noise level increases and the intervals approach classical, non-selective intervals. In general, $\eta$ is the key parameter that trades off the information used for selection versus inference: small $\eta$ corresponds to using more information for inference, while large $\eta$ corresponds to prioritizing selection quality. Figure~\ref{fig:winners_curse} illustrates how the interval width changes with $\eta$, in comparison to baselines, for $\sigma =1,\delta = 0.5$, and varying $n$.
 
\begin{figure}[t]
\centerline{
\includegraphics[width=0.33\textwidth]{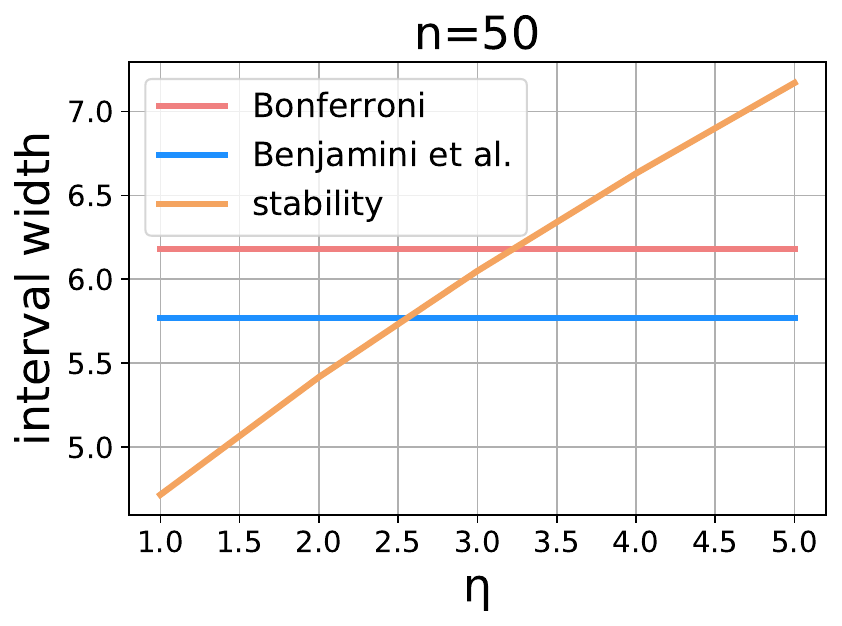}\includegraphics[width=0.33\textwidth]{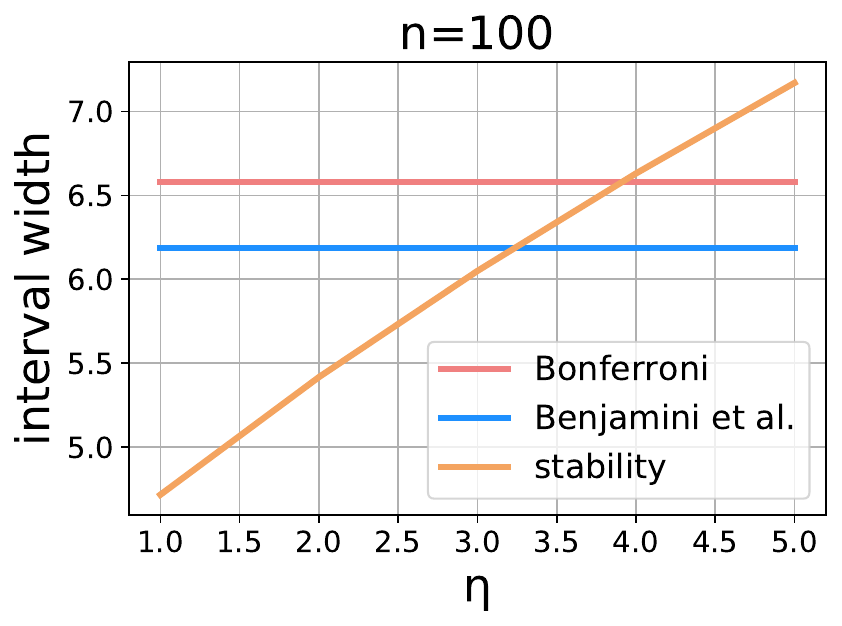}\includegraphics[width=0.33\textwidth]{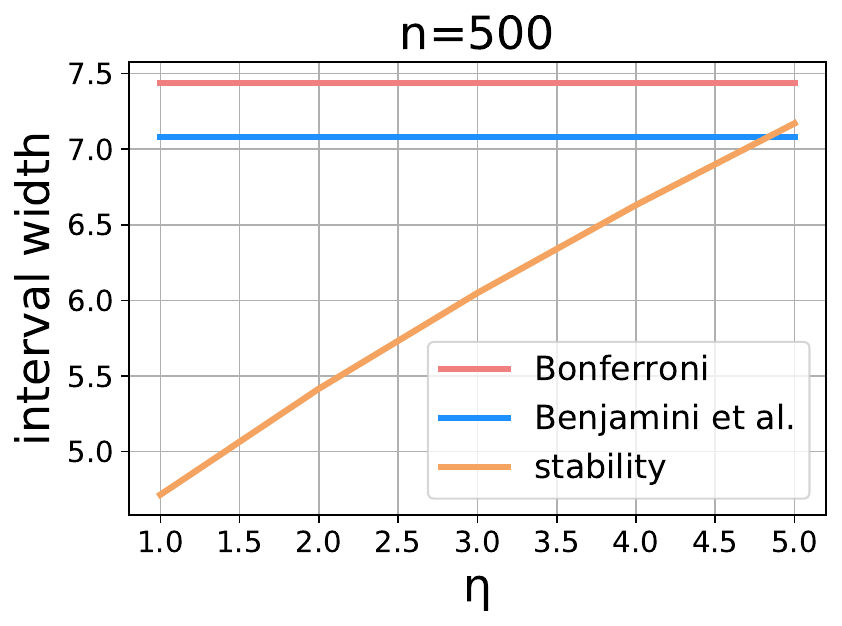}}
\caption{Confidence interval width around the ``winning'' effect, computed via the Bonferroni correction, Benjamini et al.~\cite{benjamini2019confidence} correction, and our stability-based approach. From left to right, we increase the value of $n\in\{50,100,500\}$ and keep $\sigma = 1$ fixed.} 
\label{fig:winners_curse}
\end{figure}

Note that if there is significant separation between $\mu_{i_*}$ and the other effects, $\hat i_*$ will likely be equal to $i_*$ even for small $\eta$. If, on the other hand, there are multiple effects of similar magnitude, the randomization will smooth out the selection and place non-negligible probability on all the competitive effects. In particular, to ensure $\hat i_* = i_*$ with high probability, it suffices to have $\eta$ inversely proportional to the gap between the largest and second largest observed effect, $\Delta = y_{i_*} - \max_{j\neq i_*} y_j$.

\begin{claim}
\label{claim:vignette_utility1}
If $\eta \geq \frac{4\log(n/(2\delta')) z_{1-\alpha\delta/(2n)}}{\Delta}$ for some $\delta'\in(0,1)$, then $\hat i_* = i_*$ with probability at least $1-\delta'$ over the randomness in the selection.
\end{claim}

As a result, for large enough $\Delta$, the approach of Claim \ref{claim:max_effect} allows selecting $i_*$ with high probability while providing a tighter correction than the baselines.

\paragraph{Vignette 2: Feature selection.}
In the second example we look at inference after data-driven feature selection. Suppose we have a fixed design matrix, $X\in \R^{n\times d}$, with $n$ observations and $d$ features and a corresponding outcome vector $y\sim \mathcal{N}(\mu,\sigma^2 I)\in \R^n$. Denote by $X_i$ the columns of $X$, for $i\in [d]$. We would like to select a \emph{model} corresponding to a subset of the $d$ features, and perform valid inference on the least-squares target after regressing $y$ on the \emph{selected} features only. This problem is discussed in depth by Berk et al.~\cite{berk2013valid}.

We set this problem up more generally in later sections; to keep this illustration light, assume that the features are normalized so that $\|X_i\|_2=1$ and we are selecting a single feature. Then, this problem amounts to doing inference on $X_{i_*}^\top \mu$, where $i_*$ is the selected feature. Again, we note that this is a random inferential target since $i_*$ is data-dependent.

Suppose that the goal of selection is to simply maximize the absolute correlation of the selected feature with $y$: $i_* = \argmax_i |X_i^\top y|$. Then, our results imply the following.

\begin{claim}
\label{claim:max_corr}
Suppose that we select $\hat i_* = \argmax_i |X_i^\top y + \xi_i|$, where $\xi_i\stackrel{\mathrm{i.i.d.}}{\sim} \mathrm{Lap}\left(\frac{2z_{1-\alpha\delta/(2d)}}{\eta}\right)$, for  user-chosen parameters $\eta>0, \delta\in(0,1)$. Then, 
$$\PP{X_{\hat i_*}^\top \mu \in (X_{\hat i_*}^\top y \pm z_{1-\alpha(1-\delta) e^{-\eta}/2}\sigma) }\geq 1-\alpha.$$
\end{claim}

Again, as $\eta$ and $\delta$ tend toward zero, the intervals approach non-selective intervals, and the relationship between $\eta$ and the gap between the largest and second largest correlation, $\Delta = |X_{i_*}^\top y| - \max_{j\neq i_*}|X_j^\top y|$, drives the accuracy of selection.

\begin{claim}
\label{claim:vignette_utility2}
If $\eta \geq \frac{4\log(d/(2\delta')) z_{1-\alpha\delta/(2d)}}{\Delta}$ for some $\delta'\in(0,1)$, then $\hat i_* = i_*$ with probability at least $1-\delta'$ over the randomness in the selection.	
\end{claim}

\begin{figure}[t]
\centerline{
\includegraphics[width=0.33\textwidth]{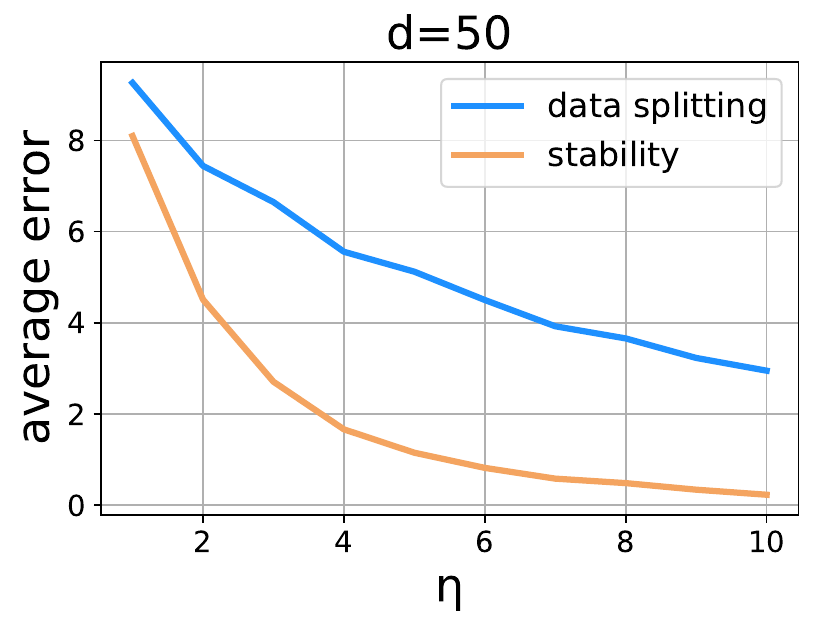}\includegraphics[width=0.33\textwidth]{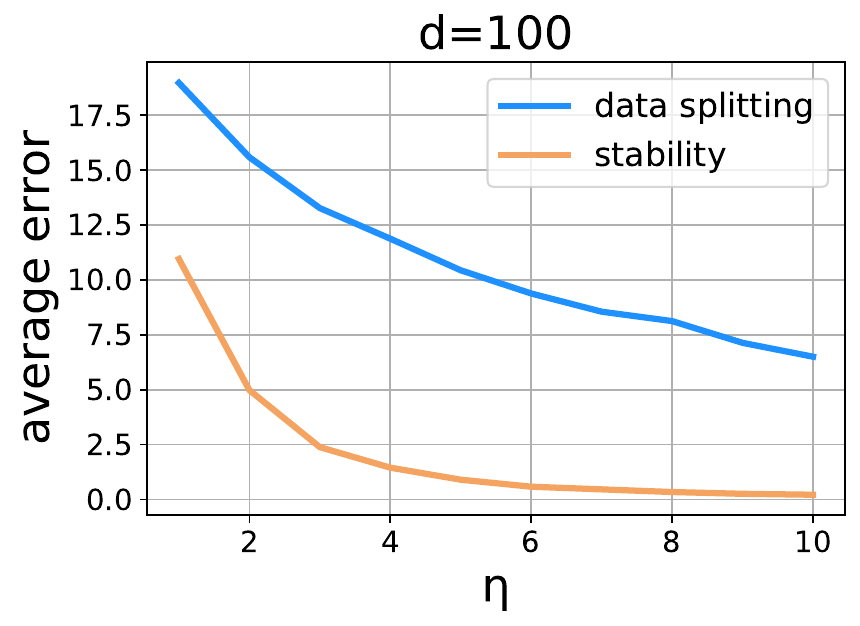}\includegraphics[width=0.33\textwidth]{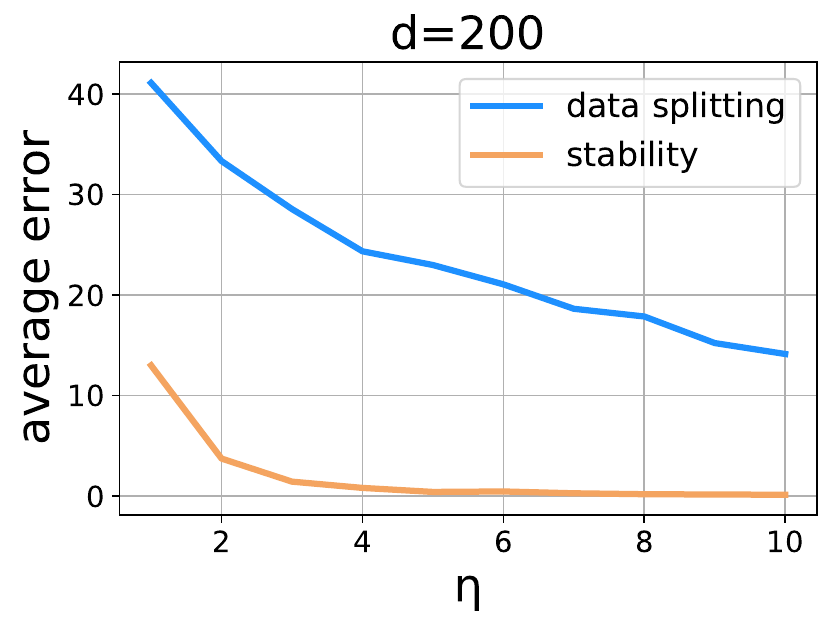}}
\caption{Average error, defined as $\max_i |X_{i}^\top y| - |X_{\hat i_*}^\top y|$, achieved by data splitting and our stability-based approach. We use one of the experimental setups in Section \ref{sec:experiments}. The rows of $X$ are drawn i.i.d.\ from an equicorrelated multivariate Gaussian distribution with pairwise correlation $\rho = 0.5$. The outcome is generated as $y = X\beta + \epsilon$, where $\epsilon$ is isotropic Gaussian noise and half of the entries in $\beta$ are zero and half are sampled independently from $\mathrm{Exp}(0.2)$. From left to right, we increase the value of $d\in\{50,100,200\}$ and keep $n=50$ fixed.} 
\label{fig:max_corr}
\end{figure}

An alternative, equally simple solution to inference after feature selection is data splitting: we use a fraction $f\in(0,1)$ of the data for selection and the remaining $1-f$ fraction for inference. In Section \ref{sec:data_splitting} we present a more detailed comparison with data splitting. For now we remark that, given that both data splitting and stability lead to intervals that look like classical intervals with an additional correction factor, every parameter $\eta$ has a corresponding parameter $f\equiv f(\eta)$ that yields data-splitting intervals of roughly the same width as stability-based intervals (given a fixed $\delta$ such as $0.5$). However, we observe that stability can be a significantly more powerful solution. Figure \ref{fig:max_corr} compares the error of data splitting with the error of the stability solution described in Claim \ref{claim:max_corr} in a simple simulation setting.

Finally, we emphasize that stability can be applied to the problem of feature selection even when data splitting is not an option, such as when there are spatial or temporal dependencies in the data.
 
 \section{Related work}
\label{sec:related_work}

In this section, we elaborate on the comparisons between our work and existing work in post-selection inference, and additionally discuss relevant work in the algorithmic stability literature.

\paragraph{Simultaneous coverage.}
In the formulation of post-selection inference by Berk et al.~\cite{berk2013valid}, the goal is to construct simultaneous confidence intervals (as per Eq.~(\ref{eqn:valid_ci})) that are valid \emph{for any model selection method} $\hat M : y \rightarrow \M$, for a pre-specified model class $\M$. The framework of Berk et al. was subsequently generalized by Bachoc et al.~\cite{bachoc2020uniformly} to handle distributions beyond the homoscedastic Gaussian, as initially assumed. These proposals are computationally infeasible in high dimensions as they essentially require looking for the ``worst possible'' model $M\in \M$, one that implies the largest so-called PoSI (Post-Selection Inference) constant, an analog of which we introduce and characterize in the context of stability. More recent work has proposed computationally efficient confidence regions via UPoSI~\cite{kuchibhotlavalid}.

Another approach to valid post-selection inferences that applies to general selection rules is \emph{data splitting} \cite{rinaldo2019bootstrapping}: split the data into two disjoint subsets, then use one subset to select the inferential target and the other subset to perform inference. Data splitting is appealing because, if the two subsets of the data are independent, classical inferences will be valid regardless of the selection procedure. However, data splitting is not universally applicable as one cannot always obtain two independent data sets, and even if applicable, it can suffer a significant loss in power, such as when only a few samples capture some relevant information. Our stability-based approach does not rely on any independence assumption between different observations, and, as illustrated in Figure~\ref{fig:max_corr}, it can be a more powerful solution than data splitting when the latter is applicable. We give a further discussion of data splitting and its relationship to stability in Section \ref{sec:data_splitting}.

All the aforementioned strategies strive for robustness: they protect against \emph{all} selection procedures. For specific selection procedures, however, the intervals computed by simultaneous methods and related approaches are unnecessarily wide, as they do not exploit any knowledge of how the analyst arrives at the selected target. Recent work aims to address this issue in the framework of simultaneous coverage over the selected variables (SoS) by constructing SoS-controlling confidence intervals for $k$ seemingly largest effects~\cite{benjamini2019confidence}. Our work likewise implies SoS intervals, by putting forward a general stability perspective and analyzing the relationship between stability and interval width for arbitrary stable procedures.

\paragraph{Conditional coverage.}
Conditional methods \cite{fithian2014optimal} exploit properties of the selection procedure. However, they control a different error criterion than simultaneous methods. In particular, the goal of conditional post-selection inference is to design $\ci_{S}$ such that
$$\PPst{\beta_{S}\in \ci_{S}}{\hat S = S}\geq 1-\alpha,$$
for all fixed selections $S$.

For a fixed selection procedure, conditional post-selection inference aims to characterize the distribution of the data given $\hat S = S$, and then using the knowledge of this conditional distribution it computes $\ci_{S}$. This approach is tailored to the selection method at hand, and existing work has derived intervals for model selection via methods such as the LASSO \cite{lee2016exact}, marginal screening, orthogonal matching pursuit~\cite{lee2014exact}, forward stepwise, and LARS \cite{tibshirani2016exact}.

It is often remarked that the conditional approach leads to \emph{overconditioning}, thus leading to wide intervals \cite{kivaranovic2018expected}. Informally, overconditioning refers to the phenomenon of overstating the cost of selection, thus leaving little information for inference. Surprisingly, it has even been observed that simultaneous approaches can in some cases yield smaller intervals, due to the intervals being unconditional rather than conditional \cite{bachoc2020uniformly}. One attempt at narrowing down the intervals involves choosing a better event on which to condition~\cite{liu2018more}. Another solution to overconditioning which is relevant to the present context is the idea of randomizing the selection procedure \cite{tian2018selective, tian2016magic, tian2016selective, kivaranovic2020tight, panigrahi2017mcmc, panigrahi2019approximate, bi2020inferactive}. Notably, the pioneering work in this direction due to Tian and Taylor~\cite{tian2018selective} proves a central limit theorem that asymptotically relates the validity of statistical inferences without selection to their selective counterparts, a result similar in flavor to our Theorem \ref{thm:conf-ints-general}. However, existing randomization proposals suffer several drawbacks. One is that they give little insight into the tradeoff between confidence interval width and the loss in utility from the additional noise. Another issue is that inference is based on a selective pivot which, unlike in exact conditional approaches, lacks closed-form expressions. As a result, to approximate the pivot, existing work resorts to computationally expensive sampling \cite{tian2016selective,tian2018selective}, which is generally infeasible in high dimensions. There are other, computationally-efficient approaches which aim to approximate the pivot \cite{panigrahi2017mcmc, panigrahi2019approximate}, although these are only approximate and the general theory applies to restricted classes of selection problems.

Although our primary goal is to provide unconditional guarantees, in Section \ref{sec:conditional} we will also discuss implications of stability for conditional inference.

We also point out the work of Andrews et al.~\cite{andrews2019inference}, who propose a hybrid approach that interpolates between unconditional and conditional post-selection inference to obtain smaller confidence intervals relative to a purely conditional approach.

\paragraph{Algorithmic stability.}
The technical tools of this paper are rooted in the theory of differential privacy~\cite{dwork2006calibrating, dwork2014algorithmic} and its extensions \cite{dwork2015generalization, bassily2016typical}. Initially, differential privacy was developed as a standard for private data analysis. A more recent line of work, typically referred to as \emph{adaptive data analysis}~\citep[see, e.g.,][]{dwork2015preserving, dwork2015generalization, bassily2016algorithmic}, has recognized that a stability concept can be extracted from differential privacy and exploited to obtain perturbation-based generalization guarantees in learning theory. Superficially, adaptive data analysis has the same goal as post-selection inference---developing statistical tools for valid inference when hypotheses about the data are also data-driven---but the typical formalization of this problem is not directly comparable to that of the canonical post-selection inference setup in regression. The conceptual connection between the two areas has, however, been recognized (dating back to at least the seminal work of Tian and Taylor~\cite{tian2018selective}), and several existing works discuss selective hypothesis tests and stability-based corrections to arbitrary selective $p$-values \cite{rogers2016max, russo2016controlling}. Our work does not aim to contribute to adaptive data analysis per se; rather, we build on and adapt existing tools in this literature for the purpose of providing post-selection corrections for common selection problems, such as within the framework put forward by Berk et al.~\cite{berk2013valid}. Finally, we note that connections between stability and generalization are not new \cite{bousquet2002stability}, and stability ideas have been utilized to construct predictive confidence intervals \cite{steinberger2018conditional, barber2019predictive}.

\section{Algorithmic stability and selection}
\label{sec:stability}

The formal theory of algorithmic stability characterizes how the output of an algorithm changes when the input is perturbed. 
Randomized algorithms have as output a \emph{random variable}; therefore, to study the stability of a randomized algorithm, an
appropriate notion of closeness of two random variables is required.  The particular notion of closeness considered in
differential privacy and related work is known as \emph{indistinguishability}, or \emph{max-divergence}.

\begin{definition}[Indistinguishability]
We say that a random variable $Q$ is $(\eta,\tau)$-indistinguishable from $W$, denoted $Q\approx_{\eta,\tau} W$, if for all measurable sets~$\O$, 
	$$\PP{Q\in \O}\leq e^{\eta}\PP{W\in \O} + \tau.$$
\end{definition}

Roughly speaking, $\tau$ bounds the probability of the event that $Q$ and $W$ are ``very different.'' For fixed $\tau\in[0,1]$, the parameter $\eta$ is meant to capture how similar the distributions of $Q$ and $W$ are---the larger $\eta$ is the larger the divergence between $Q$ and $W$ can be. One should think of $\tau$ as being at most a small factor proportional to the miscoverage probability $\alpha$.

We now formally introduce the main notion of algorithmic stability considered in this paper. The algorithm whose stability we analyze will usually be a selection algorithm.

\begin{definition}[Stability]
\label{def:stability}
Let $\A:\R^n\rightarrow \mathcal{S}$ be a randomized algorithm. We say that $\A$ is $(\eta,\tau,\nu)$-stable with respect to a distribution $\P$ supported on $\R^n$ if there exists a random variable $ A_0$, possibly dependent on $\P$, such that
$$\P\left\{\omega\in \R^n : \A(\omega)\approx_{\eta,\tau} A_0\right\}\geq 1 -\nu.$$
\end{definition}

This notion is a special case of \emph{typical stability}  introduced by Bassily and Freund~\cite{bassily2016typical}.  It is closely related to the notions of perfect generalization \cite{cummings2016adaptive} and max-information \cite{dwork2015generalization}. Unless stated otherwise, whenever we use the term stability we will assume stability in the sense of Definition \ref{def:stability}. The parameter $\nu$ can in principle take on any value in $[0,1]$ but in practice we will set it to be proportional to $\alpha$.

We will only invoke stability with respect to the data distribution, which we will denote $\P_y$. Thus, for simplicity, when we say that $\A$ is $(\eta,\tau,\nu)$-stable we are implicitly assuming that $\P_y$ is the reference distribution.

Throughout we will use $\hat S(\cdot)$ to denote a possibly randomized selection algorithm, which takes as input the data $y$ and outputs a selection that determines the inferential target. For example, $\hat S$ could be a model selection algorithm such as in the second vignette, or it could be an algorithm that selects an effect that is the focus of subsequent inference, such as in the first vignette. With a slight abuse of notation, we will use $\hat S$ to denote both the mapping from the data to the selection as well as the selection itself, $\hat S(y) \equiv \hat S$.

Intuitively, $\hat S$ will be stable if we can ``guess'' the distribution of $\hat S(y)$ (conditional on $y$) with only knowledge of the data distribution $\P_y$, not the data $y$ itself.

\begin{example}
To provide intuition for Definition \ref{def:stability}, we present one simple mechanism for achieving stability. Although basic, this mechanism will be a fundamental building block in our stability proofs. Suppose that we wish to compute $w^\top y$, for some fixed vector $w$, and suppose that we take $\P_y$ to be $\mathcal{N}(\mu,\sigma^2I)$ with known $\sigma>0$. Let $\A(y) = w^\top y + \xi$, where $\xi\sim \text{Lap}\left(\frac{z_{1-\nu/2}\sigma \|w\|_2}{\eta}\right)$, for user-specified parameters $\eta>0,\nu\in(0,1)$. Here, $\text{Lap}(b)$ denotes a draw from the zero-mean Laplace distribution with parameter $b$, independent of $y$. We argue that this mechanism is $(\eta,0,\nu)$-stable. First, we know
\begin{align*}
\lefteqn{\PP{|w^\top y - w^\top \mu| \geq z_{1-\nu/2}\sigma \|w\|_2}}\\
&= \PP{|\mathcal{N}(0,\sigma^2\|w\|_2^2)|\geq z_{1-\nu/2}\sigma \|w\|_2} = \nu.
\end{align*}

Denote $Y_\nu = \{\omega\in\R^n:|w^\top \omega - w^\top \mu| \leq z_{1-\nu/2}\sigma \|w\|_2\}$, and notice that we have shown that $\PP{y\in Y_\nu} = 1-\nu$.

Now let $A_0 = \A(\mu)$. Since the ratio of densities of $\xi\sim\text{Lap}(b)$ and its shifted counterpart $x + \xi$ is upper bounded by $e^{|x|/b}$, we can conclude that for all $\omega\in Y_\nu$ and measurable sets $\O$, 
$$\frac{\PP{\A(\omega)\in \O}}{\PP{\A(\mu)\in \O}} \leq e^\eta;$$
that is, we have $\A(\omega)\approx_{\eta,0} A_0$ for all $\omega \in Y_\nu$. Putting everything together, we see that $\A(\cdot)$ is $(\eta,0,\nu)$-stable with respect to $\P_y$.
\end{example}

\section{Confidence intervals after stable selection}
\label{sec:cis_via_stability}

Given the assumption of $(\eta, \tau, \nu)$-stability, we now show how a simple modification to classical confidence intervals suffices to correct for selective inferences. This correction is valid \emph{regardless} of any additional property of the selection criterion.

The main intuition behind this assertion is the following. If the selection algorithm is stable, then by Definition \ref{def:stability} there must exist an oracle selection $\hat S_0$ such that $\hat S(y)$ and $\hat S_0$ are indistinguishable. Note that $\hat S_0$ is independent of $y$. Since $\hat S_0$ is indistinguishable from $\hat S(y)$, we can pretend that $\hat S_0$ is the selection of interest. Furthermore, since $\hat S_0$ and $y$ are independent, we are \emph{free to use $y$ for inference}. Stability ensures that, despite data reuse, inference behaves almost like with data splitting, in which we perform selection on one batch of data and then use independent data for constructing intervals.

We state a technical lemma, similar to Lemma 3.3 by Bassily and Freund~\cite{bassily2016typical}, that we use to prove our main theorem. We include a proof of Lemma \ref{lemma:near-indep} in the Supplement.

\begin{lemma}
\label{lemma:near-indep}
	Let $\hat S:\R^n\rightarrow \mathcal{S}$ be an $(\eta,\tau,\nu)$-stable selection algorithm and let $\hat S_0$ be the corresponding oracle selection. Then, it holds that
	\begin{equation}
	\label{eqn:max-info}
(y,\hat S(y))\approx_{\eta,\tau + \nu}(y, \hat S_0).
	\end{equation}
\end{lemma}

%

Equipped with Lemma \ref{lemma:near-indep}, we can now describe how to construct post-selection-valid confidence intervals after stable selection. 

Suppose that, under selection $S$, our target of inference is $\beta_S$. Moreover, suppose that $\ci_{S}^{(\alpha)}$ are valid confidence intervals at level $1-\alpha$ for any \emph{fixed} $S$, meaning that
$$\PP{\beta_{S}\not\in \ci_{S}^{(\alpha)}} \leq \alpha.$$
Such intervals are provided by classical theory.

Theorem \ref{thm:conf-ints-general} formally states how to construct confidence intervals for an \emph{adaptive} target $\beta_{\hat S}$, when $\hat S$ is selected in a stable way. This is the key result of our paper.

\begin{theorem}
\label{thm:conf-ints-general}
Fix $\delta\in(0,1)$, and let $\hat S$ be an $(\eta,\tau,\nu)$-stable selection algorithm.
Then,
$$\PP{\beta_{\hat S}\not\in \ci_{\hat S}^{(\delta e^{-\eta})}} \leq \delta + \tau + \nu.$$
\end{theorem}

In words, if $\hat S$ is $(\eta,\tau,\nu)$-stable, we can pretend that there is no selection bias and simply construct classical intervals, albeit at a more conservative level, to achieve validity. If we set the target error level to be $\delta e^{-\eta}$, then the realized error level will be at most $\delta + \tau + \nu$. For example, if we let $\tau = \nu = \alpha/3$, then to get coverage at level $1-\alpha$ we can set the target coverage level to be $\alpha/3\cdot e^{-\eta}$.

\subsection{Comparison with data splitting}
\label{sec:data_splitting}

In many scenarios it is possible to split the data into two independent chunks, one to be used for selection and the other to be reserved for inference. Classical inferences are then valid because the inferential target is determined before seeing any of the data used in the inference step. This simple baseline for valid inference after selection is called \emph{data splitting}. In this section, we illuminate the relationship between our approach via stability and data splitting.

First we want to emphasize that the stability principle is applicable even with dependent samples: Theorem \ref{thm:conf-ints-general} can be applied even when it is not clear how to create two independent subsets of the data. Moreover, in some selection problems data splitting makes little conceptual sense, such as in our first motivating vignette about inference on the winning effect.

The appeal of data splitting lies in its broad applicability. As long as the data can be split into two independent components, the criteria for choosing the inferential target can be arbitrary. Therefore, data splitting provides a \emph{selection-agnostic} correction, universally valid across all possible selection strategies.

Conceptually, stability lies somewhere between data splitting and conditional post-selection inference. It computes a correction level as a function of how adaptive the selection is to the data, thereby adapting to some properties of the selection rule like conditional inference methods. However, at the same time it provides a correction that is universally valid across all possible selection strategies \emph{with the same level of stability}, which can be seen as a refinement of the principle of data splitting.

To illustrate the conceptual difference between the stability principle and the data splitting principle, suppose that in the latter case we allocate $f$-fraction of the data to selection, and $(1-f)$-fraction to inference. Then, the resulting intervals will roughly look like classical intervals augmented by a factor of $\sqrt{\frac{1}{1-f}}$ \emph{regardless} of how the selection is performed.

In contrast, the stability approach augments classical intervals as a function of the adaptivity of the selection algorithm. Suppose for concreteness that $y\sim \mathcal{N}(\mu, I)$ and we are considering doing inference on one of two targets, $v_0^\top \mu$ or $v_1^\top \mu$, where the selection $\hat S \in \{0,1\}$ depends on the data $y$. Consider three different selection methods:
\begin{itemize}
\item $\hat S = 1$ no matter what the data vector is.
\item $\hat S=1$ if $\bar y := \frac{1}{n} \sum_{i=1}^n y_i \geq 0$, and $\hat S=0$ otherwise.
\item $\hat S = 1$ if $X_1^\top y \geq 0$ for some unit vector $X_1$, and $\hat S=0$ otherwise.
\end{itemize}
We can write all three procedures as $\hat S = \mathbf{1}\{w^\top y \geq 0\}$; in the first case $w=0$, in the second case $w = \frac{1}{n}\mathbf{1}$, and in the third case $w = X_1$.

Let us fix the noise level $\gamma>0$ and select $\hat S = \mathbf{1}\{w^\top y + \xi \geq 0\}$, where $\xi\sim \text{Lap}(\gamma)$.
The first method is trivially $(0,0,0)$-stable for any level $\gamma$, hence we can simply use $y$ for inference without any correction. Based on the same analysis as in the example in Section \ref{sec:stability}, the second selection method is $(\sqrt{2\log(2/\nu)}/(\gamma\sqrt{n}),0,\nu)$-stable for all $\nu>~0$; i.e., it is $(\sqrt{2\log(4/\alpha)}/(\gamma\sqrt{n}),0,\alpha/2)$-stable.
Similarly, the third selection method is $(\sqrt{2\log(4/\alpha)}/\gamma,0,\alpha/2)$-stable.

We can thus observe that, even though in all three examples we perturb the selection by the same constant level of noise, the stability approach exploits the fact that some selection criteria are more stable than others and this is reflected in the resulting stability parameter. By Theorem \ref{thm:conf-ints-general}, this stability parameter, in turn, directly determines the correction factor, i.e., how conservative we need to make classical inferences for them to be valid post selection.

While data splitting and stability come with conceptual differences, they also have technical similarities. In particular, each one has a leading parameter---$f\in(0,1)$ in the case of data splitting and $\eta>0$ in the case of stability---and this parameter interpolates between two extremes. One extreme is when all information is reserved for inference (attained when $f=0$ and $\eta=0$ respectively) and the other is when all information is used for selection (attained when $f=1$ and $\eta\rightarrow\infty$ respectively). Therefore, it might make sense to ask how the two interpolations relate.

For every $\eta$, there is an $f(\eta)$ such that, if we used $f(\eta)$-fraction of the data for selection and $1-f(\eta)$ for inference, we would approximately get the same interval correction. We sketch the derivation of $f(\eta)$ in the case of normal intervals for simplicity, however this calculation can be generalized to other distributions.
We will assume that $\nu + \tau \leq \delta \alpha$ for some $\delta\in(0,1)$; then,
 the intervals resulting from $(\eta,\tau,\nu)$-stability are of width proportional to
$z_{1-(1-\delta)\frac{\alpha}{2} e^{-\eta}}.$
The intervals resulting from data splitting are of width proportional to
$z_{1-\frac{\alpha}{2}}(1-f(\eta))^{-1/2}.$
By equating the two expressions to achieve the same width and simplifying, we obtain
\begin{equation}
\label{eqn:f_eta}
f(\eta) = 1 - \left(\frac{z_{1-\frac{\alpha}{2}}}{z_{1-(1-\delta)\frac{\alpha}{2} e^{-\eta}}}\right)^2 \approx \frac{\log \frac{1}{1-\delta} + \eta}{\log\frac{2}{(1-\delta)\alpha} + \eta},
\end{equation}
where the approximation on the right-hand side follows by a subgaussian approximation.

Of course, this sketch only gives intuition for when data splitting and stability imply equally powerful inference; it does not say anything about which selection is more accurate---one where we select on $f(\eta)$-fraction of the data, or one where we select on the whole data set in an $\eta$-stable way. We will tackle this question empirically, as the notion of ``more accurate'' varies greatly depending on the context. In Figure~\ref{fig:max_corr} we used the splitting fraction in Eq.~\eqref{eqn:f_eta} and observed that stability outperforms data splitting. We provide further empirical comparisons in Section \ref{sec:experiments}.

Finally, we mention another proposal that is conceptually closely related to data splitting, namely the $(U,V)$ decomposition of Rasines and Young~\cite{rasines2021splitting}. Like stability, the $(U,V)$ decomposition allows the statistician to see all data points---more precisely, noisy versions thereof---both in the selection step and in the inference step. This is an important advantage over data splitting when there are only a few samples that capture information about certain directions. In contrast with stability, performing the $(U,V)$ decomposition does not rely on any properties of the selection method. However, finite-sample guarantees of this approach crucially rely on the data being Gaussian with known covariance, while the stability principle is applicable beyond Gaussianity and is robust to only having an estimate of the covariance.

\section{Model selection in linear regression}
\label{sec:linear_reg}

In this section, we discuss an application of our stability tools to the problem of model selection in linear regression. We focus on the framework presented in the seminal work of Berk et al.~\cite{berk2013valid}. We begin by reviewing the model and introduce the necessary notation.

Let $X\in\R^{n\times d}$ denote a fixed design matrix, and let $X_i\in\R^n$ denote the $i$-th column of $X$, for $i\in[d]$. We refer to vectors $X_i$ as variables or features. For a subset $M\subseteq[d]$, we denote by $X_M\in \R^{n\times |M|}$ the submatrix of $X$ given by selecting the columns indexed by $M$. We make no assumptions about how $n$ and $d$ relate; in particular, we could have $d\gg n$.

By $y\in\R^n$ we denote the random vector of outcomes corresponding to $X$. Importantly, we do not assume knowledge of a true data-generating process; for example, we do not assume that $\mu:=\E[y]$ can be expressed as a linear combination of $\{X_i\}_{i=1}^d$. The vector $\mu\in \R^n$ is unconstrained and need not reside in the column space of $X$. Rather, different subsets of $\{X_i\}_{i=1}^d$ provide different approximations to $\mu$, some better than others.

The statistician wishes to let the data decide how the initial pool of features should be reduced to a smaller set of seemingly relevant features, and then run linear regression on this smaller set. That is, the statistician chooses a set $\hat M\subseteq [d]$ by running a model selection method on $X,y$, and then aims to approximate $y\approx X_{\hat M}\hat \beta_{\hat M}$, for some $\hat \beta_{\hat M}$. As before, we will employ a conventional abuse of notation by letting $\hat M \equiv \hat M(y)$. 

Assuming $X_{\hat M}$ has full column rank almost surely, the unique least-squares estimate in model $\hat M$ is given by
$$\hat \beta_{\hat M} := \argmin_{\beta\in\R^{|\hat M|}} \|y - X_{\hat M}\beta\|_2^2 = (X_{\hat M}^\top X_{\hat M})^{-1}X_{\hat M}^\top y := X_{\hat M}^+ y,$$
where we define $X_{\hat M}^+:=(X_{\hat M}^\top X_{\hat M})^{-1}X_{\hat M}^\top$ to be the pseudoinverse of $X_{\hat M}$. For a \emph{fixed} model $M$, the target estimand of $\hat\beta_M$ is
$$\beta_M := \argmin_{\beta\in\R^{|M|}} \E\left[\|y - X_M\beta\|_2^2\right] = X_M^+\mu,$$
and hence for a random model $\hat M$, this implies a \emph{random} target $\beta_{\hat M} = X_{\hat M}^+\mu$.

We denote by $\beta_{j\cdot M}$ the entry of $\beta_M$ corresponding to feature $X_j$, for all $j\in M$. Note that $\beta_{j\cdot M}$ is not defined for $j\not\in M$. We adopt similar notation for the entries of $\hat \beta_M$.

Our goal is to construct simultaneous confidence intervals for the target of inference $\beta_{\hat M}$. More precisely, we wish to design $\ci_{j\cdot \hat M}^{(\alpha)}$ such that
\begin{equation}
\label{eqn:valid_ci}
\PP{\beta_{j\cdot \hat M} \in \ci_{j\cdot \hat M}^{(\alpha)},~ \forall j\in \hat M} \geq 1 - \alpha,
\end{equation}
for a fixed $\alpha\in(0,1)$ and a \emph{fixed} selection procedure $\hat M$.
Note that the work of Berk et al.\ and various extensions \cite{berk2013valid, bachoc2020uniformly, kuchibhotlavalid} provide simultaneity \emph{both} over the selected variables \emph{and} over all selection methods, while we keep the selection method fixed. Our guarantees are \emph{simultaneous over the selected}~\citep[cf.][]{benjamini2019confidence}.

The confidence intervals resulting from our approach take the usual form,
$$\ci_{j\cdot \hat M}(K) := \left(\hat \beta_{j\cdot \hat M} \pm K \hat\sigma_{j\cdot \hat M}\right),$$
where $\hat\sigma_{j\cdot \hat M}^2$ is an estimator of variance for the OLS estimate $\hat \beta_{j\cdot \hat M}$; e.g., the ``sandwich'' variance estimator~\cite{buja2019models}. Our goal is to find a suitable value of $K$ such that $\ci_{j\cdot\hat M}(K)$ are valid $(1-\alpha)$-confidence intervals, as per Eq.~\eqref{eqn:valid_ci}. 
By analogy with Berk et al.~\cite{berk2013valid}, we refer to the minimal such valid $K$ as the \emph{PoSI constant}. It is important to remember that, unlike in Berk et al., our PoSI constant depends on the selection procedure, rather than a family of all possible models.

The PoSI constant is well characterized when the model is fixed rather than determined in a data-driven fashion. For a fixed model $M$ and given $\alpha\in(0,1)$, we define $K_{M,\alpha}$ to be the minimum value of $K$ such that
$$\PP{\max_{j\in M} \left|\frac{\hat\beta_{j\cdot M} - \beta_{j\cdot M}}{\hat \sigma_{j\cdot M}}\right| \geq K}\leq \alpha.$$
In other words, $K_{M,\alpha}$ defines the PoSI constant when the model $M$ is specified up front and does not depend on the data; in this case, $\ci_{j\cdot M}(K_{M,\alpha})$ are valid simultaneous intervals at level $1-\alpha$. 
For example, when $y\sim \mathcal{N}(\mu,\sigma^2I)$, one simple way of providing a valid upper bound on $K_{M,\alpha}$ is via standard z-scores or t-scores, after doing a Bonferroni correction over $j\in M$. Sharper estimates of $K_{M,\alpha}$ can be obtained by exploiting the correlations between the regression coefficients to estimate the maximum z-score or t-score. Even in a distribution-free setting, it is common to determine $K_{M,\alpha}$ via normal approximation~\cite{rinaldo2019bootstrapping, kuchibhotla2020berry}.


We are now ready to state a corollary of Theorem \ref{thm:conf-ints-general} that focuses on the problem of model selection in linear regression.

\begin{corollary}
\label{corollary:conf-ints-linreg}
Fix $\delta\in(0,1)$. Let $\hat M$ be an $(\eta,\tau,\nu)$-stable model selection algorithm. For all $j\in \hat M$, let:
$$\ci_{j\cdot \hat M}(K_{\hat M,\delta e^{-\eta}}) = \left(\hat \beta_{j\cdot \hat M} \pm K_{\hat M,\delta e^{-\eta}} \hat\sigma_{j\cdot \hat M}\right).$$
Then,
$$\PP{\exists j \in \hat M : \beta_{j\cdot \hat M}\not\in \ci_{j\cdot \hat M}\left(K_{\hat M,\delta e^{-\eta}}\right)} \leq \delta + \tau + \nu.$$
\end{corollary}

To provide further intuition, we instantiate Corollary \ref{corollary:conf-ints-linreg} in the canonical setting of Gaussian observations. Let $y\sim \mathcal{N}(\mu,\sigma^2 I)$. If $\sigma>0$ is known, we let $\hat \sigma_{j\cdot M} = \sigma\sqrt{((X_M^\top X_M)^{-1})_{jj}}$; otherwise, we assume we have access to an estimate of $\sigma$, denoted $\hat\sigma$, and let $\hat \sigma_{j\cdot M} = \hat \sigma\sqrt{((X_M^\top X_M)^{-1})_{jj}}$. Following the treatment of Berk et al.~\cite{berk2013valid}, we assume that $\hat \sigma^2 \sim \sigma^2 \frac{\chi_r^2}{r}$ for $r$ degrees of freedom and assume that $\hat \sigma^2 \perp \hat \beta_{j\cdot M}$ for all possible OLS estimates $\hat \beta_{j\cdot M}$. If the full model is assumed to be correct, that is $y\sim \mathcal{N}(X\beta,\sigma^2I)$, and $n>d$, then this assumption is satisfied for $r=n-d$ by setting $\hat\sigma^2 = \|y - X\hat\beta\|_2^2/(n-d)$, where $\hat \beta$ is the OLS estimate in the full model. Even if the full model is not correct, there exist other ways of producing such a valid estimate of $\sigma$; we refer the reader to Berk et al.~\cite{berk2013valid} for further discussion.

We denote by $z_{1-\alpha}$ the $1-\alpha$ quantile of the standard normal distribution, and by $t_{r,1-\alpha}$ the $1-\alpha$ quantile of the $t$-distribution with $r$ degrees of freedom.

\begin{corollary}
\label{corollary:gaussians}
	Fix $\delta\in(0,1)$, and suppose $y\sim \mathcal{N}(\mu,\sigma^2I)$. Further, let $\hat M$ be an $(\eta,\tau,\nu)$-stable model selection algorithm. If $\sigma$ is known, let:
$$\ci_{j\cdot \hat M} = \left(\hat \beta_{j\cdot \hat M} \pm z_{1-\delta/(2|\hat M| e^\eta)}\sigma\sqrt{((X_{\hat M}^\top X_{\hat M})^{-1})_{jj}}\right).$$
If, on the other hand, $\sigma$ is not known but there exists an estimate, $\hat\sigma^2\sim\sigma^2 \frac{\chi_r^2}{r}$, independent of the OLS estimates, let:
$$\ci_{j\cdot \hat M} = \left(\hat \beta_{j\cdot \hat M} \pm t_{r,1-\delta/(2|\hat M| e^\eta)}\hat \sigma\sqrt{((X_{\hat M}^\top X_{\hat M})^{-1})_{jj}}\right).$$
In either case, we have
$$\PP{\exists j \in \hat M : \beta_{j\cdot \hat M}\not\in \ci_{j\cdot \hat M}} \leq \delta + \tau + \nu.$$
\end{corollary}

The proof follows by a direct application of Corollary \ref{corollary:conf-ints-linreg}, together with a Bonferroni correction over $j\in \hat M$ when computing $K_{\hat M,\delta e^{-\eta}}$. 
Approximating Gaussian quantiles by subgaussian concentration, we observe that the PoSI constant in Corollary \ref{corollary:gaussians} scales roughly as $\sqrt{2\left(\log(2|\hat M|/\delta) + \eta\right)}$ (when $\sigma$ is known, or as $r\rightarrow\infty$ when $\sigma$ is estimated from data).

\paragraph{Recovering the Scheff\'e rate.}
Our main technical step in deriving selective confidence intervals is Lemma \ref{lemma:near-indep}, which argues that the joint distribution of $(y,\hat S)$ cannot be too different from the joint distribution of $(y,\hat S_0)$, where $\hat S_0$ is the oracle from the definition of stability, in the indistinguishability metric. In the context of model selection in linear regression, we verify that the confidence intervals resulting from this approach are not vacuously wide in the two most extreme settings: the first, in which the model selection is independent of the data, and the second, in which the model selection is arbitrarily complex and dependent on the data.

Suppose that $\hat M$ is independent of $y$. Then, the distribution of $\hat M(y)$, conditional on $y$, is equal to the distribution of $\hat M(\omega)$ for any point $\omega$, hence $\hat M(\omega)$ is an oracle which trivially implies $(0,0,0)$-stability. In this case, the intervals in Corollary \ref{corollary:conf-ints-linreg} reduce to $\ci_{j\cdot \hat M}(K_{\hat M,\delta})$ and are valid at level $1-\delta$, as expected.

Now suppose that $\hat M$ is allowed to have arbitrary dependence on $y$; in particular, it can attain the ``significant triviality bound'' of Berk et al.~\cite{berk2013valid}. While arguing stability in the sense of Definition~\ref{def:stability} would require additional assumptions, the only property of stability used to prove Theorem \ref{thm:conf-ints-general}---the indistinguishability bound in Eq.~\eqref{eqn:max-info}---can be obtained. This allows for the proof of Theorem \ref{thm:conf-ints-general} to go through, thus recovering the tight rate of existing analyses.

\begin{proposition}
\label{prop:posi-rate}
Let $\hat M$ be an arbitrary, possibly randomized model selection procedure, such that $|\hat M|\leq s$ almost surely. Then, for any $\P_y$, there exists an oracle selection $\hat M_0$ such that for any $\tau\in(0,1)$,
$$(y,\hat M(y))\approx_{\eta,\tau}(y,\hat M_0), \text{ for some } \eta = O(s\log(d/s)) + \log(1/\tau).$$
Consequently, there exists a value $\eta = O(s\log(d/s)) + \log(1/\tau)$ such that the intervals $\ci_{j\cdot \hat M}(K_{\hat M,\delta e^{-\eta}}) = \left(\hat \beta_{j\cdot \hat M} \pm K_{\hat M,\delta e^{-\eta}} \hat\sigma_{j\cdot \hat M}\right)$ satisfy
$$\PP{\exists j \in \hat M : \beta_{j\cdot \hat M}\not\in \ci_{j\cdot \hat M}\left(K_{\hat M,\delta e^{-\eta}}\right)} \leq \delta + \tau.$$
\end{proposition}

By approximating Gaussian quantiles via subgaussian concentration, we obtain confidence intervals which are universally valid for \emph{all} $s$-sparse selections under Gaussian outcomes and scale as $O(\sqrt{\eta}) = O(\sqrt{s\log(d/s))})$. This rate is in general tight \cite{kuchibhotla2019all}, and as $s$ approaches $d$, it matches the rate given by the Scheff\'e protection \cite{scheffe1999analysis, berk2013valid}.

\section{Conditional coverage}
\label{sec:conditional}

So far all results we have presented have been about marginal coverage. Sometimes it is desirable to provide \emph{conditional} coverage, whereby we condition on the event that a given selection was made. We discuss how stability can provide guarantees that closely resemble those of conditional post-selection inference.

We start by stating an implication of Lemma \ref{lemma:near-indep} for a specific oracle selection $\hat S_0$. In what follows, let $E$ denote the subset of $\text{supp}(\P_y)$ over which a selection $\hat S$ is indistinguishable from $\hat S_0$, i.e., if $\hat S$ is $(\eta,\tau,\nu)$-stable, we let
$$E = \{\omega: \hat S(\omega)\approx_{\eta,\tau} \hat S_0\}.$$
Note that we know $\PP{y\in E}\geq 1 -\nu$ by definition.

\begin{lemma}
\label{lemma:near-indep2}
	Suppose that $\hat S$ is $(\eta,0,\nu)$-stable with respect to oracle $\hat S_0 = \hat S(y_E')$, where $y'_E$ is a sample from $\P_y$ truncated to $E$. Then, it holds that
	\begin{equation}
	\label{eqn:conditional}	
	\PPst{y\in \O_S}{\hat S(y) = S, y\in E} \leq e^\eta \PPst{y\in \O_S}{y\in E},
	\end{equation}
	for all selections $S$ and measurable sets $\O_S$.
\end{lemma}

As suggested by Lemma \ref{lemma:near-indep2}, the main difference between conditional post-selection inference and the conditional guarantees implied by stability is that in the latter case we additionally truncate the distribution of $y$ to a high-probability set $E$. Note that on the right-hand side of Eq.~\eqref{eqn:conditional} there is no dependence on the selection event, which makes inference, despite selection, essentially as easy as classical inference.

We illustrate the conditional properties of stability with an example.

\paragraph{Example: publication bias.}
We consider an illustration of the publication bias problem, also known as the file-drawer problem \cite{rosenthal1979file, fithian2014optimal, tian2018selective}. Suppose we observe an effect $y\sim \P_y$ with $\E[y] = \mu$, $\text{supp}(\P_y)\subseteq \R$. We are interested in constructing an interval for $\mu$ only if the observed effect is deemed ``interesting'' enough, for example if $y> T$ for some threshold $T$. Denote by $\mathrm{report}(y)$ the event that we decide to report the confidence interval, given that we observe effect $y$.
  
One approach to this problem is to evaluate the distribution of the data \emph{conditional} on the selection event. For example, we could find $K$ such that $\PPst{|y - \mu| > K}{\mathrm{report}(y)}\leq \alpha$, and report $\text{CI}(K) = (y\pm K)$ on the event $\mathrm{report}(y)$. Importantly, this approach generally requires an explicit characterization of the event $\mathrm{report}(y)$.

 
 Our theory suggests a \emph{criterion-agnostic} solution based on randomizing the selection, whose validity we prove in the Supplement.
  
 \begin{claim}
 \label{claim:conditional}
Let $y\sim \mathcal{N}(\mu,\sigma^2)$. Suppose that we apply the selection criterion to $y+\xi$, where $\xi \sim \mathrm{Lap}(b)$ for some user-chosen parameter $b>0$; that is, we report the confidence interval on the event $\mathrm{report}(y+\xi)$. Then, for any user-chosen parameter $\nu\in(0,1)$, we have 
$$\PPst{\mu\not\in(y\pm z_{1-\frac{\alpha}{2}(1-\nu)e^{-\eta}}\sigma)}{\mathrm{report}(y+\xi), y\in E} \leq \alpha,$$
where
$$\eta = \frac{z_{1-\nu/2}\sigma}{b} - \frac{\sigma^2}{2b^2} + \log\left(\frac{1-\nu}{2(\Phi(z_{1-\nu/2} + \frac{\sigma}{b}) - \Phi(\frac{\sigma}{b}))}\right) $$
and $E$ is an event such that $\PP{y\in E}\geq 1-\nu$.
\end{claim}

\begin{figure}[b]
\centerline{
\includegraphics[width=0.33\textwidth]{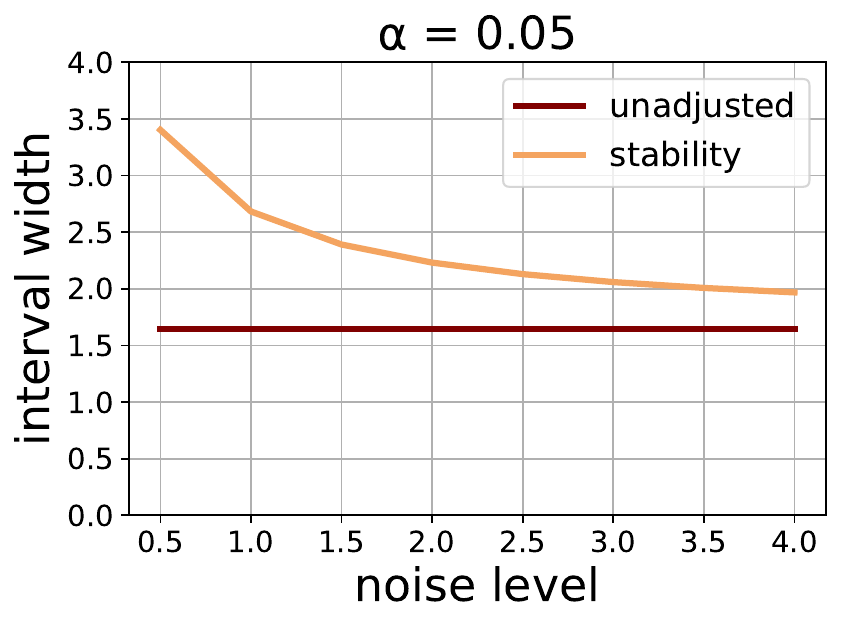}\includegraphics[width=0.33\textwidth]{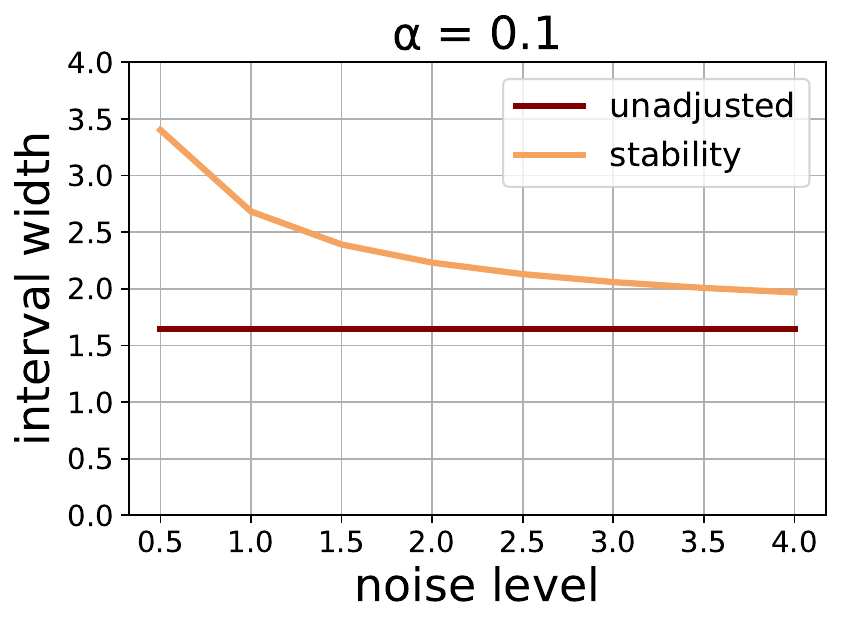}\includegraphics[width=0.33\textwidth]{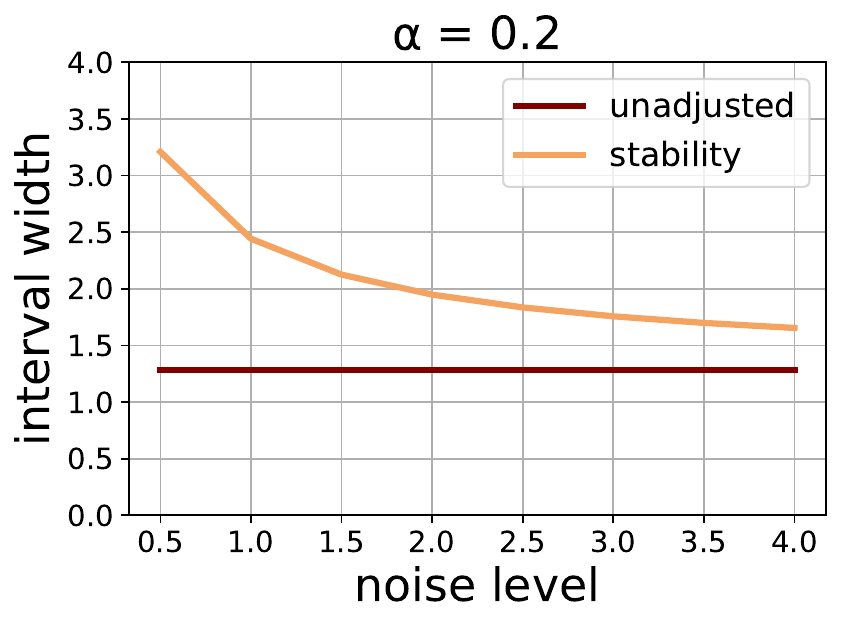}}
\caption{Normalized interval width implied by stability solution and unadjusted width for different noise levels $b$ and error levels $\alpha$. We fix $\nu = 0.05$.} 
\label{fig:pub_bias}
\end{figure}

Although in the main body of the paper we state the result when $\P_y =\mathcal{N}(\mu,\sigma^2)$ for simplicity, the Laplace noise addition strategy is valid for \emph{arbitrary} $\P_y$; only the expression for $\eta$ changes as a function of $\P_y$. In the Supplement we state a general version of Claim~\ref{claim:conditional}.

The proof of Claim \ref{claim:conditional} relies on showing that the selection to report is stable with respect to the oracle $\hat S(y'_E)$, and hence we can invoke Lemma \ref{lemma:near-indep2}.

We can see that, by choosing $\nu\rightarrow 0, b\rightarrow\infty$, we recover the non-selective confidence intervals, albeit at the cost of making the decision to report virtually independent of $y$. As we decrease $b$ (and keep $\nu$ bounded away from zero), the decision to report becomes more reflective of the event $\mathrm{report}(y)$ and the inference level smoothly becomes more stringent. In Figure \ref{fig:pub_bias} we plot the normalized interval width, $z_{1-\frac{\alpha}{2}(1-\nu)e^{-\eta}}$, together with the unadjusted normalized width $z_{1-\frac{\alpha}{2}}$, for several different noise levels $b$ and error levels $\alpha$ and $\nu = 0.05$.

As a concrete example, suppose that $\mathrm{report}(y) = \{y> T\}$, for some threshold $T$. Then, as in Claim~\ref{claim:vignette_utility1} and Claim~\ref{claim:vignette_utility2}, we can conclude that $\mathbf{1}\{\mathrm{report}(y)\} = \mathbf{1}\{\mathrm{report}(y + \xi)\}$ with probability $1-\delta$ over the choice of $\xi$ as long as $b \leq \frac{|y-T|}{\log(1/(2\delta))}$.

\section{The design of stable selection algorithms}
\label{sec:algorithms}

We discuss general tools for designing stable selection methods and present an application of these tools to variable selection in linear regression. We begin with an overview of the basic properties of stability, which are key to efficient design of stable selections.

\subsection{Properties of stability}

Stability satisfies two key algorithmic properties: \emph{closure under post-processing} and \emph{composition}. We provide precise definitions of the two shortly. The reason why these properties enable efficient stability designs is that many selection rules can be written as post-processing and composition of simple computations, such as linear functions of the data or finding maxima of a sequence. As long as we know how to stabilize the necessary simple computations, closure under post-processing and composition provide rules for computing the overall stability parameter of the whole algorithm efficiently.

\paragraph{Post-processing.}
First, stability is \emph{closed under post-processing}: if $\A:\R^n\rightarrow \mathcal{S}$ is $(\eta,\tau,\nu)$-stable, then for any (possibly randomized) map $\B:\mathcal{S}\rightarrow \G$, the composition $\B~\circ~\A$ is also $(\eta,\tau,\nu)$-stable. While the proof of this fact is a straightforward consequence of the definition of stability, the implications are significant. Suppose for the moment that the statistician is given a stable version of the LASSO algorithm, and denote its solution by $\thetalasso$. Since $\hat \theta_{\text{LASSO}}$ is stable, then so is 
$$\hat M = \{j\in [d]~:~ \hat \theta_{\text{LASSO},j} \neq 0\}.$$
In fact, the statistician need not necessarily choose the model corresponding \emph{exactly} to the support of $\thetalasso$; for example, they could choose $\hat M = \{j\in[d]: |\hat \theta_{\text{LASSO},j}| \geq \epsilon\}$, for some constant threshold $\epsilon$, or they could pick $\dsel\leq d$ entries with the maximum absolute value. More generally, any model chosen solely as a function of $\thetalasso$ inherits the same stability parameters as $\thetalasso$. And, according to Corollary \ref{corollary:conf-ints-linreg}, the same PoSI constant suffices to correct the confidence intervals resulting from any such model.

\paragraph{Composition.}
The second important property is \emph{composition}. In Algorithm \ref{alg:comp}, we define adaptive composition, after which we discuss simpler, non-adaptive composition.

\begin{algorithm}[H]
\begin{flushleft}
\SetAlgoLined
\textbf{input: } data $y\in\R^n$, sequence of algorithms $\A_t:\mathcal{S}_1\times \dots\times \mathcal{S}_{t-1}\times \R^n\rightarrow \mathcal{S}_t, ~t\in[k]$\\
\textbf{output: } $(a_1,\dots,a_k)\in \mathcal{S}_1\times \dots\times \mathcal{S}_k$\\
\For{$t=1,2,\dots,k$}{
Compute $a_t = \A_t(a_1,\dots,a_{t-1},y)\in \mathcal{S}_t$}
 Return $(a_1,\dots,a_k)$
 \end{flushleft}
\caption{Adaptive composition}
\label{alg:comp}
\end{algorithm}

Adaptive composition consists of $k$ sequential rounds in which the analyst observes the outcomes of all previous computations and selects the next computation \emph{adaptively}---as a function of the previous evaluations. The adaptive composition property bounds the stability parameters of Algorithm \ref{alg:comp} in terms of the stability parameters of $\A_t$. In its simplest form, it says that Algorithm \ref{alg:comp} is $(k\eta,0,0)$-stable if for all $t\in[k]$, $\A_t(a_1,\dots,a_{t-1},\cdot)$ is $(\eta,0,0)$-stable for all fixed $a_1,\dots,a_{t-1}$. 
For example, for some selection algorithms such as forward stepwise, it is clear to see how they can be represented using adaptive composition. In forward stepwise, $\A_t$ outputs an index $i_t\in[d]$, which corresponds to the variable $i$ that minimizes the squared error resulting from adding $i$ to the current pool of selected features; $i_t = \A_t(i_1,\dots,i_{t-1},y)$. It suffices to prove that any given step of forward stepwise selection is stable, in order to infer that the overall algorithm is stable as well.

Our proofs will only require adaptive composition for algorithms with $\nu=0$; such results follow from classical theory on differential privacy. More advanced (and naturally more conservative) adaptive composition theorems which allow $\nu> 0$ can be found in the context of typical stability~\cite{bassily2016typical}. In the Supplement we formally state the adaptive composition results we will need in our proofs.

A simpler kind of composition is non-adaptive composition. Here, the algorithms $\A_t$ have no dependence on the past computations. Non-adaptive composition can capture a protocol that involves running multiple selection methods and choosing a final selection target as an arbitrary function of all the outputs. As we state formally in the Supplement, the resulting stability parameters simply add up. This is a rather appealing property of stability, as it suggests that the statistician only needs to keep track of the stability parameters of each selection algorithm they run, in order to derive valid selective confidence intervals. An analogous combination of the results of different selection methods was considered by Markovic and Taylor~\cite{markovic2016bootstrap}; their approach, however, relies on a sophisticated Monte Carlo sampling scheme.

\subsection{Model selection algorithms: examples}

We now consider several algorithms for variable selection in linear regression through the lens of stability. While many of the principles presented in this section can be adapted to different distributional assumptions, for the sake of clarity and interpretability we assume that $y\sim \mathcal{N}(\mu,\sigma^2 I)$, where $\sigma^2$ is \emph{unknown} but we have access to an estimate $\hat \sigma^2\sim \sigma^2 \frac{\chi_r^2}{r}$, independent of $y$. This is the setup studied by Berk et al.~\cite{berk2013valid}. More generally, we only need to know the decay of the tail of the distribution of $y$ in order to enforce stability.  In the Supplement, we extend the algorithms in this section to the case of outcome vectors with a known bound on their Orlicz norm, for any Orlicz function. This includes cases such as general subgaussian and subexponential outcome vectors.


\paragraph{Model selection via the LASSO.}
We begin by considering the canonical example of the LASSO estimator \cite{tibshirani1996regression}.
The LASSO estimate is the solution to the usual least-squares problem with an additional $\ell_1$-constraint on the regression coefficients:
\begin{equation}
\label{eqn:constr-lasso}
\thetalasso \in \argmin_{\theta\in\R^d} \frac{1}{2}\|y - X\theta\|_2^2 \text{ s.t. } \|\theta\|_1 \leq C_1,
\end{equation}
where $C_1>0$ is a tuning parameter. This problem is sometimes referred to as the LASSO in constrained/bound form, to contrast it with the LASSO in penalized form:
\begin{equation}
\label{eqn:pen-lasso}
\thetalasso^\lambda \in \argmin_{\theta\in \R^d} \frac{1}{2}\|y - X\theta\|_2^2 + \lambda \|\theta\|_1,
\end{equation}
where $\lambda>0$ is now the tuning parameter. These two problems are equivalent: for any $C_1>0$, there exists a corresponding $\lambda>0$ such that $\thetalasso$ is an optimal solution for the problem in Eq.~(\ref{eqn:pen-lasso}), and vice versa. In our analysis we focus on the formulation (\ref{eqn:constr-lasso}).

The LASSO objective induces sparse solutions, and a common way of declaring that a feature is relevant is to check for a corresponding non-zero entry in the LASSO solution vector. That is, the model ``selected'' by the LASSO is:
$$\hat M = \{j\in[d]~:~ \hat \theta_{\text{LASSO},j} \neq 0\}.$$

Model selection via the LASSO has been of great interest in prior work on selective inference, starting with Lee et al.~\cite{lee2016exact}. While this work provides exact confidence intervals, it has been observed that these intervals (which do not make use of randomization) have infinite expected length \cite{kivaranovic2018expected}. Subsequent work has improved upon these often large confidence intervals by choosing a better event to condition on \cite{liu2018more}, or by applying randomization \cite{tian2016selective, tian2016magic, tian2018selective, kivaranovic2020tight, panigrahi2017mcmc, panigrahi2019approximate}. 

We now formulate a stable version of the LASSO algorithm. It is inspired by the differentially private LASSO algorithm of Talwar et al.~\cite{talwar2015nearly}, although the noise variables are calibrated somewhat differently due to different modeling assumptions.

We use $e_i$ to denote the $i$-th standard basis vector in $\R^d$, and $\{\pm e_i\}_{i=1}^d$ to denote the set of $2d$ standard basis vectors, multiplied by $1$ and $-1$. We also let $\|X\|_{2,\infty}$ denote the $L_{2,\infty}$ norm of $X$, $\|X\|_{2,\infty} := \max_{i\in[d]} \|X_i\|_2$.

\begin{algorithm}[H]
\SetAlgoLined
\begin{flushleft}
\textbf{input: }design matrix $X\in\R^{n\times d}$, outcome vector $y\in\R^{n}$, variance estimate $\hat \sigma^2 \sim \sigma^2 \frac{\chi_r^2}{r}$, $\ell_1$-constraint $C_1$, number of steps $k$, parameters $\delta\in(0,1), \eta>0$\\
\textbf{output: } LASSO solution $\thetalasso \in\R^d$\newline
Initialize $\theta_1=0$\newline
\For{$t=1,2,\dots,k$}{
\ $\forall\phi \in C_1 \cdot \{\pm e_i\}_{i=1}^d$, sample $\xi_{t,\phi} \stackrel{\text{i.i.d.}}{\sim} \text{Lap}\left(\frac{4 t_{r,1-\delta/(2d)}C_1 \|X\|_{2,\infty}}{\eta n}\right)$\newline
  $\forall\phi \in C_1 \cdot \{\pm e_i\}_{i=1}^d$, let $\alpha_\phi = -\frac{2}{n \hat \sigma} \phi^\top X^\top(y - X\theta_t) + \xi_{t,\phi}$\newline
Set $\phi_t = \argmin_{\phi\in C_1 \cdot \{\pm e_i\}_{i=1}^d} \alpha_\phi$\newline
Set $\theta_{t+1} = (1-\Delta_t)\theta_t + \Delta_t \phi_t$, where $\Delta_t = \frac{2}{t+1}$
}
 Return $\thetalasso = \theta_{k+1}$
 \end{flushleft}
\caption{Stable LASSO algorithm}
\label{alg:lasso}
\end{algorithm}

In essence, Algorithm \ref{alg:lasso} is a randomized version of the classical Frank-Wolfe algorithm from constrained optimization~\cite{frank1956algorithm}.

We now argue that $\thetalasso$ is stable. The proof is based on a composition argument: namely, we can view $\thetalasso$ as the result of a composition of $k$ subroutines, each given by one optimization step which produces $\theta_{t}$. The stability of each subroutine is proved by extending an argument related to the ``report noisy max'' mechanism from differential privacy~\cite{dwork2014algorithmic}. The full proof of Proposition \ref{prop:lasso-stability} can be found in the Supplement.

\begin{proposition}[LASSO stability]
\label{prop:lasso-stability}
Algorithm \ref{alg:lasso} is both
\begin{itemize}
\item[(a)] $\left(\frac{1}{2}k\eta^2 + \sqrt{2k\log(1/\delta)}\eta,\delta,\delta\right)$-stable, and
\item[(b)] $(k\eta,0,\delta)$-stable.
\end{itemize}
\end{proposition}

We state two rates because there exist parameter regimes where either rate leads to tighter confidence intervals than the other (the first rate being tighter when $\eta$ is small).

By the post-processing property, Proposition \ref{prop:lasso-stability} implies stability of any model $\hat M$ obtained as a function of $\thetalasso$, such as the model corresponding to its non-zero entries.

Notice that the noise level in Algorithm \ref{alg:lasso} is an explicit function of $\eta$. This allows the statistician to understand the loss in utility---that is, how much worse $\thetalasso$ is relative to an exact LASSO solution---due to randomization. In fact, building on work by Jaggi~\cite{jaggi2013revisiting} and Talwar et al.~\cite{talwar2015nearly}, we can upper bound the excess risk resulting from randomization.

\begin{proposition}[LASSO utility]
\label{prop:lasso-utility}
Suppose we run Algorithm \ref{alg:lasso} for $k = \left\lceil\frac{n\|X\|_\infty^2 C_1 \eta}{\hat \sigma \|X\|_{2,\infty}}\right\rceil$ steps. Then, 
$$\frac{1}{n}\E[\|y - X\thetalasso\|_2^2 ~|~ y] - \min_{\theta:\|\theta\|_1\leq C_1} \frac{1}{n} \|y - X\theta\|_2^2 = \tilde O\left(\frac{C_1 \|X\|_{2,\infty} \log(d) t_{r,1-\delta/(2d)} \sigma}{n\eta}\right).$$	
\end{proposition}

The proof is deferred to the Supplement.


\paragraph{Model selection via marginal screening.}
One of the most commonly used model selection methods involves simply picking a constant number of the features with the largest absolute inner product with the outcome $y$ \cite{guyon2003introduction, fan2008sure}. That is, one selects features $i$ corresponding to the top $k$ values of $|X_i^\top y|$, for a pre-specified parameter $k$. This strategy is known as \emph{marginal screening}, and it was first analyzed in the context of selective inference by Lee and Taylor~\cite{lee2014exact}.

In Algorithm \ref{alg:marginal}, we state a stable version of marginal screening. Notice that the randomization scheme is similar to that of the stable LASSO method. As before, we let $\|X\|_{2,\infty}$ denote the $L_{2,\infty}$ norm of $X$.

\begin{algorithm}[H]
\SetAlgoLined
\begin{flushleft}
\textbf{input: }design matrix $X\in\R^{n\times d}$, outcome vector $y\in\R^{n}$, variance estimate $\hat\sigma^2 \sim \sigma^2 \frac{\chi_r^2}{r}$, model size $k$, parameters $\delta\in(0,1), \eta>0$\\
\textbf{output: } $\hat M = \{i_1,\dots,i_k\}$\newline
Compute $(c_1,\dots,c_d) = \frac{1}{n \hat \sigma}X^\top y\in \R^d$\newline
$\text{res}_1 = [d]$\newline
\For{$t=1,2,\dots,k$}{
\ $\forall i \in \text{res}_i$, sample $\xi_{t,i} \stackrel{\text{i.i.d.}}{\sim} \text{Lap}\left(\frac{2 t_{r,1-\delta/(2d)}\|X\|_{2,\infty}}{n\eta}\right)$\newline
$i_t = \argmax_{i\in\text{res}_t} |c_i + \xi_{t,i}|$\newline
$\text{res}_{t+1} = \text{res}_t \setminus i_t$
}
 Return $\hat M = \{i_1,\dots,i_k\}$
 \end{flushleft}
\caption{Stable marginal screening algorithm}
\label{alg:marginal}
\end{algorithm}

The high-level idea behind the proof of stability of Algorithm \ref{alg:marginal} is similar to that of Algorithm~\ref{alg:lasso}, and we present it in the Supplement.

\begin{proposition}[Marginal screening stability]
\label{prop:marginal-stability}
Algorithm \ref{alg:marginal} is both
\begin{itemize}
 \item[(a)] $\left(\frac{1}{2}k\eta^2 + \sqrt{2k\log(1/\delta)}\eta,\delta,\delta\right)$-stable, and
 \item[(b)] $(k\eta,0,\delta)$-stable.
 \end{itemize}
\end{proposition}

As for the LASSO, we aim to quantify the loss in utility due to randomization. Given that the goal of marginal screening is to detect the largest $k$ values $|c_i| = |X_i^\top y|$, a reasonable notion of utility loss is the difference between the values $c_i$ corresponding to the variables in $\hat M$, and the actual largest values of~$c_i$.

\begin{proposition}[Marginal screening utility]
\label{prop:marginal-utility}
Let $m_i$ denote the index of the $i$-th largest value $c_j$ in absolute value, so that $(|c_{m_1}|,\dots,|c_{m_d}|)$ is the decreasing order statistic of $\{|c_i|\}_{i=1}^d$. Then, for any $\delta'\in (0,1)$, Algorithm \ref{alg:marginal} satisfies:
$$\PPst{\max_{j\in[k]}  |c_{m_j}| - |c_{i_j}| \leq \frac{4t_{r,1-\delta/(2d)}\log(dk/\delta')\|X\|_{2,\infty}}{n\eta} }{y} \geq 1 - \delta'.$$
\end{proposition}

\section{Experimental results}
\label{sec:experiments}

In this section, we evaluate our selective intervals for the LASSO and marginal screening and compare our solution with data splitting.

For a fixed sample size $n$ we vary the number of features $d$. We consider two different data-generating processes for the design matrix: one in which the rows of $X$ are drawn independently from an equicorrelated multivariate Gaussian distribution with pairwise correlation $\rho = 0.5$, and the second one in which all entries of $X$ are drawn as independent Bernoulli random variables with parameter $0.1$. In the former case, $X$ is normalized to have columns of unit norm. The outcome is generated as $y = X\beta + \epsilon$, where $\epsilon_i \stackrel{\mathrm{i.i.d.}}{\sim} \mathcal{N}(0,1), i\in[n]$, and the entries of $\beta$ are sampled according to
$$\beta_i = \begin{cases}
 	\mathrm{Exp}(\rho), ~i\in \{1,\dots,sd\},\\
 	0, ~i\in\{sd + 1,\dots,d\},
 \end{cases}
$$
for a signal parameter $\rho>0$ and a sparsity parameter $s\in(0,1)$, which we vary. In the Supplement we provide additional experiments when the errors are drawn from a heavier-tailed, Laplace distribution. We fix the target miscoverage level to be $\alpha = 0.1$. In all experiments we vary $\eta\in\{1,2,3,4,5,6,7,8,9,10\}$. For the comparison with data splitting, we use the splitting fraction derived in Section~\ref{sec:data_splitting}. Further experimental details are given in the Supplement.
\begin{figure}[t]
\centerline{\includegraphics[width=0.25\textwidth]{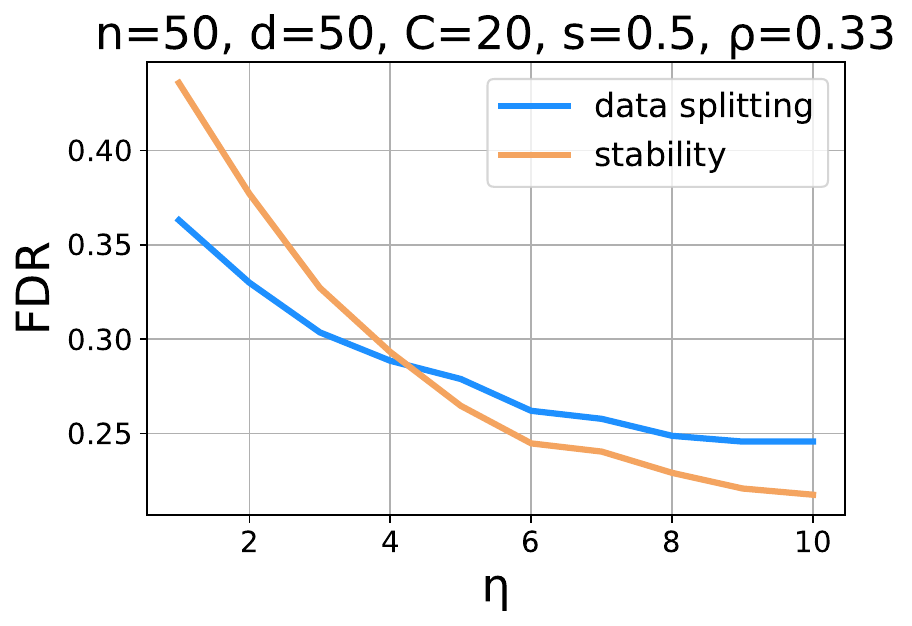}
\includegraphics[width=0.25\textwidth]{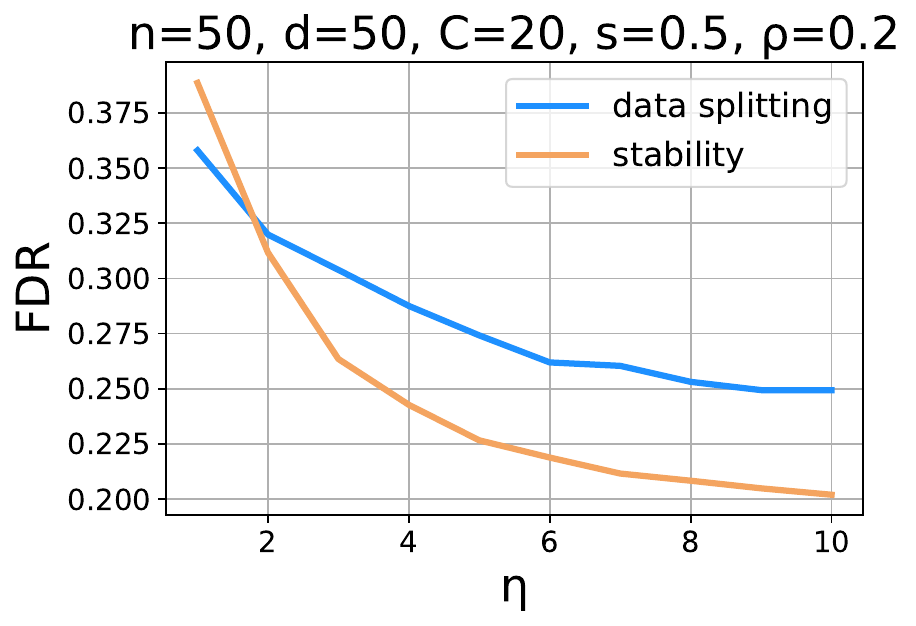}
\includegraphics[width=0.25\textwidth]{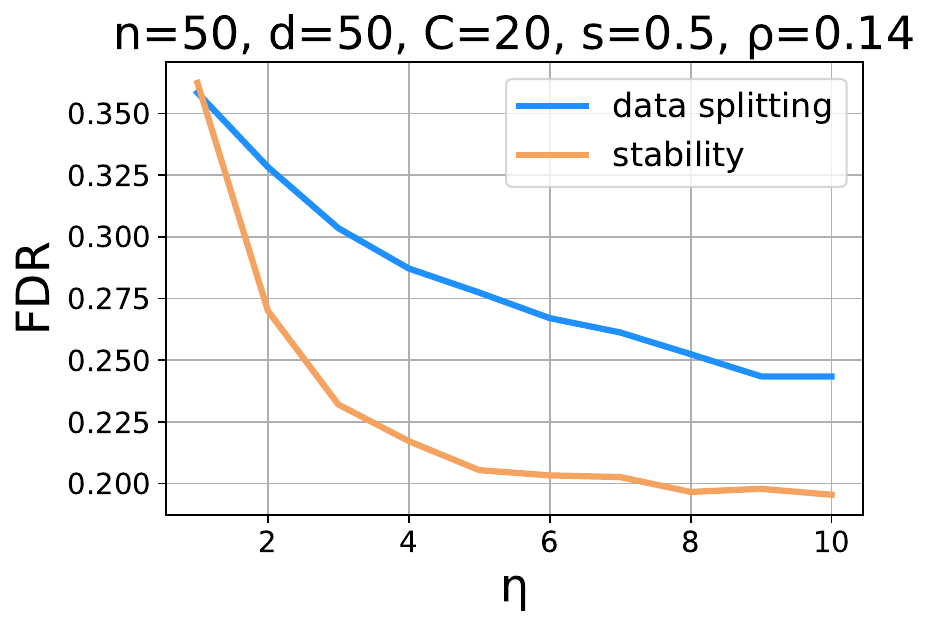}
\includegraphics[width=0.25\textwidth]{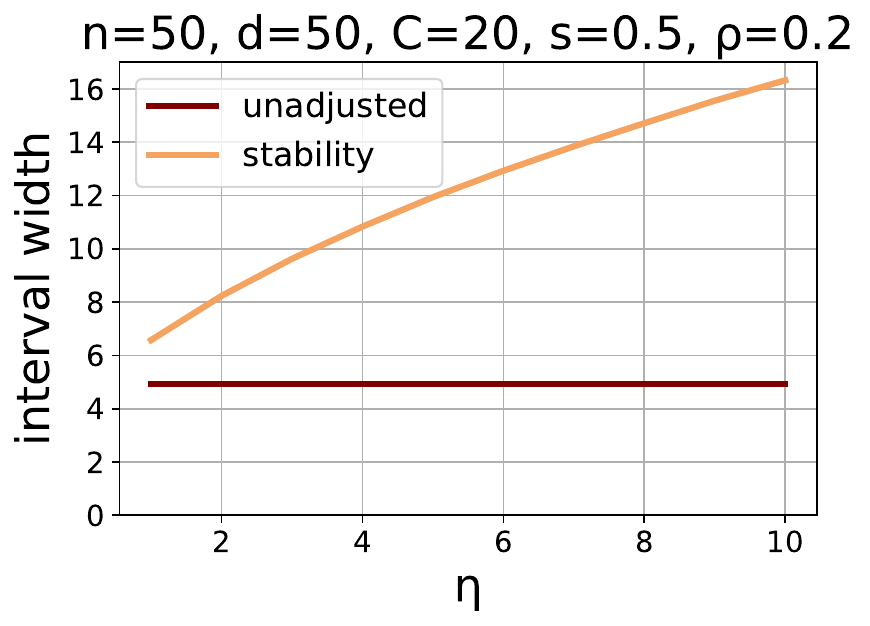}
}
\centerline{\includegraphics[width=0.25\textwidth]{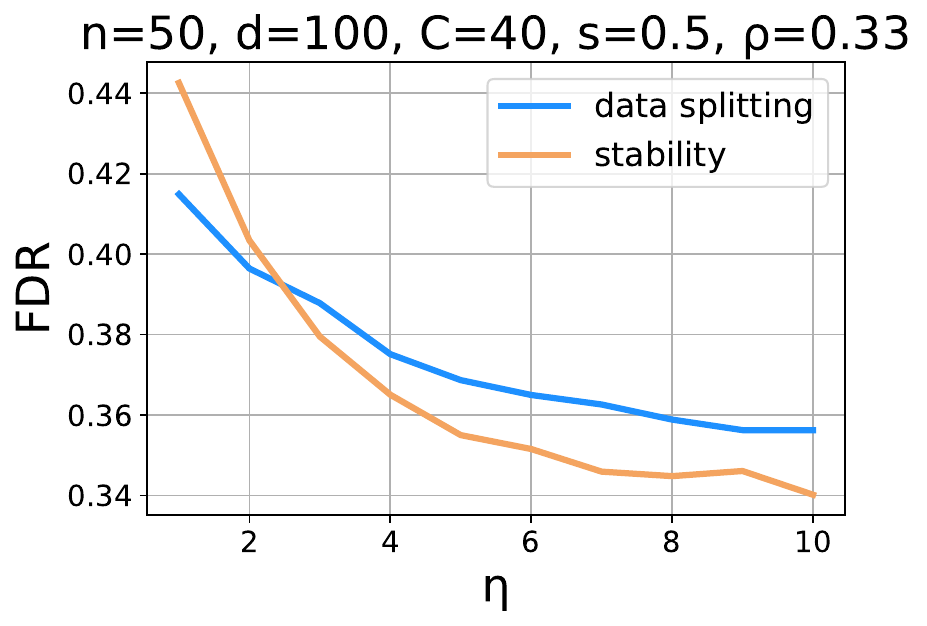}
\includegraphics[width=0.25\textwidth]{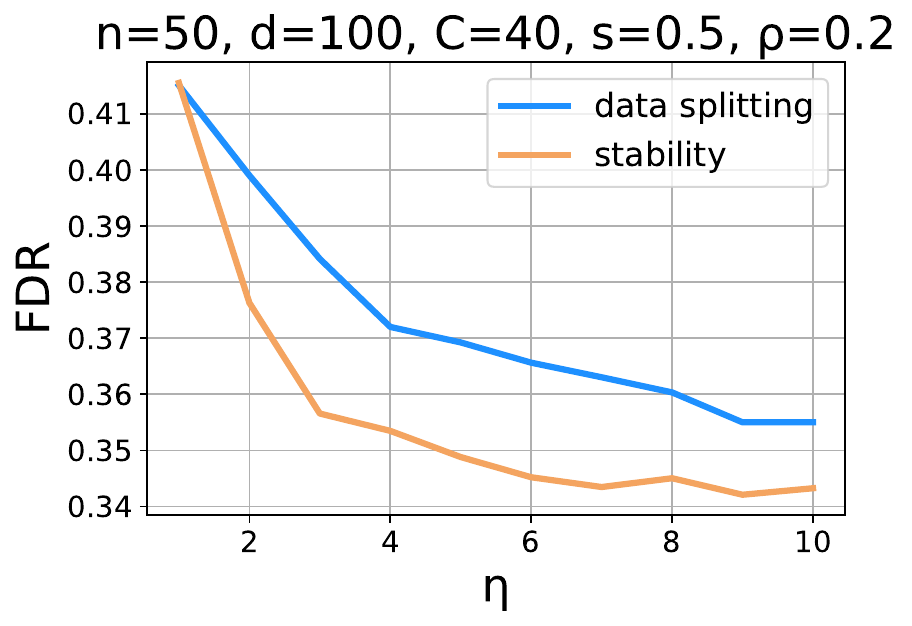}
\includegraphics[width=0.25\textwidth]{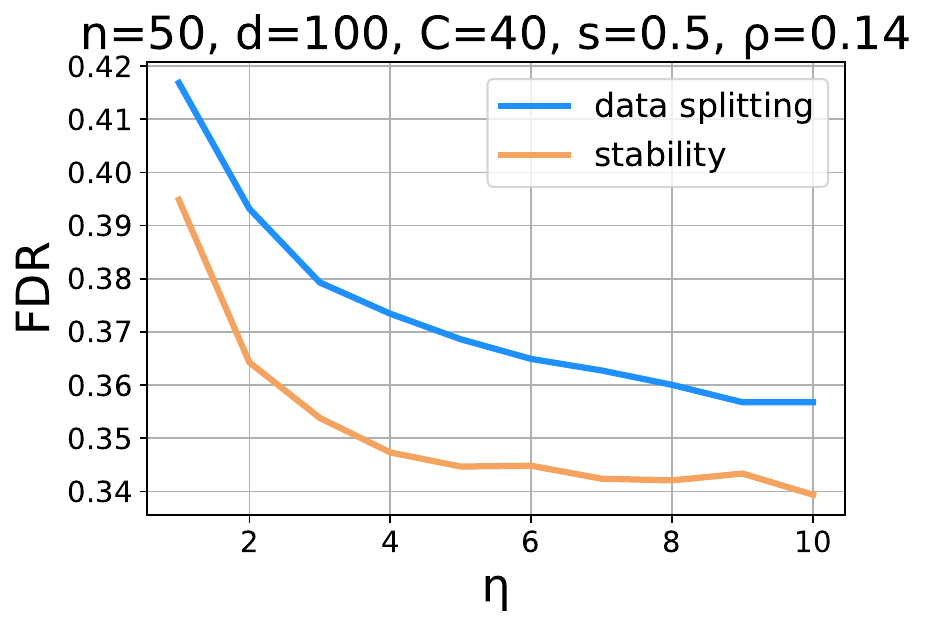}
\includegraphics[width=0.25\textwidth]{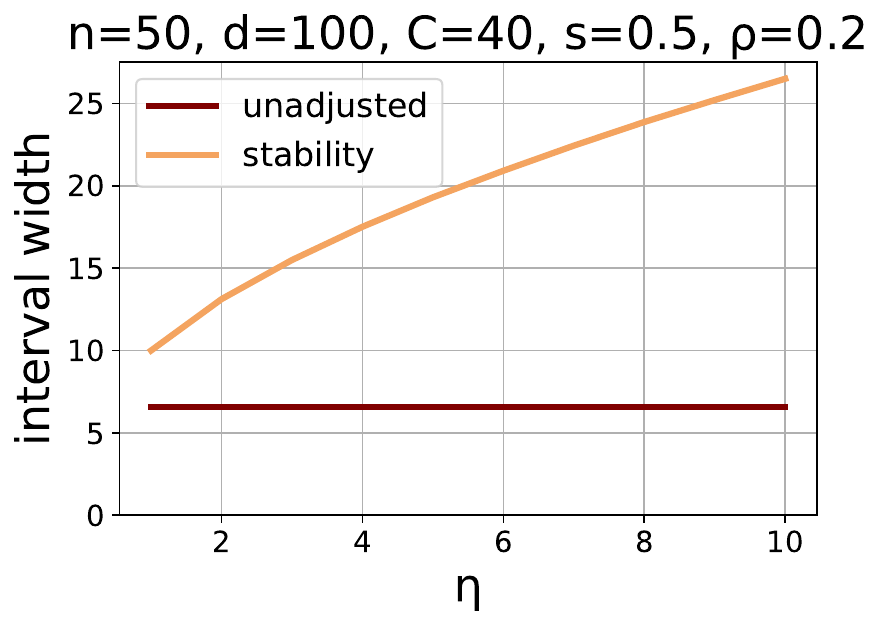}
}
\centerline{\includegraphics[width=0.25\textwidth]{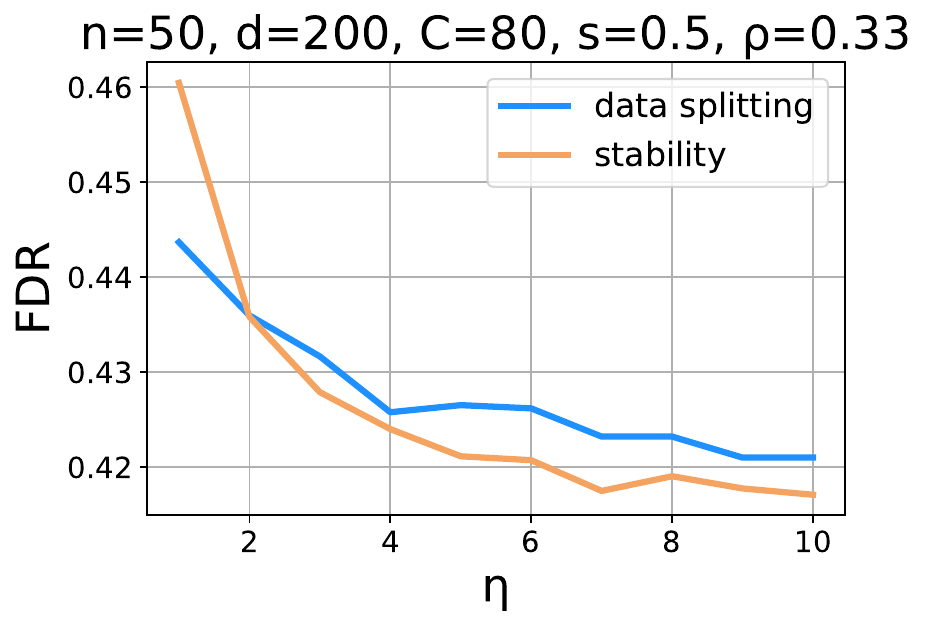}
\includegraphics[width=0.25\textwidth]{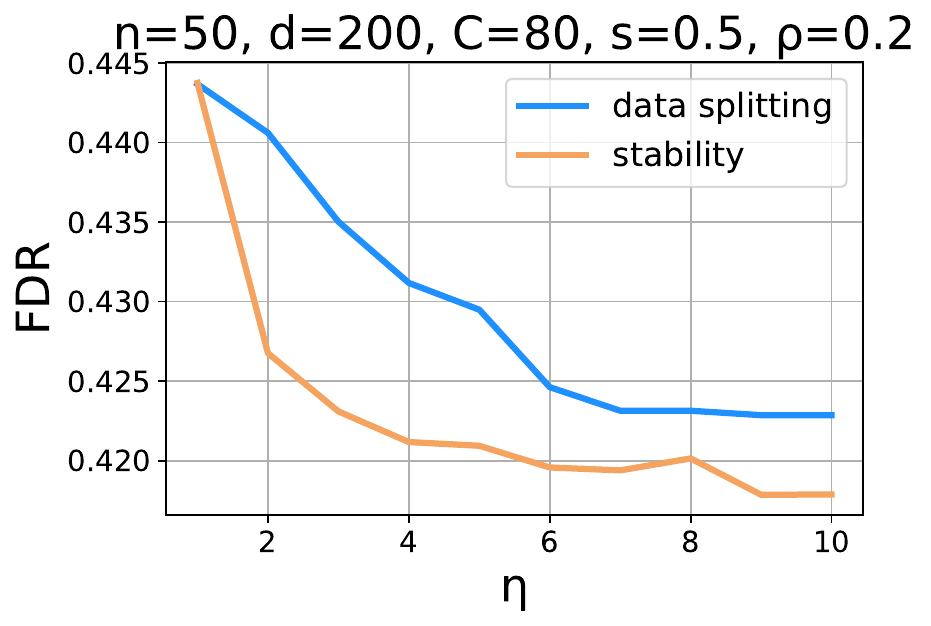}
\includegraphics[width=0.25\textwidth]{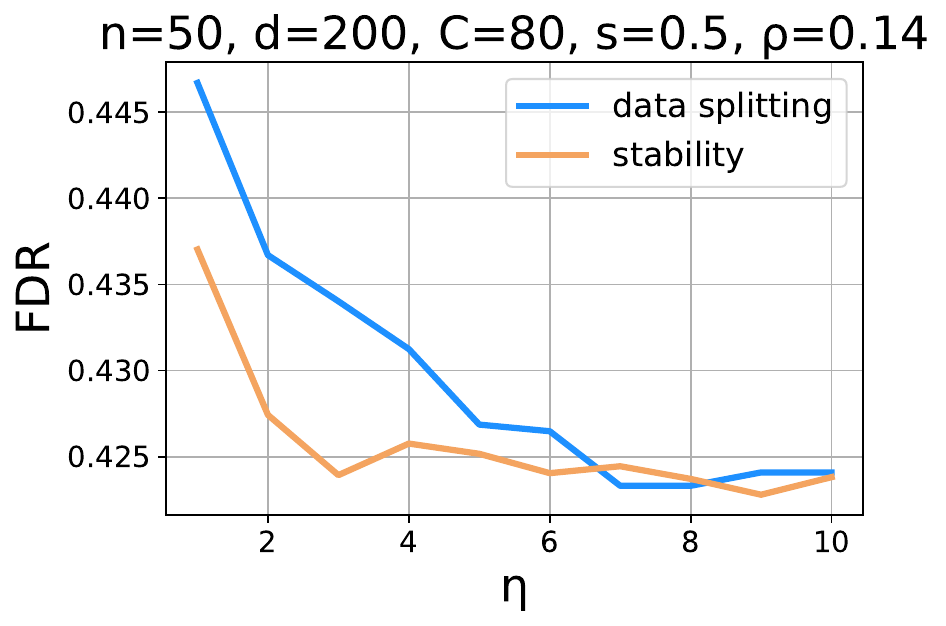}
\includegraphics[width=0.25\textwidth]{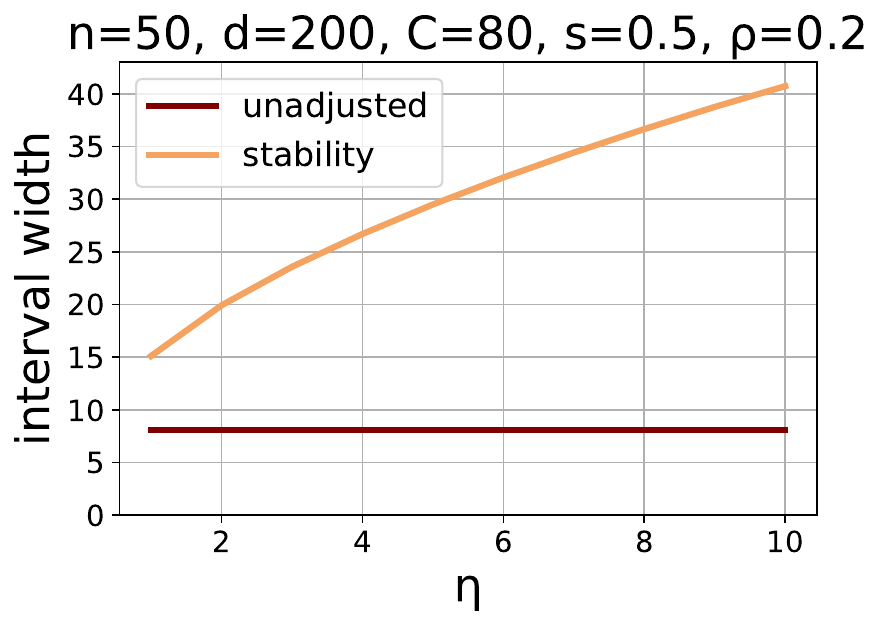}
}
\caption{Comparison of FDR after stable LASSO and LASSO with data splitting, with varying dimension and signal strength, in the Gaussian design case. In addition, we plot the average interval width (at $\rho=0.2$ only, however the width varies minimally with $\rho$) and the average unadjusted width.} 
\label{fig:lasso_comparison}
\end{figure}

\begin{figure}[b]
\centerline{\includegraphics[width=0.25\textwidth]{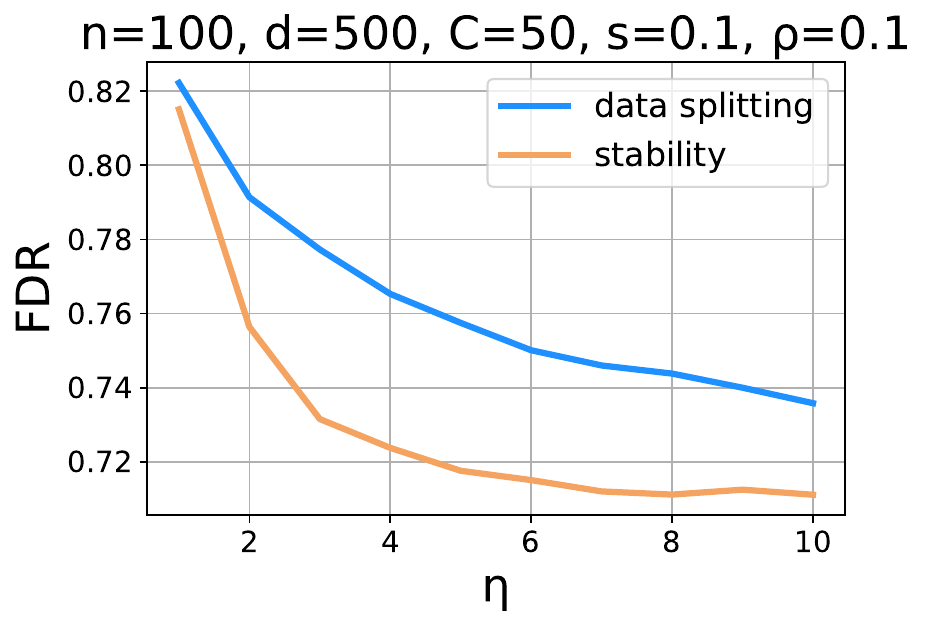}
\includegraphics[width=0.25\textwidth]{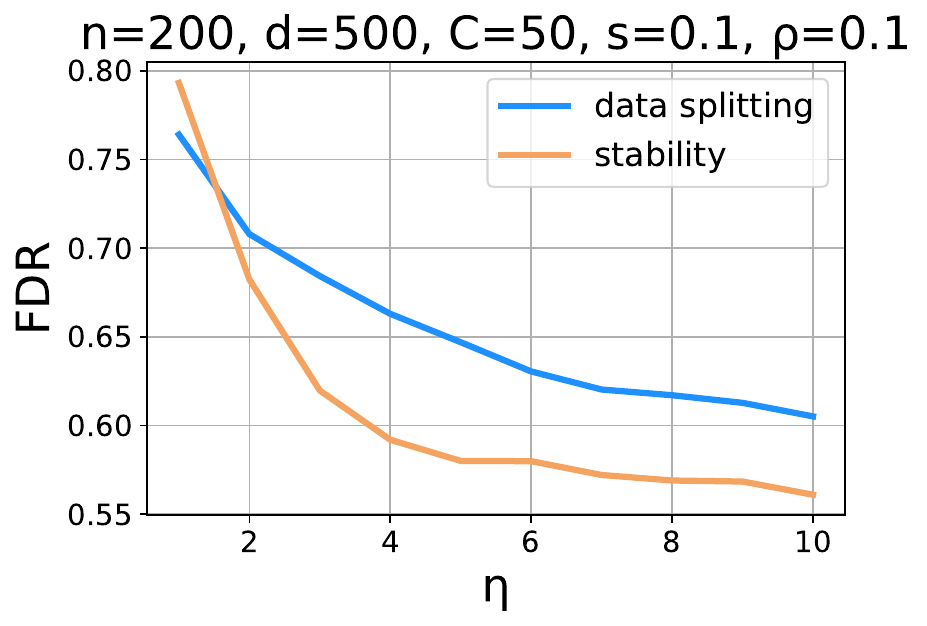}
\includegraphics[width=0.25\textwidth]{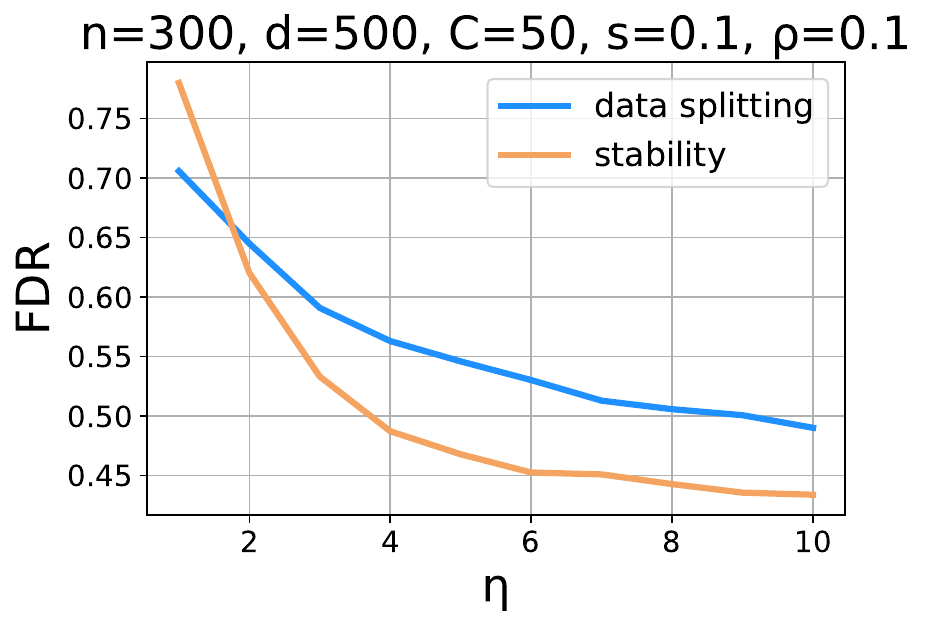}
\includegraphics[width=0.25\textwidth]{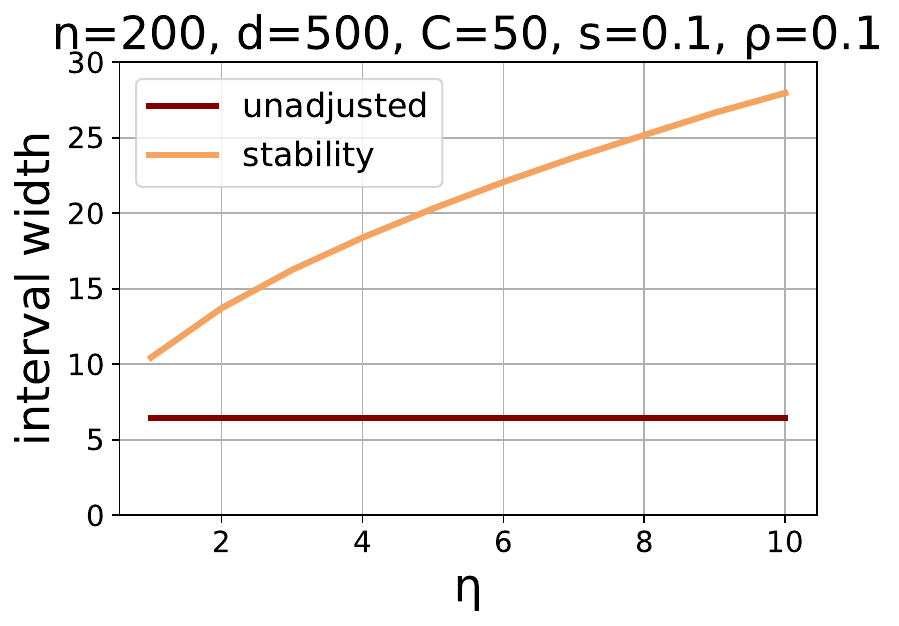}
}
\caption{Comparison of FDR after stable LASSO and LASSO with data splitting, with varying sample size, in the Gaussian design case. In addition, we plot the average interval width at $n=200$ and the average unadjusted width.} 
\label{fig:lasso_highdim}
\end{figure}

\subsection{Gaussian design} We first state the results for the Gaussian design case.

\paragraph{LASSO.}
In Figure \ref{fig:lasso_comparison} we compare the false discovery rate (FDR) of the stable LASSO algorithm and the LASSO algorithm with data splitting. In all plots $n=50$ is fixed and we vary $d\in\{50,100,200\}$. As we increase $d$, we also increase the size of the constraint set $C_1 \in \{20,40,80\}$ to allow more selections. We consider signal levels $\rho\in\{0.33, 0.2, 0.14\}$, which corresponds to an expected value of the non-null $\beta_i$ lying in $\{3,5,7\}$, and we fix $s=0.5$.

We observe that stability generally outperforms data splitting as $\eta$ grows, equivalently when the splitting fraction $f(\eta)$ grows, as well as when the signal strength grows. In Figure~\ref{fig:lasso_comparison} we additionally plot the average width of stable intervals against the average width of naive, unadjusted intervals. Note that the intervals obtained via data splitting have essentially the same width (and are hence not plotted), based on how $f(\eta)$ is chosen. We only plot interval width for $\rho = 0.2$ since the width varies minimally for different values of $\rho$. For completeness we include all plots of interval width in the Supplement.
 
 \begin{figure}[t]
\centerline{\includegraphics[width=0.25\textwidth]{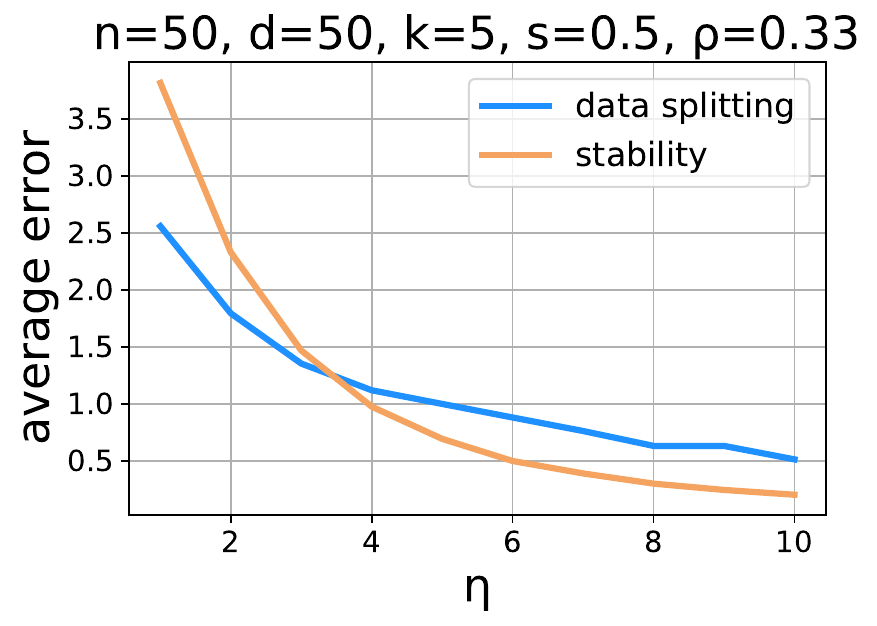}
\includegraphics[width=0.25\textwidth]{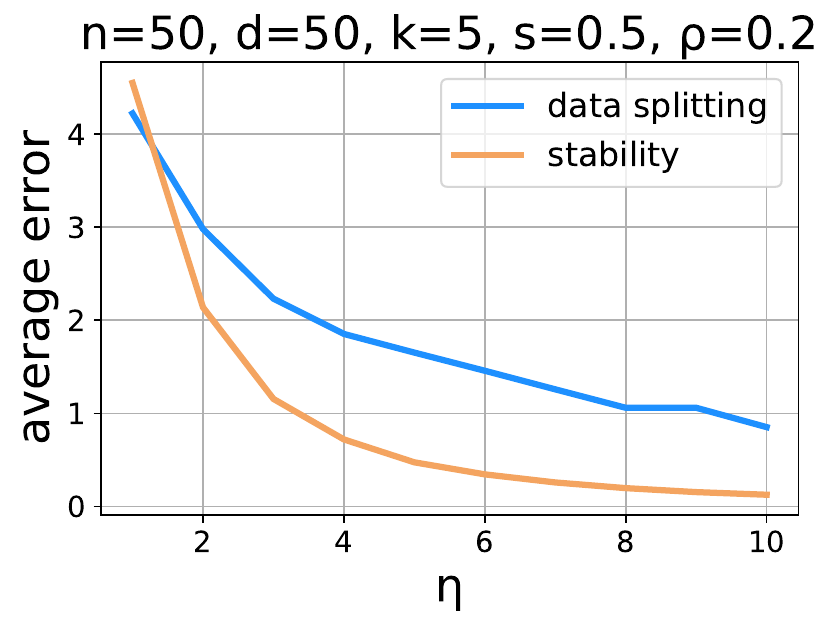}
\includegraphics[width=0.25\textwidth]{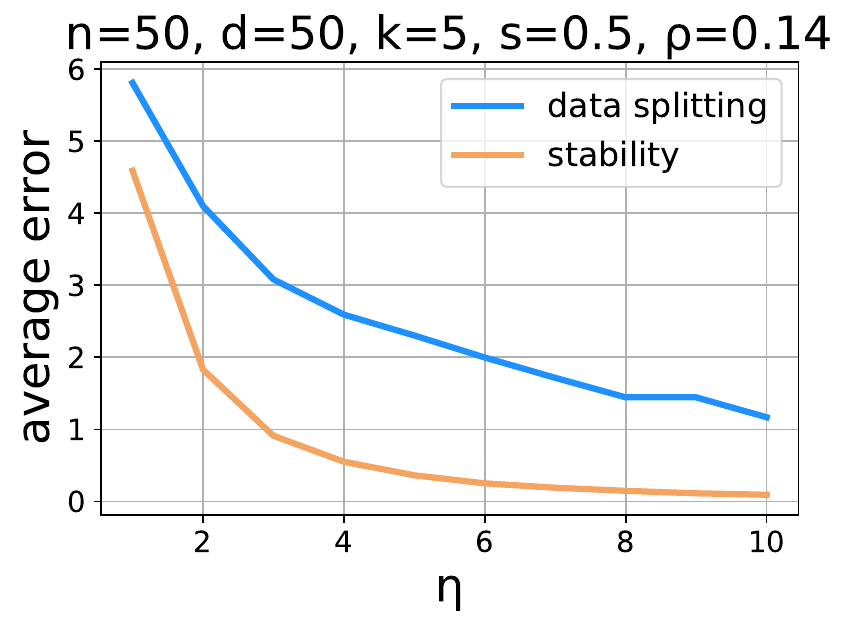}
\includegraphics[width=0.25\textwidth]{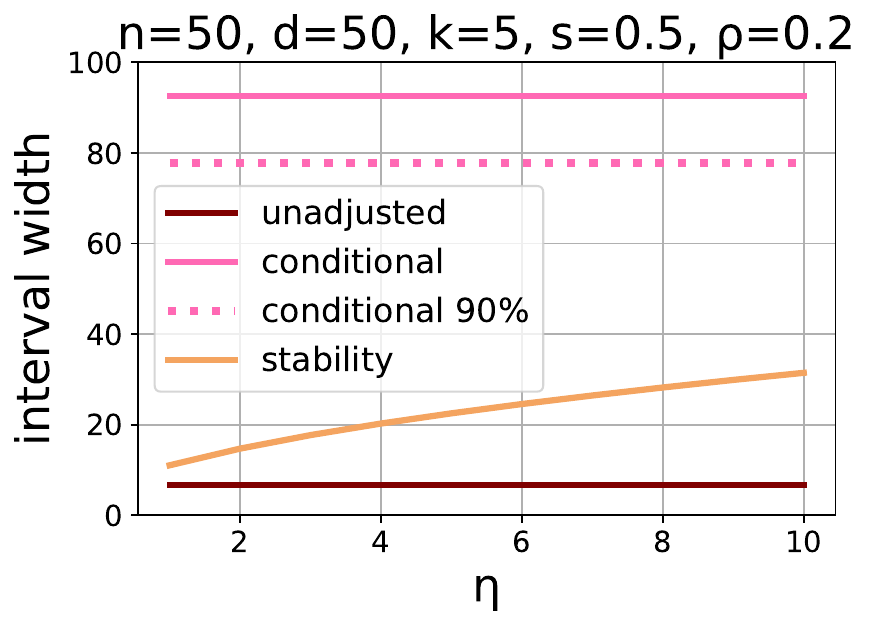}
}
\centerline{\includegraphics[width=0.25\textwidth]{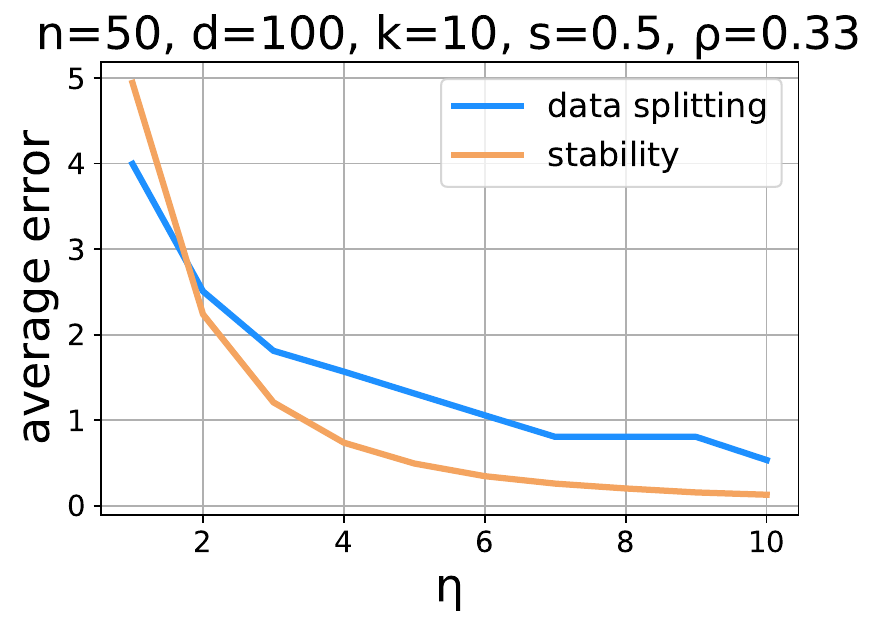}
\includegraphics[width=0.25\textwidth]{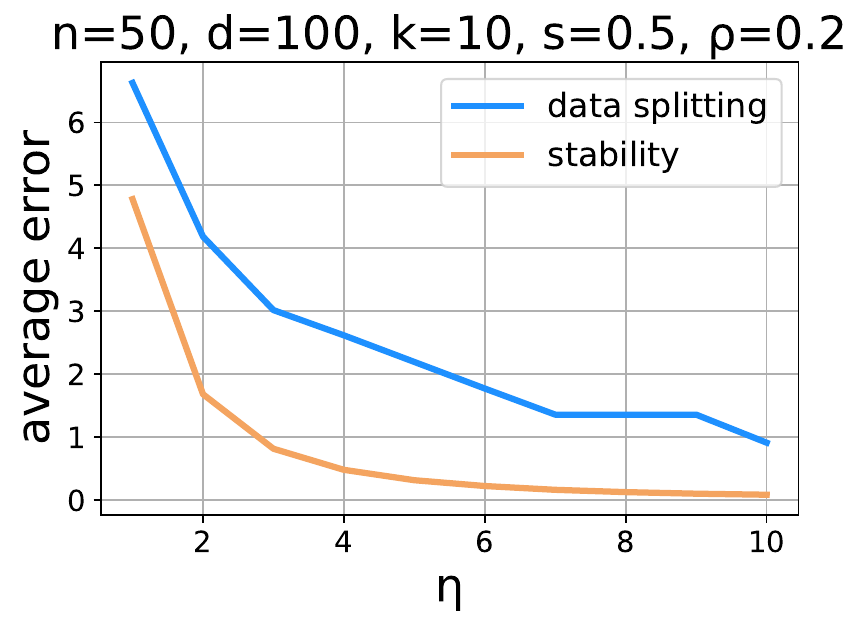}
\includegraphics[width=0.25\textwidth]{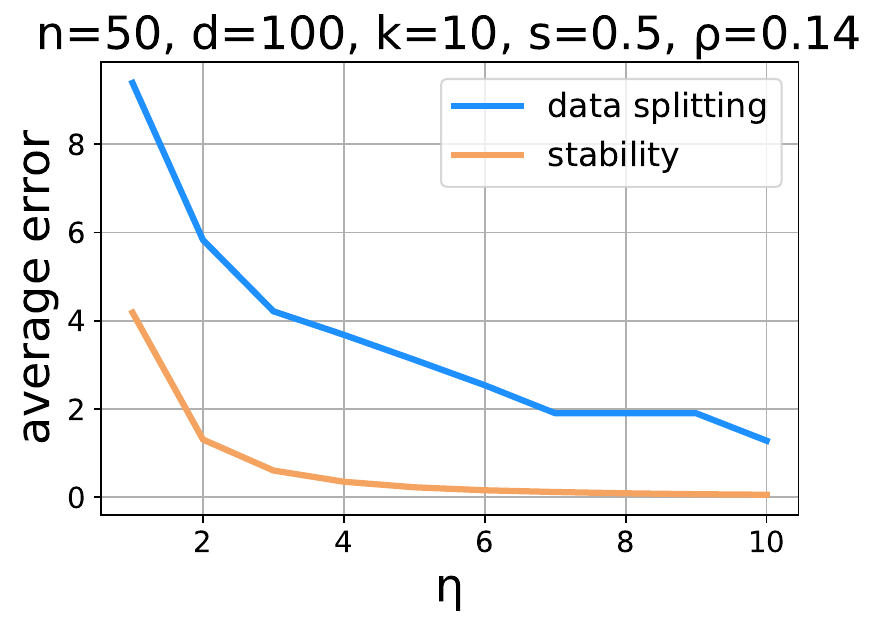}
\includegraphics[width=0.25\textwidth]{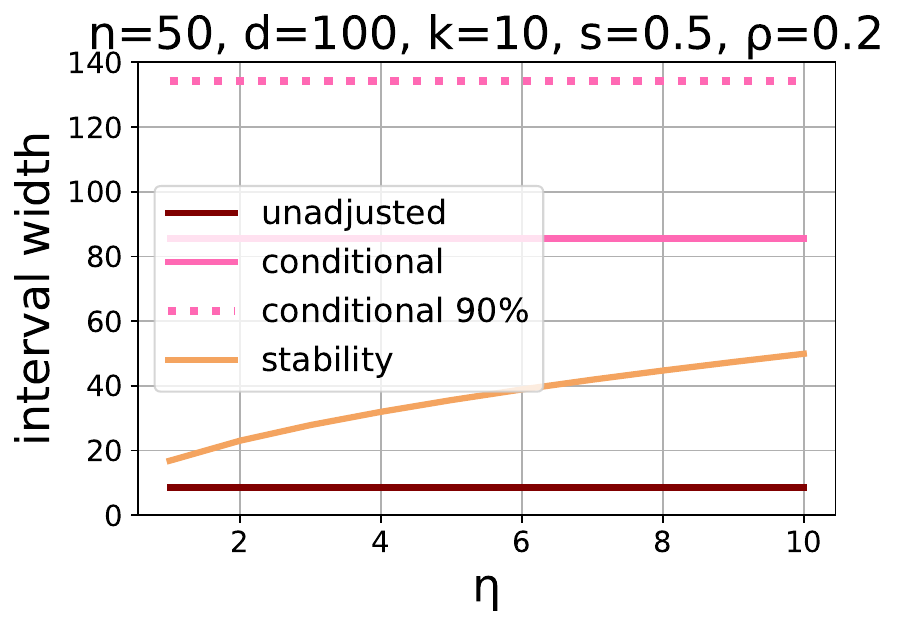}
}
\centerline{\includegraphics[width=0.25\textwidth]{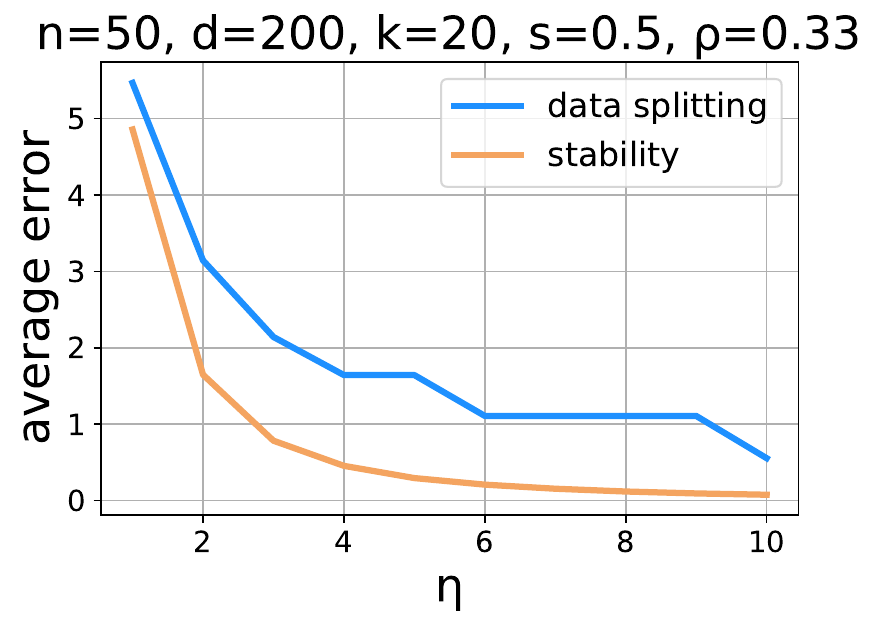}
\includegraphics[width=0.25\textwidth]{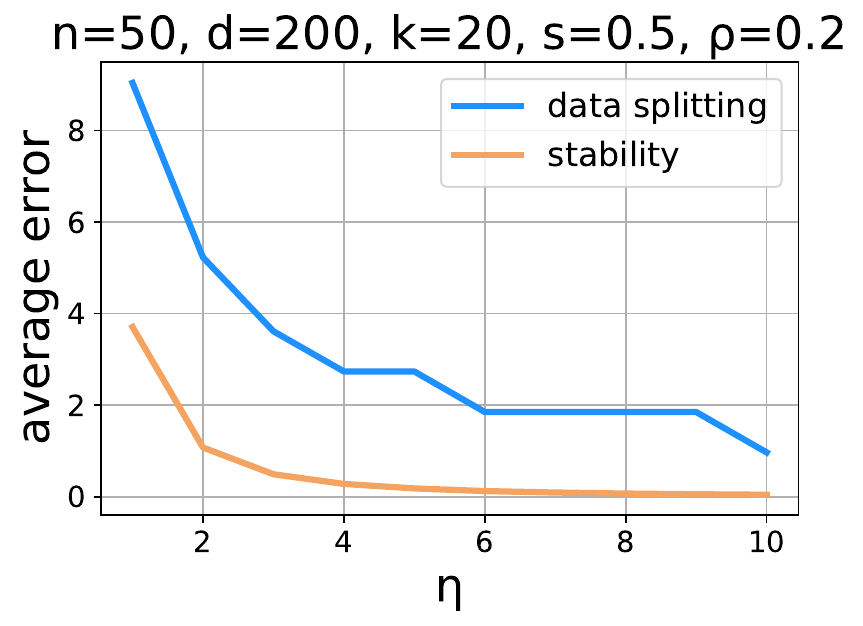}
\includegraphics[width=0.25\textwidth]{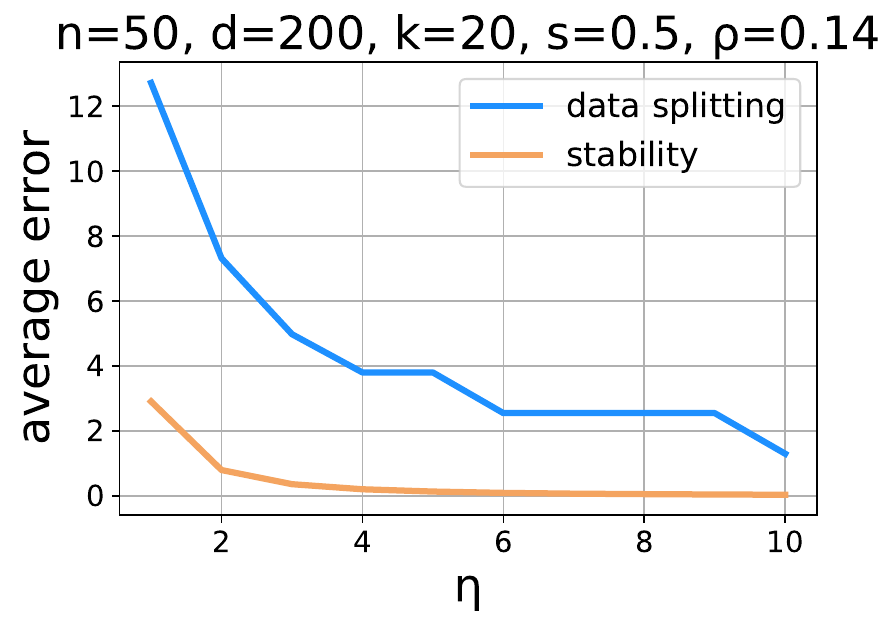}
\includegraphics[width=0.25\textwidth]{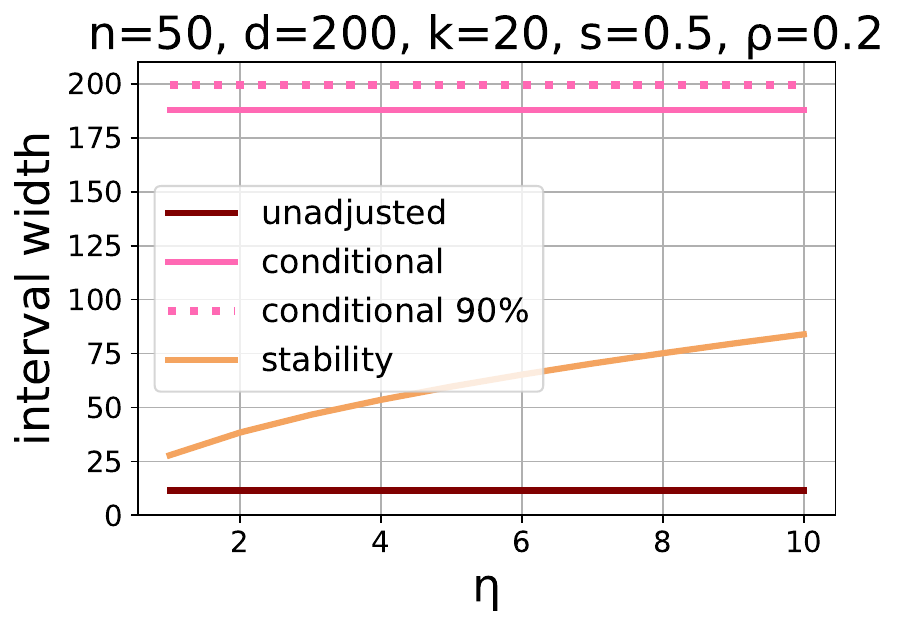}
}
\caption{Comparison of average error after stable marginal screening and marginal screening with data splitting, with varying dimension and signal strength, in the Gaussian design case. In addition, we plot the average interval width (at $\rho=0.2$ only, however the width varies minimally with $\rho$), together with the average unadjusted width and the width obtained via the conditional correction of Lee and Taylor~\cite{lee2014exact}. We also plot the $90\%$ quantile of the conditional width because it varies greatly across realizations.} 
\label{fig:screening_comparison}
\end{figure}

\begin{figure}[b]
\centerline{\includegraphics[width=0.25\textwidth]{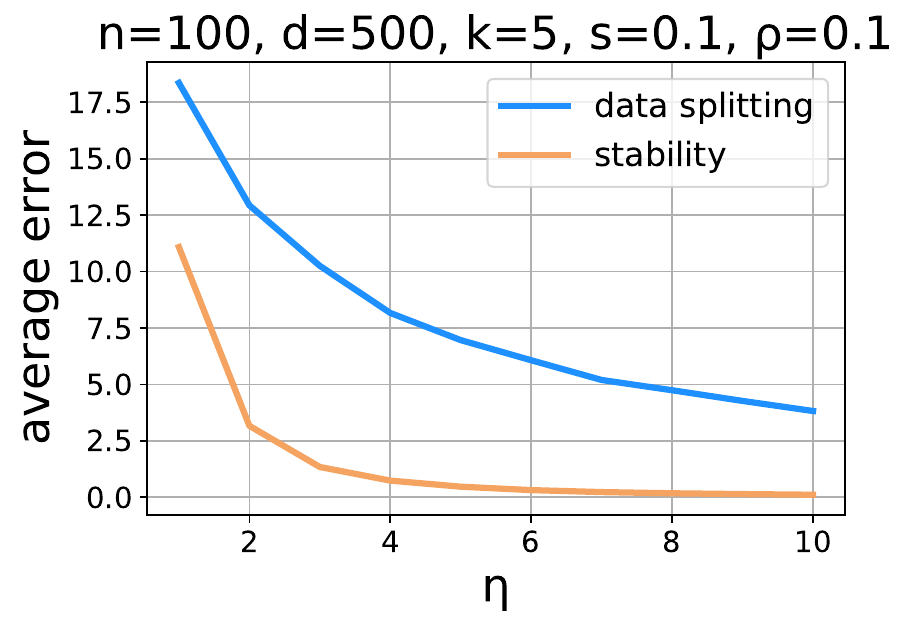}
\includegraphics[width=0.25\textwidth]{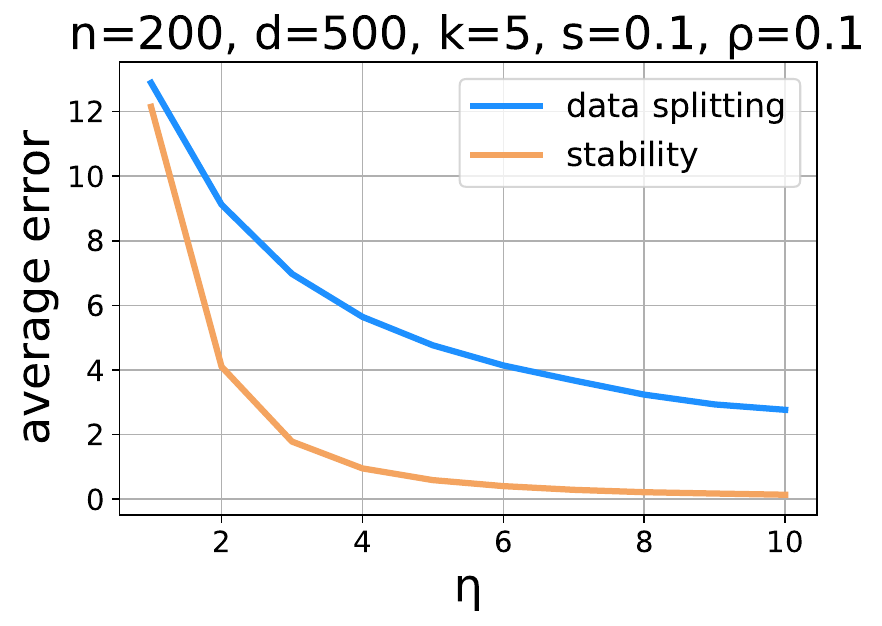}
\includegraphics[width=0.25\textwidth]{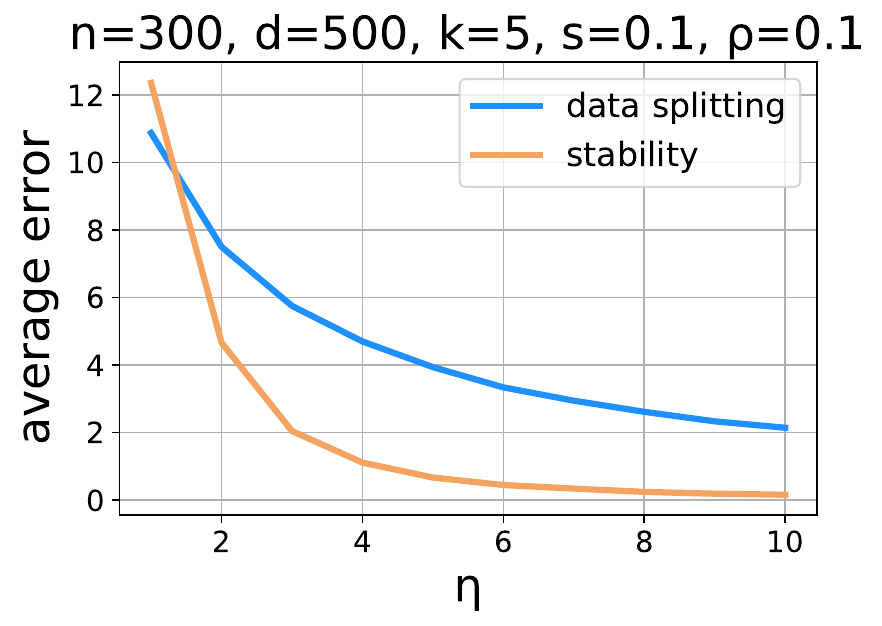}
\includegraphics[width=0.25\textwidth]{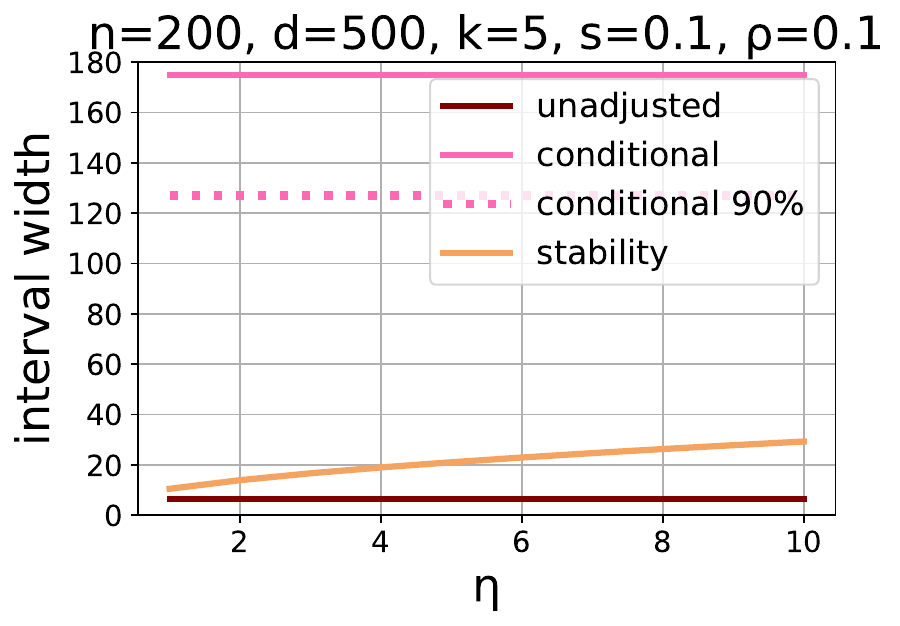}
}
\caption{Comparison of average error after stable marginal screening and marginal screening with data splitting, with varying sample size, in the Gaussian design case. In addition, we plot the average interval width at $n=200$, together with the average unadjusted width and the width implied by the conditional approach of Lee and Taylor~\cite{lee2014exact}.} 
\label{fig:ms_highdim}
\end{figure}

In Figure \ref{fig:lasso_highdim} we compare the stable LASSO algorithm and the LASSO with data splitting in a sparse high-dimensional setting with $d=500, s = 0.1$, and we vary the sample size $n\in\{100,200,300\}$. We fix $\rho = 0.1$. We observe that stability consistently outperforms data splitting for large enough $\eta$ and this gap grows with $n$. In addition, we plot the average interval width implied by stability against the average unadjusted interval width at $n=200$ (again we do not plot the interval width given by data splitting for the same reason as in Figure \ref{fig:lasso_comparison}). We include the plots of all interval widths in the Supplement.

\begin{figure}[t]
\centerline{\includegraphics[width=0.25\textwidth]{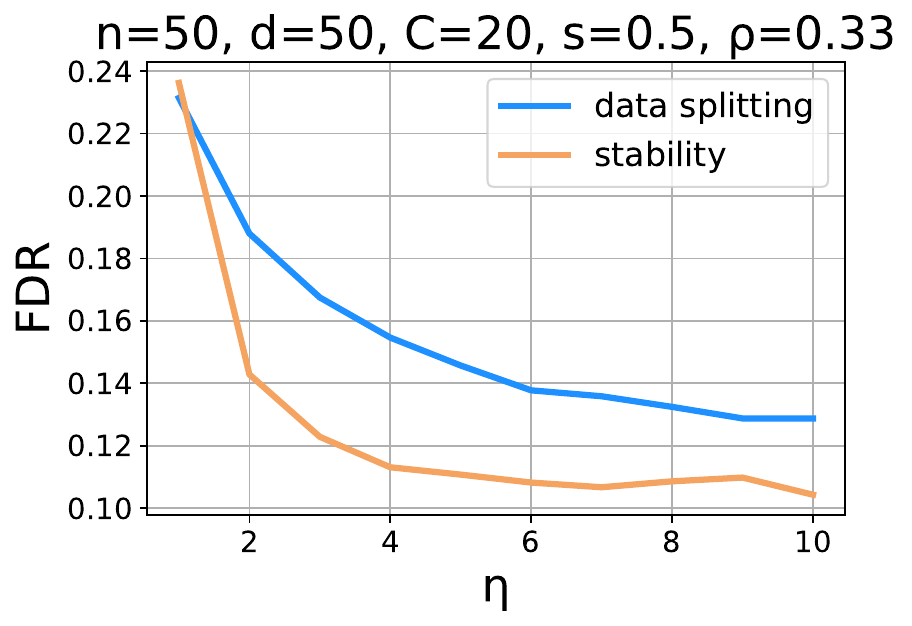}
\includegraphics[width=0.25\textwidth]{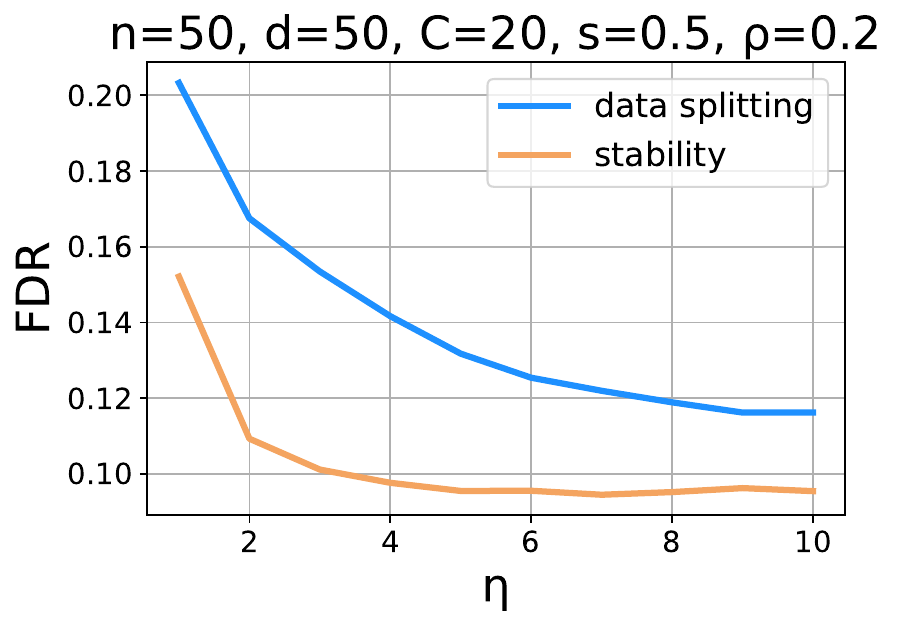}
\includegraphics[width=0.25\textwidth]{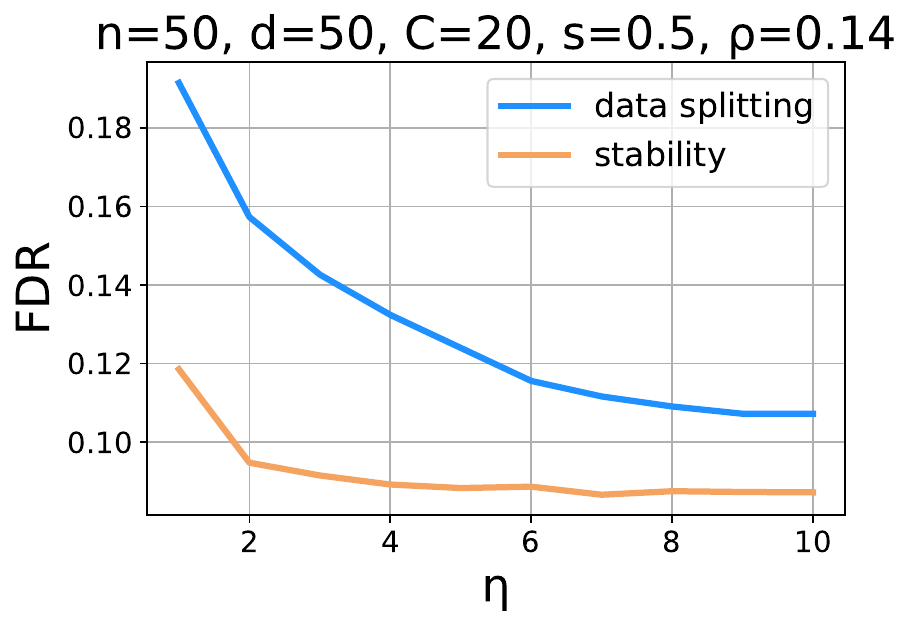}
\includegraphics[width=0.25\textwidth]{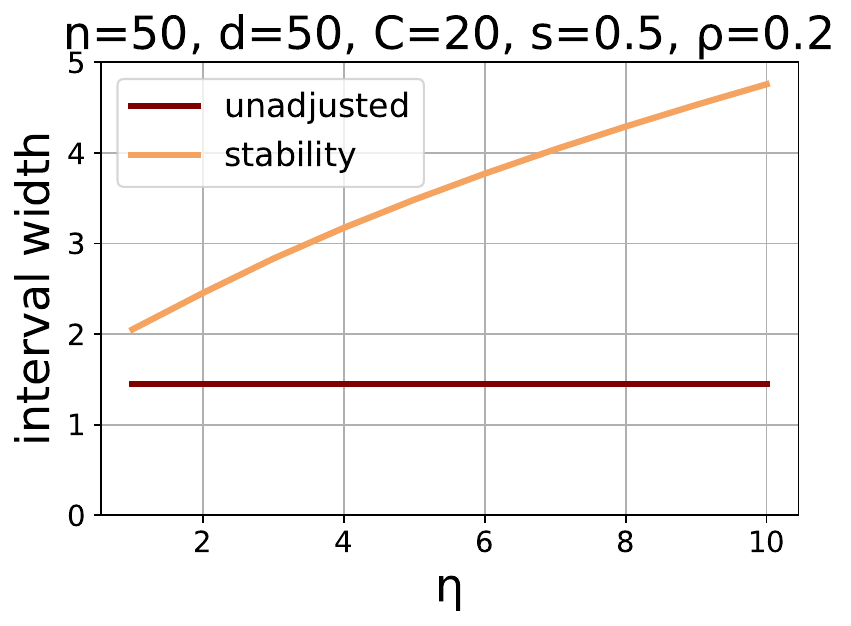}
}
\centerline{\includegraphics[width=0.25\textwidth]{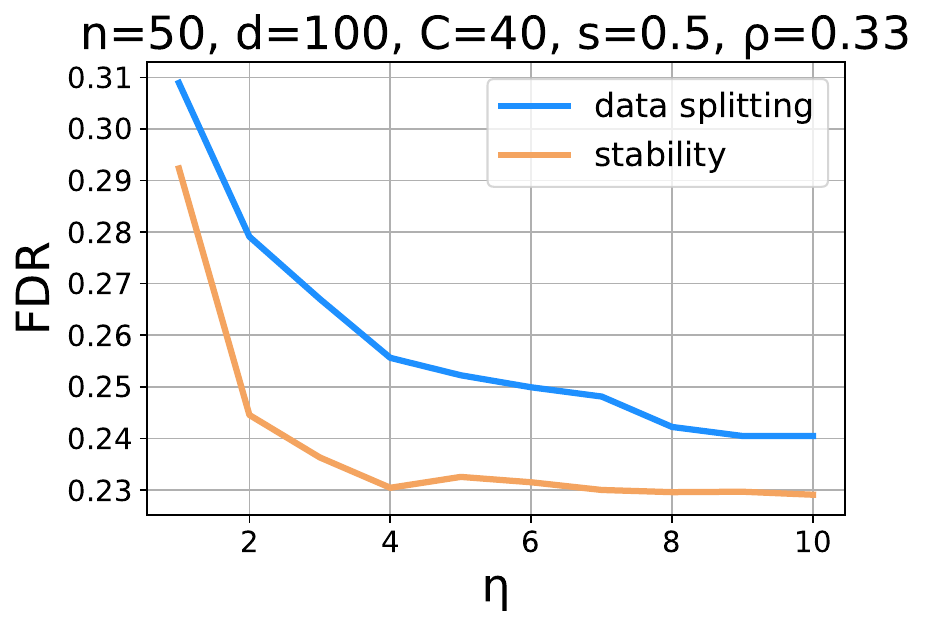}
\includegraphics[width=0.25\textwidth]{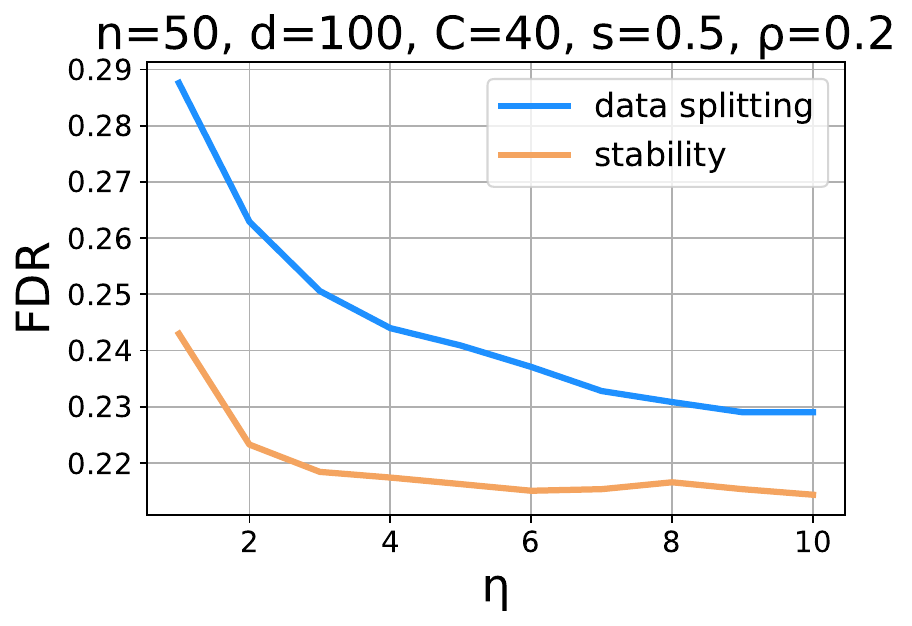}
\includegraphics[width=0.25\textwidth]{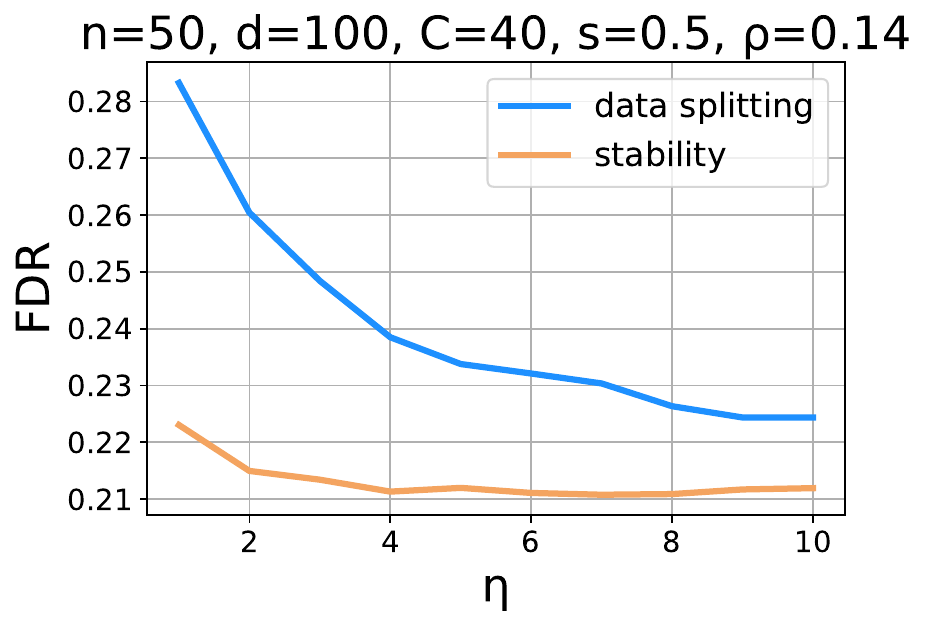}
\includegraphics[width=0.25\textwidth]{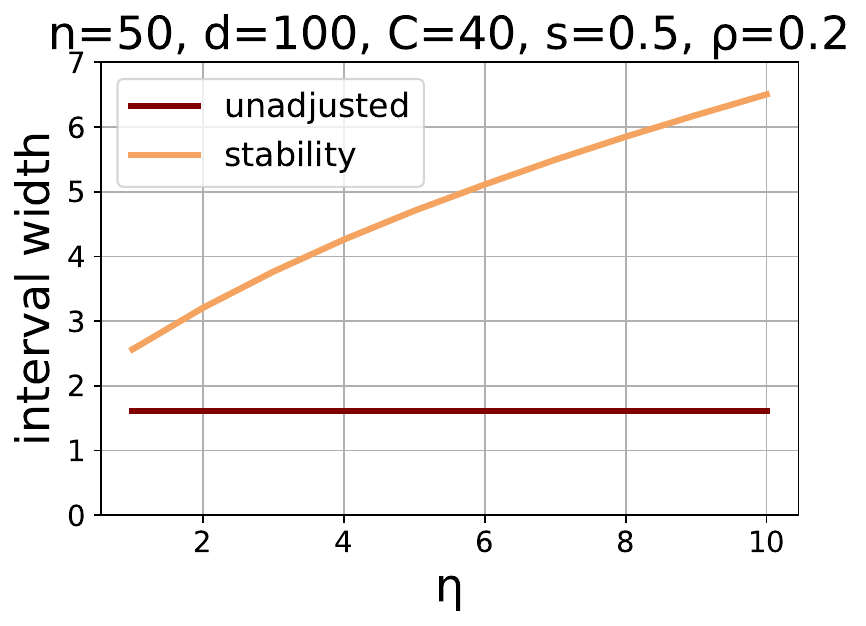}
}
\centerline{\includegraphics[width=0.25\textwidth]{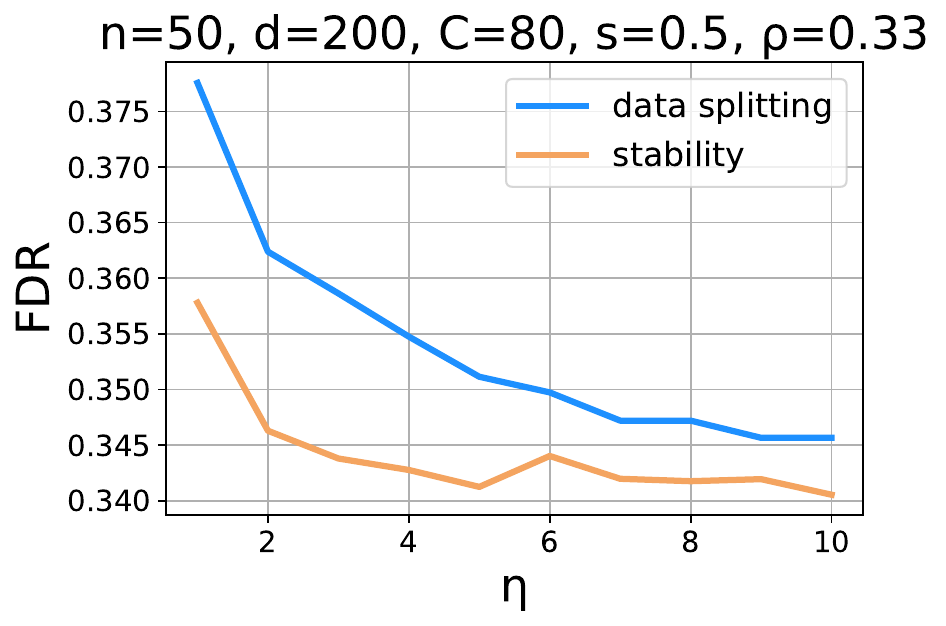}
\includegraphics[width=0.25\textwidth]{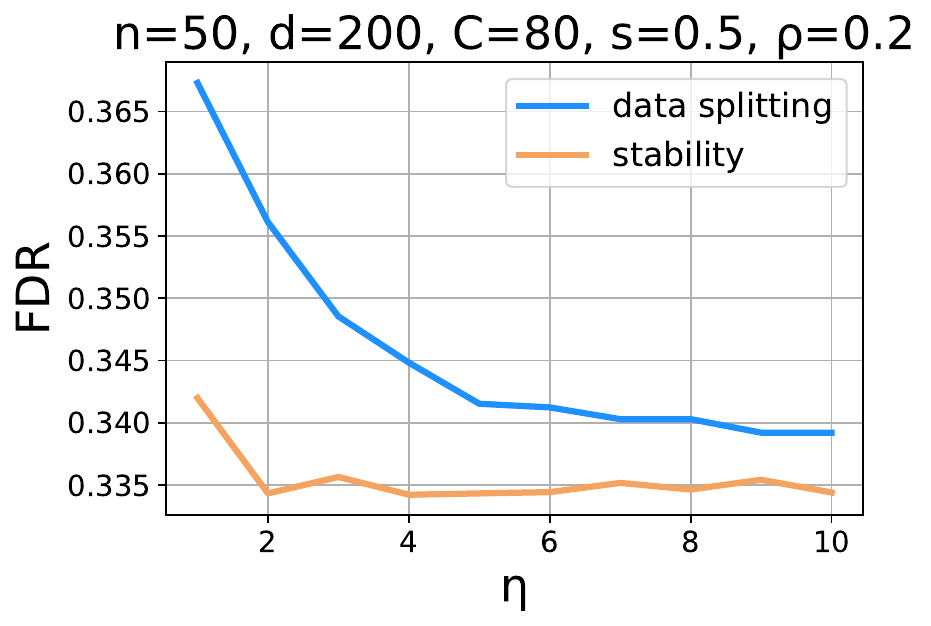}
\includegraphics[width=0.25\textwidth]{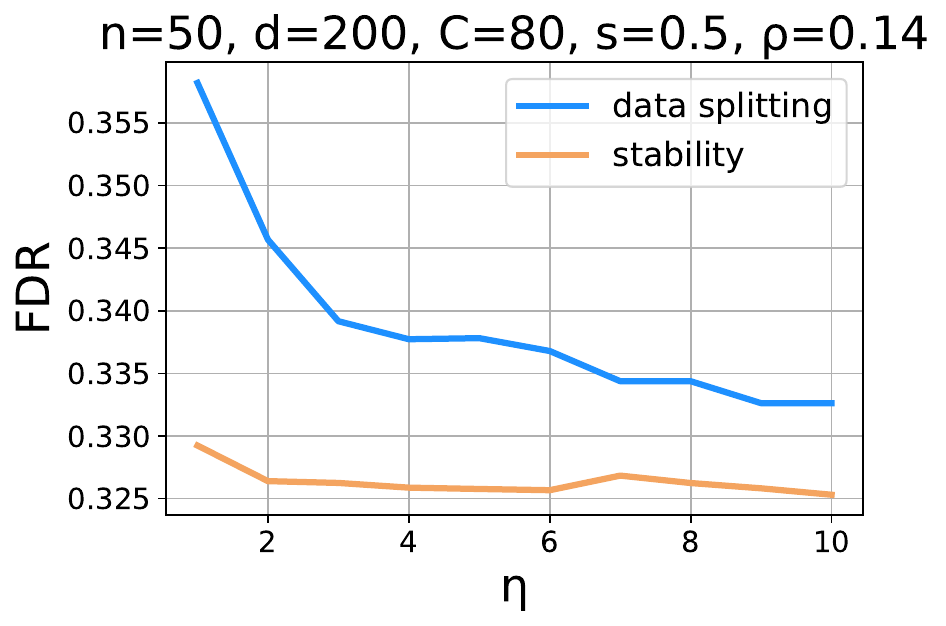}
\includegraphics[width=0.25\textwidth]{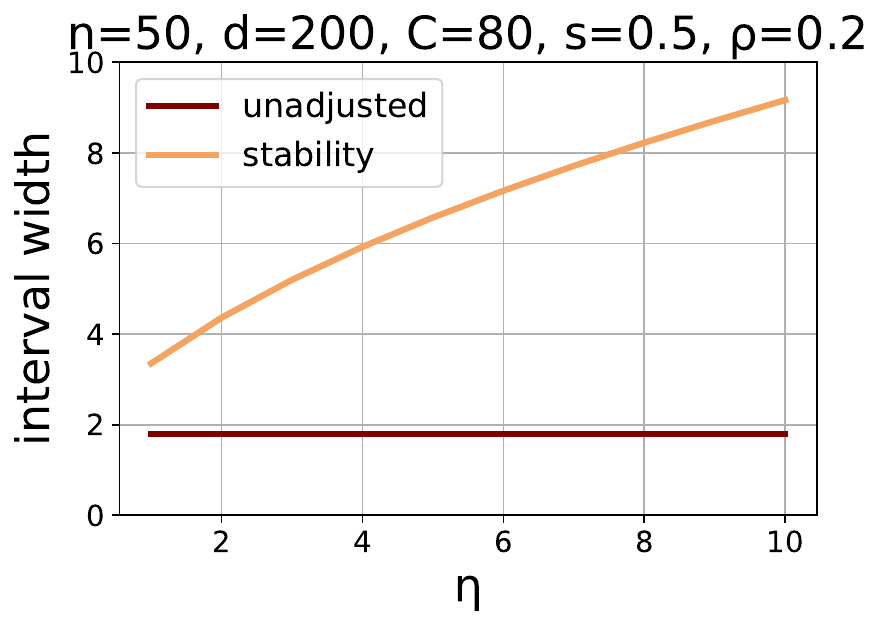}
}
\caption{Comparison of FDR after stable LASSO and LASSO with data splitting, with varying dimension and signal strength, in the Bernoulli design case. In addition, we plot the average interval width (at $\rho=0.2$ only, however the width varies minimally with $\rho$) and the average unadjusted width.} 
\label{fig:lasso_comparison_bern}
\end{figure}

\begin{figure}[b]
\centerline{\includegraphics[width=0.25\textwidth]{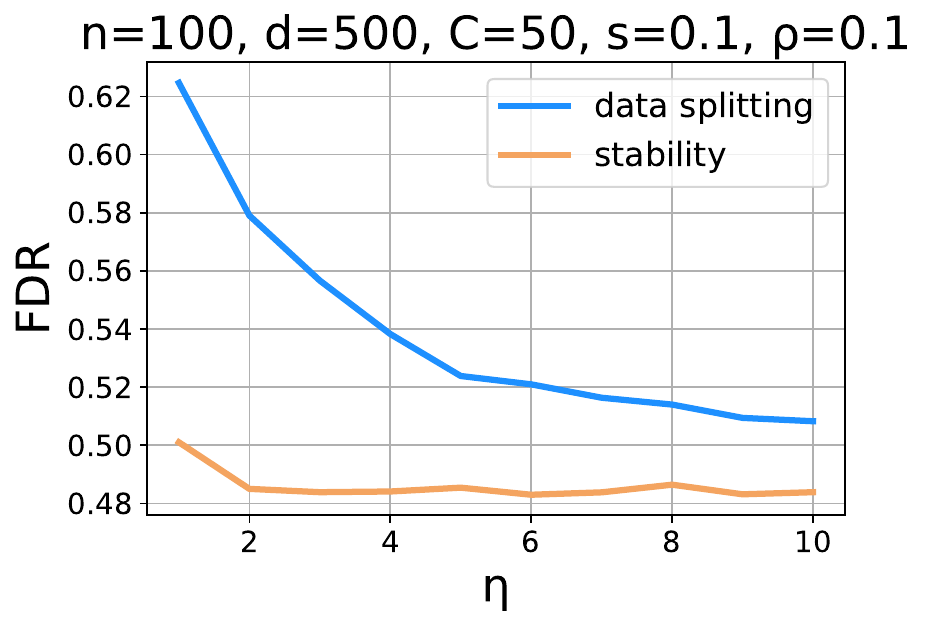}
\includegraphics[width=0.25\textwidth]{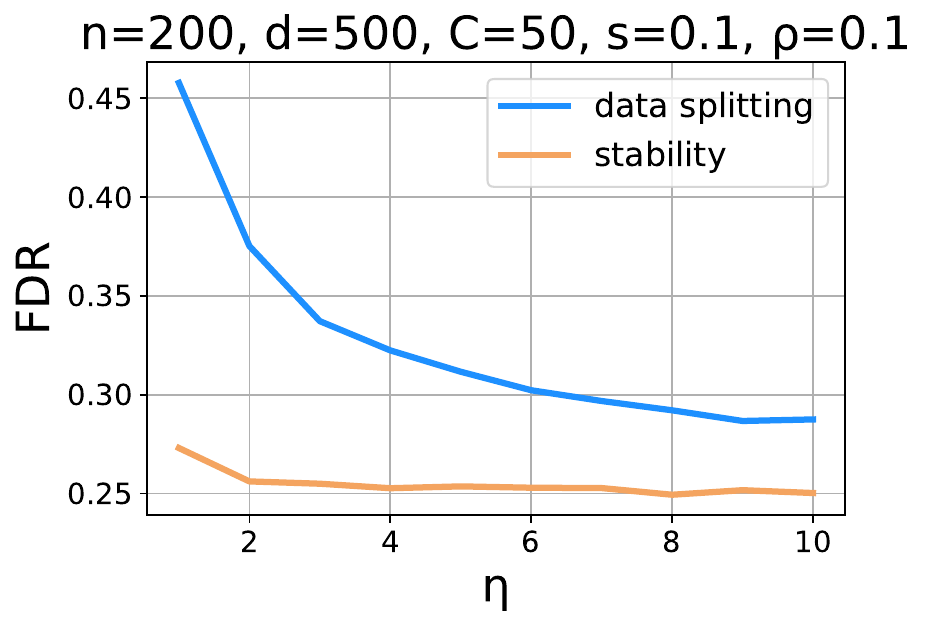}
\includegraphics[width=0.25\textwidth]{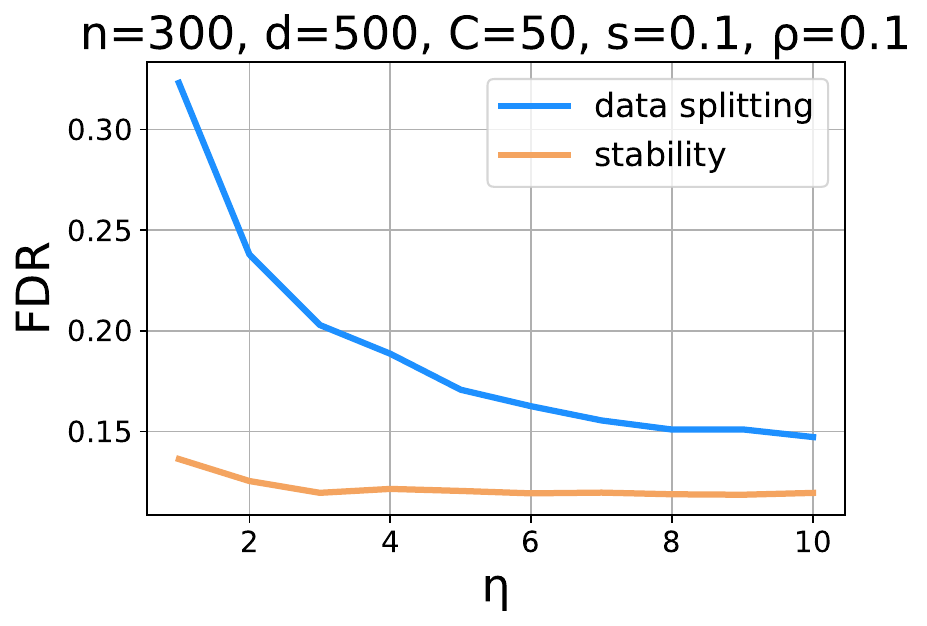}
\includegraphics[width=0.25\textwidth]{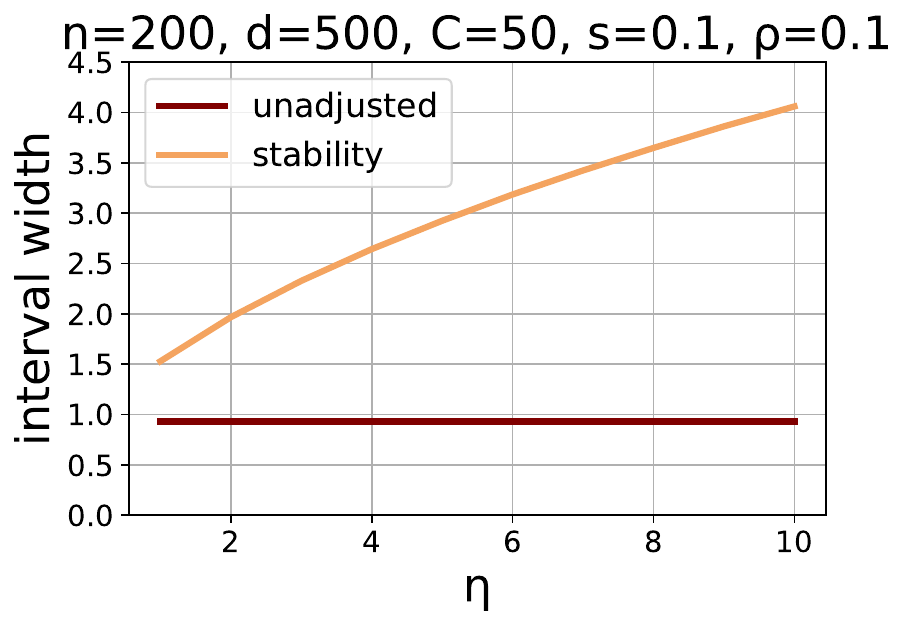}
}
\caption{Comparison of FDR after stable LASSO and LASSO with data splitting, with varying sample size, in the Bernoulli design case. In addition, we plot the average interval width at $n=200$ and the average unadjusted width.} 
\label{fig:lasso_highdim_bern}
\end{figure}

\paragraph{Marginal screening.}
In Figure \ref{fig:screening_comparison} we compare the average error of stable marginal screening and marginal screening with data splitting. Since marginal screening explicitly aims to maximize the values $|X_i^\top y|$ for selected variables $X_i$, we quantify the error as $\frac{1}{k}\sum_{t=1}^k (|X_{i_t^*}^\top y| - |X_{i_t}^\top y|)$, where $i_t$ is the estimated index of the $t$-th largest absolute inner product (based on a subsample in the case of data splitting, or based on a randomized sample in the case of stability), and $i_t^*$ is the true index of the $t$-th largest absolute inner product in the data set. We vary the parameters as in the LASSO comparison in Figure \ref{fig:lasso_comparison}, only instead of varying $C_1$ we vary $k\in\{5,10,20\}$. We also plot the average interval width with stability, together with the unadjusted interval width and the average width obtained via the conditional method of Lee and Taylor~\cite{lee2014exact} with no randomization. For the conditional method, since the intervals are sometimes orders of magnitude larger than the average width, we also plot the $90\%$ quantile of interval width. We see that stability typically outperforms data splitting in terms of the average error, and this benefit is more pronounced for larger $\eta$ and signal strength. In terms of interval width, we observe that stability leads to significantly smaller intervals than the conditional approach. We only plot interval width when $\rho=0.2$, and defer the remaining plots to the Supplement.

In Figure \ref{fig:ms_highdim} we consider a setting analogous to that of Figure \ref{fig:lasso_highdim}, and we analogously vary the sample size $n$. We again see that stability generally dominates data splitting. Moreover, the gap between the intervals obtained via stability and those of Lee and Taylor~\cite{lee2014exact} is even more pronounced than in Figure~\ref{fig:screening_comparison}. We provide the plots of all interval widths in the Supplement.

\subsection{Bernoulli design} Now we consider the Bernoulli design case. The motivation for considering a sparse Bernoulli design lies in the fact that certain directions in the column space of $X$ are captured by only a few samples, hence missing out on them---as is possible with data splitting---can significantly affect the quality of selection.

\begin{figure}[t]
\centerline{\includegraphics[width=0.25\textwidth]{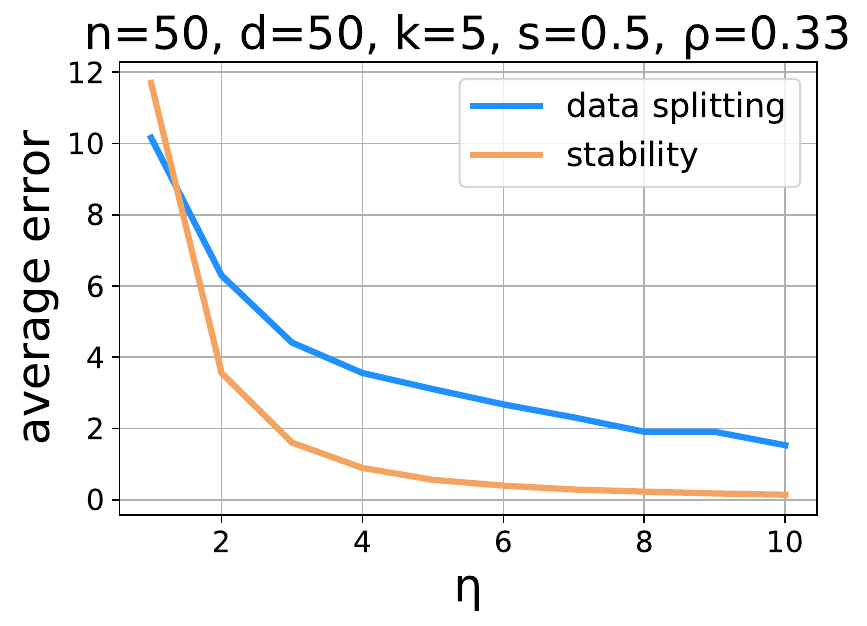}
\includegraphics[width=0.25\textwidth]{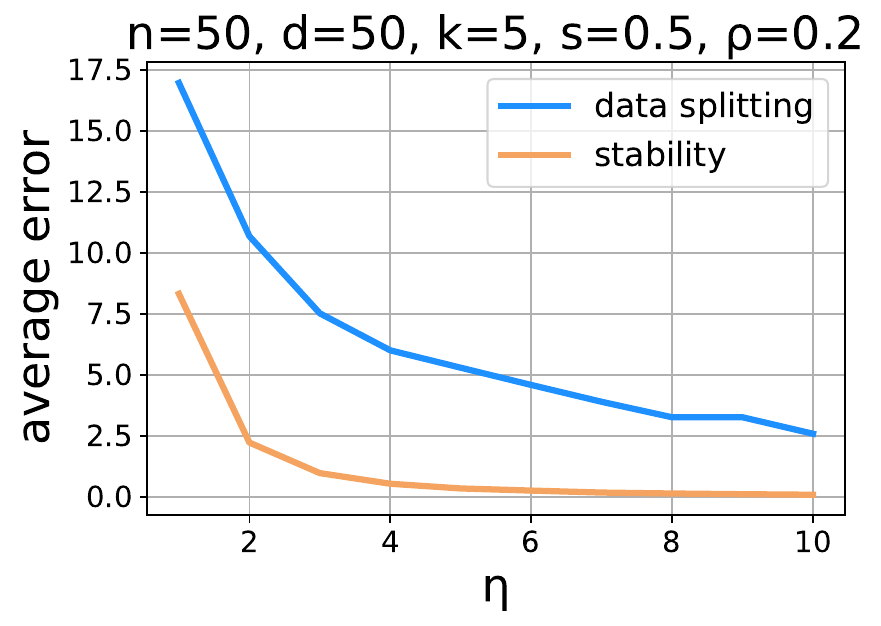}
\includegraphics[width=0.25\textwidth]{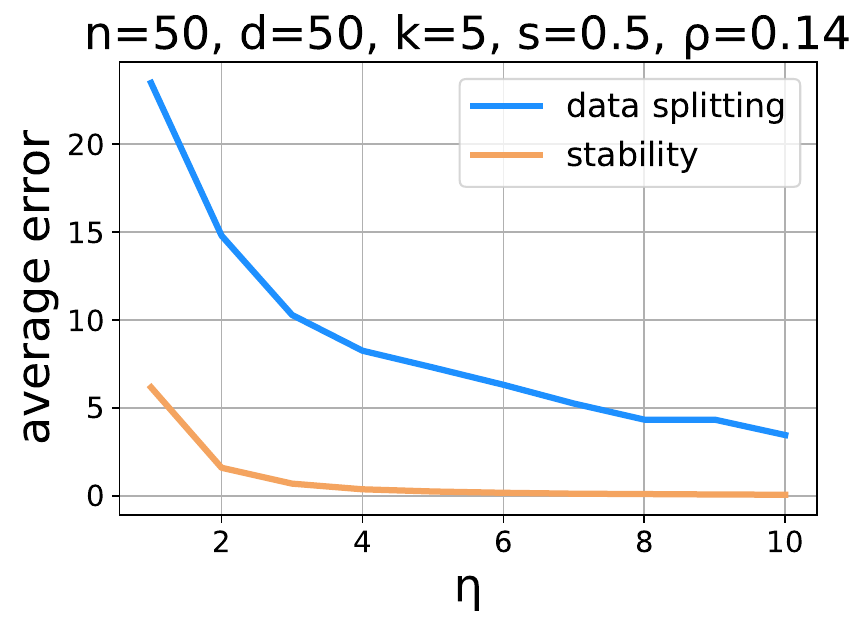}
\includegraphics[width=0.25\textwidth]{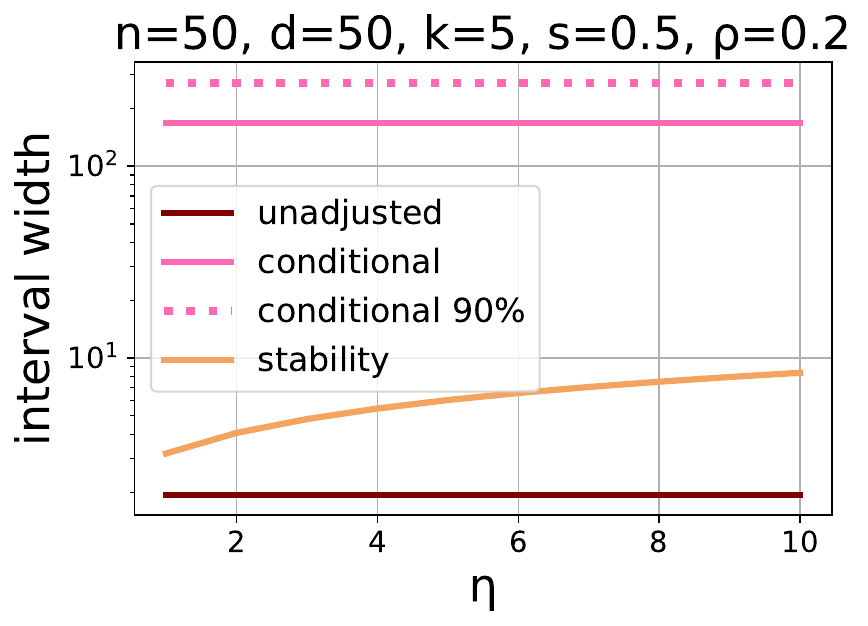}
}
\centerline{
\includegraphics[width=0.25\textwidth]{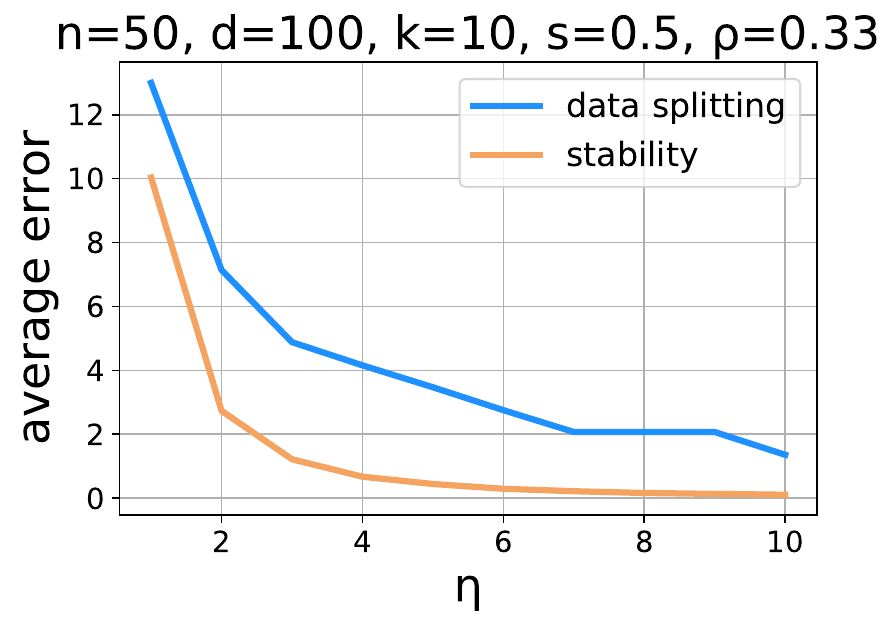}
\includegraphics[width=0.25\textwidth]{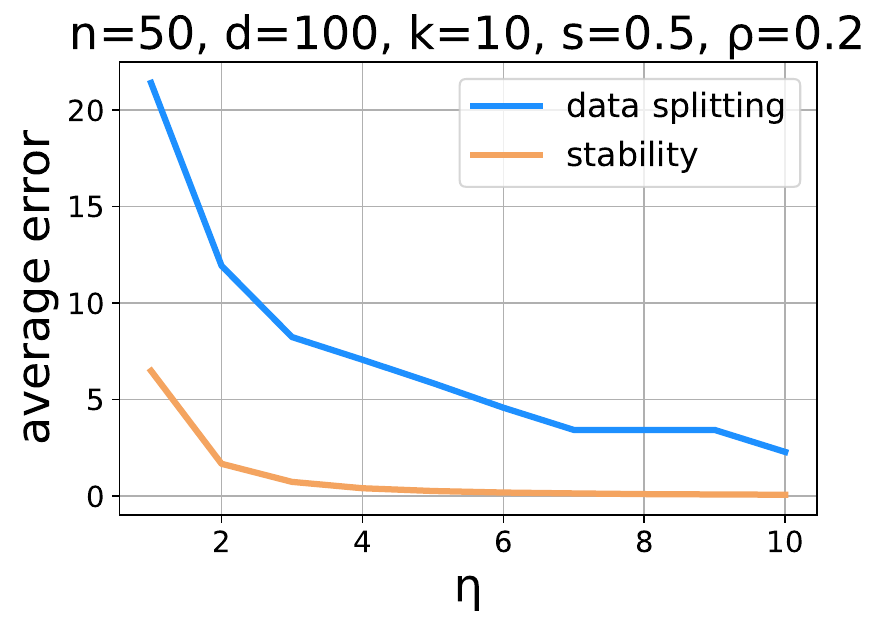}
\includegraphics[width=0.25\textwidth]{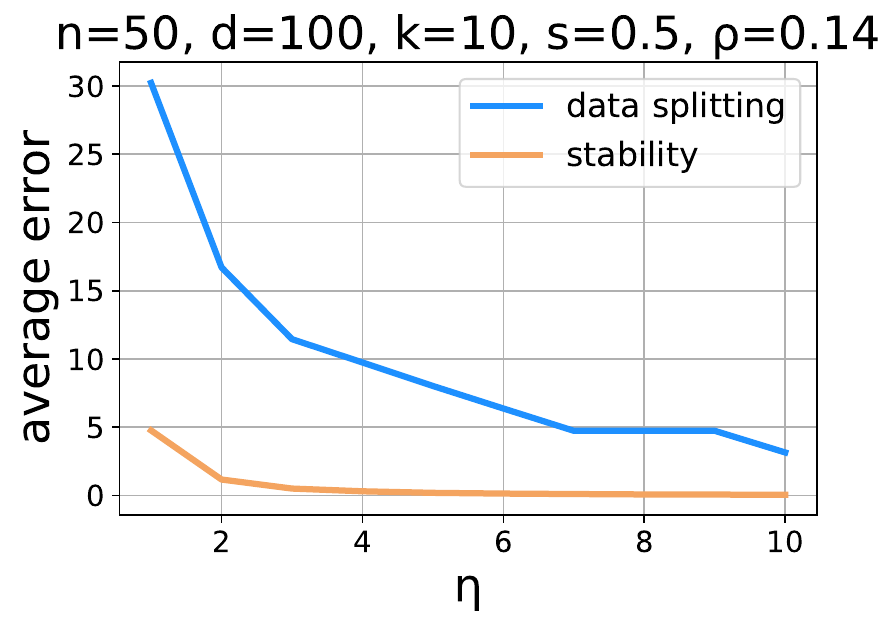}
\includegraphics[width=0.25\textwidth]{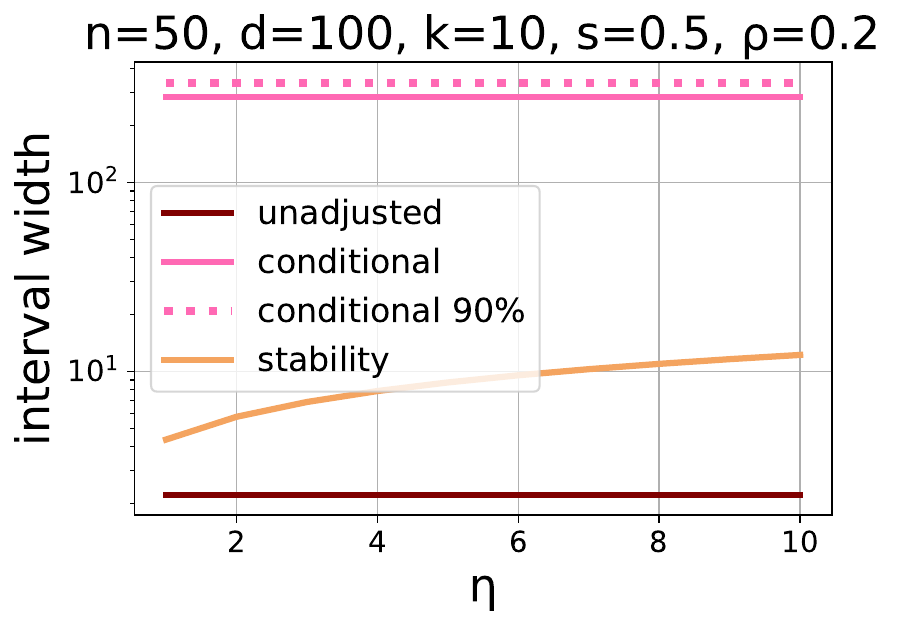}
}
\centerline{\includegraphics[width=0.25\textwidth]{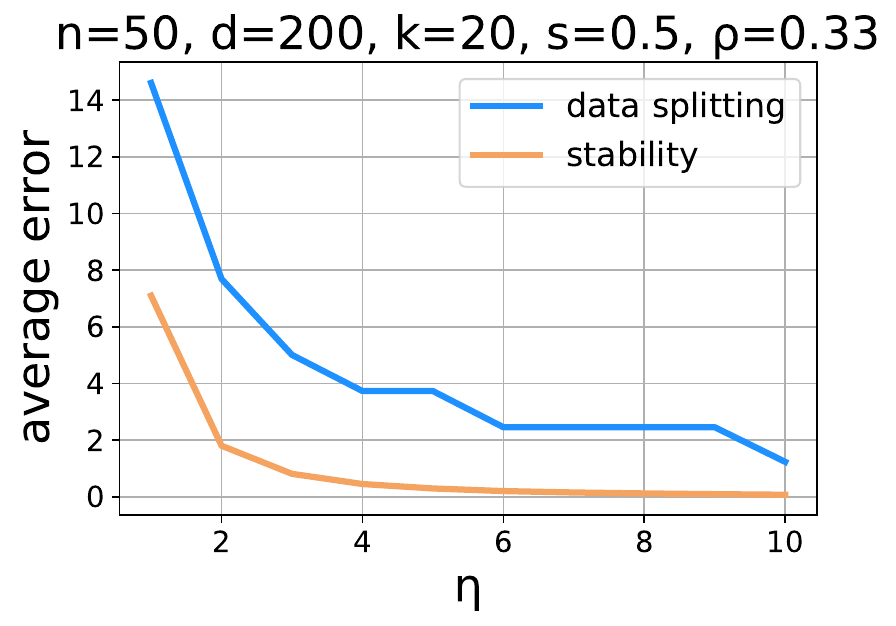}
\includegraphics[width=0.25\textwidth]{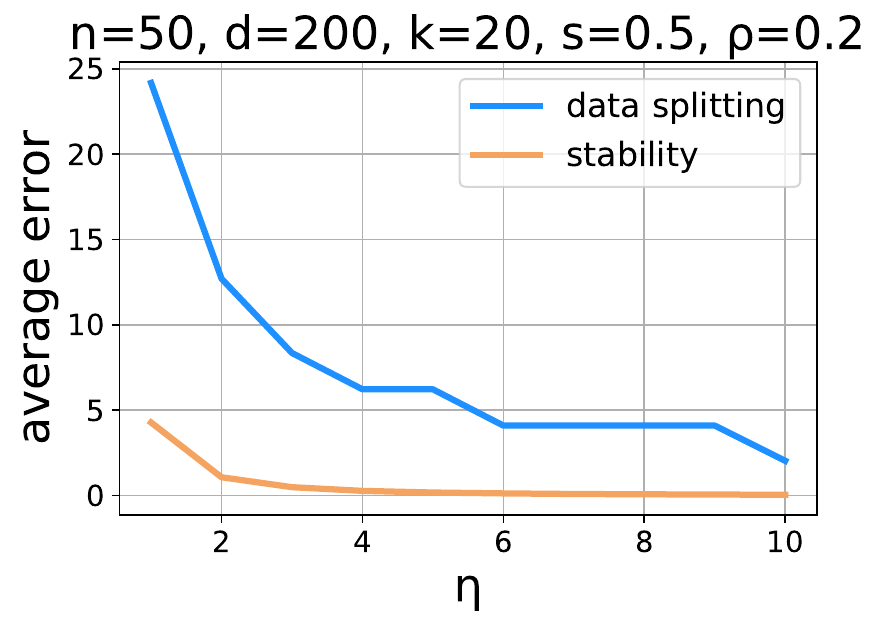}
\includegraphics[width=0.25\textwidth]{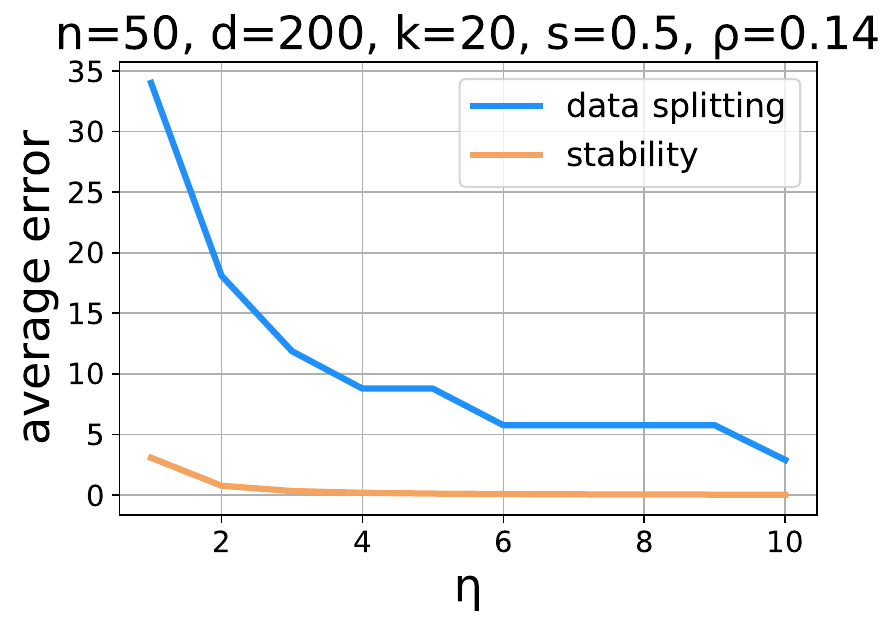}
\includegraphics[width=0.25\textwidth]{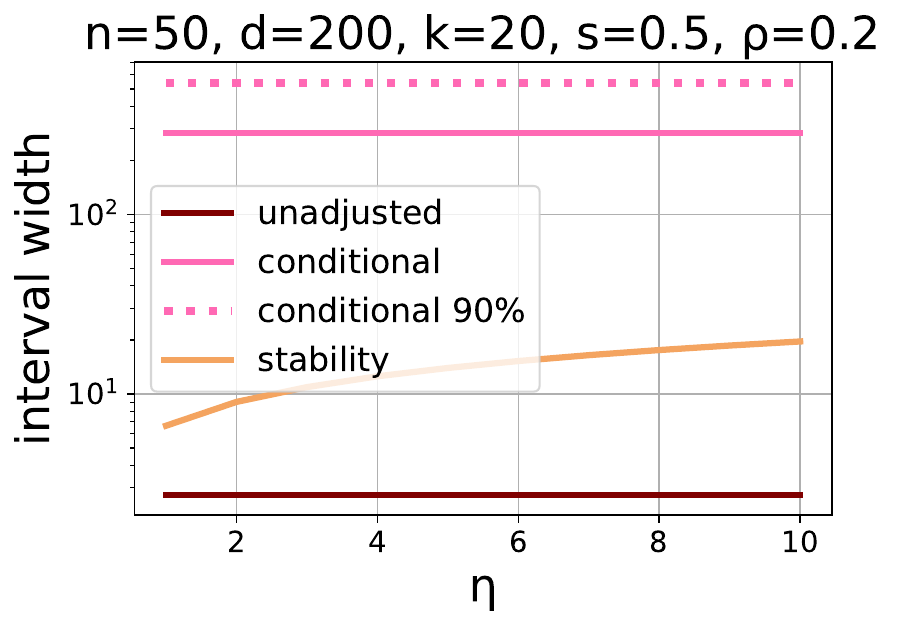}
}
\caption{Comparison of average error after stable marginal screening and marginal screening with data splitting, with varying dimension and signal strength, in the Bernoulli design case. In addition, we plot the average interval width (at $\rho=0.2$ only, however the width varies minimally with $\rho$), together with the average unadjusted width and the width obtained via the conditional correction of Lee and Taylor~\cite{lee2014exact}. We also plot the $90\%$ quantile of the conditional width because it varies greatly across realizations. Since the conditional widths are of a higher order of magnitude, the scale on the $y$-axis in the widths plots is logarithmic.} 
\label{fig:screening_comparison_bern}
\end{figure}

\begin{figure}[t]
\centerline{\includegraphics[width=0.25\textwidth]{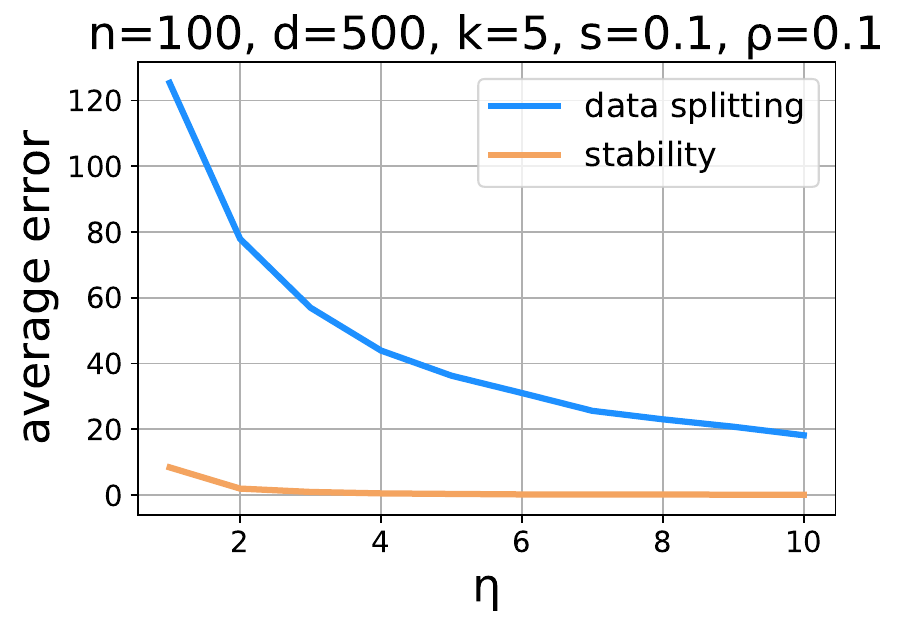}
\includegraphics[width=0.25\textwidth]{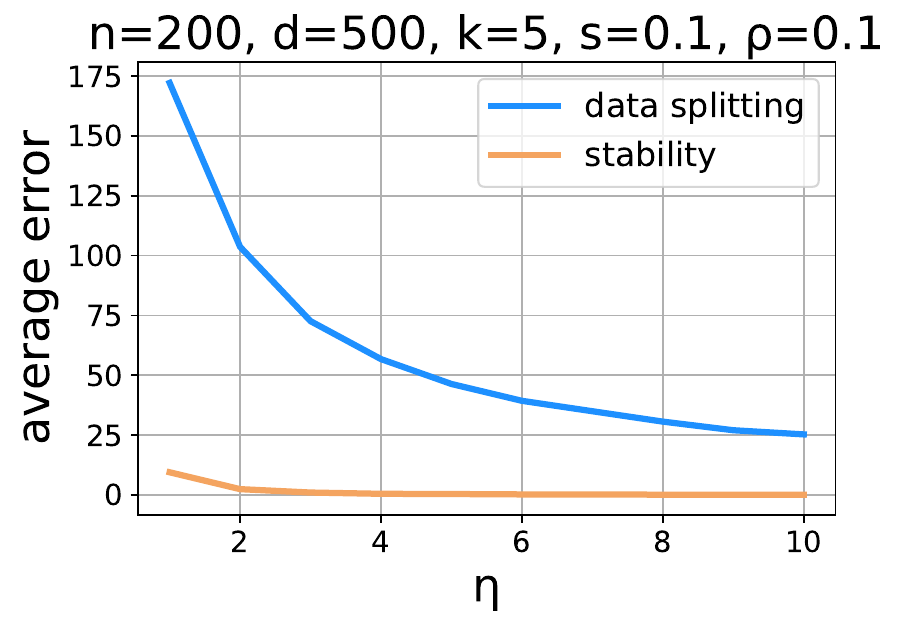}
\includegraphics[width=0.25\textwidth]{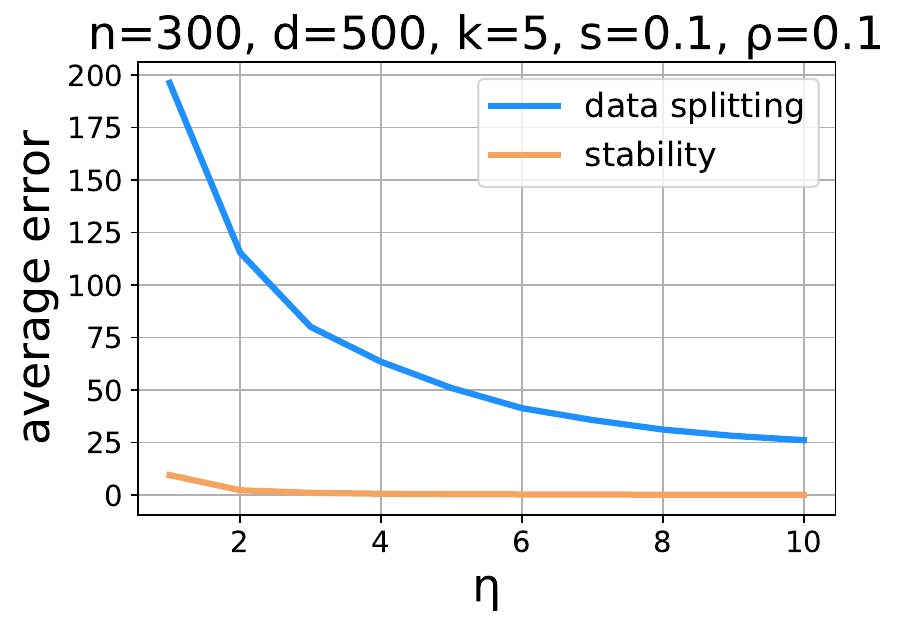}
\includegraphics[width=0.25\textwidth]{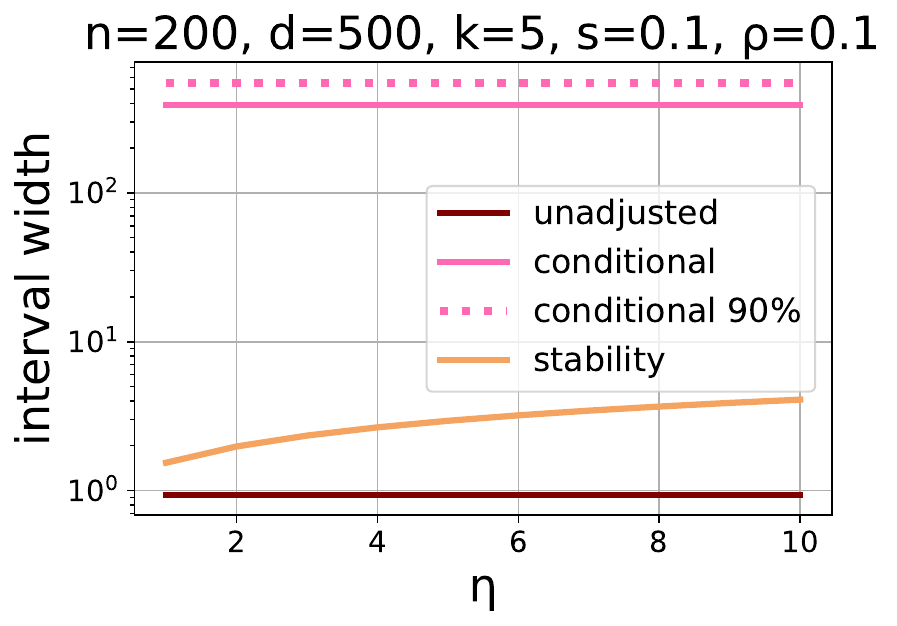}
}
\caption{Comparison of average error after stable marginal screening and marginal screening with data splitting, with varying sample size, in the Bernoulli design case. In addition, we plot the average interval width at $n=200$, together with the average unadjusted width and the width implied by the conditional approach of Lee and Taylor~\cite{lee2014exact}. Since the conditional widths are of a higher order of magnitude, the scale on the $y$-axis in the widths plot is logarithmic.} 
\label{fig:ms_highdim_bern}
\end{figure}

\paragraph{LASSO.}
In Figure \ref{fig:lasso_comparison_bern} and Figure \ref{fig:lasso_highdim_bern} we provide comparisons analogous to those of Figure~\ref{fig:lasso_comparison} and Figure~\ref{fig:lasso_highdim}, using the same parameter configurations. We observe a larger gap between data splitting and stability than in the Gaussian design case, and observe the same trends: as $\eta$ and the signal strength grow, the performance gap increases. As before, we defer the remaining plots of interval widths to the Supplement.

\paragraph{Marginal screening.}
In Figure \ref{fig:screening_comparison_bern} and Figure \ref{fig:ms_highdim_bern} we provide comparisons analogous to those of Figure~\ref{fig:screening_comparison} and Figure~\ref{fig:ms_highdim}, using the same parameter configurations. We observe a larger gap between data splitting and stability both than in the Gaussian design case, as well as in the LASSO experiments using the Bernoulli design. In addition, we observe an even more pronounced gap between stable confidence interval widths and widths of intervals obtained via a conditional correction \cite{lee2014exact}. For this reason, the $y$-axis in the widths plots is logarithmic. We defer the remaining plots of interval widths to the Supplement.

\section{Discussion}
\label{sec:discussion}

Building on concepts from algorithmic stability, as originally developed for applications in differential privacy, we have provided general theory for designing post-selection confidence intervals when the selection procedure is stable. The stability principle is broadly applicable, ranging from inference on the winning effect to model selection in linear regression. In particular, stability is applicable even when data splitting is not, such as when there are dependencies between observations.

Performing inference after a stable selection is simple: it merely requires discounting the type I error based on the level of stability. Moreover, stability comes with several practically appealing properties, namely robustness to post-processing and composition. Thus, for example, the statistician can run various selection methods, and essentially only needs to keep track of the stability parameters of each in order to obtain valid confidence intervals for the final target, which could combine the results of all the selections in an arbitrary way.


There are numerous other potential applications of algorithmic stability to the problem of post-selection inference that would be worthwhile to explore. For example, it would be valuable to understand bootstrapping \cite{rinaldo2019bootstrapping} from the perspective of stability, due to its conceptual relations to the ``privacy amplification by subsampling'' principle in differential privacy, which argues that privacy is amplified when run on a random subsample of the entire data set \cite{kasiviswanathan2011can, beimel2010bounds}. More broadly, selective mechanisms have been long analyzed in the context of differential privacy \cite{thakurta2013differentially, lei2018differentially, steinke2017tight, durfee2019practical, dwork2015private}, and we believe that some of these developments could be imported to selective inference via stability.

\section*{Acknowledgements}
We are grateful to Vitaly Feldman, Will Fithian, Moritz Hardt, Arun Kumar Kuchibhotla, and Adam Sealfon for many helpful discussions and feedback which has lead to improvements of this work. In particular, we thank Will Fithian for pointing out the advantages of the oracle definition of stability. This work was supported by the Army Research Office (ARO) under contract W911NF-17-1-0304 as
part of the collaboration between US DOD, UK MOD and UK Engineering and Physical Research Council (EPSRC) under the Multidisciplinary University Research Initiative (MURI).

\bibliography{POSI}

\begin{thebibliography}{56}
\providecommand{\natexlab}[1]{#1}
\providecommand{\url}[1]{\texttt{#1}}
\expandafter\ifx\csname urlstyle\endcsname\relax
  \providecommand{\doi}[1]{doi: #1}\else
  \providecommand{\doi}{doi: \begingroup \urlstyle{rm}\Url}\fi

\bibitem[Andrews et~al.(2019)Andrews, Kitagawa, and
  McCloskey]{andrews2019inference}
Isaiah Andrews, Toru Kitagawa, and Adam McCloskey.
\newblock Inference on winners.
\newblock \emph{National Bureau of Economic Research}, 2019.

\bibitem[Bachoc et~al.(2020)Bachoc, Preinerstorfer, and
  Steinberger]{bachoc2020uniformly}
Fran{\c{c}}ois Bachoc, David Preinerstorfer, and Lukas Steinberger.
\newblock Uniformly valid confidence intervals post-model-selection.
\newblock \emph{Annals of Statistics}, 48\penalty0 (1):\penalty0 440--463,
  2020.

\bibitem[Barber et~al.(2021)Barber, Candes, Ramdas, and
  Tibshirani]{barber2019predictive}
Rina~Foygel Barber, Emmanuel~J Candes, Aaditya Ramdas, and Ryan~J Tibshirani.
\newblock Predictive inference with the jackknife+.
\newblock \emph{Annals of Statistics}, 49\penalty0 (1):\penalty0 486--507,
  2021.

\bibitem[Bassily and Freund(2016)]{bassily2016typical}
Raef Bassily and Yoav Freund.
\newblock Typical stability.
\newblock \emph{arXiv preprint arXiv:1604.03336}, 2016.

\bibitem[Bassily et~al.(2016)Bassily, Nissim, Smith, Steinke, Stemmer, and
  Ullman]{bassily2016algorithmic}
Raef Bassily, Kobbi Nissim, Adam Smith, Thomas Steinke, Uri Stemmer, and
  Jonathan Ullman.
\newblock Algorithmic stability for adaptive data analysis.
\newblock In \emph{Proceedings of the 48th Annual ACM Symposium on Theory of
  Computing (STOC)}, pages 1046--1059, 2016.

\bibitem[Beimel et~al.(2010)Beimel, Kasiviswanathan, and
  Nissim]{beimel2010bounds}
Amos Beimel, Shiva~Prasad Kasiviswanathan, and Kobbi Nissim.
\newblock Bounds on the sample complexity for private learning and private data
  release.
\newblock In \emph{Theory of Cryptography Conference}, pages 437--454.
  Springer, 2010.

\bibitem[Benjamini(2010)]{benjamini2010simultaneous}
Yoav Benjamini.
\newblock Simultaneous and selective inference: Current successes and future
  challenges.
\newblock \emph{Biometrical Journal}, 52\penalty0 (6):\penalty0 708--721, 2010.

\bibitem[Benjamini et~al.(2019)Benjamini, Hechtlinger, and
  Stark]{benjamini2019confidence}
Yoav Benjamini, Yotam Hechtlinger, and Philip~B Stark.
\newblock Confidence intervals for selected parameters.
\newblock \emph{arXiv preprint arXiv:1906.00505}, 2019.

\bibitem[Berk et~al.(2013)Berk, Brown, Buja, Zhang, and Zhao]{berk2013valid}
Richard Berk, Lawrence Brown, Andreas Buja, Kai Zhang, and Linda Zhao.
\newblock Valid post-selection inference.
\newblock \emph{Annals of Statistics}, 41\penalty0 (2):\penalty0 802--837,
  2013.

\bibitem[Bi et~al.(2020)Bi, Markovic, Xia, and Taylor]{bi2020inferactive}
Nan Bi, Jelena Markovic, Lucy Xia, and Jonathan Taylor.
\newblock Inferactive data analysis.
\newblock \emph{Scandinavian Journal of Statistics}, 47\penalty0 (1):\penalty0
  212--249, 2020.

\bibitem[Bousquet and Elisseeff(2002)]{bousquet2002stability}
Olivier Bousquet and Andr{\'e} Elisseeff.
\newblock Stability and generalization.
\newblock \emph{Journal of Machine Learning Research}, 2\penalty0
  (Mar):\penalty0 499--526, 2002.

\bibitem[Buja et~al.(2019)Buja, Brown, Berk, George, Pitkin, Traskin, Zhang,
  and Zhao]{buja2019models}
Andreas Buja, Lawrence Brown, Richard Berk, Edward George, Emil Pitkin, Mikhail
  Traskin, Kai Zhang, and Linda Zhao.
\newblock Models as approximations {I}: Consequences illustrated with linear
  regression.
\newblock \emph{Statistical Science}, 34\penalty0 (4):\penalty0 523--544, 2019.

\bibitem[Bun and Steinke(2016)]{bun2016concentrated}
Mark Bun and Thomas Steinke.
\newblock Concentrated differential privacy: Simplifications, extensions, and
  lower bounds.
\newblock In \emph{Theory of Cryptography Conference}, pages 635--658.
  Springer, 2016.

\bibitem[Clarkson(2010)]{clarkson2010coresets}
Kenneth~L Clarkson.
\newblock Coresets, sparse greedy approximation, and the {Frank-Wolfe}
  algorithm.
\newblock \emph{ACM Transactions on Algorithms}, 6\penalty0 (4):\penalty0
  1--30, 2010.

\bibitem[Cummings et~al.(2016)Cummings, Ligett, Nissim, Roth, and
  Wu]{cummings2016adaptive}
Rachel Cummings, Katrina Ligett, Kobbi Nissim, Aaron Roth, and Zhiwei~Steven
  Wu.
\newblock Adaptive learning with robust generalization guarantees.
\newblock In \emph{Conference on Learning Theory (COLT)}, pages 772--814, 2016.

\bibitem[Durfee and Rogers(2019)]{durfee2019practical}
David Durfee and Ryan~M Rogers.
\newblock Practical differentially private top-k selection with
  pay-what-you-get composition.
\newblock In \emph{Advances in Neural Information Processing Systems
  (NeurIPS)}, pages 3532--3542, 2019.

\bibitem[Dwork and Roth(2014)]{dwork2014algorithmic}
Cynthia Dwork and Aaron Roth.
\newblock \emph{{The Algorithmic Foundations of Differential Privacy}},
  volume~9.
\newblock Now Publishers, Inc., 2014.

\bibitem[Dwork et~al.(2006)Dwork, McSherry, Nissim, and
  Smith]{dwork2006calibrating}
Cynthia Dwork, Frank McSherry, Kobbi Nissim, and Adam Smith.
\newblock Calibrating noise to sensitivity in private data analysis.
\newblock In \emph{Theory of Cryptography Conference}, pages 265--284.
  Springer, 2006.

\bibitem[Dwork et~al.(2010)Dwork, Rothblum, and Vadhan]{dwork2010boosting}
Cynthia Dwork, Guy~N Rothblum, and Salil Vadhan.
\newblock Boosting and differential privacy.
\newblock In \emph{IEEE 51st Annual Symposium on Foundations of Computer
  Science (FOCS)}, pages 51--60, 2010.

\bibitem[Dwork et~al.(2015{\natexlab{a}})Dwork, Feldman, Hardt, Pitassi,
  Reingold, and Roth]{dwork2015generalization}
Cynthia Dwork, Vitaly Feldman, Moritz Hardt, Toni Pitassi, Omer Reingold, and
  Aaron Roth.
\newblock Generalization in adaptive data analysis and holdout reuse.
\newblock In \emph{Advances in Neural Information Processing Systems (NIPS)},
  pages 2350--2358, 2015{\natexlab{a}}.

\bibitem[Dwork et~al.(2015{\natexlab{b}})Dwork, Feldman, Hardt, Pitassi,
  Reingold, and Roth]{dwork2015preserving}
Cynthia Dwork, Vitaly Feldman, Moritz Hardt, Toniann Pitassi, Omer Reingold,
  and Aaron~Leon Roth.
\newblock Preserving statistical validity in adaptive data analysis.
\newblock In \emph{Proceedings of the 47th Annual ACM Symposium on Theory of
  Computing (STOC)}, pages 117--126, 2015{\natexlab{b}}.

\bibitem[Dwork et~al.(2015{\natexlab{c}})Dwork, Su, and
  Zhang]{dwork2015private}
Cynthia Dwork, Weijie Su, and Li~Zhang.
\newblock Private false discovery rate control.
\newblock \emph{arXiv preprint arXiv:1511.03803}, 2015{\natexlab{c}}.

\bibitem[Fan and Lv(2008)]{fan2008sure}
Jianqing Fan and Jinchi Lv.
\newblock Sure independence screening for ultrahigh dimensional feature space.
\newblock \emph{Journal of the Royal Statistical Society: Series B (Statistical
  Methodology)}, 70\penalty0 (5):\penalty0 849--911, 2008.

\bibitem[Fithian et~al.(2014)Fithian, Sun, and Taylor]{fithian2014optimal}
William Fithian, Dennis Sun, and Jonathan Taylor.
\newblock Optimal inference after model selection.
\newblock \emph{arXiv preprint arXiv:1410.2597}, 2014.

\bibitem[Frank and Wolfe(1956)]{frank1956algorithm}
Marguerite Frank and Philip Wolfe.
\newblock An algorithm for quadratic programming.
\newblock \emph{Naval Research Logistics Quarterly}, 3\penalty0 (1-2):\penalty0
  95--110, 1956.

\bibitem[Guyon and Elisseeff(2003)]{guyon2003introduction}
Isabelle Guyon and Andr{\'e} Elisseeff.
\newblock An introduction to variable and feature selection.
\newblock \emph{Journal of Machine Learning Research}, 3\penalty0
  (Mar):\penalty0 1157--1182, 2003.

\bibitem[Jaggi(2013)]{jaggi2013revisiting}
Martin Jaggi.
\newblock Revisiting {Frank-Wolfe}: Projection-free sparse convex optimization.
\newblock In \emph{Proceedings of the 30th International Conference on Machine
  Learning}, pages 427--435, 2013.

\bibitem[Kasiviswanathan et~al.(2011)Kasiviswanathan, Lee, Nissim,
  Raskhodnikova, and Smith]{kasiviswanathan2011can}
Shiva~Prasad Kasiviswanathan, Homin~K Lee, Kobbi Nissim, Sofya Raskhodnikova,
  and Adam Smith.
\newblock What can we learn privately?
\newblock \emph{SIAM Journal on Computing}, 40\penalty0 (3):\penalty0 793--826,
  2011.

\bibitem[Kivaranovic and Leeb(2020{\natexlab{a}})]{kivaranovic2018expected}
Danijel Kivaranovic and Hannes Leeb.
\newblock On the length of post-model-selection confidence intervals
  conditional on polyhedral constraints.
\newblock \emph{Journal of the American Statistical Association}, pages 1--13,
  2020{\natexlab{a}}.

\bibitem[Kivaranovic and Leeb(2020{\natexlab{b}})]{kivaranovic2020tight}
Danijel Kivaranovic and Hannes Leeb.
\newblock A (tight) upper bound for the length of confidence intervals with
  conditional coverage.
\newblock \emph{arXiv preprint arXiv:2007.12448}, 2020{\natexlab{b}}.

\bibitem[Kuchibhotla et~al.(2019)Kuchibhotla, Brown, Buja, and
  Cai]{kuchibhotla2019all}
Arun~K Kuchibhotla, Lawrence~D Brown, Andreas Buja, and Junhui Cai.
\newblock All of linear regression.
\newblock \emph{arXiv preprint arXiv:1910.06386}, 2019.

\bibitem[Kuchibhotla et~al.(2020{\natexlab{a}})Kuchibhotla, Brown, Buja, Cai,
  George, and Zhao]{kuchibhotlavalid}
Arun~K Kuchibhotla, Lawrence~D Brown, Andreas Buja, Junhui Cai, Edward~I
  George, and Linda~H Zhao.
\newblock Valid post-selection inference in model-free linear regression.
\newblock \emph{Annals of Statistics}, 48\penalty0 (5):\penalty0 2953--2981,
  2020{\natexlab{a}}.

\bibitem[Kuchibhotla et~al.(2020{\natexlab{b}})Kuchibhotla, Rinaldo, and
  Wasserman]{kuchibhotla2020berry}
Arun~Kumar Kuchibhotla, Alessandro Rinaldo, and Larry Wasserman.
\newblock Berry-{E}sseen bounds for projection parameters and partial
  correlations with increasing dimension.
\newblock \emph{arXiv preprint arXiv:2007.09751}, 2020{\natexlab{b}}.

\bibitem[Lee and Taylor(2014)]{lee2014exact}
Jason~D Lee and Jonathan~E Taylor.
\newblock Exact post model selection inference for marginal screening.
\newblock In \emph{Advances in Neural Information Processing Systems (NIPS)},
  pages 136--144, 2014.

\bibitem[Lee et~al.(2016)Lee, Sun, Sun, and Taylor]{lee2016exact}
Jason~D Lee, Dennis~L Sun, Yuekai Sun, and Jonathan~E Taylor.
\newblock Exact post-selection inference, with application to the lasso.
\newblock \emph{Annals of Statistics}, 44\penalty0 (3):\penalty0 907--927,
  2016.

\bibitem[Lei et~al.(2018)Lei, Charest, Slavkovic, Smith, and
  Fienberg]{lei2018differentially}
Jing Lei, Anne~Sophie Charest, Aleksandra Slavkovic, Adam Smith, and Stephen
  Fienberg.
\newblock Differentially private model selection with penalized and constrained
  likelihood.
\newblock \emph{Journal of the Royal Statistical Society. Series A: Statistics
  in Society}, 181\penalty0 (3):\penalty0 609--633, 2018.

\bibitem[Liu et~al.(2018)Liu, Markovic, and Tibshirani]{liu2018more}
Keli Liu, Jelena Markovic, and Robert Tibshirani.
\newblock More powerful post-selection inference, with application to the
  {L}asso.
\newblock \emph{arXiv preprint arXiv:1801.09037}, 2018.

\bibitem[Markovic and Taylor(2016)]{markovic2016bootstrap}
Jelena Markovic and Jonathan Taylor.
\newblock Bootstrap inference after using multiple queries for model selection.
\newblock \emph{arXiv preprint arXiv:1612.07811}, 2016.

\bibitem[Panigrahi and Taylor(2019)]{panigrahi2019approximate}
Snigdha Panigrahi and Jonathan Taylor.
\newblock Approximate selective inference via maximum likelihood.
\newblock \emph{arXiv preprint arXiv:1902.07884}, 2019.

\bibitem[Panigrahi et~al.(2017)Panigrahi, Markovic, and
  Taylor]{panigrahi2017mcmc}
Snigdha Panigrahi, Jelena Markovic, and Jonathan Taylor.
\newblock An {MCMC}-free approach to post-selective inference.
\newblock \emph{arXiv preprint arXiv:1703.06154}, 2017.

\bibitem[Rasines and Young(2021)]{rasines2021splitting}
Daniel~G Rasines and G~Alastair Young.
\newblock Splitting strategies for post-selection inference.
\newblock \emph{arXiv preprint arXiv:2102.02159}, 2021.

\bibitem[Rinaldo et~al.(2019)Rinaldo, Wasserman, and
  G'Sell]{rinaldo2019bootstrapping}
Alessandro Rinaldo, Larry Wasserman, and Max G'Sell.
\newblock Bootstrapping and sample splitting for high-dimensional,
  assumption-lean inference.
\newblock \emph{Annals of Statistics}, 47\penalty0 (6):\penalty0 3438--3469,
  2019.

\bibitem[Rogers et~al.(2016)Rogers, Roth, Smith, and Thakkar]{rogers2016max}
Ryan Rogers, Aaron Roth, Adam Smith, and Om~Thakkar.
\newblock Max-information, differential privacy, and post-selection hypothesis
  testing.
\newblock In \emph{IEEE 57th Annual Symposium on Foundations of Computer
  Science (FOCS)}, pages 487--494, 2016.

\bibitem[Rosenthal(1979)]{rosenthal1979file}
Robert Rosenthal.
\newblock The file drawer problem and tolerance for null results.
\newblock \emph{Psychological Bulletin}, 86\penalty0 (3):\penalty0 638, 1979.

\bibitem[Russo and Zou(2016)]{russo2016controlling}
Daniel Russo and James Zou.
\newblock Controlling bias in adaptive data analysis using information theory.
\newblock In \emph{Artificial Intelligence and Statistics}, pages 1232--1240,
  2016.

\bibitem[Scheffe(1999)]{scheffe1999analysis}
Henry Scheffe.
\newblock \emph{The Analysis of Variance}.
\newblock John Wiley \& Sons, 1999.

\bibitem[Steinberger and Leeb(2018)]{steinberger2018conditional}
Lukas Steinberger and Hannes Leeb.
\newblock Conditional predictive inference for high-dimensional stable
  algorithms.
\newblock \emph{arXiv preprint arXiv:1809.01412}, 2018.

\bibitem[Steinke and Ullman(2017)]{steinke2017tight}
Thomas Steinke and Jonathan Ullman.
\newblock Tight lower bounds for differentially private selection.
\newblock In \emph{IEEE 58th Annual Symposium on Foundations of Computer
  Science (FOCS)}, pages 552--563, 2017.

\bibitem[Talwar et~al.(2015)Talwar, Thakurta, and Zhang]{talwar2015nearly}
Kunal Talwar, Abhradeep~Guha Thakurta, and Li~Zhang.
\newblock Nearly optimal private {LASSO}.
\newblock In \emph{Advances in Neural Information Processing Systems (NIPS)},
  pages 3025--3033, 2015.

\bibitem[Taylor and Tibshirani(2015)]{taylor2015statistical}
Jonathan Taylor and Robert~J Tibshirani.
\newblock Statistical learning and selective inference.
\newblock \emph{Proceedings of the National Academy of Sciences}, 112\penalty0
  (25):\penalty0 7629--7634, 2015.

\bibitem[Thakurta and Smith(2013)]{thakurta2013differentially}
Abhradeep~Guha Thakurta and Adam Smith.
\newblock Differentially private feature selection via stability arguments, and
  the robustness of the lasso.
\newblock In \emph{Conference on Learning Theory (COLT)}, pages 819--850, 2013.

\bibitem[Tian and Taylor(2018)]{tian2018selective}
Xiaoying Tian and Jonathan Taylor.
\newblock Selective inference with a randomized response.
\newblock \emph{Annals of Statistics}, 46\penalty0 (2):\penalty0 679--710,
  2018.

\bibitem[Tian et~al.(2016)Tian, Bi, and Taylor]{tian2016magic}
Xiaoying Tian, Nan Bi, and Jonathan Taylor.
\newblock {MAGIC}: a general, powerful and tractable method for selective
  inference.
\newblock \emph{arXiv preprint arXiv:1607.02630}, 2016.

\bibitem[Tian~Harris et~al.(2016)Tian~Harris, Panigrahi, Markovic, Bi, and
  Taylor]{tian2016selective}
Xiaoying Tian~Harris, Snigdha Panigrahi, Jelena Markovic, Nan Bi, and Jonathan
  Taylor.
\newblock Selective sampling after solving a convex problem.
\newblock \emph{arXiv preprint arXiv:1609.05609}, 2016.

\bibitem[Tibshirani(1996)]{tibshirani1996regression}
Robert Tibshirani.
\newblock Regression shrinkage and selection via the lasso.
\newblock \emph{Journal of the Royal Statistical Society: Series B
  (Methodological)}, 58\penalty0 (1):\penalty0 267--288, 1996.

\bibitem[Tibshirani et~al.(2016)Tibshirani, Taylor, Lockhart, and
  Tibshirani]{tibshirani2016exact}
Ryan~J Tibshirani, Jonathan Taylor, Richard Lockhart, and Robert Tibshirani.
\newblock Exact post-selection inference for sequential regression procedures.
\newblock \emph{Journal of the American Statistical Association}, 111\penalty0
  (514):\penalty0 600--620, 2016.

\end{thebibliography}
\bibliographystyle{plainnat}

\appendix

\section{Proofs}

\subsection{Composition of stability}

Here we summarize prior work on adaptive composition theorems for differential privacy, which will be crucial in our proofs. To facilitate readability, we restate the definition of adaptive composition.

\begin{algorithm}[H]
\SetAlgoLined
\begin{flushleft}
\textbf{input: }data $y\in\R^n$, sequence of algorithms $\A_i:\F_1\times \dots\times \F_{i-1}\times \R^n\rightarrow \F_i, ~i\in[k]$\\
\textbf{output: } $(a_1,\dots,a_k)\in \F_1\times \dots\times \F_k$\newline
\For{$i=1,2,\dots,k$}{
Compute $a_i = \A_i(a_1,\dots,a_{i-1},y)\in \F_i$}
 Return $\A^{(k)}(y) = (a_1,\dots,a_k)$
\end{flushleft}
\caption{Adaptive composition}
\label{alg:comp2}
\end{algorithm}

The first statement below is a reformulation of the ``simple'' composition property of differential privacy~\cite{dwork2006calibrating}. The second statement is a slightly stronger reformulation \cite{bun2016concentrated} of the so-called ``advanced'' composition theorem for differential privacy \cite{dwork2010boosting}.

\begin{lemma}[\cite{dwork2010boosting, bun2016concentrated}]
\label{lemma:fixedinputcomp}
Fix two vectors $y,y'\in\R^n$ and suppose that $\A_t(a_1,\dots,a_{t-1},y)\approx_{\eta,\tau} \A_t(a_1,\dots,a_{t-1},y')$, for every fixed sequence $a_1,\dots,a_{t-1}$, and all $t\in[k]$. Then,
\begin{itemize}
\item[(a)] $\A^{(k)}(y)\approx_{k\eta,k\tau} \A^{(k)}(y')$,
\item[(b)] $\A^{(k)}(y)\approx_{\frac{1}{2}k\eta^2 + \sqrt{2k\log(1/\delta)}\eta,k\tau + \delta} \A^{(k)}(y')$, for all $\delta\in(0,1)$.
\end{itemize}
\end{lemma}

We also state a non-adaptive composition property of stability, which is relevant when running multiple model selection algorithms. It says that the privacy parameters of all outputs combined simply add up.

\begin{lemma}
\label{lemma:nonadaptive_comp}
Suppose $\A_i:\R^n\rightarrow \F_i$ is $(\eta_i,\tau_i,\nu_i	)$-stable, for all $i\in[k]$. Then, $\A^{(k)}:~\R^n\rightarrow \F_1\times \dots \times \F_k$ defined as $\A^{(k)}(\cdot) = (\A_1(\cdot),\dots,\A_k(\cdot))$ is $(\sum_{i=1}^k\eta_i,\sum_{i=1}^k\tau_i,\sum_{i=1}^k \nu_i)$-stable.
\end{lemma}

The proof of Lemma \ref{lemma:nonadaptive_comp} follows by the simple composition theorem for differential privacy \cite{dwork2010boosting}, together with a union bound over all failure events with probability $\nu_i$.

\subsection{Proof of Claim \ref{claim:max_effect}}

Let $E = \{\omega\in\R^n:\|\omega - \mu\|_\infty \leq z_{1-\alpha\delta/(2n)}\}$. Then, by a Bonferroni correction, we know that $\PP{y\in E} \geq 1-\alpha\delta$.

For any $\omega\in\R^n$, let $\hat i_*(\omega) = \argmax_i (\omega_i + \xi_i)$, where $\xi_i\stackrel{\mathrm{i.i.d.}}{\sim}\mathrm{Lap}\left(\frac{2z_{1-\alpha\delta/(2n)}}{\eta}\right)$. Define
$$\xi^* = \argmin_{\xi} y_i + \xi_i > y_j + \xi_j,~\forall j \neq i.$$
Then, $\hat i_*(y) = i$ if and only if $\xi_i \geq \xi^*$. By the definition of $E$ we have that for all $y\in E$
\begin{align*}
z_{1-\alpha\delta/(2n)} + \mu_i + \xi^* \geq y_i + \xi^* > y_j + \xi_j \geq \mu_j + \xi_j - z_{1-\alpha\delta/(2n)},
\end{align*}
for all $j\neq i$. Rearranging the terms, we get
$$2z_{1-\alpha\delta/(2n)} + \mu_i + \xi^* \geq \mu_j + \xi_j.$$
Thus, if $\xi_{i} \geq \xi^* +2z_{1-\alpha\delta/(2n)}$, then $\hat i_*(\mu) = i$ if the noise levels are $(\xi_{1},\dots,\xi_{i},\dots,\xi_{n})$. Putting everything together, for fixed $y\in E$, we have
\begin{align*}
\PPst{\hat i_*(\mu) = i}{\{\xi_{t,j}\}_{j\neq i}} &\geq \PPst{\xi_{i} \geq \xi^* + 2z_{1-\alpha\delta/(2n)}}{\{\xi_{t,j}\}_{j\neq i}}\\
	&\geq e^{-\eta} \PPst{\xi_{i} \geq \xi^*}{\{\xi_{t,j}\}_{j\neq i}}
	= e^{-\eta}\PPst{\hat i_*(y) = i}{\{\xi_{t,j}\}_{j\neq i}}.
\end{align*}
Multiplying by $e^{\eta}$ and applying the law of iterated expectations completes the proof that $\hat i_*(y)\approx \hat i_*(\mu)$ for all $y\in E$, which implies that $\hat i_*(\mu)$ is a valid oracle and $\hat i_*(\cdot)$ is $(\eta,0,\alpha\delta)$-stable.

Finally, we invoke Theorem \ref{thm:conf-ints-general} to complete the proof.

\subsection{Proof of Claim \ref{claim:max_corr}}

A subsequent proof of Proposition \ref{prop:marginal-stability} implies as a special case that $\hat i_*$ is $(\eta,0,\delta\alpha)$-stable. Putting this fact together with Theorem \ref{thm:conf-ints-general} completes the proof.

\subsection{Proof of Claim \ref{claim:vignette_utility1} and Claim \ref{claim:vignette_utility2}}

We prove Claim \ref{claim:vignette_utility1}; the proof of Claim \ref{claim:vignette_utility2} is essentially identical. The event $\{\hat i_* \neq i_*\}$ must imply that at least one of two events: $\{\xi_{i_*} \leq -\frac{\Delta}{2}\}$ or $\{\exists j\neq i_*: \xi_{j} \geq \frac{\Delta}{2}\}$. By a union bound and concentration of Laplace random variables, we can thus conclude
$$\PP{\hat i_*\neq i_*} \leq \frac{d}{2} \exp\left(-\frac{\Delta \eta}{4z_{1-\alpha\delta/(2d)}}\right).$$
Setting this expression to $\delta'$ and substituting for $\eta$ completes the proof.

\subsection{Proof of Lemma \ref{lemma:near-indep}}

Denote by $\hat S_0$ the oracle from Definition \ref{def:stability} and let $E = \{\omega\in \R^n : \hat S(\omega) \approx_{\eta,\tau} \hat S_0\}$. Fix an event $\O\subseteq \R^n\times \mathcal{S}$, and let $\O_\omega = \{S\in \mathcal{S}:(\omega,S)\in\O\}$. Notice that $\one{(y,\hat S(y))\in\O} = \one{\hat S(y)\in\O_y}$, and hence $\E[\one{(y,\hat S(y))\in\O}|y] = \E[\one{\hat S(y)\in\O_y}|y]$.

With this, we can write:
	\begin{align*}
	\PP{(y,\hat S(y))\in \O, y\in E} &= \E[\E[\one{\hat S(y)\in \O_y}| y]\one{y\in E}]\\
	&= \E[\PPst{\hat S(y)\in \O_y}{y}\one{y\in E}]\\
	&\leq \E[(e^{\eta}\PPst{\hat S_0\in \O_y}{y} + \tau)\one{y\in E}]\\
	&= \E[(e^{\eta}\one{\hat S_0\in \O_y} + \tau)\one{y\in E}]\\
	&\leq e^{\eta}\PP{(y,\hat S_0)\in \O, y\in E} + \tau.
	\end{align*}
Since $\PP{y\in E}\geq 1 - \nu$, we can conclude:
\begin{align*}
\PP{(y,\hat S(y))\in \O} &= \PP{(y,\hat S(y))\in \O, y\in E} + \PP{(y,\hat S(y))\in \O, y\not\in E}\\
&\leq \PP{(y,\hat S(y))\in \O, y\in E} + \nu\\
&\leq e^{\eta}\PP{(y,\hat S_0)\in \O, y\in E} + \tau + \nu\\
&\leq e^{\eta}\PP{(y,\hat S_0)\in \O} + \tau + \nu.
\end{align*}

\subsection{Proof of Theorem \ref{thm:conf-ints-general}}

By Lemma \ref{lemma:near-indep}, we know that
\begin{align*}
\PP{\beta_{\hat S} \not\in \ci_{\hat S}^{(\delta e^{-\eta})}} &\leq e^\eta \PP{\beta_{\hat S_0} \not\in \ci_{\hat S_0}^{(\delta e^{-\eta})}} + \tau + \nu\\
&= e^\eta \E\left[\PPst{\beta_{\hat S_0} \not\in \ci_{\hat S_0}^{(\delta e^{-\eta})}}{\hat S_0} \right] + \tau + \nu,
\end{align*}
where $\ci_{\hat S_0}^{(\delta e^{-\eta})}$ are confidence intervals computed on $y$ and $\hat S_0$ is an oracle selection independent of $y$. By the construction of $\ci_{\hat S_0}^{(\delta e^{-\eta})}$, we know $\PPst{\beta_{\hat S_0} \not\in \ci_{\hat S_0}^{(\delta e^{-\eta})}}{\hat S_0}\leq \delta e^{-\eta}$, and therefore
$$\PP{\beta_{\hat S} \not\in \ci_{\hat S}^{(\delta e^{-\eta})}} \leq e^\eta e^{-\eta}\delta + \tau + \nu = \delta + \tau + \nu.$$

\subsection{Proof of Proposition \ref{prop:posi-rate}}

Denote by $\M_s$ the set of all models of size at most $s$ and fix any $\tau\in(0,1)$. Let $y'\sim\P_y$ be an i.i.d. copy of $y$. Define the set of bad models to be
$$\M^* = \left\{M\in \M_s ~:~ \exists \omega^*\in\text{supp}(\P_y) \text{ such that }\frac{\PP{\hat M(\omega^*) = M}}{\PP{\hat M(y') = M}}\geq \frac{\sum_{k=1}^s{d\choose k}}{\tau}\right\}.$$
 
By definition, we see 
$$\PP{\hat M(y') \in \M^*}  \leq \sum_{M\in \M^*}\PP{\hat M(y') = M}\leq  \tau,$$
which follows by taking a union bound over all $\sum_{k=1}^s {d\choose k}$ possible models. Consequently, for any event $\O\subseteq  \R^n \times \M_s$ such that $\{M:\exists \omega \text{ s.t. } (\omega,M)\in\O\}\subseteq \M^*$, we have
$$\PP{(y, \hat M(y))\in\O}\leq \PP{\hat M(y) \in \M^*} = \PP{\hat M(y') \in \M^*} \leq \tau.$$
Now denote $\O_\omega = \{M\in \M_s:(\omega,M)\in\O\}$, and notice that $\{(y, \hat M(y))\in\O\} = \{\hat M(y)\in\O_y\}$. Then, for all $\O\subseteq  \R^n \times \M_s$ such that $\{M:\exists \omega \text{ s.t. } (\omega,M)\in\O\}\cap \M^* = \emptyset$, we know
\begin{align*}
\PP{(y,\hat M(y))\in\O} &= \PP{\hat M(y)\in\O_y} = \E\left[\PPst{\hat M(y)\in\O_y}{y}\right]\\
  &\leq \frac{\sum_{k=1}^s{d\choose k}}{\tau}\E\left[\PPst{\hat M(y')\in\O_y}{y}\right] = \frac{\sum_{k=1}^s{d\choose k}}{\tau}\PP{(y,\hat M(y'))\in\O}.
\end{align*}
Finally, take an arbitrary $\O\subseteq  \R^n \times \M_s$, and partition it as follows:
 $$\O_{\text{bad}}= \{(\omega,M)\in\O: M\in\M^*\}, ~~ \O_{\text{good}}= \{(\omega,M)\in\O: M\not \in\M^*\}.$$
Putting everything together, we have shown
\begin{align*}
\PP{(y,\hat M(y))\in\O} &= \PP{(y,\hat M(y))\in\O_{\text{bad}}} + \PP{(y,\hat M(y))\in\O_{\text{good}}}\\
 &\leq \tau + \frac{\sum_{k=1}^s{d\choose k}}{\tau}\PP{(y,\hat M(y'))\in\O}.
\end{align*}

In other words, we can conclude that $(y,\hat M(y))\approx_{\eta,\tau}(y,\hat M(y'))$, with $\eta= \log\left(\frac{\sum_{k=1}^s{d\choose k}}{\tau}\right) = O(s\log(d/s)) + \log(1/\tau)$, as desired.

Applying the same steps as in Theorem \ref{thm:conf-ints-general} allows us to conclude that $\ci_{j\cdot \hat M}(K_{\hat M,\delta e^{-\eta}}) = \left(\hat \beta_{j\cdot \hat M} \pm K_{\hat M,\delta e^{-\eta}} \hat\sigma_{j\cdot \hat M}\right)$, where $\eta=O(s\log(d/s)) + \log(1/\tau)$, are valid confidence intervals at level $\delta + \tau$.

A related argument is given in Theorem 6 of Dwork et al.~\cite{dwork2015generalization}.

\subsection{Proof of Lemma \ref{lemma:near-indep2}}

An intermediate result in the proof of Lemma \ref{lemma:near-indep} says:
\begin{align*}
	\PP{(y,\hat S(y))\in \O, y\in E} \leq e^{\eta}\PP{(y,\hat S(y_E'))\in \O, y\in E} + \tau,
	\end{align*}
where we set $\hat S_0 = \hat S(y_E')$ by assumption.
If we set $\tau=0$ and normalize both sides by $\PP{y\in E}$, we get
\begin{align}
\label{eqn:near_indep_condE}
\PPst{(y,\hat S(y))\in \O}{y\in E} \leq e^{\eta}\PPst{(y,\hat S(y_E'))\in \O}{y\in E}.
\end{align}
Now we choose $\O = \O_S\times \{S\}$. The left-hand side above can be rewritten as
$$\PPst{y\in \O_S, \hat S(y) = S}{y\in E} = \PPst{y\in \O_S}{\hat S(y) = S, y\in E}\PPst{\hat S(y) = S}{y\in E},$$
while the right-hand side is equal to
\begin{align*}
\PPst{y\in\O_S,\hat S(y_E')=S}{y\in E} &= \PPst{y\in \O_S}{\hat S(y_E')=S, y\in E} \PPst{\hat S(y_E')=S}{y\in E}\\
&= \PPst{y\in \O_S}{ y\in E} \PP{\hat S(y_E')=S},
\end{align*}
where we use the fact that $y$ and $y_E'$ are independent.

Since $\PP{\hat S(y_E')=S} = \PPst{\hat S(y)=S}{y\in E}$ , Eq.~\eqref{eqn:near_indep_condE} is equivalent to
$$\PPst{y\in \O_S}{\hat S(y) = S, y\in E} \leq e^\eta \PPst{y\in \O_S}{ y\in E},$$
which completes the proof.

\subsection{Generalization and proof of Claim \ref{claim:conditional}}

First we state a generalization of Claim \ref{claim:conditional} for general $\P_y$. In the following let $\{\P(\mu)\}_{\mu\in\R}$ be a location family obtained by translating a zero-mean distribution $\P_0$ by $\mu$. Furthermore, denote by $q_\delta$ the $\delta$ quantile of $|Q|, Q\sim\P_0$ and by $\P_0^\nu$ the distribution $\P_0$ restricted to $(-q_{1-\nu}, q_{1-\nu})$:
$$\P_0^\nu(\O) = \frac{\P_0(\O\cap (-q_{1-\nu}, q_{1-\nu}) )}{1-\nu},$$
for all measurable sets $\O$. Finally, we let $|\P_0^\nu|$ be the distribution of $|Q|, Q\sim \P_0^\nu$.

\begin{claim}
Let $y\sim \P_y = \P(\mu)$ for some $\mu\in\R$. Suppose that we apply the selection criterion to $y+\xi$, where $\xi \sim \mathrm{Lap}(b)$ for some user-chosen parameter $b>0$; that is, we report the confidence interval on the event $\mathrm{report}(y+\xi)$. Then, for any user-chosen parameter $\nu\in(0,1)$, we have 
$$\PPst{\mu\not\in(y\pm q_{1-\alpha(1-\nu)e^{-\eta}})}{\mathrm{report}(y+\xi), y\in E} \leq \alpha,$$
where
$$\eta = \frac{q_{1-\nu}}{b} - \log \mathrm{MGF}_{|\P_0^\nu|}\left(-\frac{1}{b}\right),$$
$\mathrm{MGF}_{|\P_0^{\nu}|}$ is the moment-generating function of $|\P_0^\nu|$ and $E$ is an event such that $\PP{y\in E}\geq 1-\nu$.
\end{claim}

\begin{proof}
To prove this, first let $E=\{\omega\in\R^n:|\omega-\mu| \leq q_{1-\nu}\}$ and note that $\PP{y\in E}\geq 1-\nu$. Denote by $p_Q(a)$ the density of a random variable $Q$ at $a$. Letting $Z\sim\P_0^\nu$, we have for all $\omega\in E$,
\begin{align*}
\frac{p_{\omega + \xi}(a)}{p_{\mu + Z + \xi}(a)}	 &= \frac{\frac{1}{2b}\exp\left(-\frac{|a-\omega|}{b}\right)}{\frac{1}{2b}\int_{-\infty}^\infty\exp\left(-\frac{|u-\mu|}{b}\right)p_Z(a-u)du}\\
&\leq \frac{\exp\left(\frac{|\omega-\mu|}{b}\right)}{\int_{-\infty}^\infty \exp\left(-\frac{|a-u|}{b}\right) p_Z(a-u)du}\\
&= \frac{\exp\left(\frac{|\omega-\mu|}{b}\right)}{\int_{-\infty}^\infty \exp\left(-\frac{|u|}{b}\right) p_Z(-u)du}\\
&= \frac{\exp\left(\frac{|\omega-\mu|}{b}\right)}{\int_{0}^\infty \exp\left(-\frac{|u|}{b}\right) (p_Z(u) + p_Z(-u))du}\\
&\leq \frac{\exp\left(\frac{q_{1-\nu}}{b}\right)}{\mathrm{MGF}_{|\P_0^\nu|}\left(-\frac{1}{b}\right)}.
\end{align*}
The value of $\eta$ follows by taking the logarithm of this expression, and this proves that $\mathrm{report}(\cdot + \xi)$ is $(\eta,0,\nu)$-stable with respect to an oracle that reports on the event $\mathrm{report}(\mu + Z + \xi)$. The final claim follows by invoking Lemma \ref{lemma:near-indep2}.
\end{proof}

In Claim \ref{claim:conditional}, we instantiate the value of $\eta$ when $\P_0 = \mathcal{N}(0,\sigma^2)$. In that case, the moment-generating function of the truncated normal $|\P_0^\nu|$ is equal to
$$\mathrm{MGF}_{|\P_0^\nu|}\left(-\frac{1}{b}\right) = \frac{2}{1-\nu} e^{\frac{\sigma^2}{2b^2}}\left(\Phi\left(z_{1-\nu/2} + \frac{\sigma}{b}\right) - \Phi\left( \frac{\sigma}{b}\right)\right).$$

\subsection{Proof of Proposition \ref{prop:lasso-stability} (LASSO stability)}

For the sake of readability, we denote the squared loss, rescaled by $\hat \sigma$, by $L(\theta;X,y) := \frac{1}{n \hat \sigma}\|y-X\theta\|_2^2$; hence, $\nabla L(\theta;X,y) = \frac{2}{n \hat \sigma}X^\top(y - X\theta)$. Also, we denote by $S_{C_1} := C_1 \cdot \{\pm e_i\}_{i=1}^d$ the set of $2d$ extreme points of the $\ell_1$-ball in $\R^d$, scaled by the LASSO constraint $C_1$. Similarly, we let $S_{C_1}^+ := C_1 \cdot \{ e_i\}_{i=1}^d$ denote half of the points in $S_{C_1}$ that correspond to the extreme points with non-negative coordinates.

Let $y\sim \mathcal{N}(\mu,\sigma^2 I)$. Fix $t\in[k]$ and $\theta$ such that $\|\theta\|_1\leq C_1$. For all $\phi\in S_{C_1}$, we have
\begin{align*}
 \phi^\top (\nabla L(\theta;X,y) - \nabla L(\theta;X,\mu)) &= \phi^\top\left(\frac{2}{n\hat \sigma} X^\top (y - X\theta) - \frac{2}{n\hat \sigma} X^\top(\mu - X\theta)\right)\\
  &= \frac{2}{n\hat \sigma} \phi^\top X^\top(y-\mu).
 \end{align*}

Notice that $\|X\phi\|_2 \leq C_1 \|X\|_{2,\infty} = C_1 \max_{i\in[d]}\|X_i\|_2$ for all $\phi\in S_{C_1}$. By a union bound, we can write:
\begin{align*}
\PP{\frac{2}{n\hat \sigma} \max_{\phi\in S_{C_1}}|\phi^\top X^\top(y-\mu)|\geq s} &= \PP{\frac{2}{n\hat \sigma} \max_{\phi\in S_{C_1}^+}|\phi^\top X^\top(y-\mu)|\geq s}\\
 &\leq \sum_{\phi \in S_{C_1}^+}\PP{\frac{2}{n\hat \sigma} |\phi^\top X^\top(y-\mu)|\geq s}.
\end{align*}
Since $\frac{2}{n\hat \sigma} \phi^\top X^\top(y-\mu)$ follows a rescaled $t$-distribution with $r$ degrees of freedom and there are $d$ terms in the sum on the right-hand side, for $s = s^*:= \frac{2 t_{r,1-\delta/(2d)} C_1 \|X\|_{2,\infty}}{n}$, the probability above is at most $\delta$. Denote $Y_\delta = \{\omega:\max_{\phi\in S_{C_1}} |\frac{2}{n\hat\sigma} \phi^\top X^\top (\omega - \mu)| \leq s^*\}$; we have thus shown $\PP{y\in Y_\delta} \geq 1-\delta$.

We now show that, whenever $y\in Y_\delta$, stable LASSO with input $y$ is indistinguishable from stable LASSO with input $\mu$. From here on, we fix $y\in Y_\delta$ and only consider the randomness of the algorithm.

The output of Algorithm \ref{alg:lasso} can be written as a function of $(\theta_1,\dots,\theta_{k+1})$, and hence proving that $(\theta_1,\dots,\theta_{k+1})$ is indistinguishable when computed on $y$ and $\mu$ is sufficient to argue that $\thetalasso$ is indistinguishable on the two inputs, by the post-processing property.

For all $t\leq k$, we can write $\theta_{t+1} = g_t(\theta_{t}, y)$ for some randomized function $g_t$; in Algorithm \ref{alg:g_t} we express $g_t$ as an algorithm. If we show $g_t(\theta, y)\approx_{\eta,0} g_t(\theta, \mu)$ for every fixed $\theta$ such that $\|\theta\|_1\leq C_1$, then we can apply Lemma \ref{lemma:fixedinputcomp} to conclude indistinguishability of the whole sequence $(\theta_1,\dots,\theta_{k+1})$.

\begin{algorithm}[H]
\SetAlgoLined
\begin{flushleft}
\textbf{input: }$\theta_{t},y$\\
\textbf{output: } $\theta_{t+1}$\newline
$\forall\phi \in C_1 \cdot \{\pm e_i\}_{i=1}^d$, sample $\xi_{t,\phi} \stackrel{\text{i.i.d.}}{\sim} \text{Lap}\left(\frac{4  t_{r,1-\delta/(2d)} C_1 \|X\|_{2,\infty}}{n\eta}\right)$\newline
  $\forall\phi \in C_1 \cdot \{\pm e_i\}_{i=1}^d$, let $\alpha_\phi = -\frac{2}{n \hat \sigma}\phi^\top X^\top(y - X\theta_t) + \xi_{t,\phi}$\newline
Set $\phi_t = \argmin_{\phi\in S_{C_1}} \alpha_\phi$\newline
Set $\theta_{t+1} = (1-\Delta_t)\theta_t + \Delta_t \phi_t$, where $\Delta_t = \frac{2}{t+1}$\\
 Return $\theta_{t+1}$
\end{flushleft}
\caption{The $g_t$ subroutine of the stable LASSO algorithm}
\label{alg:g_t}
\end{algorithm}

Let $\phi_t$ and $\phi_t^\mu$ denote the minimizers of $\alpha_\phi$ when the input is $y$ and $\mu$, respectively, and fix an arbitrary point $\phi^*\in S_{C_1}$. Let $\{\xi_{t,\phi}\}_{\phi\in S_{C_1}}$ be independent samples from $\text{Lap}\left(\frac{2s^*}{\eta}\right)$. Denote
$$\xi^* = \argmax_{\xi} \nabla L(\theta;X,y)^\top \phi^* + \xi \leq  \nabla L(\theta;X,y)^\top \phi + \xi_{t,\phi}, \forall \phi\in S_{C_1}\setminus\{\phi^*\}.$$
Conditional on $\xi_{t,\phi}$, $\phi\in S_{C_1}\setminus\{\phi^*\}$, we get $\phi_t = \phi^*$ if and only if $\xi_{t,\phi^*}\leq \xi^*$.

By the definition of $Y_\delta$, we have:
\begin{align*}
(\phi^*)^\top\nabla L(\theta;X,\mu) - s^* + \xi^* &\leq (\phi^*)^\top\nabla  L(\theta;X,y) + \xi^*\\
 &\leq \phi^\top\nabla  L(\theta;X,y) + \xi_{t,\phi} \leq \phi^\top\nabla  L(\theta;X,\mu) + s^* + \xi_{t,\phi},
\end{align*}
for all $\phi\in S_{C_1}\setminus\{\phi^*\}$. As a result, conditional on $\xi_{t,\phi}$, $\phi\in S_{C_1}\setminus\{\phi^*\}$, the event $\xi_{t,\phi^*}\leq \xi^* - 2s^*$ implies $\phi_t^\mu = \phi^*$. Thus, we get:
\begin{align*}
\PPst{\phi_t^\mu = \phi^*}{\xi_{t,\phi}, \phi\in S_{C_1}\setminus\{\phi^*\}} &\geq \PPst{\xi_{t,\phi^*} \leq \xi^* - 2s^*}{\xi_{t,\phi}, \phi\in S_{C_1}\setminus\{\phi^*\}}\\
 &\geq e^{-\eta}\PPst{\xi_{t,\phi^*} \leq \xi^*}{\xi_{t,\phi}, \phi\in S_{C_1}\setminus\{\phi^*\}}\\
 &= e^{-\eta}\PPst{\phi_t = \phi^*}{\xi_{t,\phi}, \phi\in S_{C_1}\setminus\{\phi^*\}}.
\end{align*}
Applying an expectation to both sides yields
$$\PP{\phi_t = \phi^*}\leq e^{\eta}\PP{\phi_t^\mu = \phi^*},$$
and this is true for all $\phi^*\in S_{C_1}$.
Therefore, for all $y\in Y_{\delta}$, $\phi_t\approx_{\eta,0}\phi_t^\mu$. By post-processing, this also implies $g_t(\theta,y)\approx_{\eta,0} g_t(\theta,\mu)$, for all $\theta$.

By Lemma \ref{lemma:fixedinputcomp}, we finally conclude that, for all $y\in Y_\delta$, the output of the stable LASSO algorithm when applied to $y$ is $(\frac{1}{2}k\eta^2 + \sqrt{2k\log(1/\delta)}\eta,\delta)$-indistinguishable from the output implied by the oracle input $\mu$, for all $\delta\in(0,1)$, or alternatively it is $(k\eta,0)$-indistinguishable. Since this holds with $1-\delta$ probability over the choice of $y$, we see that Algorithm \ref{alg:lasso} is stable with the desired parameters.

\subsection{Proof of Proposition \ref{prop:lasso-utility} (LASSO utility)}

As in the proof of Proposition \ref{prop:lasso-stability}, we denote the rescaled squared loss by $L(\theta;X,y) := \frac{1}{n \hat \sigma}\|y-X\theta\|_2^2$, and by $S_{C_1} := C_1 \cdot \{\pm e_i\}_{i=1}^d$ we denote the set of $2d$ extreme points of the $\ell_1$-ball in $\R^d$, scaled by the constraint $C_1$.

We begin by stating a convergence result for the Frank-Wolfe algorithm due to Jaggi~\cite{jaggi2013revisiting}, which forms the core of our analysis.

\begin{lemma}[\cite{jaggi2013revisiting}]
\label{lemma:frank-wolfe}
Fix $s>0$ and $\theta_1\in \D\subseteq \R^d$. Let $(\phi_1,\dots,\phi_{k})$ be a sequence of vectors from $\D$ and let $\theta_{t+1} = (1 - \Delta_t)\theta_t + \Delta_t \phi_t$, for arbitrary $\Delta_t\in [0,1]$. Define the curvature constant of $L$ as
$$C_L := \sup_{\theta_1,\theta_2\in\D, \gamma\in[0,1], \theta_3 = (1-\gamma)\theta_1 + \gamma\theta_2} \frac{2}{\gamma^2} (L(\theta_3) - L(\theta_1) - (\theta_3 - \theta_1)^\top \nabla L(\theta_1)).$$
Suppose that for all $t\in[k]$, it holds that: $\phi_t^\top \nabla L(\theta_t) \leq \min_{\phi\in C_1\cdot \{\pm e_i\}_{i=1}^d} \phi^\top \nabla L(\theta_t) + \frac{s\Delta_t C_L}{2}.$
Then,
$$L(\theta_{k+1})- \min_{\theta:\|\theta\|_1\leq C_1}L(\theta) \leq \frac{2C_L}{k + 2} (1 + s).$$
\end{lemma}


Denote by $b:=\frac{4t_{r,1-\delta/(2d)}C_1\|X\|_{2,\infty}}{n\eta}$ the parameter of the Laplace noise in Algorithm \ref{alg:lasso}. Fix $s>0$. Denoting by $C_L$ the curvature constant of $L$, as defined in Lemma \ref{lemma:frank-wolfe}, and by applying subexponential concentration of the Laplace distribution, we know:
\begin{align*}
&\lefteqn{\PP{\exists t \in [k]: \phi_t^\top \nabla L(\theta_t; X,y) > \min_{\phi\in S_{C_1}}\phi^\top \nabla L(\theta_t; X,y) + \frac{s\Delta_t C_L}{2} }}\\
 &\leq \PP{\exists t \in [k]:  \max_{\phi\in S_{C_1}} |\xi_{t,\phi}| > \frac{s\Delta_t C_L}{4}}\\
&\leq \PP{\max_{t\in[k], \phi\in S_{C_1}} |\xi_{t,\phi}| > \frac{s\Delta_k C_L}{4} }\\
&\leq k|S_{C_1}|\exp\left(-\frac{s\Delta_k C_L}{4b}\right),
\end{align*}
where the last step follows by a union bound. Setting $s = \frac{4b}{\Delta_k C_L}\log(k|S_{C_1}|/\zeta)$ controls this probability to be at most $\zeta$.

We use a standard fact from convex geometry: for any set $S_{\D}$ such that its convex hull is equal to $\D$, it holds that $\min_{\phi\in \D} \phi^\top \nabla L(\theta_t; X,y) = \min_{\phi\in S_{\D}} \phi^\top \nabla L(\theta_t; X,y)$. In our setting, $\D = \{\theta:\|\theta\|_1\leq C_1\}$, and it can be obtained as the convex hull of $S_{C_1}$.

With this, we can apply Lemma \ref{lemma:frank-wolfe}, as well as the fact that $|S_{C_1}|=2d$, to get that with probability $1-\zeta$ over the Laplace noise variables:
$$L(\theta_{k+1}; X,y) - \min_{\theta:\|\theta\|_1 \leq C_1} L(\theta; X,y) \leq \frac{2C_L}{k+2} + \frac{8 C_L  b\log(2kd/\zeta)}{(k+2)\Delta_k C_L}.$$
By the curvature characterization for quadratics due to Clarkson~\cite{clarkson2010coresets}, we can bound the curvature constant as 
$$C_L \leq \frac{1}{n \hat \sigma}\max_{\theta,\theta': \|\theta\|_1\leq C_1, \|\theta'\|_1\leq C_1} \|X(\theta-\theta')\|_2^2 \leq \frac{1}{n\hat \sigma} \max_{\varphi: \|\varphi\|_1\leq 2C_1} \|X\varphi\|_2^2 \leq \frac{4}{\hat \sigma}\|X\|_\infty^2C_1^2.$$
Therefore, we can conclude
$$L(\theta_{k+1}; X,y) - \min_{\theta:\|\theta\|_1 \leq C_1} L(\theta; X,y) \leq \frac{8\|X\|_\infty^2 C_1^2}{\hat \sigma(k+2)} + 4 b\log(2kd/\zeta).$$
Further, notice that for all $\theta,\theta'$ such that $ \max\{\|\theta\|_1,\|\theta'\|_1\}\leq C_1$, by H\"older's inequality we have:
\begin{align*}
|L(\theta; X,y) - L(\theta'; X,y)| &= \left|\frac{1}{n\hat \sigma}\|y-X\theta\|_2^2 - \frac{1}{n\hat \sigma}\|y-X\theta'\|_2^2\right|\\
 &\leq \frac{2}{\hat \sigma} \|X\|_\infty (\|X\|_\infty C_1 + \|y\|_\infty)\|\theta'-\theta\|_1 := L_1 \|\theta'-\theta\|_1 \leq 2L_1 C_1,
\end{align*}
where by $L_1$ we denote the $\ell_1$-Lipschitz constant of the squared loss restricted to the LASSO domain. Now we pick $\zeta = \frac{\gamma}{2C_1 L_1}$ for some constant $\gamma>0$, which gives:
\begin{align*}
\E[L(\theta_{k+1}; X,y)|y, \hat \sigma]- \min_{\theta:\|\theta\|_1 \leq C_1} L(\theta; X,y) \leq \gamma + \frac{8\|X\|_\infty^2 C_1^2}{\hat \sigma(k+2)} + 4 b\log(4kd C_1 L_1/\gamma)&\\
= \gamma + \frac{8\|X\|_\infty^2 C_1^2}{\hat \sigma(k+2)} + \frac{16 t_{r,1-\delta/(2d)}C_1 \|X\|_{2,\infty}\log(4kd C_1 L_1/\gamma)}{n\eta}&,
\end{align*}
where in the last step we use the noise level from Algorithm \ref{alg:lasso}.
Now we set $k = \left\lceil\frac{n\|X\|_\infty^2 C_1 \eta}{\hat \sigma \|X\|_{2,\infty}}\right\rceil$, and get the following utility upper bound:
\begin{align*}
&\lefteqn{\E[L(\theta_{k+1};X,y)|y, \hat \sigma]- \min_{\theta:\|\theta\|_1 \leq C_1} L(\theta;X,y)}\\
 &\leq \gamma + \frac{8 C_1\|X\|_{2,\infty}}{n\eta} + \frac{16 t_{r,1-\delta/(2d)}C_1 \|X\|_{2,\infty}\log(4kd C_1 L_1/\gamma)}{n\eta}.
\end{align*}
Note that the above inequality is true for all $\gamma>0$. After optimizing over $\gamma$, the right-hand side reduces to
$$\frac{8C_1\|X\|_{2,\infty}}{n\eta} + \frac{16t_{r,1-\delta/(2d)}C_1 \|X\|_{2,\infty}\left(1 +\log(kd L_1 n \eta/(4 t_{r,1-\delta/(2d)}\|X\|_{2,\infty}))\right)}{n\eta}.$$
Using $k\leq \frac{2n\|X\|_\infty^2 C_1\eta}{\hat \sigma \|X\|_{2,\infty}}$ and the value of $L_1$, we finally get
\begin{align*}
	&\E[L(\theta_{k+1};X,y)|y, \hat \sigma]- \min_{\theta:\|\theta\|_1 \leq C_1} L(\theta;X,y) \leq \frac{8C_1\|X\|_{2,\infty}}{n\eta} + \\
	&\frac{16 t_{r,1-\delta/(2d)}C_1 \|X\|_{2,\infty}}{n\eta} \left(1 + \log\left(\frac{d C_1 n^2 \eta^2 \|X\|_\infty^3 (\|X\|_\infty C_1 + \|y\|_\infty)}{2 t_{r,1-\delta/(2d)}\|X\|_{2,\infty}^2\hat \sigma^2}\right)\right).
\end{align*}
Focusing on the relevant parameters, this bound can be simplified as
$$\frac{1}{n}\E[\|y - X\theta_{k+1}\|_2^2~|~y]- \min_{\theta:\|\theta\|_1 \leq C_1} \frac{1}{n} \|y-X\theta\|_2^2 = \tilde O\left(\frac{C_1 \|X\|_{2,\infty} \log(d) t_{r,1-\delta/(2d)} \sigma}{n\eta}\right).$$

Note that similar guarantees follow without conditioning on $y$, by taking iterated expectations, applying Jensen's inequality, and using subgaussianity to bound  $\E[\|y\|_\infty]$.

\subsection{Proof of Proposition \ref{prop:marginal-stability} (marginal screening stability)}

Let $y\sim\mathcal{N}(\mu,\sigma^2 I)$ and define $c_i^\omega:=\frac{1}{n \hat \sigma}X_i^\top \omega$ for all $\omega\in\R^n$. Let $Y_\delta = \{\omega:\|c^\omega-c^{\mu}\|_\infty \leq \frac{t_{r,1-\delta/(2d)} \|X\|_{2,\infty}}{n}\}$. First we prove that $\PP{y\in Y_\delta}\geq 1 - \delta$:
\begin{align*}
\PP{\|c^y - c^{\mu}\|_\infty \geq \frac{t_{r,1-\delta/(2d)} \|X\|_{2,\infty}}{n}} &= \PP{\exists i:\frac{1}{n \hat \sigma}|X_i^\top y - X_i^\top \mu| \geq \frac{t_{r,1-\delta/(2d)} \|X\|_{2,\infty}}{n}}\\
&= \PP{\exists i:\left|\frac{X_i^\top (y - \mu)}{\hat \sigma}\right| \geq  t_{r,1-\delta/(2d)} \|X\|_{2,\infty}}\\
&\leq d \cdot \frac{\delta}{d} = \delta.
\end{align*}

Now we appeal to a similar composition argument as in Proposition \ref{prop:lasso-stability}. From here on, fix $y\in Y_\delta$. We will show that the output of stable marginal screening, when applied to $y$, is indistinguishable from the output of stable marginal screening given the oracle input $\mu$.
 
The selected model $\hat M$ can be written as the output of a composition of $k$ functions $g_t(i_1,\dots,i_{t-1}, y)$, $t\in[k]$. In particular, the feature ``peeled off'' at time $t$, $i_t$, is equal to $g_t(i_1,\dots,i_{t-1},y)$. We show that $g_t(i_1,\dots,i_{t-1},y)\approx_{\eta,0} g_t(i_1,\dots,i_{t-1},\mu)$ holds true for all fixed $i_1,\dots,i_{t-1}$. By Lemma \ref{lemma:fixedinputcomp}, that will imply that the overall selected model under input $y$ and under input $\mu$ is indistinguishable as well.

Fix a round $t\in[k]$, as well as an index $i\in\text{res}_t$. Suppose that we add independent draws $\xi_{t,j}\sim \text{Lap}\left(\frac{2 t_{r,1-\delta/(2d)} \|X\|_{2,\infty}}{n\eta}\right)$ to each value $c_j$, where $j\in \text{res}_t$. Define
$$\xi^*_+ = \argmin_{\xi\geq -c^y_i} c_i^y + \xi > |c^y_j + \xi_{t,j}|, ~~\xi^*_- = \argmax_{\xi< -c^y_i} -c^y_i - \xi > |c^y_j + \xi_{t,j}|, ~\forall j \neq i.$$
Then, $g_t(i_1,\dots,i_{t-1},y) = i$ if and only if $\xi_{t,i} \geq \xi^*_+$ or $\xi_{t,i} \leq \xi^*_-$. Moreover, since $y\in Y_\delta$, we have
\begin{align*}
\frac{t_{r,1-\delta/(2d)} \|X\|_{2,\infty}}{n} + c_i^{\mu} + \xi^*_+ &\geq c^y_i + \xi^*_+ > |c^y_j + \xi_{t,j}|\geq |c_j^{\mu} + \xi_{t,j}| - \frac{t_{r,1-\delta/(2d)} \|X\|_{2,\infty}}{n},
\end{align*}
\begin{align*}
\frac{t_{r,1-\delta/(2d)} \|X\|_{2,\infty}}{n} - c_i^{\mu} - \xi^*_- &\geq -c^y_i - \xi^*_- > |c^y_j + \xi_{t,j}|\geq |c_j^{\mu} + \xi_{t,j}| - \frac{t_{r,1-\delta/(2d)} \|X\|_{2,\infty}}{n}.
\end{align*}

Rearranging the terms, we get
$$\frac{2t_{r,1-\delta/(2d)} \|X\|_{2,\infty}}{n} + c_i^{\mu} + \xi^*_+ \geq |c_j^{\mu} + \xi_{t,j}|,~\frac{2t_{r,1-\delta/(2d)} \|X\|_{2,\infty}}{n} - c_i^{\mu} - \xi^*_- \geq |c_j^{\mu} + \xi_{t,j}|.$$
Thus, if $\xi_{t,i} \geq \xi^*_+ + \frac{2t_{r,1-\delta/(2d)} \|X\|_{2,\infty}}{n}$ or $\xi_{t,i} \leq \xi^*_- - \frac{2t_{r,1-\delta/(2d)} \|X\|_{2,\infty}}{n}$, then $i = g_t(i_1,\dots,i_{t-1}, \mu)$ if the noise levels are $(\xi_{t,1},\dots,\xi_{t,i},\dots,\xi_{t,d})$. Finally, for fixed $y\in Y_\delta$, we have
\begin{align*}
&\lefteqn{\PPst{g_t(i_1,\dots,i_{t-1}, \mu) = i}{\{\xi_{t,j}\}_{j\neq i}}}\\
&\geq \PPst{\xi_{t,i} \geq \xi^*_+ + \frac{2t_{r,1-\delta/(2d)} \|X\|_{2,\infty}}{n}}{\{\xi_{t,j}\}_{j\neq i}}\\
	 &+ \PPst{\xi_{t,i} \leq \xi^*_- - \frac{2t_{r,1-\delta/(2d)} \|X\|_{2,\infty}}{n}}{\{\xi_{t,j}\}_{j\neq i}}\\
	&\geq e^{-\eta} \PPst{\xi_{t,i} \geq \xi^*_+}{\{\xi_{t,j}\}_{j\neq i}} +e^{-\eta}\PPst{\xi_{t,i} \leq \xi^*_-}{\{\xi_{t,j}\}_{j\neq i}}\\
	 &= e^{-\eta}\PPst{g_t(i_1,\dots,i_{t-1}, y) = i}{\{\xi_{t,j}\}_{j\neq i}}.
\end{align*}
Multiplying by $e^{\eta}$ and applying the law of iterated expectations completes the proof that $g_t(i_1,\dots,i_{t-1},y)\approx_{\eta,0} g_t(i_1,\dots,i_{t-1},\mu)$ for all $y\in Y_\delta$.

Finally, by Lemma \ref{lemma:fixedinputcomp} we conclude that for all fixed $y\in Y_\delta$, the output of stable marginal screening under input $y$ and under the oracle input $\mu$ is $(\frac{1}{2}k\eta^2 + \sqrt{2k\log(1/\delta)}\eta,\delta)$-indistinguishable for all $\delta\in(0,1)$, or alternatively $(k\eta,0)$-indistinguishable. Since this holds with $1-\delta$ probability over the choice of $y$, we see that stable marginal screening satisfies stability with the desired parameters.

\begin{figure}[t]
\centerline{\includegraphics[width=0.3\textwidth]{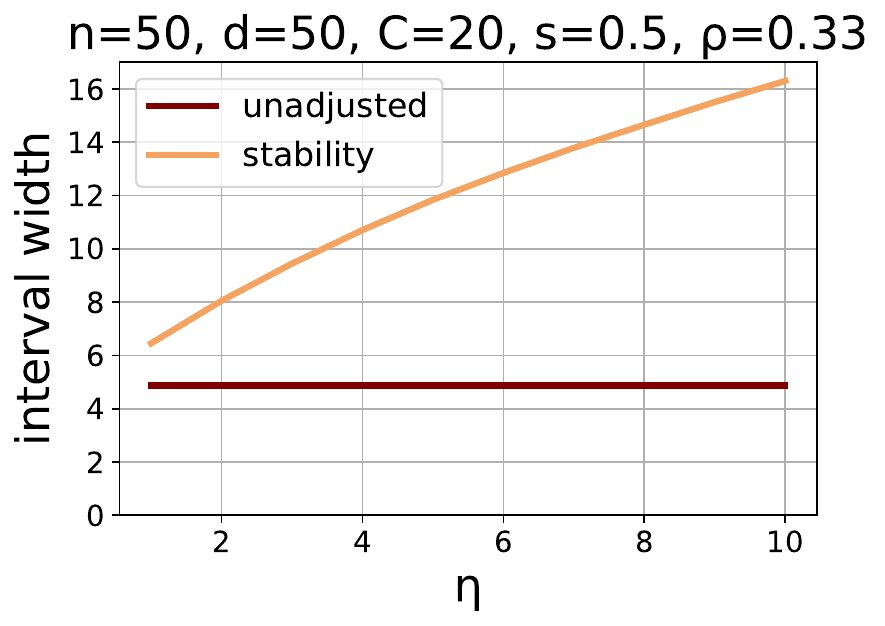}
\includegraphics[width=0.3\textwidth]{plots/LASSOw_d50_C20_s05_sig5.pdf}
\includegraphics[width=0.3\textwidth]{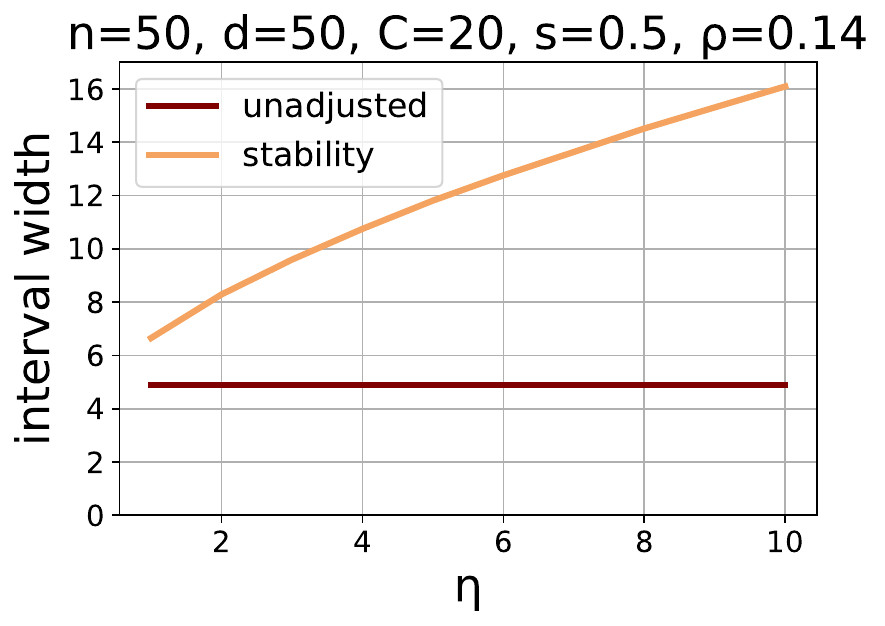}
}
\centerline{\includegraphics[width=0.3\textwidth]{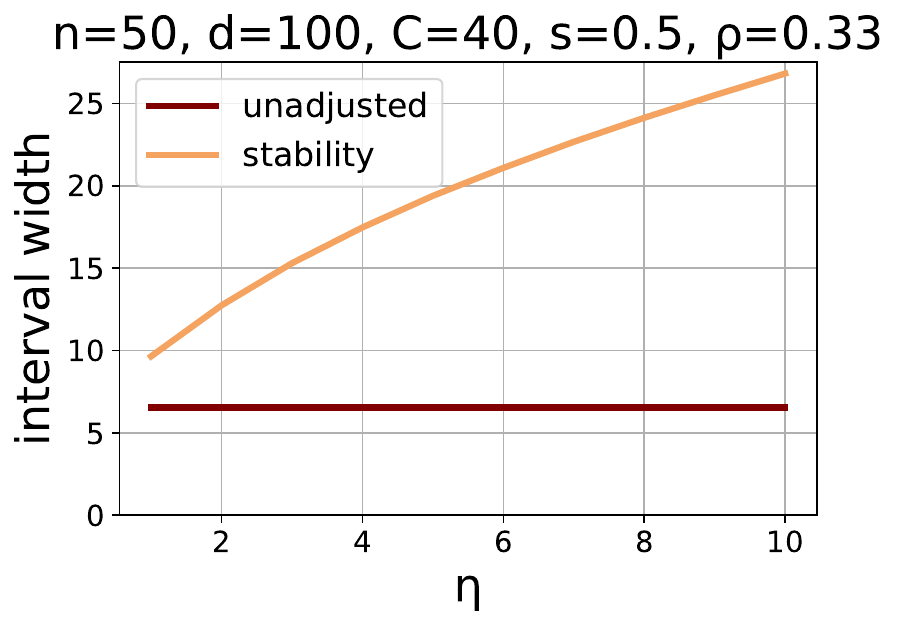}
\includegraphics[width=0.3\textwidth]{plots/LASSOw_d100_C40_s05_sig5.pdf}
\includegraphics[width=0.3\textwidth]{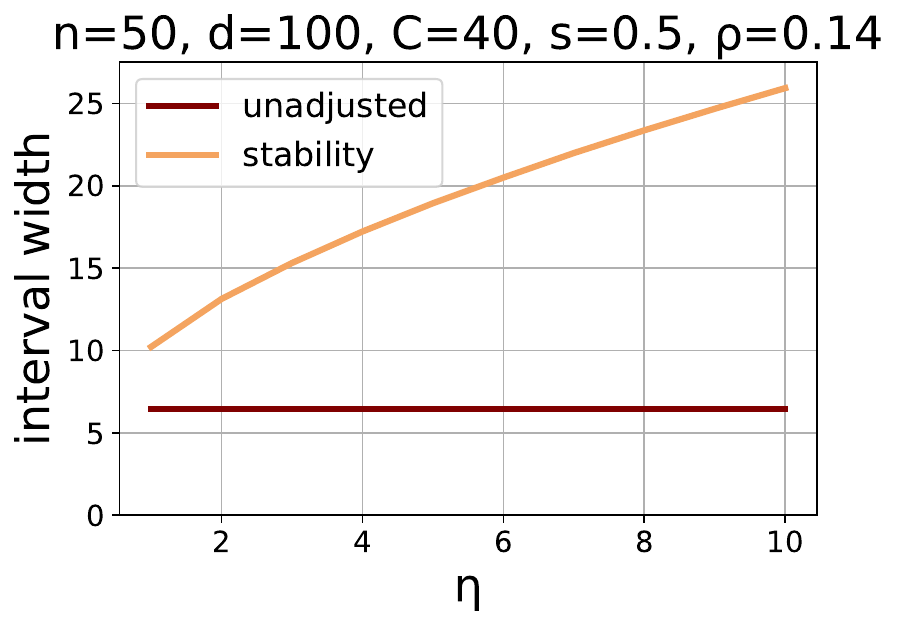}
}
\centerline{\includegraphics[width=0.3\textwidth]{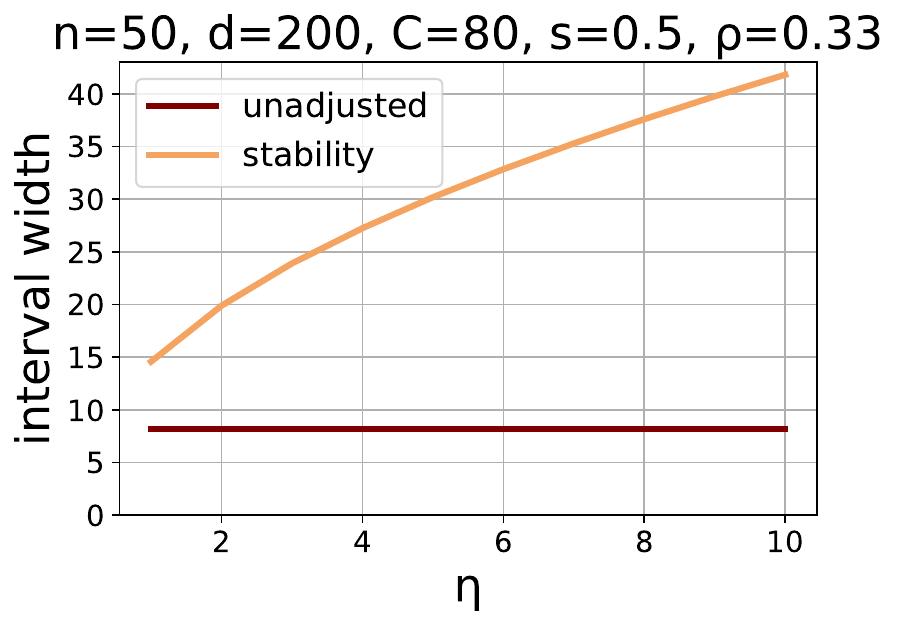}
\includegraphics[width=0.3\textwidth]{plots/LASSOw_d200_C80_s05_sig5.pdf}
\includegraphics[width=0.3\textwidth]{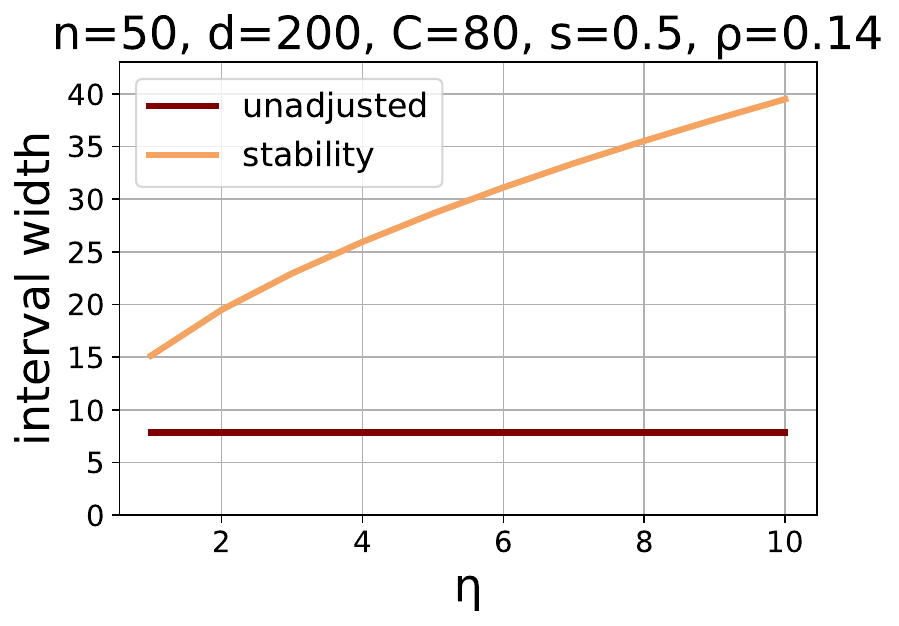}
}
\caption{Plots of average confidence interval widths of stable LASSO, with varying dimension and signal strength, in the Gaussian design case.} 
\label{fig:lassow_comparison}
\end{figure}

\begin{figure}[b]
\centerline{\includegraphics[width=0.3\textwidth]{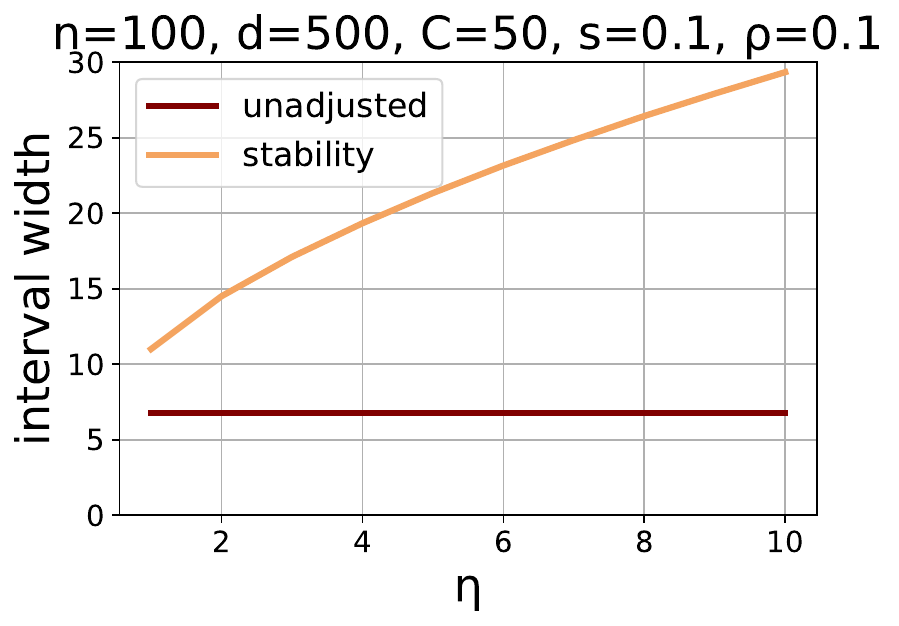}
\includegraphics[width=0.3\textwidth]{plots/LASSOw_n200_d500_C50_s01_sig10.pdf}
\includegraphics[width=0.3\textwidth]{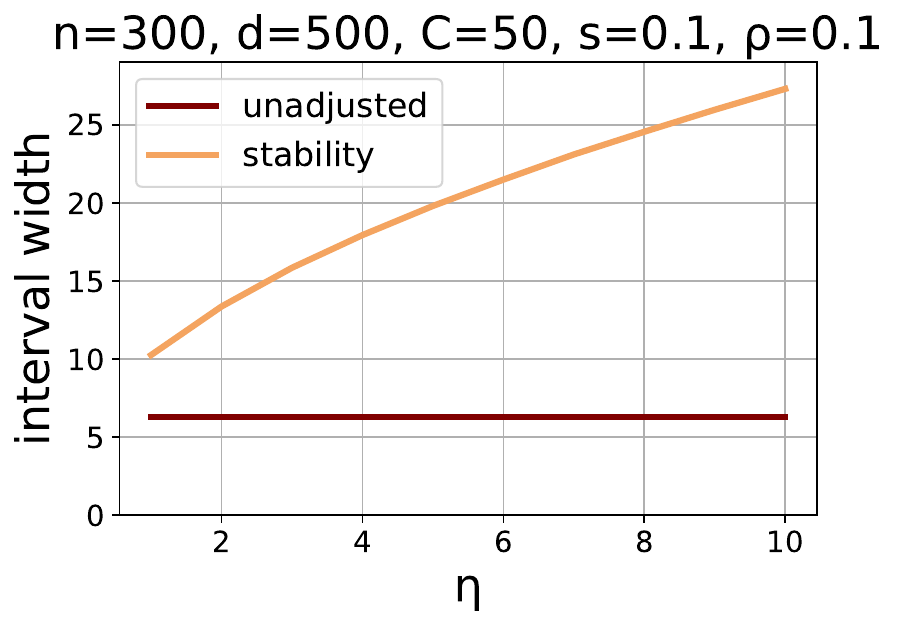}
}
\caption{Plots of average confidence interval widths of stable LASSO, with varying sample size, in the Gaussian design case.} 
\label{fig:lassow_highdim}
\end{figure}

\subsection{Proof of Proposition \ref{prop:marginal-utility} (marginal screening utility)}

Fix $s>0$. Taking a union bound, we get:
$$\PPst{\max_{j\in[k]}  |c_{m_j}| - |c_{i_j}| \geq s}{y} \leq \sum_{j=1}^k \PPst{|c_{m_j}| - |c_{i_j}| \geq s}{y}.$$
At the time when $i_j$ is chosen, exactly $j-1$ items have been selected; therefore, at least one of $m_1,\dots,m_j$ has still not been selected. The event that $|c_{m_j}| - |c_{i_j}| \geq s$ implies that $i_j$ ``beat'' one of $m_1,\dots,m_j$, which further implies that
$\max_{i\in[d]} |\xi_{j,i}|\geq \frac{s}{2}$.  By a union bound, this happens with probability at most $d\exp(-sn\eta/(4t_{r,1-\delta/(2d)}\|X\|_{2,\infty}))$. Putting everything together, we get
$$\sum_{j=1}^k \PPst{|c_{m_j}| - |c_{i_j}| \geq s}{y} \leq kd\exp\left(-\frac{sn\eta}{4t_{r,1-\delta/(2d)}\|X\|_{2,\infty}}\right).$$
Plugging in $s = \frac{4t_{r,1-\delta/(2d)}\log(dk/\delta')\|X\|_{2,\infty}}{n\eta}$ completes the proof.

\begin{figure}[t]
\centerline{\includegraphics[width=0.3\textwidth]{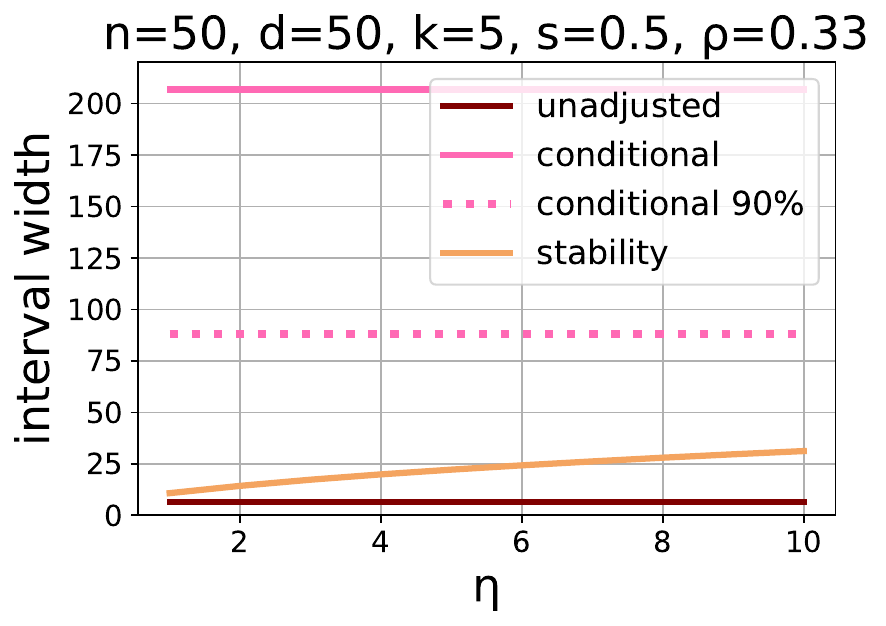}
\includegraphics[width=0.3\textwidth]{plots/MSw_d50_k5_s05_sig5.pdf}
\includegraphics[width=0.3\textwidth]{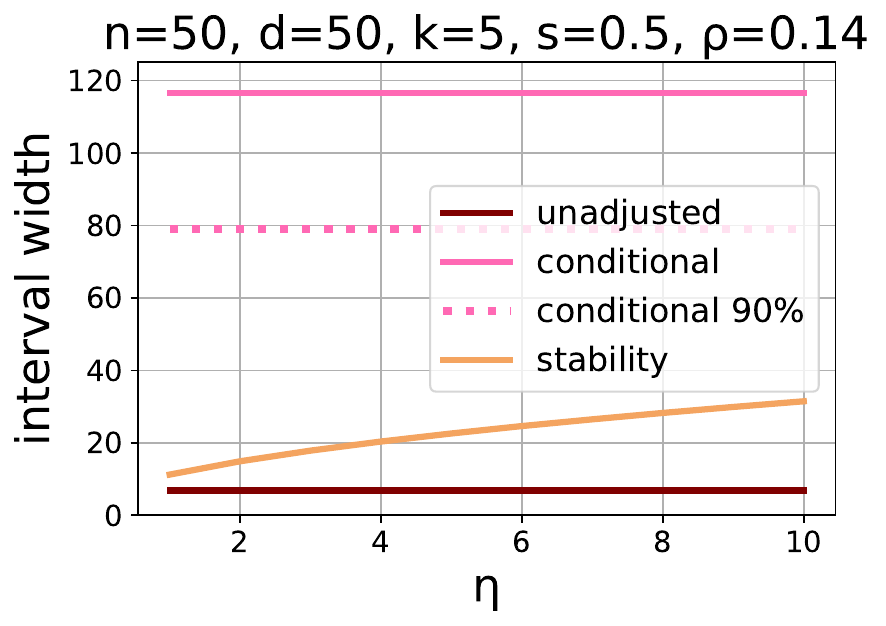}
}
\centerline{\includegraphics[width=0.3\textwidth]{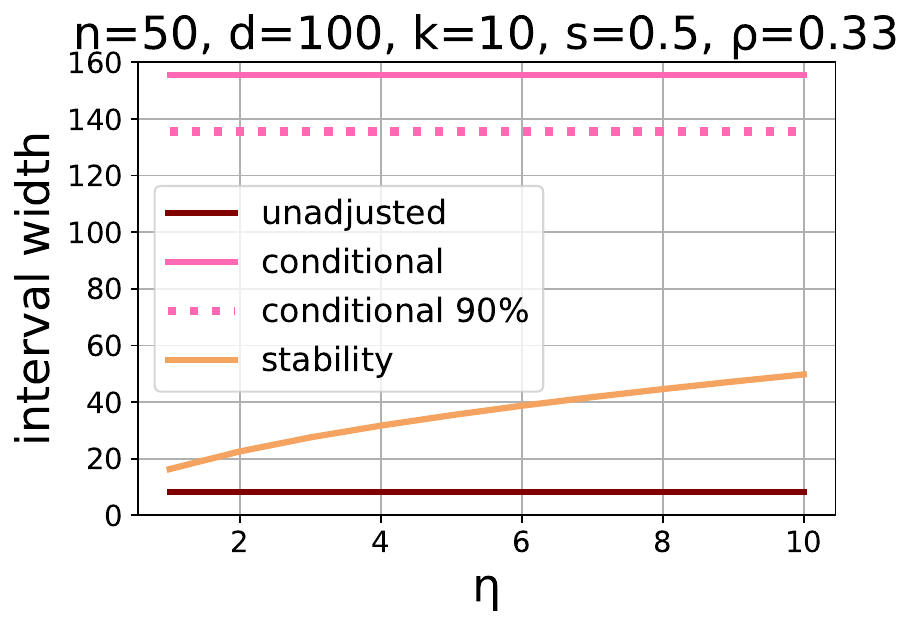}
\includegraphics[width=0.3\textwidth]{plots/MSw_d100_k10_s05_sig5.pdf}
\includegraphics[width=0.3\textwidth]{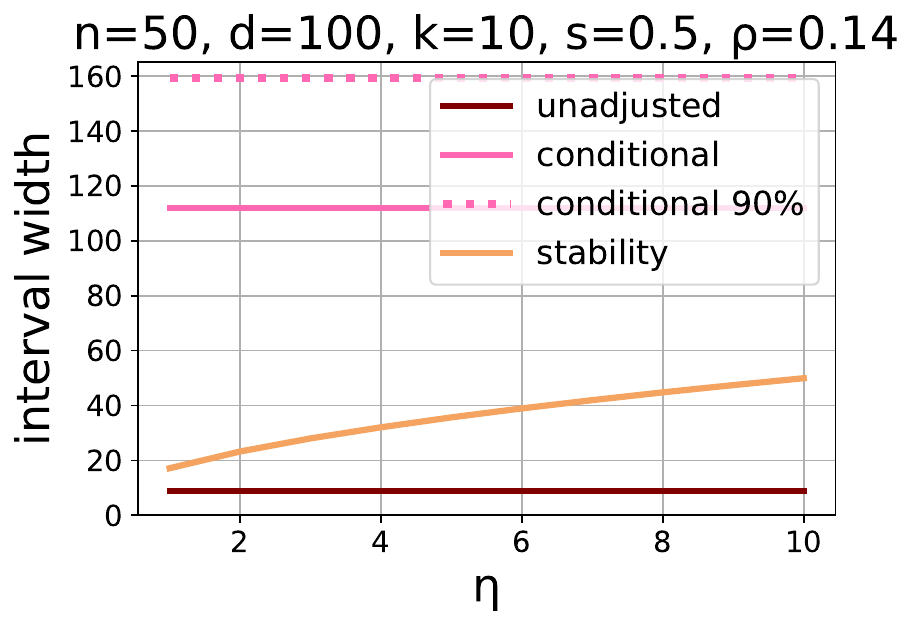}
}
\centerline{\includegraphics[width=0.3\textwidth]{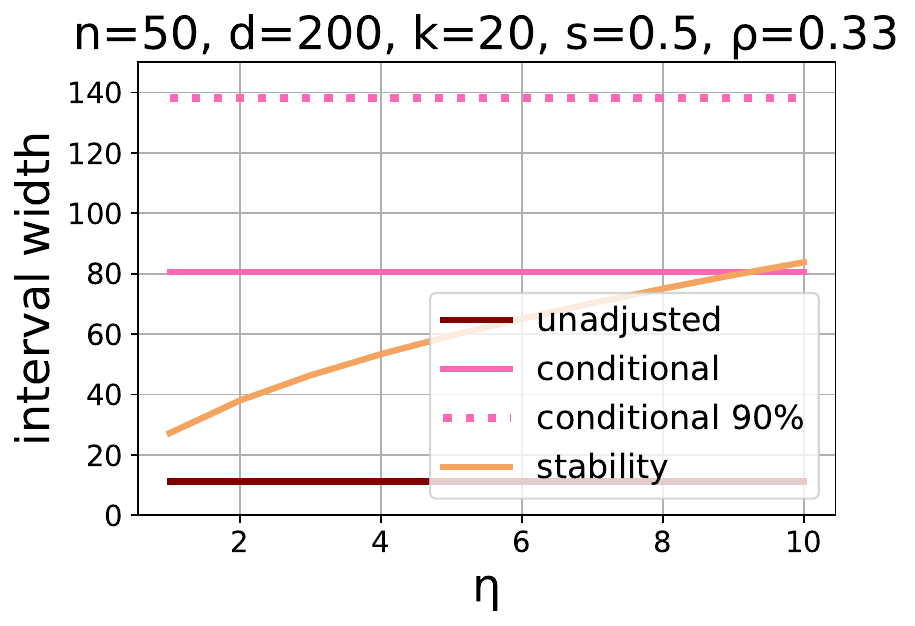}
\includegraphics[width=0.3\textwidth]{plots/MSw_d200_k20_s05_sig5.pdf}
\includegraphics[width=0.3\textwidth]{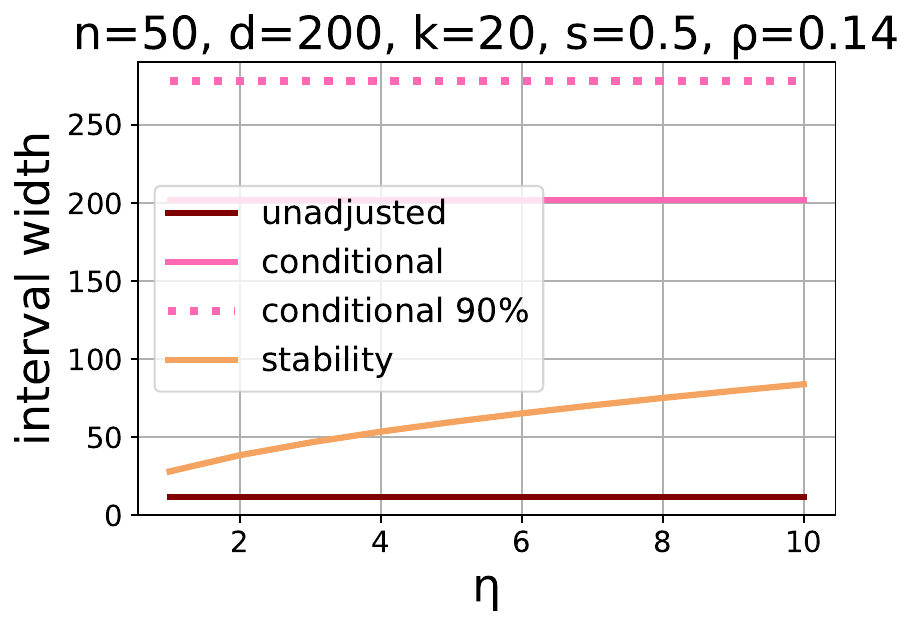}
}
\caption{Plots of average confidence interval widths of stable marginal screening and exact marginal screening with a conditional correction, with varying dimension and signal strength, in the Gaussian design case. We also plot the $90\%$ quantile of the conditional width because it varies greatly across realizations.} 
\label{fig:screeningw_comparison}
\end{figure}

\begin{figure}[b]
\centerline{\includegraphics[width=0.3\textwidth]{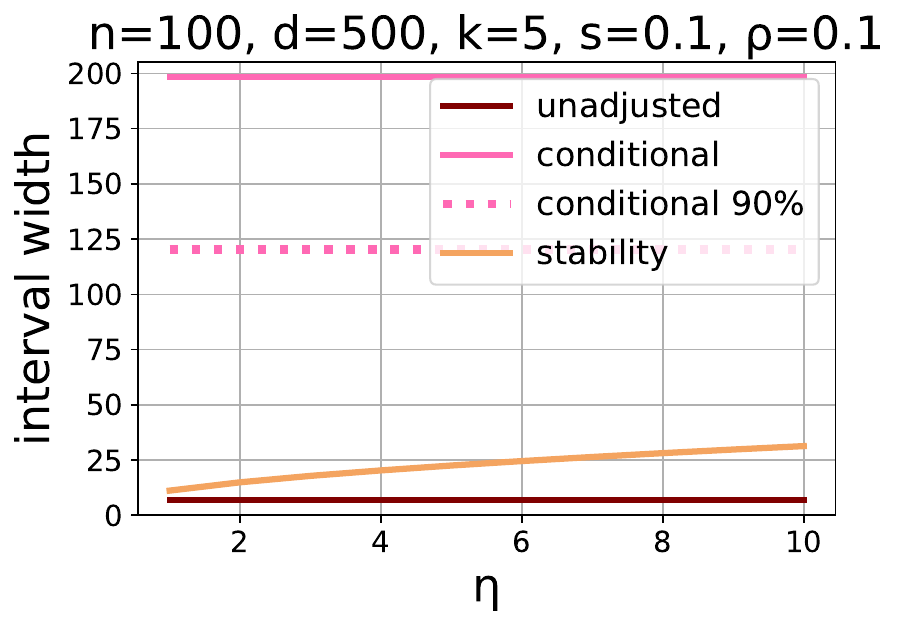}
\includegraphics[width=0.3\textwidth]{plots/MSw_n200_d500_k5_s01_sig10.pdf}
\includegraphics[width=0.3\textwidth]{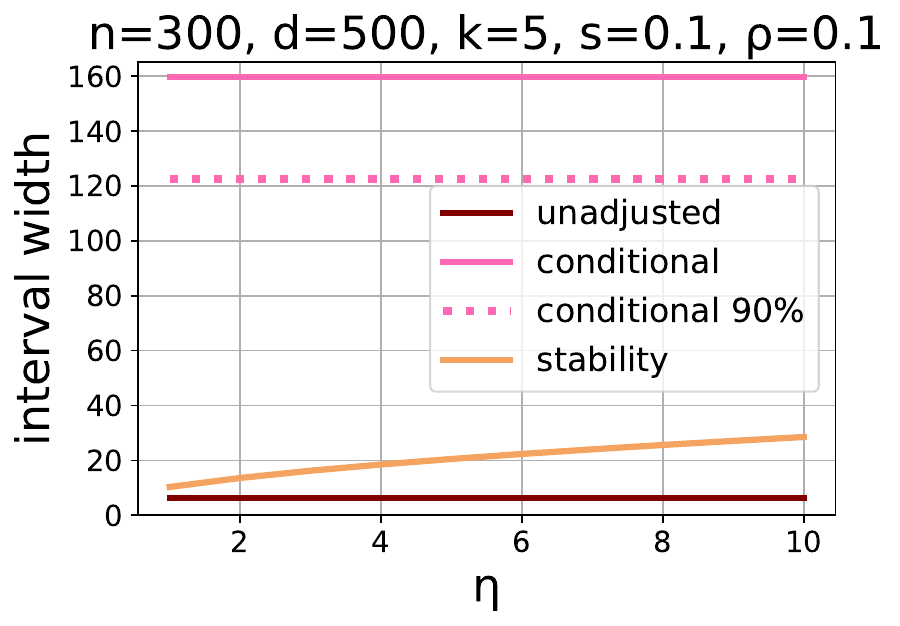}
}
\caption{Plots of average confidence interval widths of stable marginal screening and exact marginal screening with a conditional correction, with varying sample size, in the Gaussian design case. In addition, we plot the $90\%$ quantile of the conditional width because it varies greatly across realizations.} 
\label{fig:msw_highdim}
\end{figure}

\section{Deferred numerical results and simulation details}
\label{sec:deferred_plots}

\paragraph{Deferred plots.}
We include plots of confidence interval widths, which were deferred from Section \ref{sec:experiments}.

For the Gaussian design case, in Figure~\ref{fig:lassow_comparison} we plot the widths corresponding to the experiments in Figure~\ref{fig:lasso_comparison}; in Figure~\ref{fig:lassow_highdim} we plot the widths corresponding to the experiments in Figure~\ref{fig:lasso_highdim}; in Figure~\ref{fig:screeningw_comparison} we plot the widths corresponding to the experiments in Figure~\ref{fig:screening_comparison}; in Figure~\ref{fig:msw_highdim} we plot the widths corresponding to the experiments in Figure~\ref{fig:ms_highdim}.

 For the Bernoulli design case, in Figure~\ref{fig:lassow_comparison_bern} we plot the widths corresponding to the experiments in Figure~\ref{fig:lasso_comparison_bern}; in Figure~\ref{fig:lassow_highdim_bern} we plot the widths corresponding to the experiments in Figure~\ref{fig:lasso_highdim_bern}; in Figure~\ref{fig:screeningw_comparison_bern} we plot the widths corresponding to the experiments in Figure~\ref{fig:screening_comparison_bern}; in Figure~\ref{fig:msw_highdim_bern} we plot the widths corresponding to the experiments in Figure~\ref{fig:ms_highdim_bern}. In Figure~\ref{fig:screeningw_comparison_bern} and Figure~\ref{fig:msw_highdim_bern}, since the conditional widths are of a higher order of magnitude, the scale on the $y$-axis in the widths plots is logarithmic.

\begin{figure}[t]
\centerline{\includegraphics[width=0.3\textwidth]{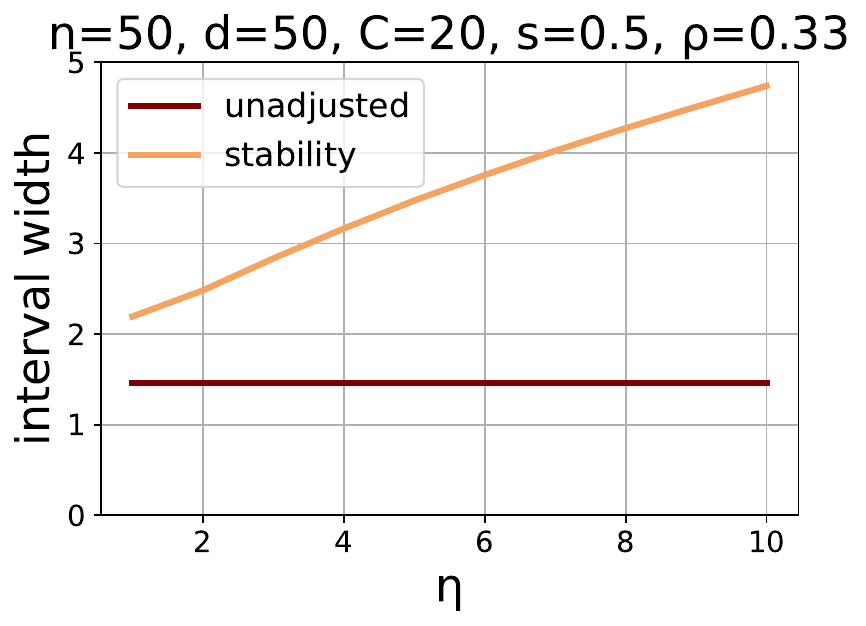}
\includegraphics[width=0.3\textwidth]{plots/LASSOwbern_d50_C20_s05_sig5.pdf}
\includegraphics[width=0.3\textwidth]{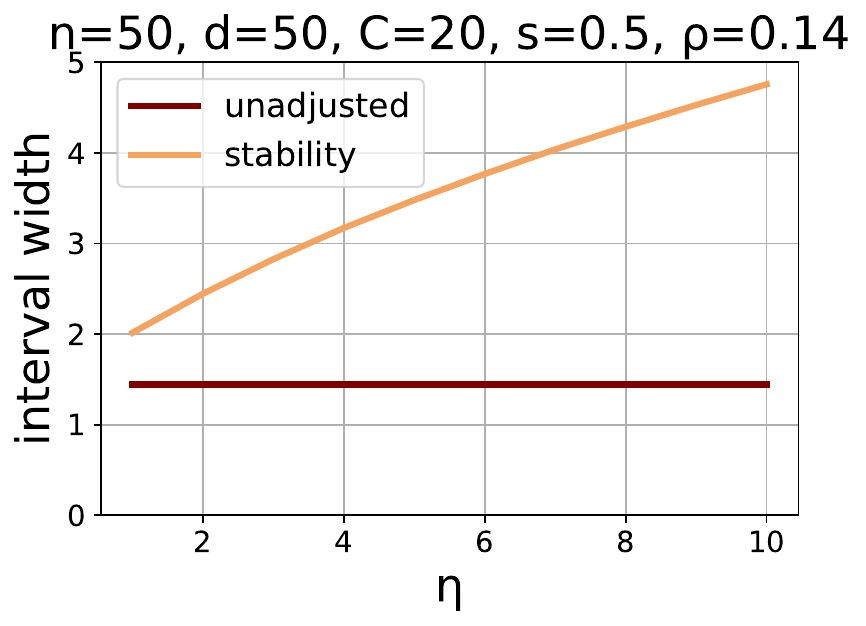}
}
\centerline{\includegraphics[width=0.3\textwidth]{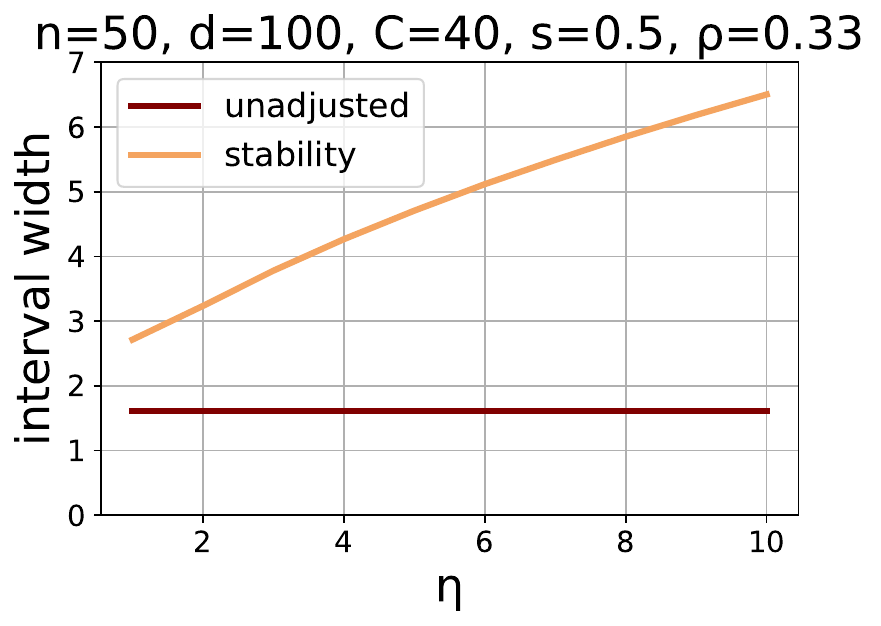}
\includegraphics[width=0.3\textwidth]{plots/LASSOwbern_d100_C40_s05_sig5.pdf}
\includegraphics[width=0.3\textwidth]{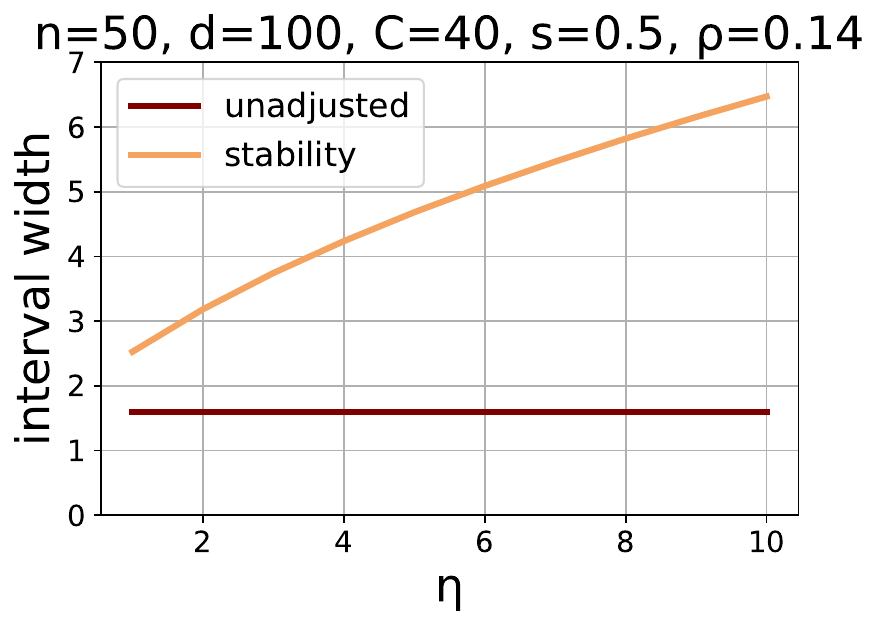}
}
\centerline{\includegraphics[width=0.3\textwidth]{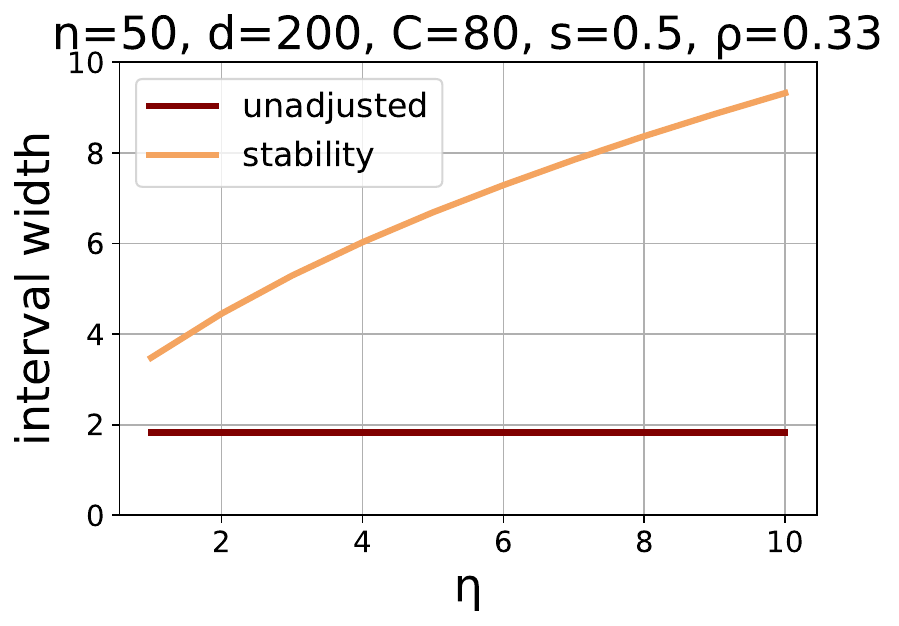}
\includegraphics[width=0.3\textwidth]{plots/LASSOwbern_d200_C80_s05_sig5.pdf}
\includegraphics[width=0.3\textwidth]{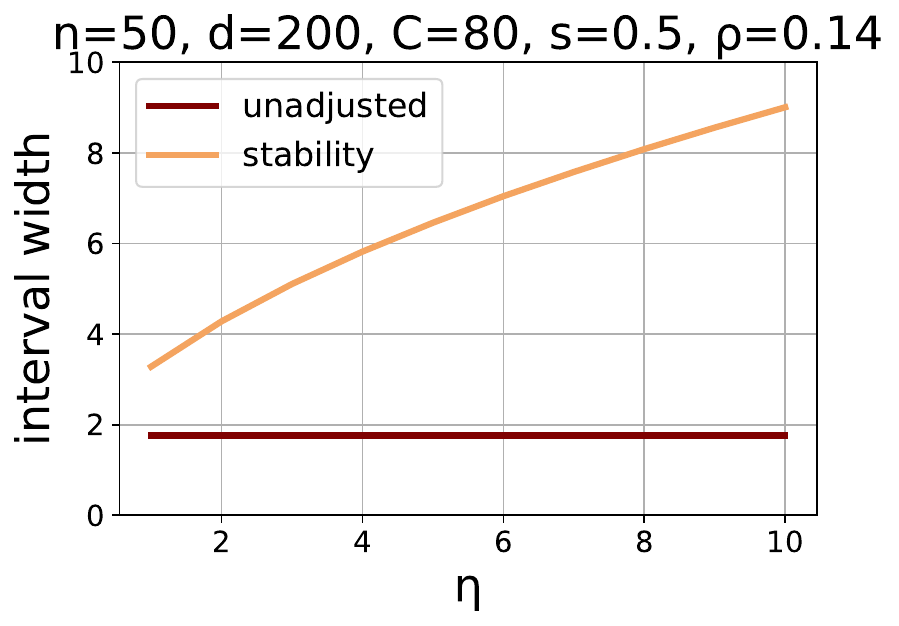}
}
\caption{Plots of average confidence interval widths of stable LASSO, with varying dimension and signal strength, in the Bernoulli design case.} 
\label{fig:lassow_comparison_bern}
\end{figure}

\begin{figure}[b]
\centerline{\includegraphics[width=0.3\textwidth]{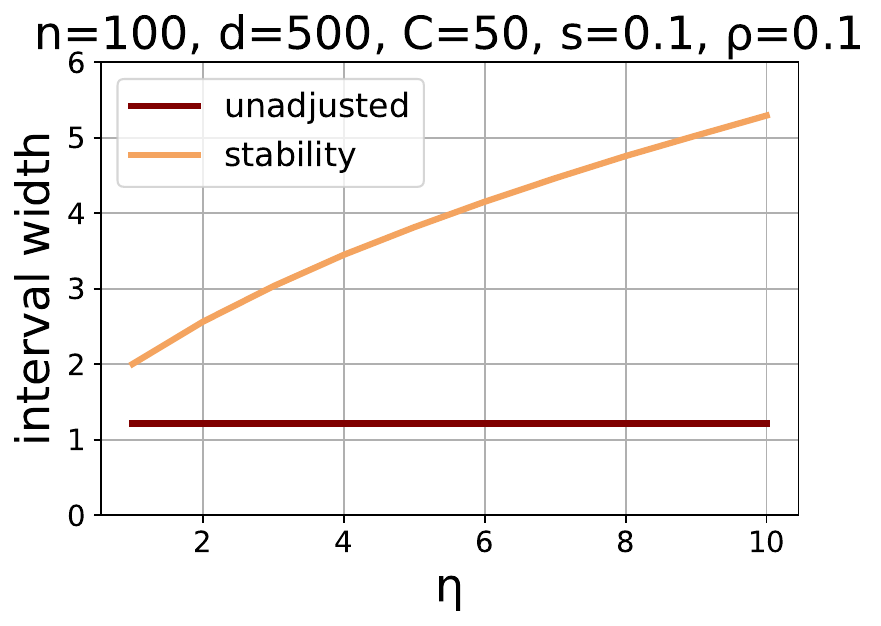}
\includegraphics[width=0.3\textwidth]{plots/LASSOwbern_n200_d500_C50_s01_sig10.pdf}
\includegraphics[width=0.3\textwidth]{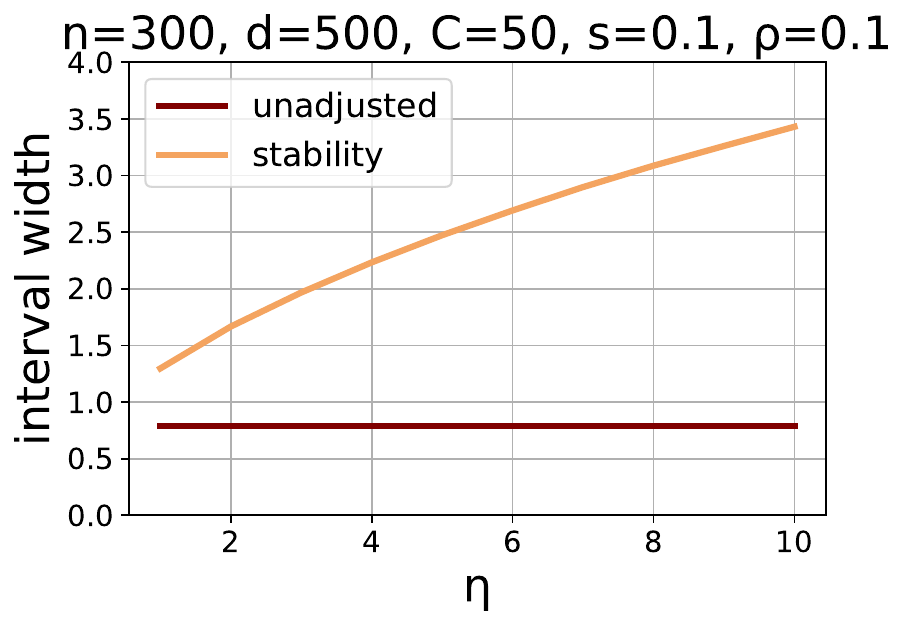}
}
\caption{Plots of average confidence interval widths of stable LASSO, with varying sample size, in the Bernoulli design case.} 
\label{fig:lassow_highdim_bern}
\end{figure}

\begin{figure}[t]
\centerline{\includegraphics[width=0.3\textwidth]{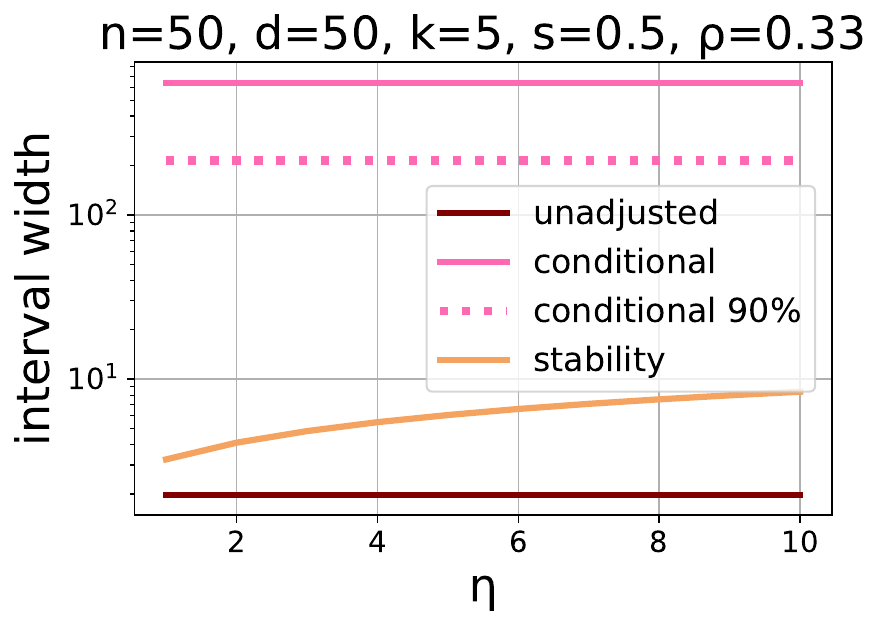}
\includegraphics[width=0.3\textwidth]{plots/MSwbern_d50_k5_s05_sig5.pdf}
\includegraphics[width=0.3\textwidth]{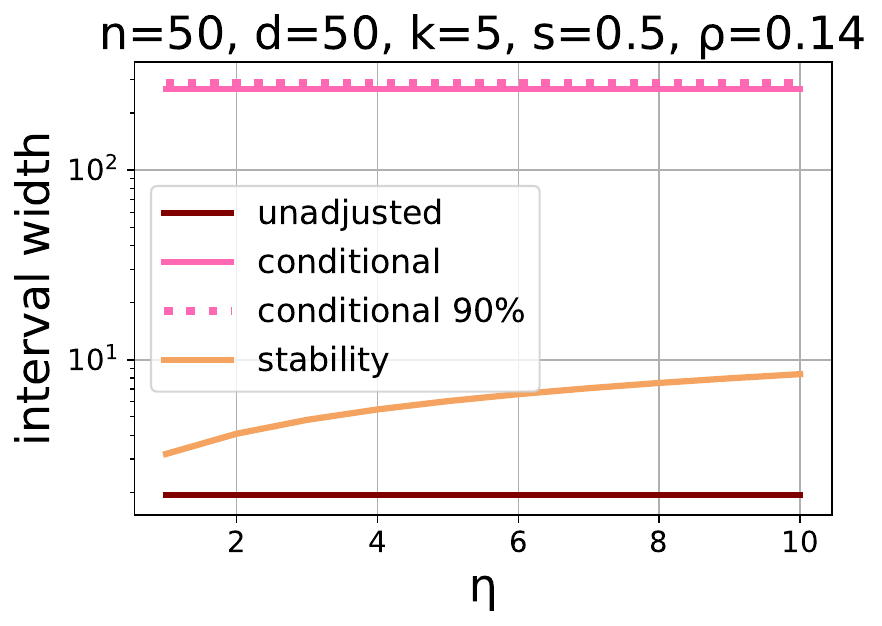}
}
\centerline{\includegraphics[width=0.3\textwidth]{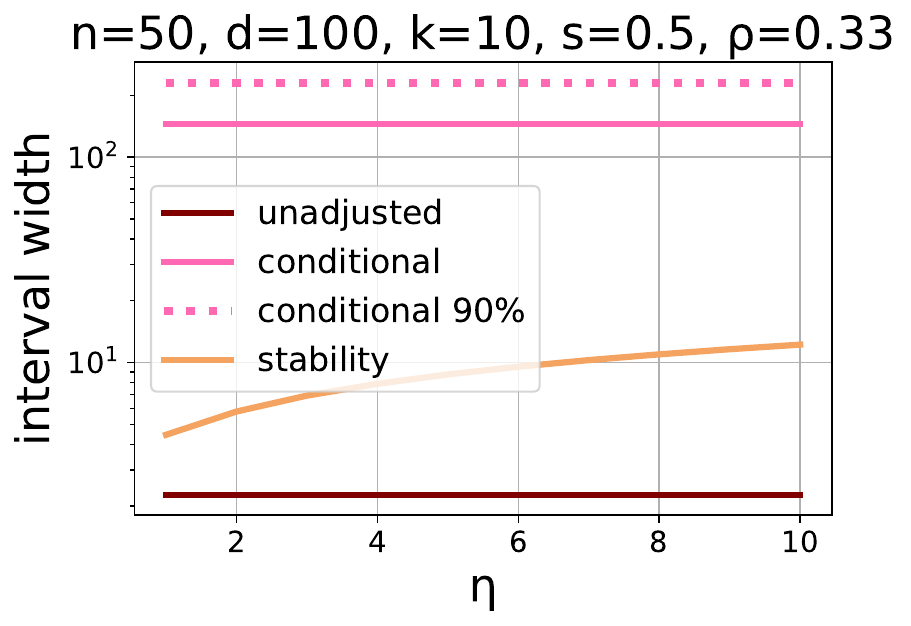}
\includegraphics[width=0.3\textwidth]{plots/MSwbern_d100_k10_s05_sig5.pdf}
\includegraphics[width=0.3\textwidth]{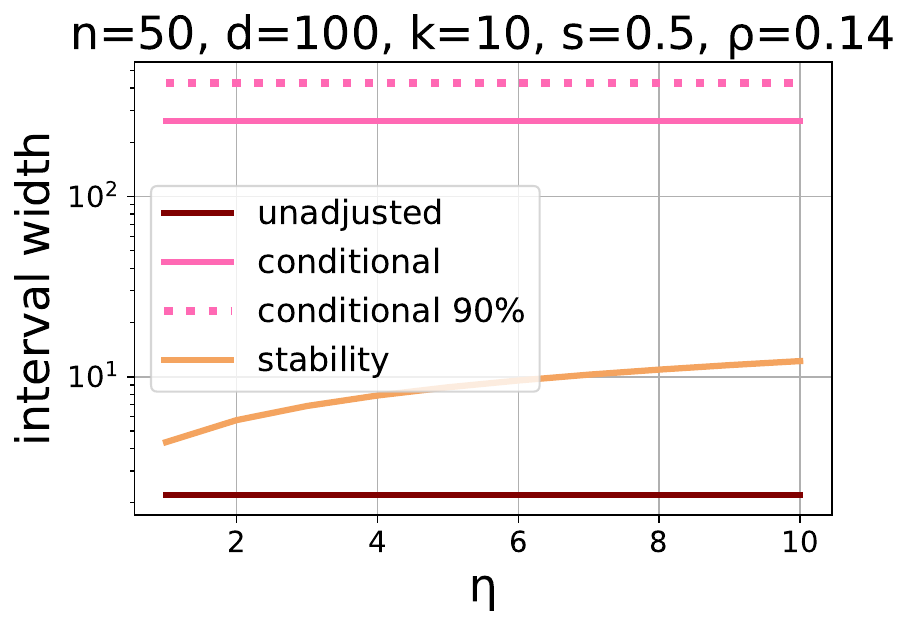}
}
\centerline{\includegraphics[width=0.3\textwidth]{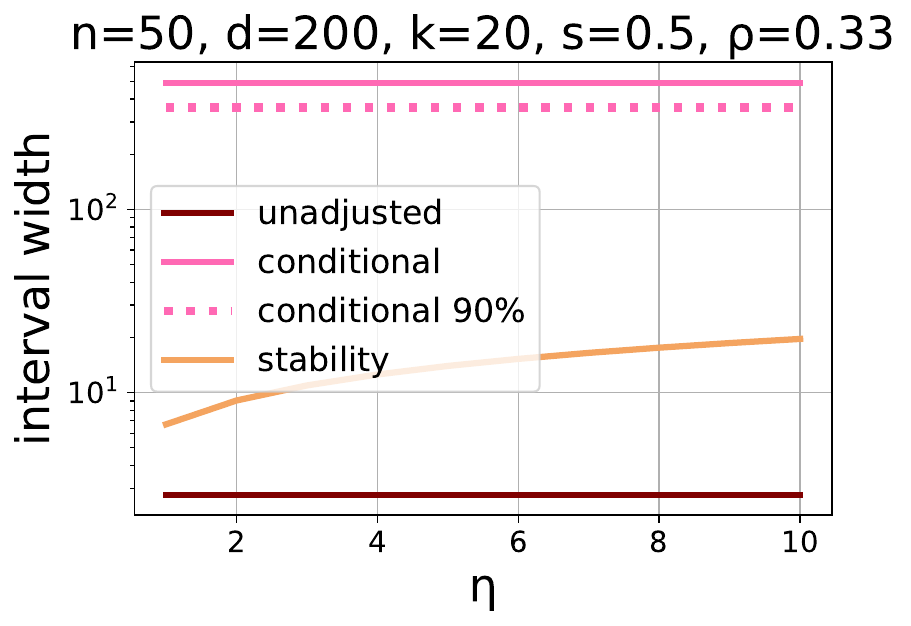}
\includegraphics[width=0.3\textwidth]{plots/MSwbern_d200_k20_s05_sig5.pdf}
\includegraphics[width=0.3\textwidth]{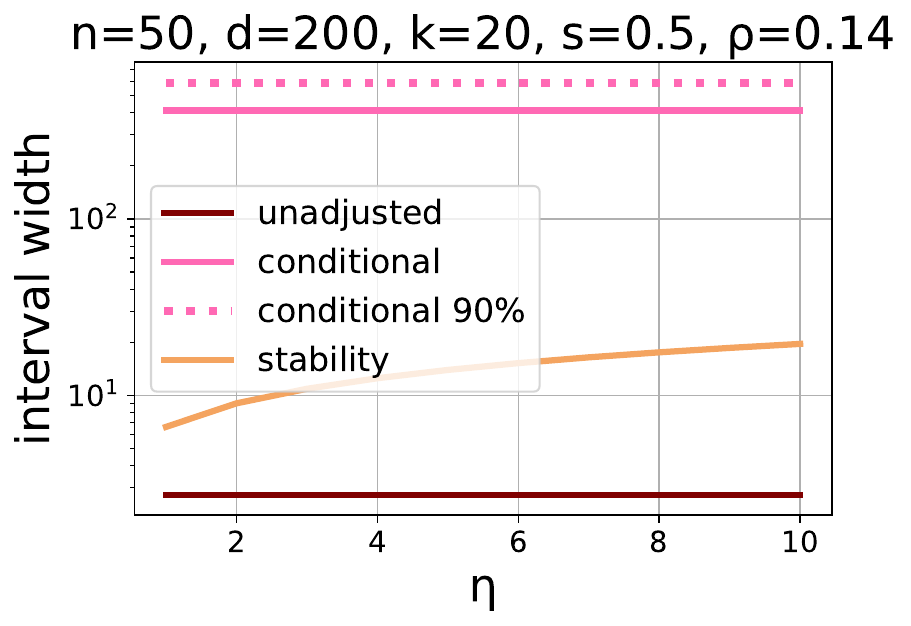}
}
\caption{Plots of average confidence interval widths of stable marginal screening and exact marginal screening with a conditional correction, with varying dimension and signal strength, in the Bernoulli design case. We also plot the $90\%$ quantile of the conditional width because it varies greatly across realizations. 
} 
\label{fig:screeningw_comparison_bern}
\end{figure}

\begin{figure}[b]
\centerline{\includegraphics[width=0.3\textwidth]{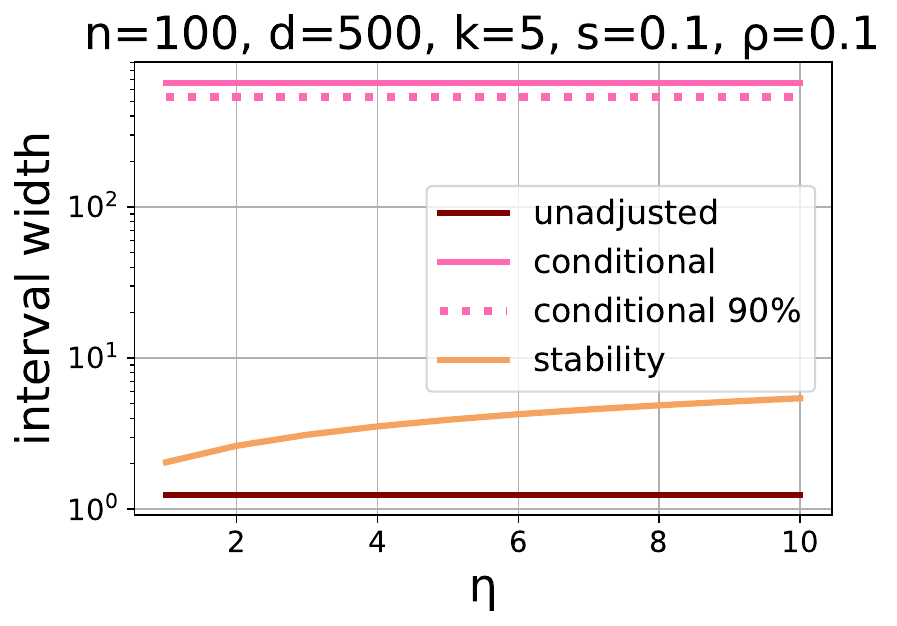}
\includegraphics[width=0.3\textwidth]{plots/MSwbern_n200_d500_k5_s01_sig10.pdf}
\includegraphics[width=0.3\textwidth]{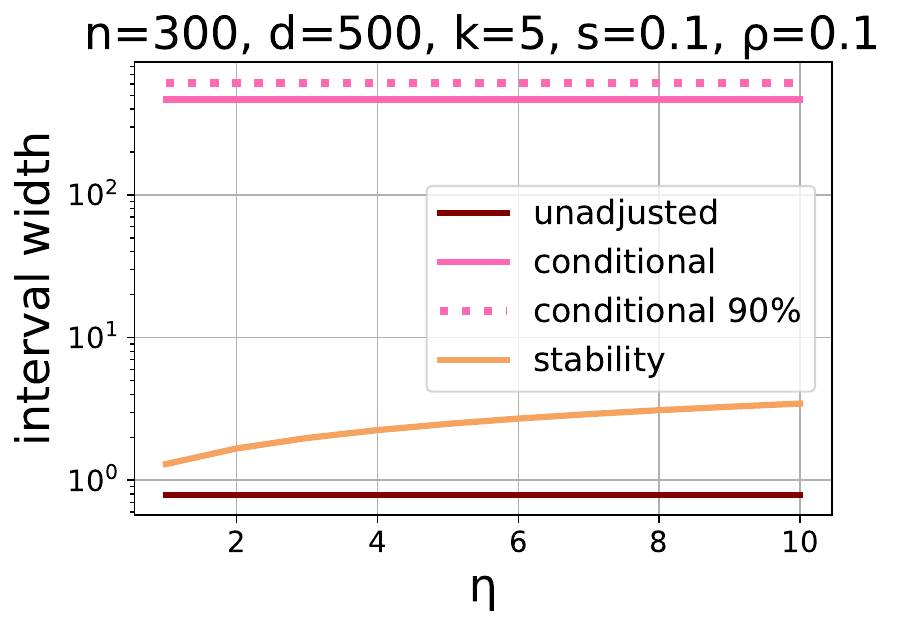}
}
\caption{Plots of average confidence interval widths of stable marginal screening and exact marginal screening with a conditional correction, with varying sample size, in the Bernoulli design case. We also plot the $90\%$ quantile of the conditional width because it varies greatly across realizations. 
} 
\label{fig:msw_highdim_bern}
\end{figure}

\paragraph{Experiments with Laplace errors.}
We include additional experiments when the errors have an exponential, rather than Gaussian, tail: the outcome is generated as $y = X\beta +~\epsilon$, where $\epsilon_i \stackrel{\mathrm{i.i.d.}}{\sim} \mathrm{Lap}\left(\frac{1}{\sqrt{2}}\right), i\in[n]$. The Laplace parameter is chosen so that the errors have variance equal to one. 

For the Gaussian design case, Figure \ref{fig:lasso_comparison_experr} shows plots analogous to those in Figure \ref{fig:lasso_comparison} and Figure \ref{fig:lasso_highdim} in this setting, and Figure \ref{fig:screening_comparison_experr} is the analogue of Figure \ref{fig:screening_comparison} and Figure \ref{fig:ms_highdim}. Similarly, for the Bernoulli design case, Figure \ref{fig:lasso_comparison_bern_experr} is analogous to Figure \ref{fig:lasso_comparison_bern} and Figure \ref{fig:lasso_highdim_bern}, and Figure \ref{fig:screening_comparison_bern_experr} is analogous to Figure~\ref{fig:screening_comparison_bern} and Figure~\ref{fig:ms_highdim_bern}.

\begin{figure}[h!]
\centerline{\includegraphics[width=0.25\textwidth]{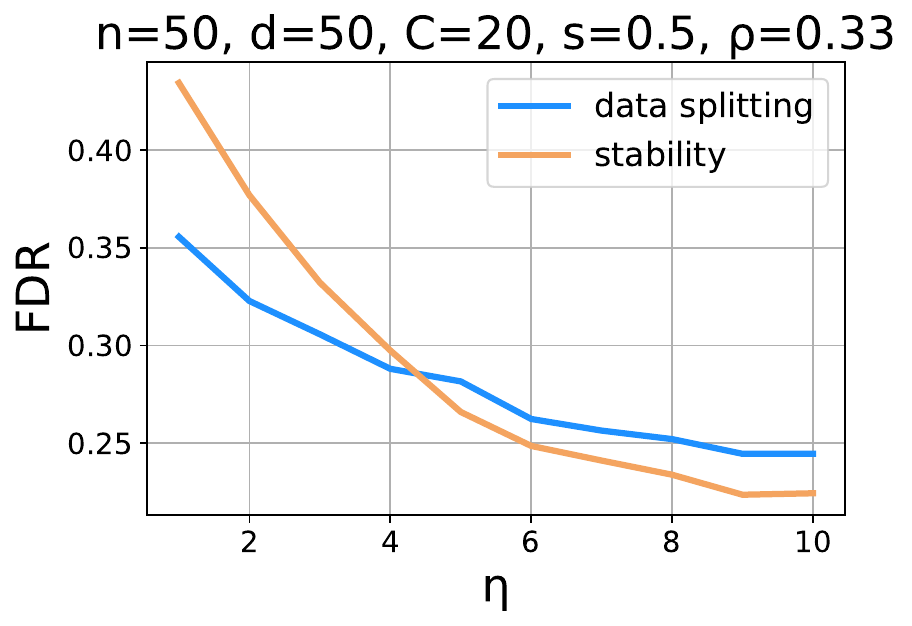}
\includegraphics[width=0.25\textwidth]{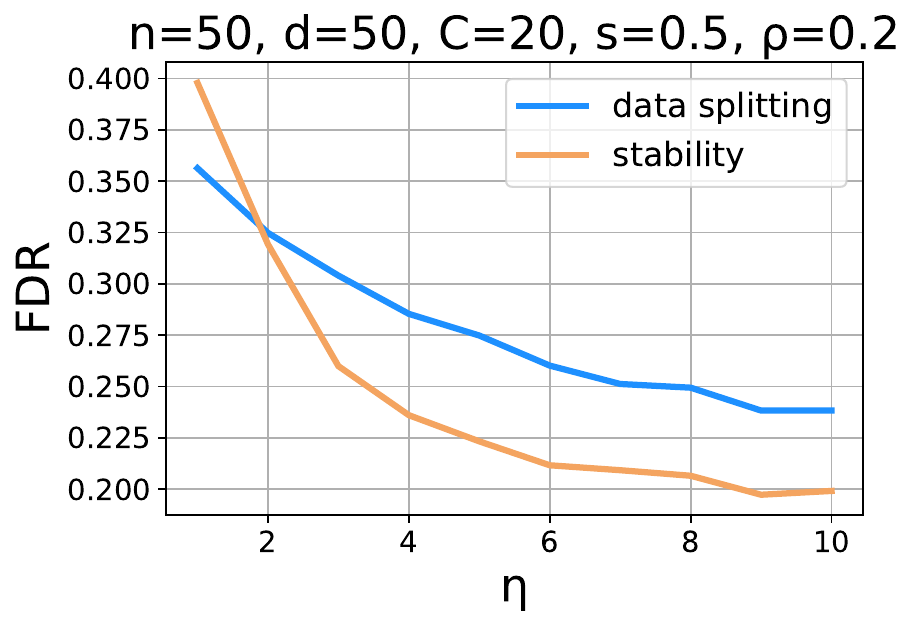}
\includegraphics[width=0.25\textwidth]{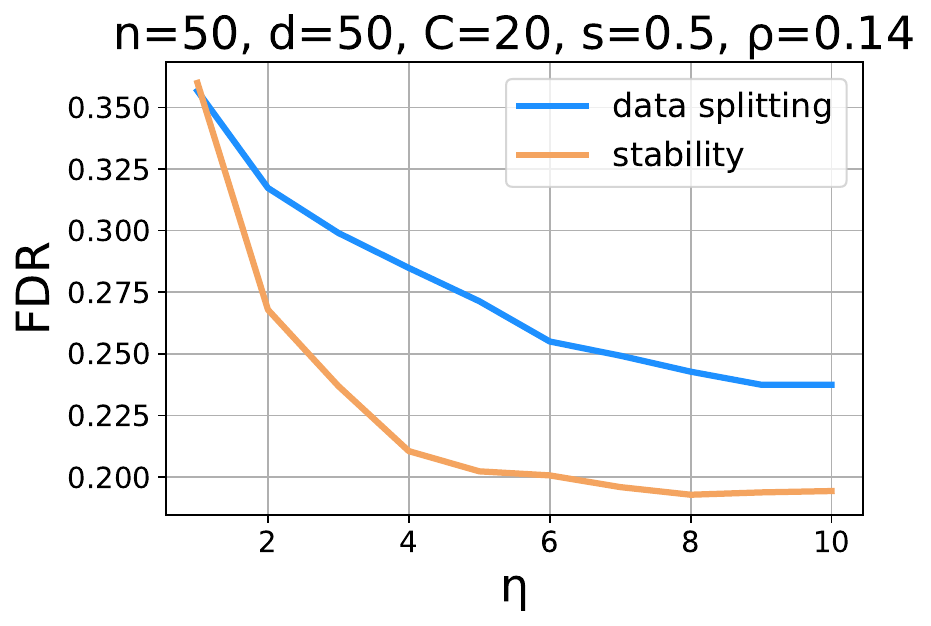}
\includegraphics[width=0.25\textwidth]{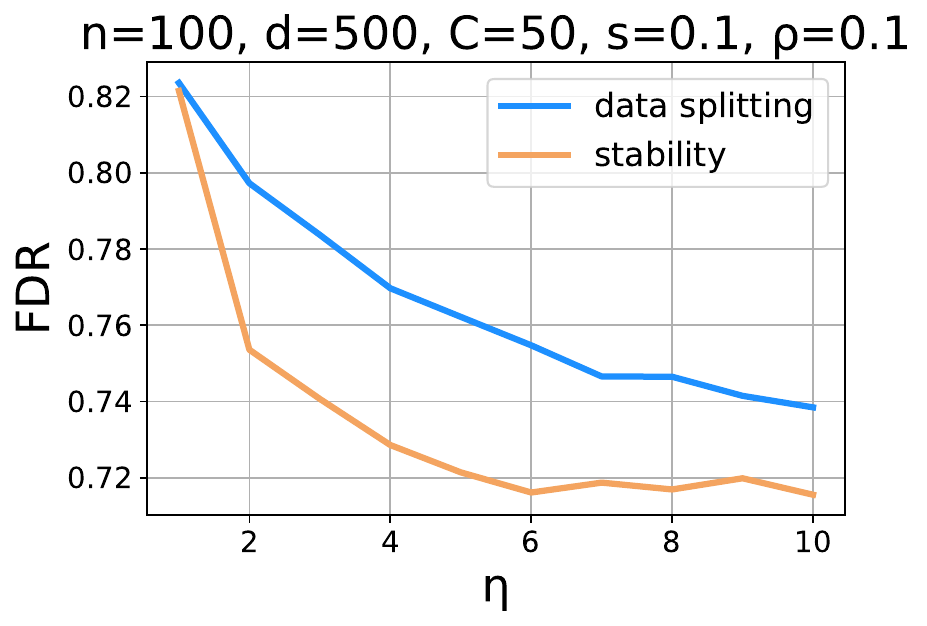}
}
\centerline{\includegraphics[width=0.25\textwidth]{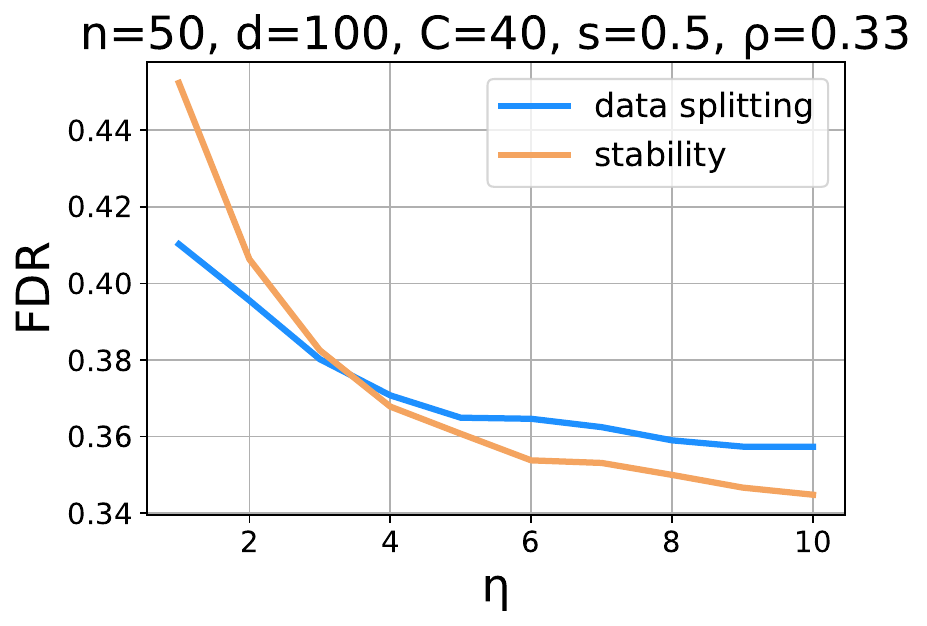}
\includegraphics[width=0.25\textwidth]{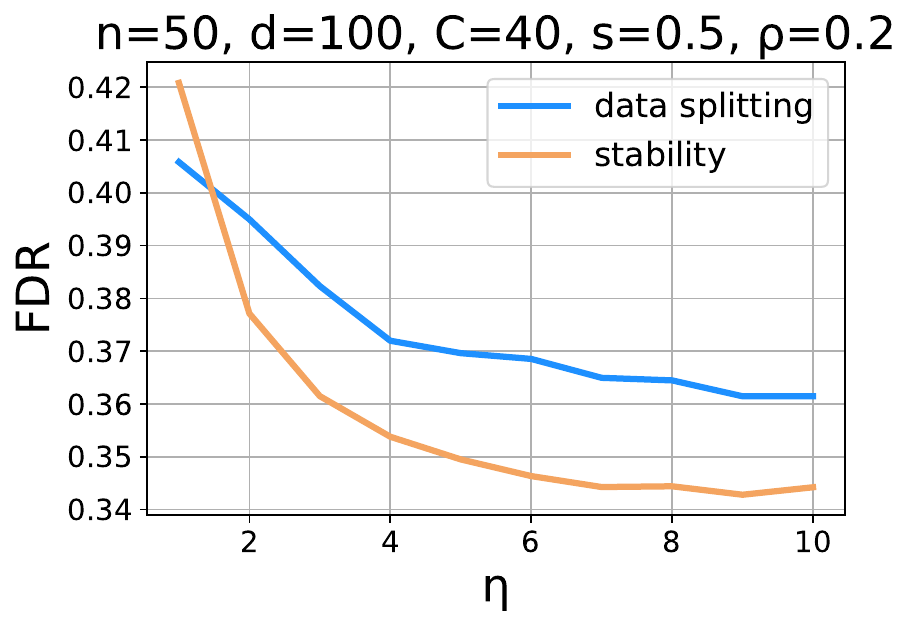}
\includegraphics[width=0.25\textwidth]{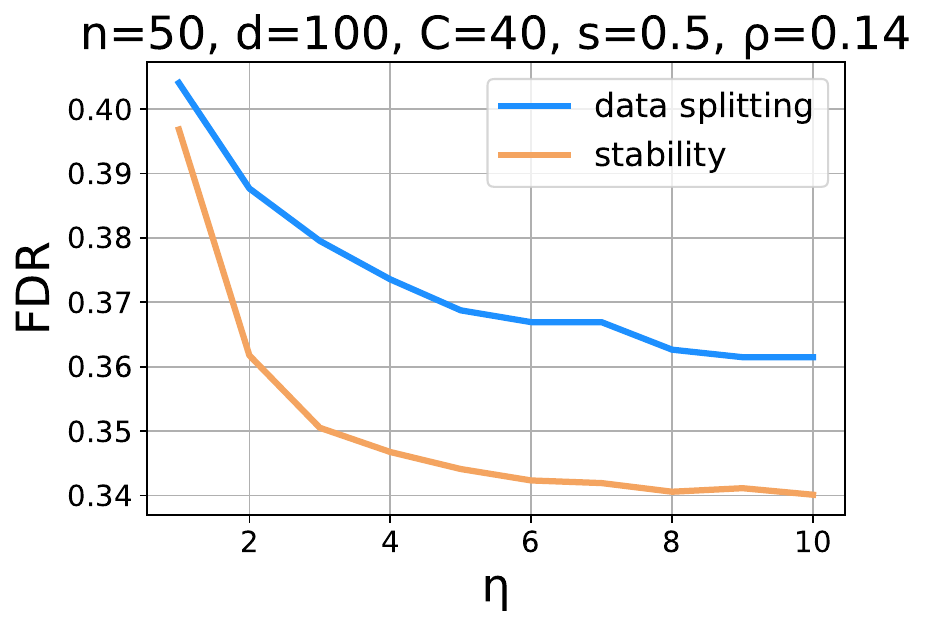}
\includegraphics[width=0.25\textwidth]{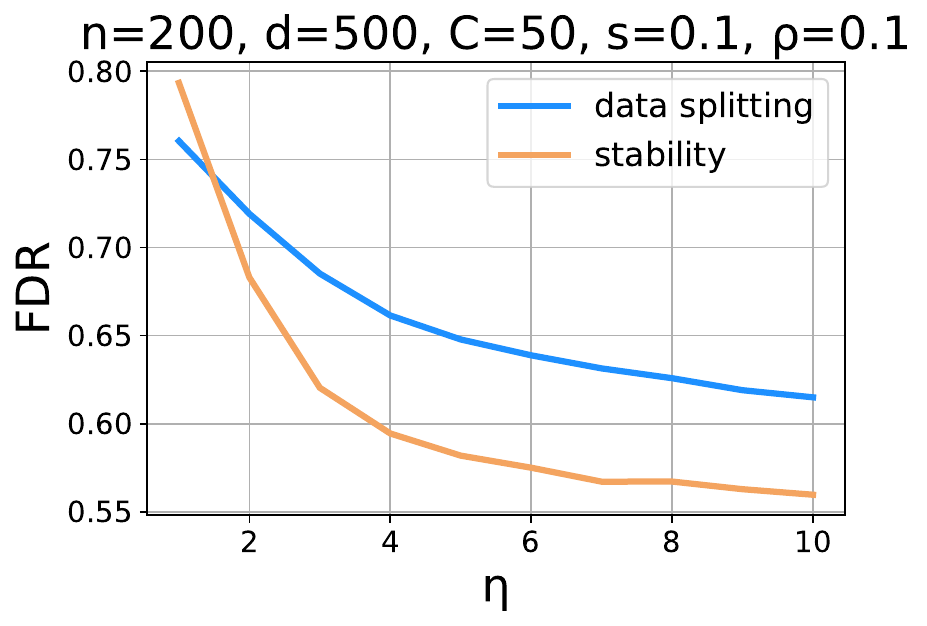}
}
\centerline{\includegraphics[width=0.25\textwidth]{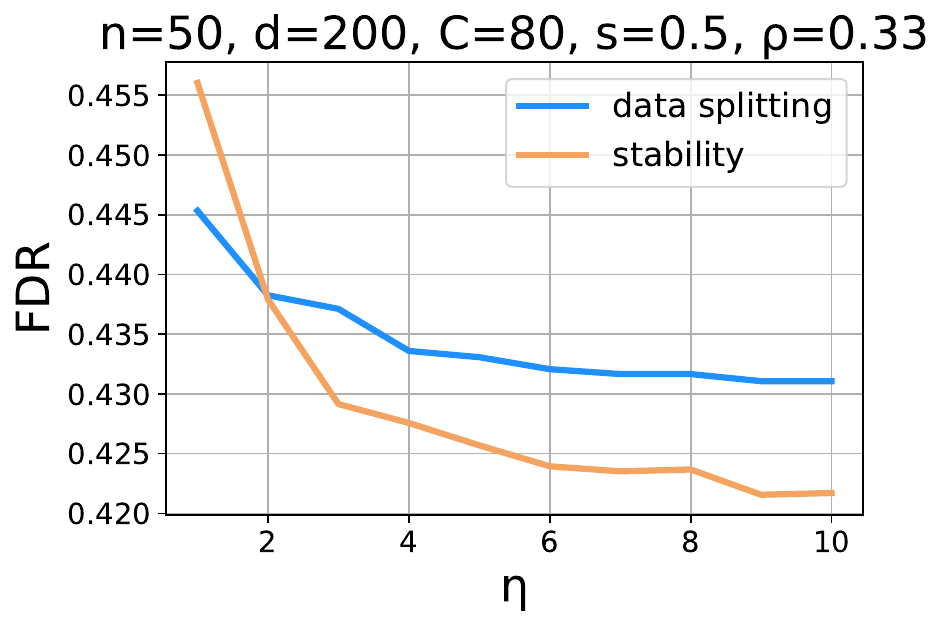}
\includegraphics[width=0.25\textwidth]{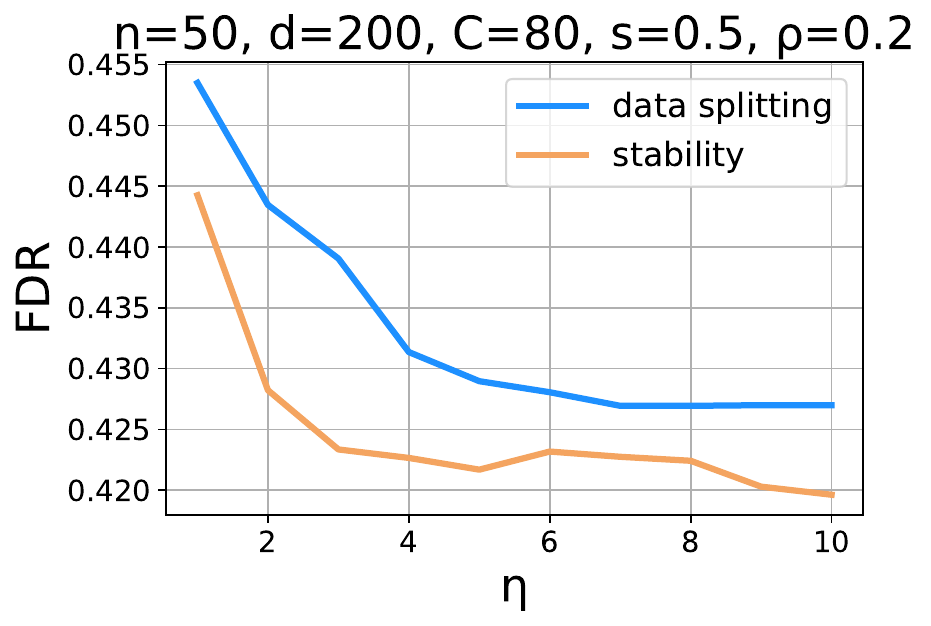}
\includegraphics[width=0.25\textwidth]{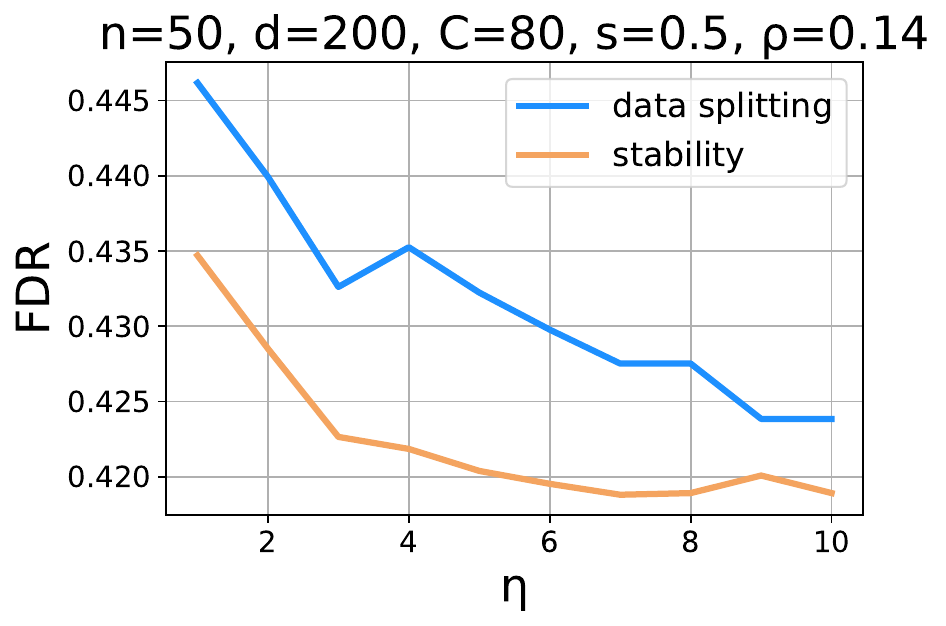}
\includegraphics[width=0.25\textwidth]{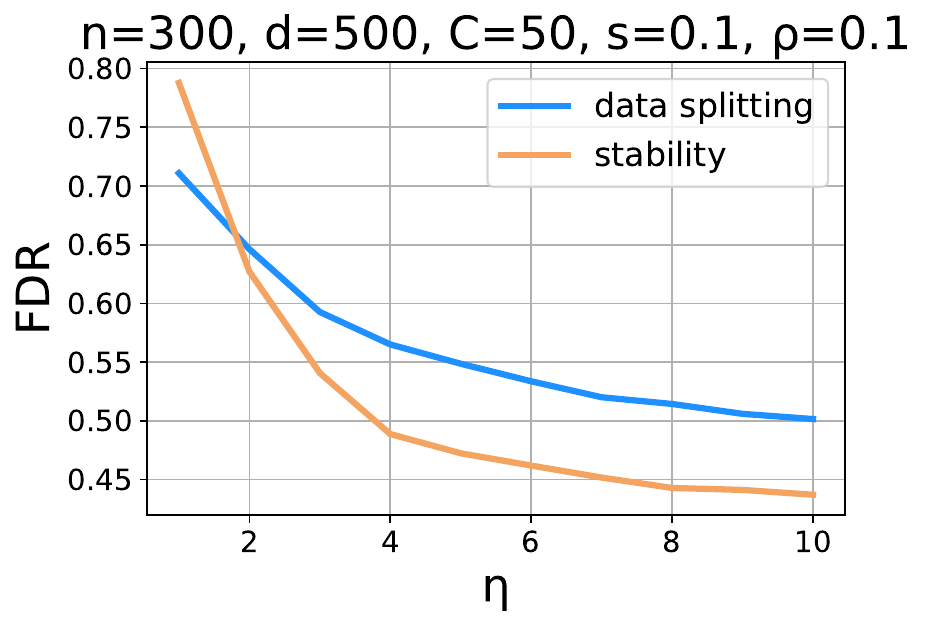}
}
\caption{Comparison of FDR after stable LASSO and LASSO with data splitting, with varying dimension and signal strength (first three columns) and sample size (last column), in the Gaussian design case. The errors are sampled from a Laplace distribution.} 
\label{fig:lasso_comparison_experr}
\end{figure}

\begin{figure}[h!]
\centerline{\includegraphics[width=0.25\textwidth]{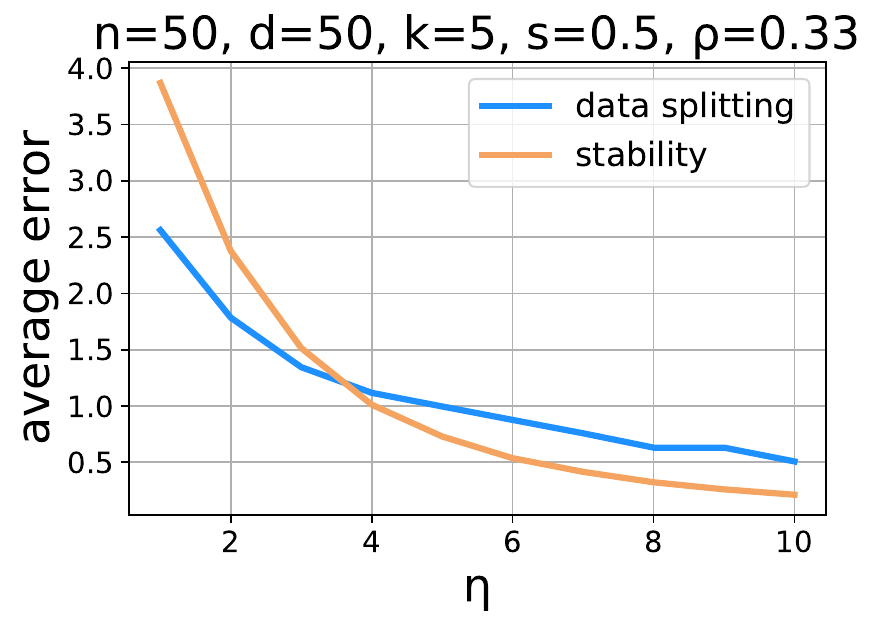}
\includegraphics[width=0.25\textwidth]{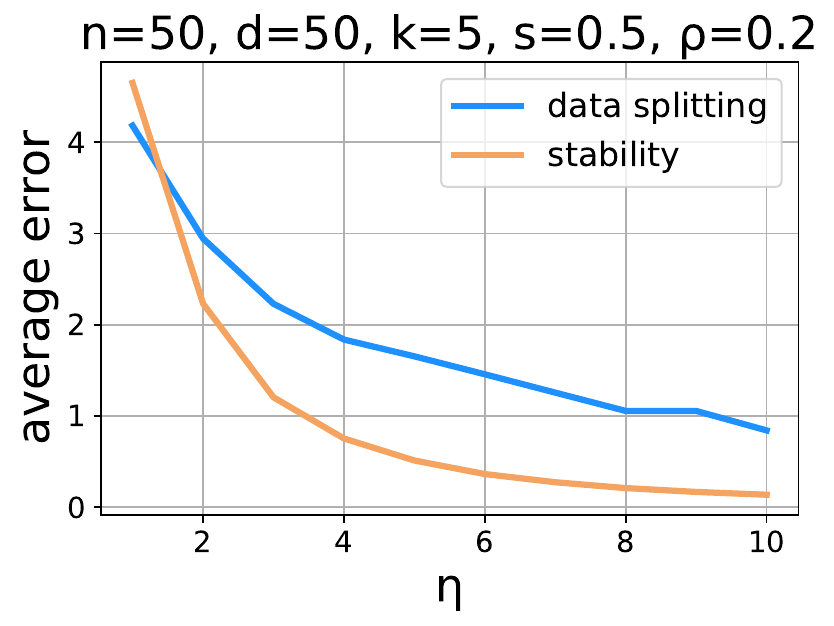}
\includegraphics[width=0.25\textwidth]{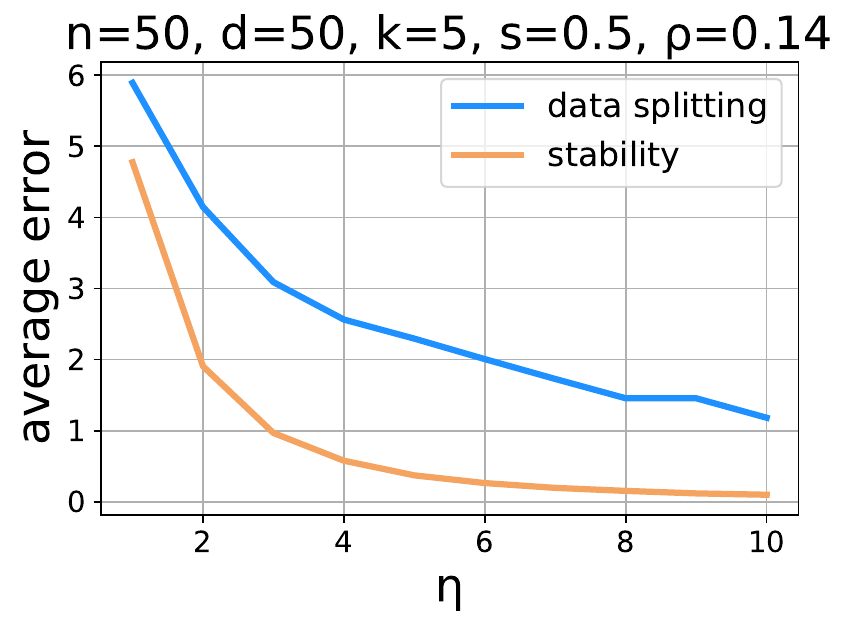}
\includegraphics[width=0.25\textwidth]{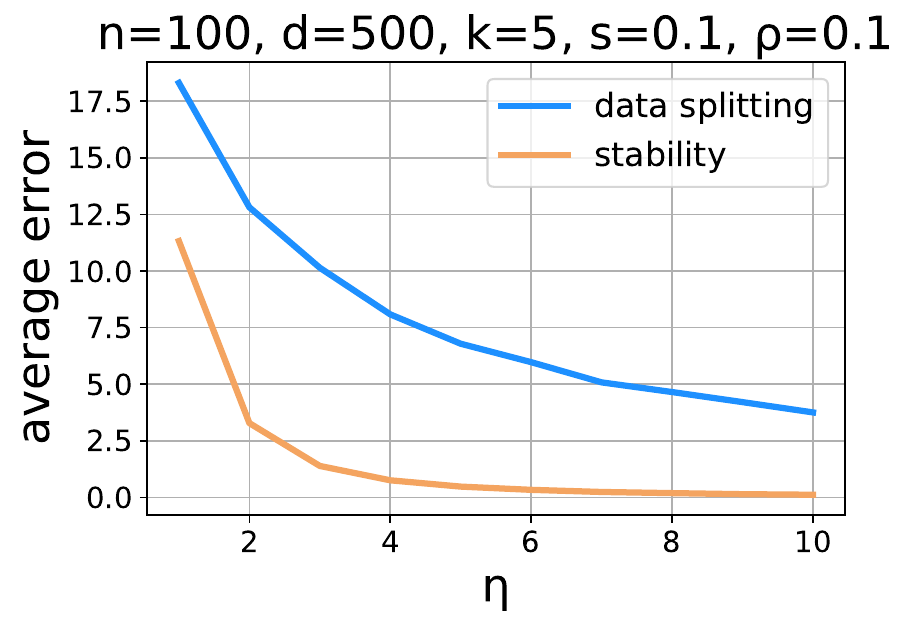}
}
\centerline{\includegraphics[width=0.25\textwidth]{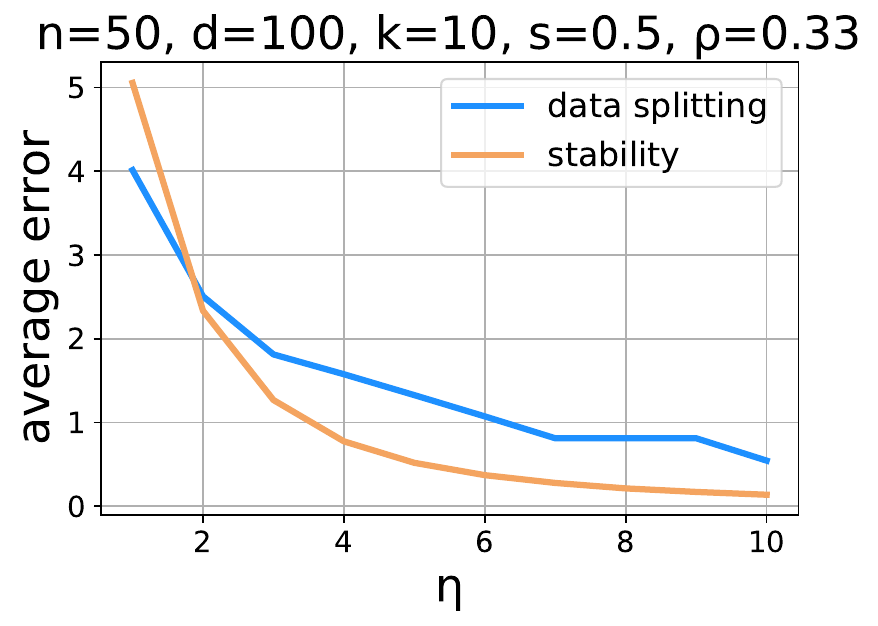}
\includegraphics[width=0.25\textwidth]{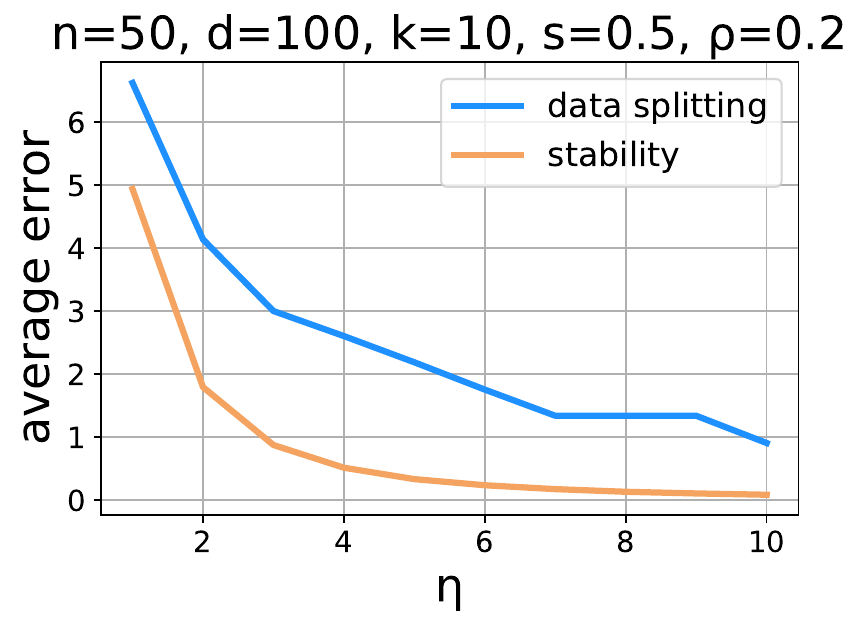}
\includegraphics[width=0.25\textwidth]{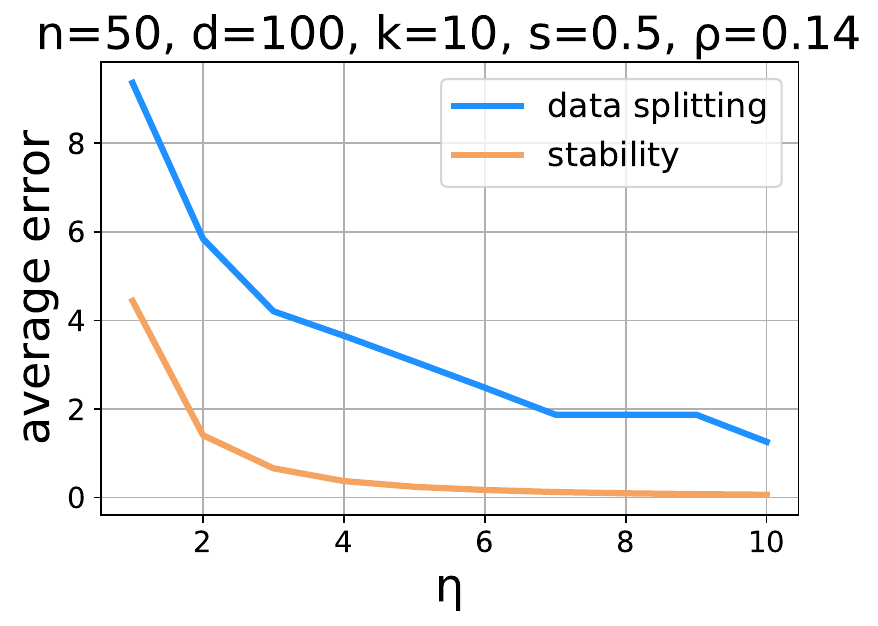}
\includegraphics[width=0.25\textwidth]{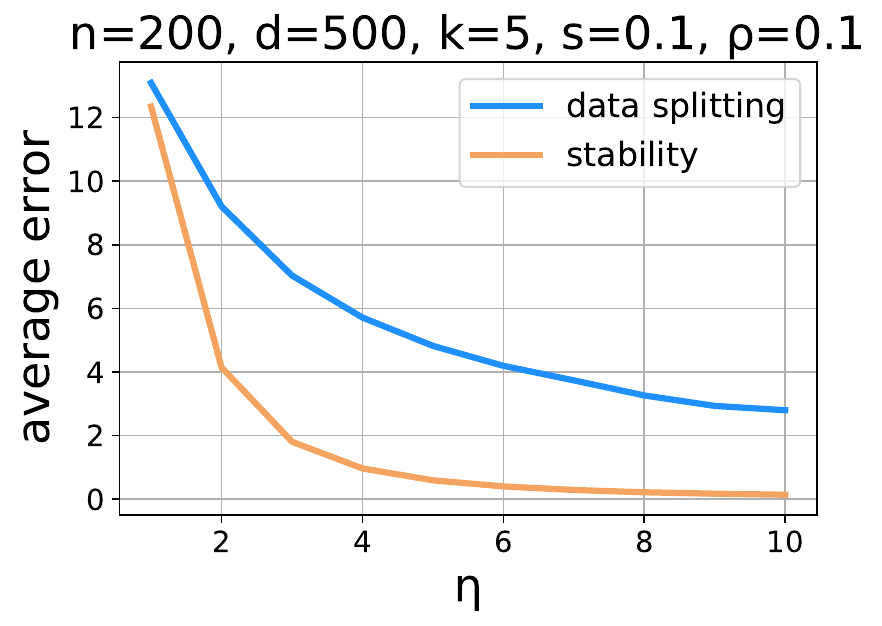}
}
\centerline{\includegraphics[width=0.25\textwidth]{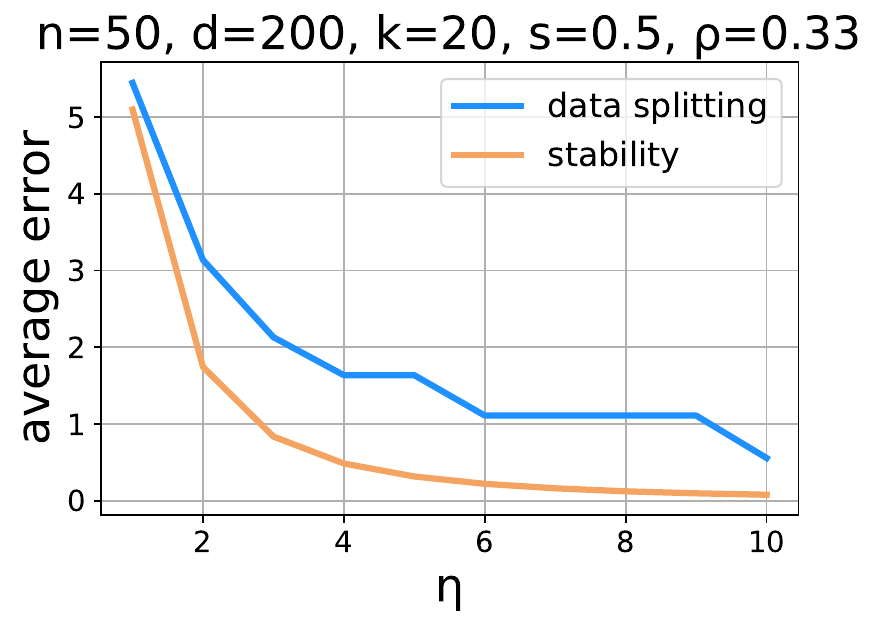}
\includegraphics[width=0.25\textwidth]{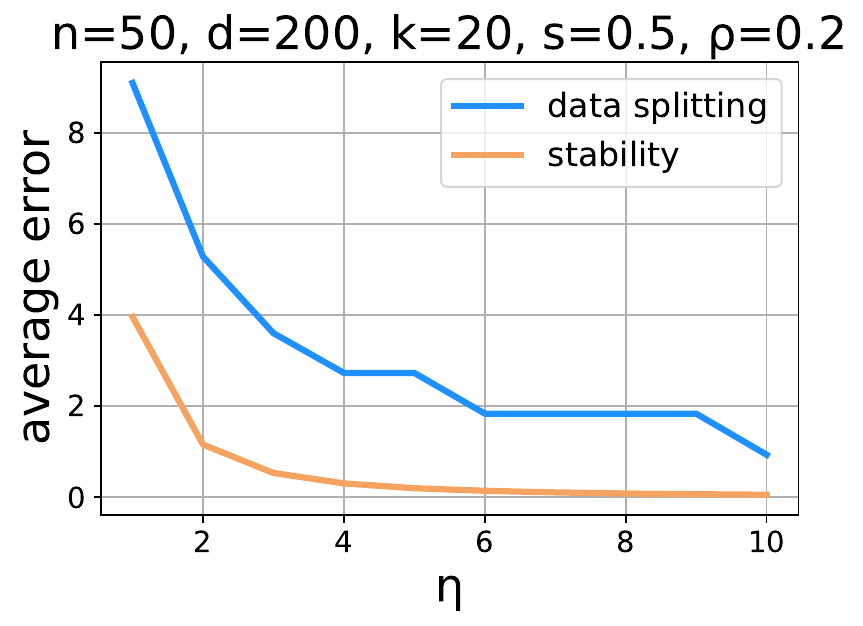}
\includegraphics[width=0.25\textwidth]{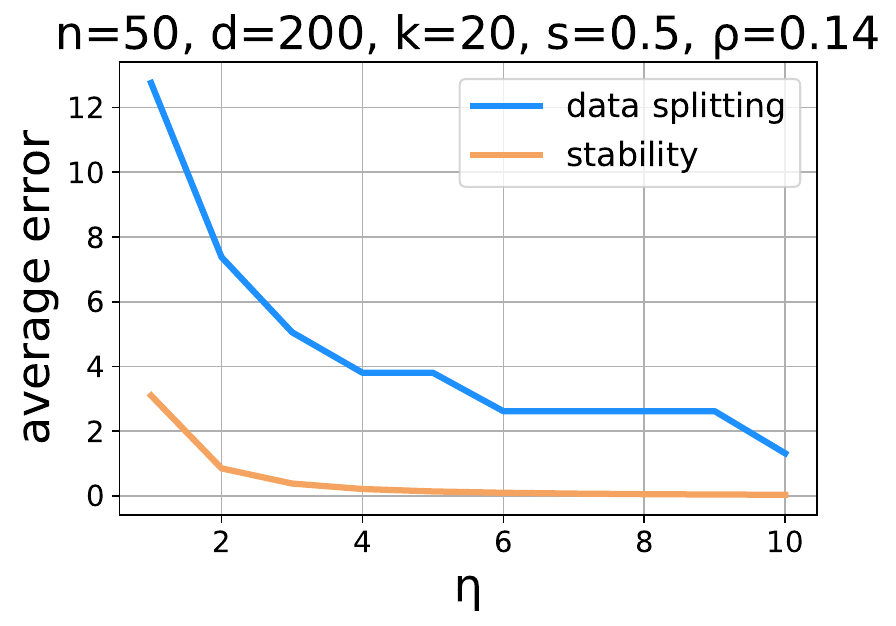}
\includegraphics[width=0.25\textwidth]{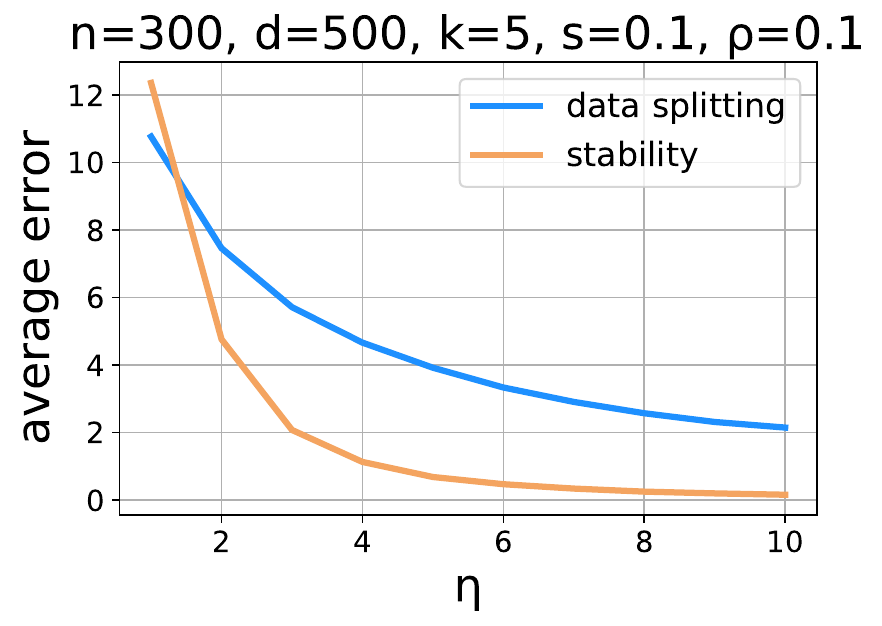}
}
\caption{Comparison of average error after stable marginal screening and marginal screening with data splitting, with varying dimension and signal strength (first three columns) and sample size (last column), in the Gaussian design case. The errors are sampled from a Laplace distribution.} 
\label{fig:screening_comparison_experr}
\end{figure}

\begin{figure}[h!]
\centerline{\includegraphics[width=0.25\textwidth]{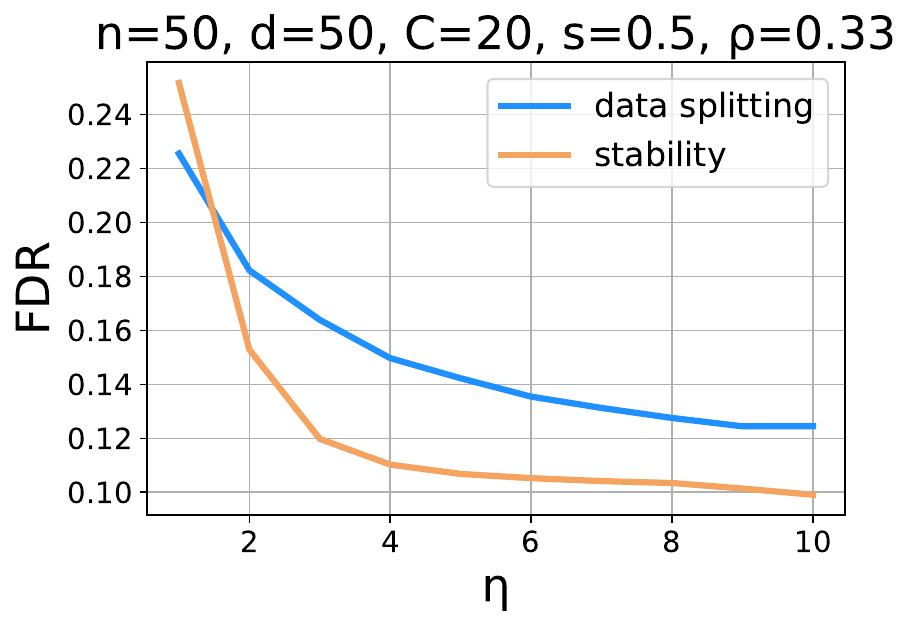}
\includegraphics[width=0.25\textwidth]{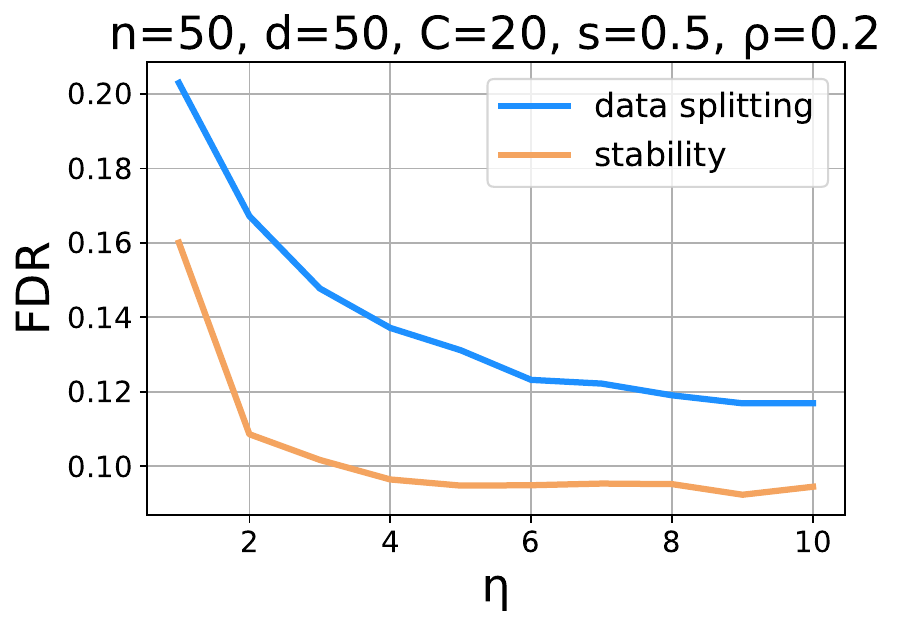}
\includegraphics[width=0.25\textwidth]{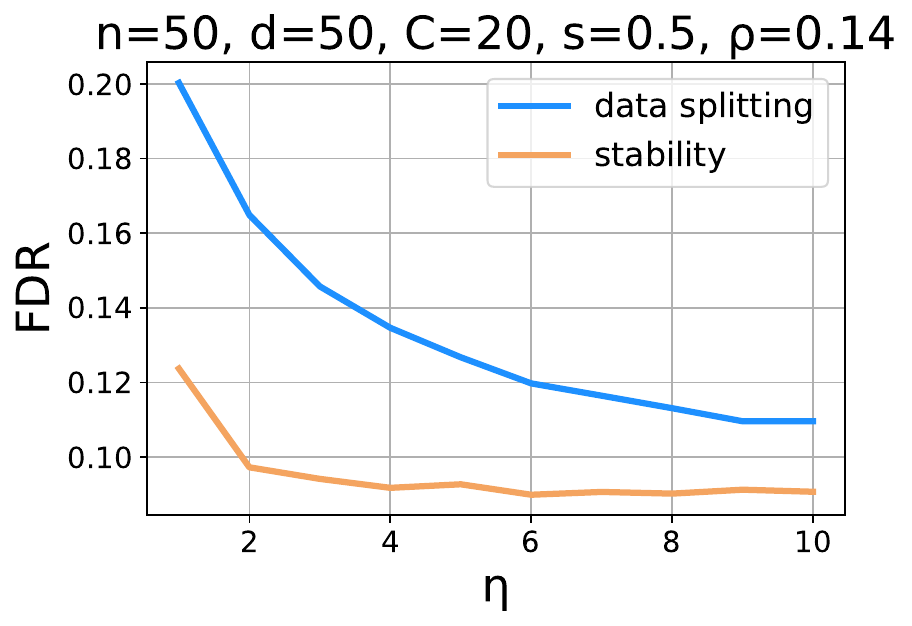}
\includegraphics[width=0.25\textwidth]{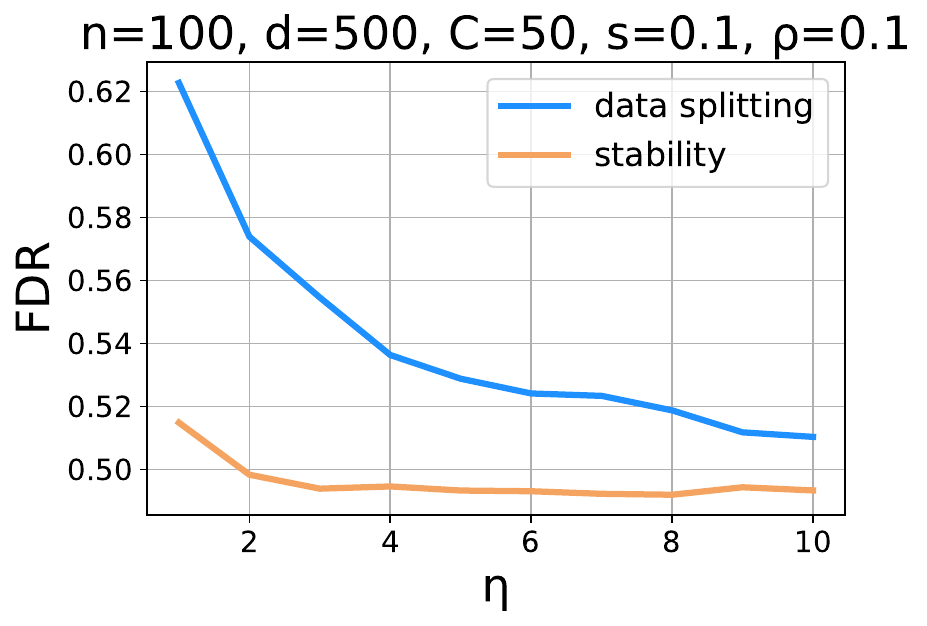}
}
\centerline{\includegraphics[width=0.25\textwidth]{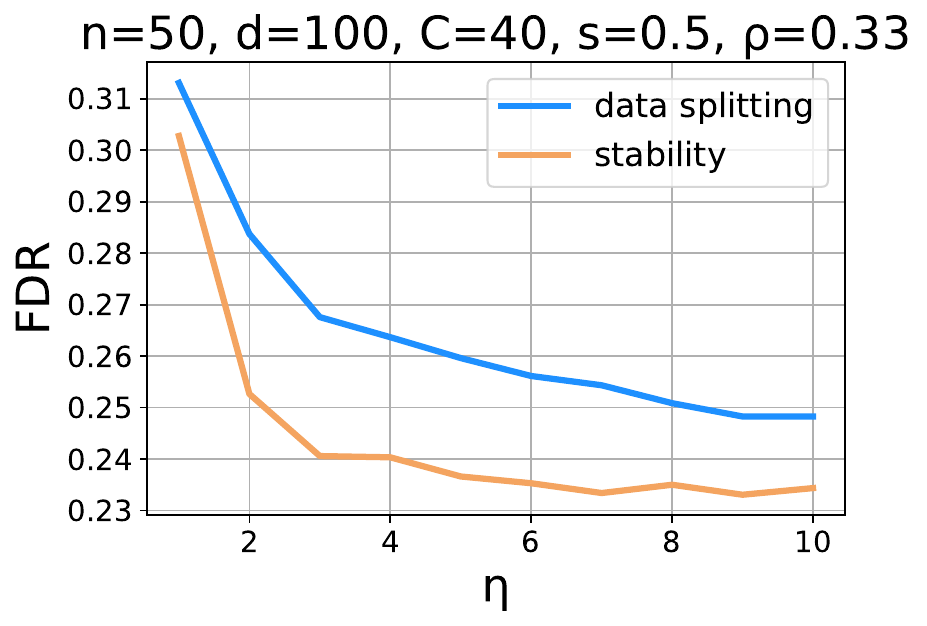}
\includegraphics[width=0.25\textwidth]{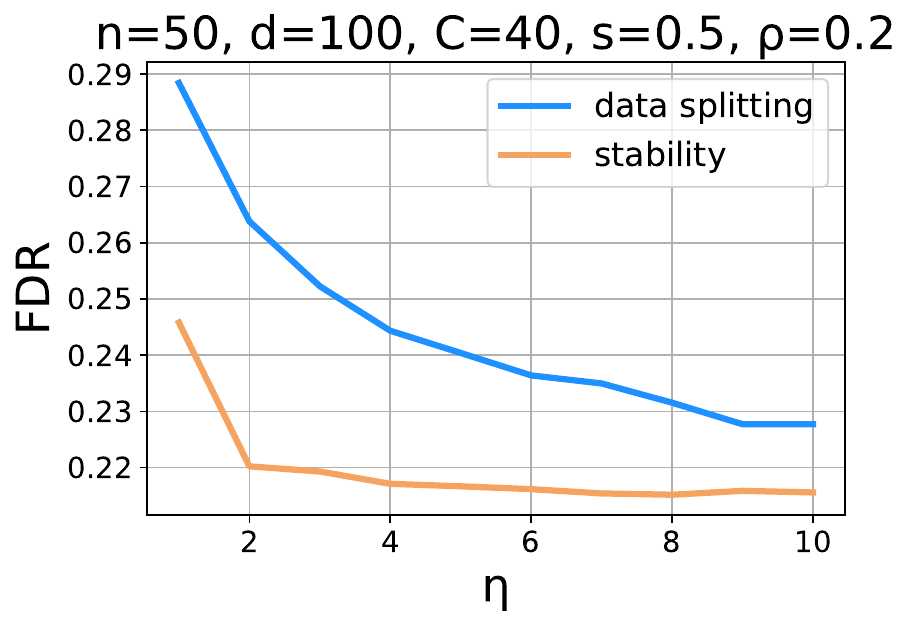}
\includegraphics[width=0.25\textwidth]{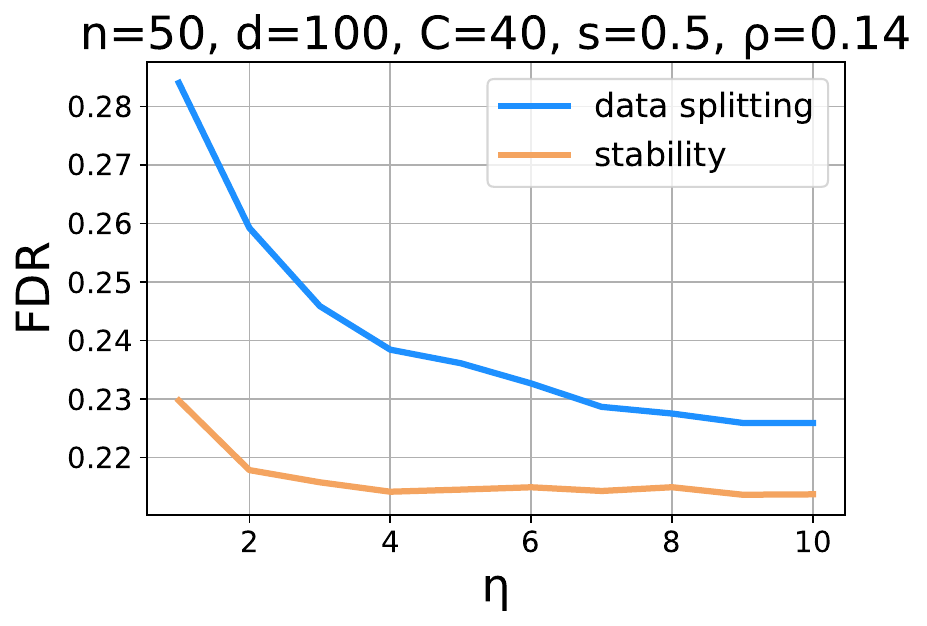}
\includegraphics[width=0.25\textwidth]{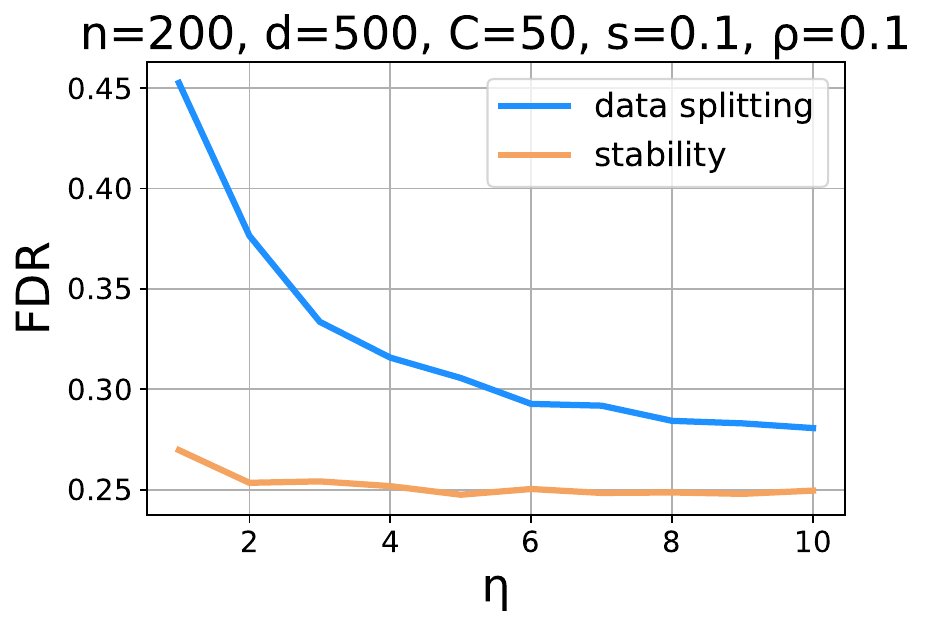}
}
\centerline{\includegraphics[width=0.25\textwidth]{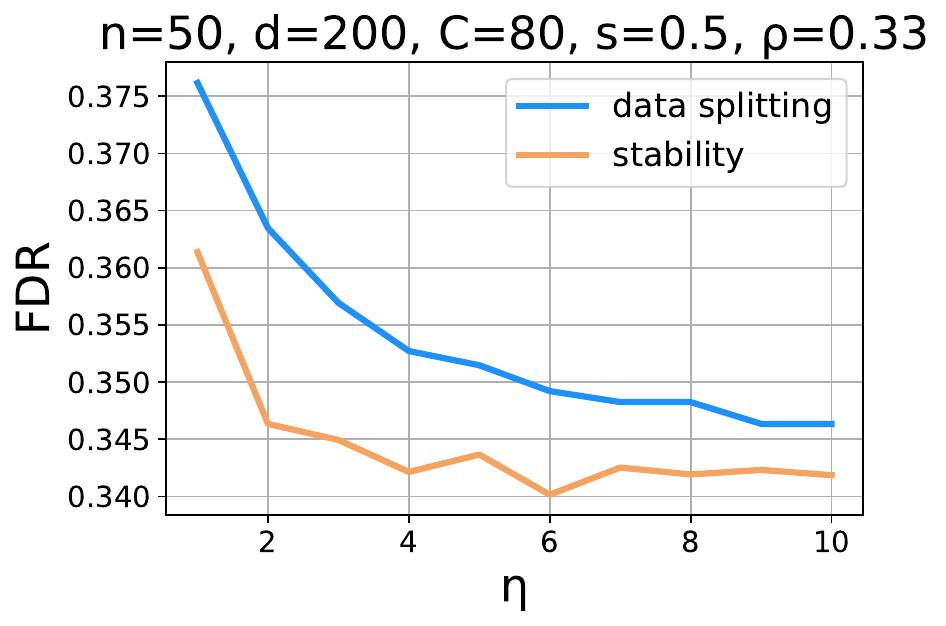}
\includegraphics[width=0.25\textwidth]{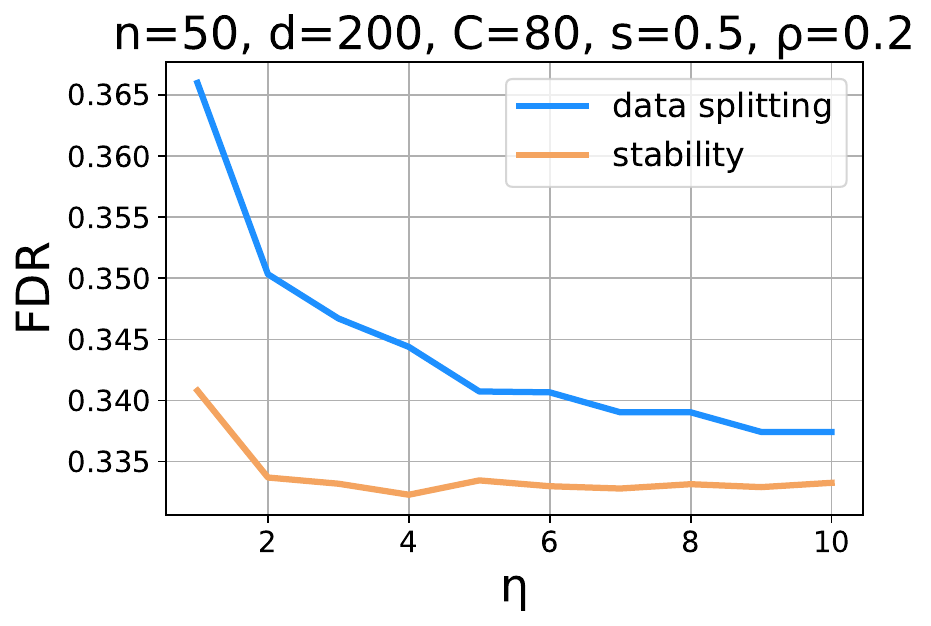}
\includegraphics[width=0.25\textwidth]{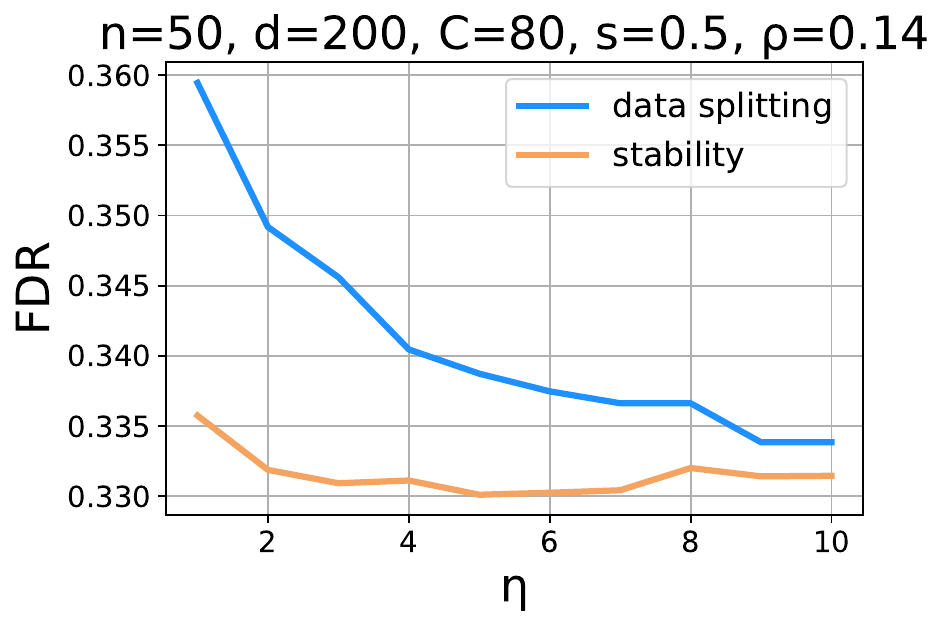}
\includegraphics[width=0.25\textwidth]{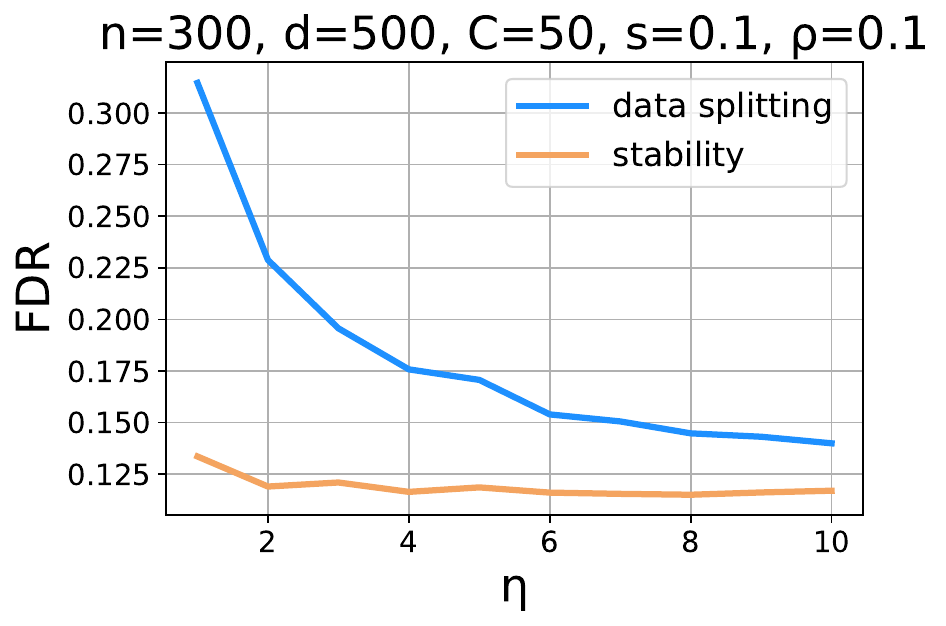}
}
\caption{Comparison of FDR after stable LASSO and LASSO with data splitting, with varying dimension and signal strength (first three columns) and sample size (last column), in the Bernoulli design case. The errors are sampled from a Laplace distribution.}
\label{fig:lasso_comparison_bern_experr}
\end{figure}

\begin{figure}[h!]
\centerline{\includegraphics[width=0.25\textwidth]{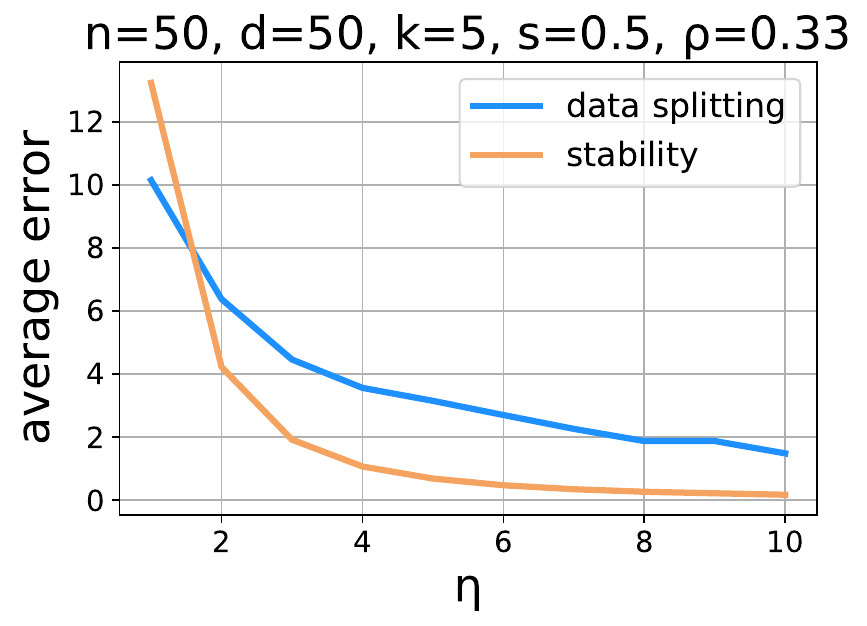}
\includegraphics[width=0.25\textwidth]{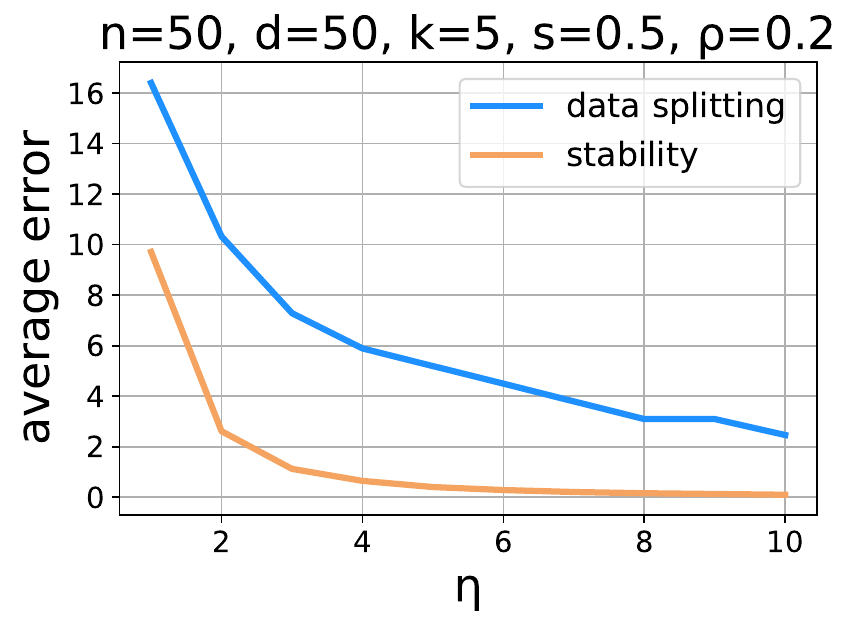}
\includegraphics[width=0.25\textwidth]{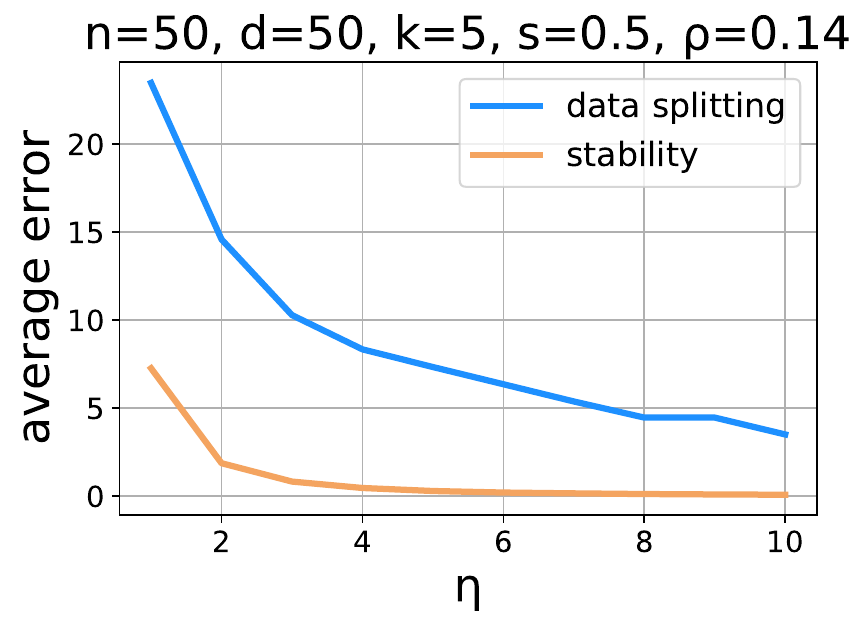}
\includegraphics[width=0.25\textwidth]{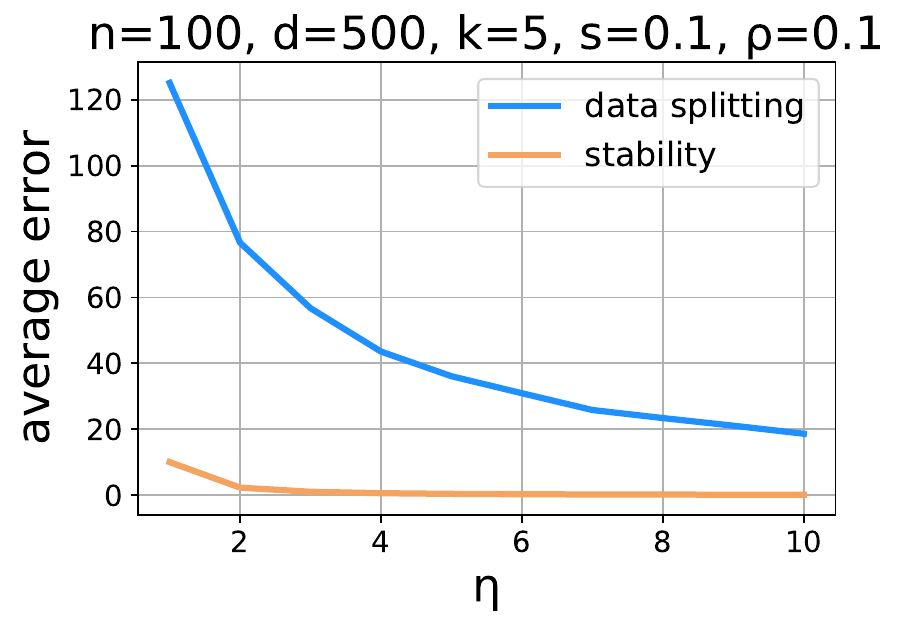}
}
\centerline{\includegraphics[width=0.25\textwidth]{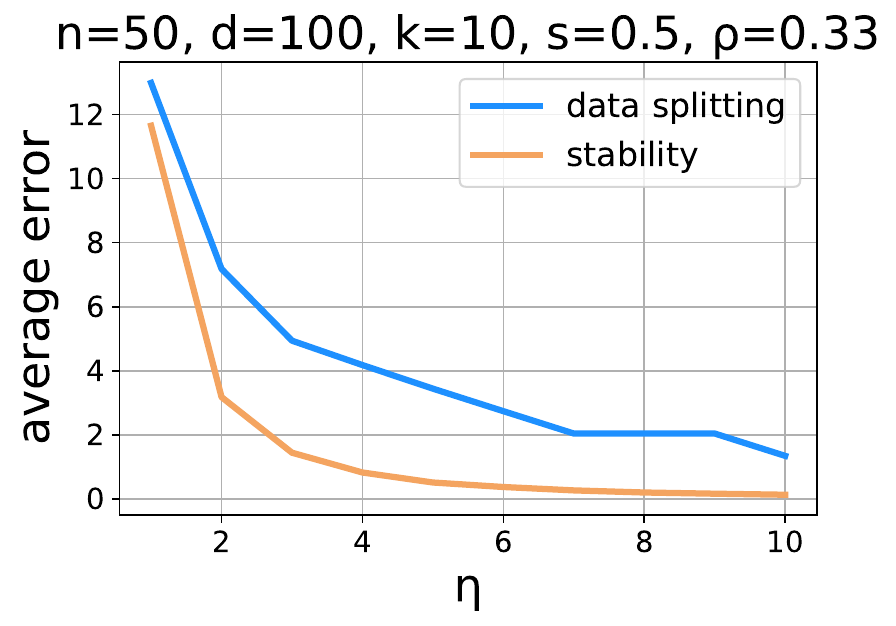}
\includegraphics[width=0.25\textwidth]{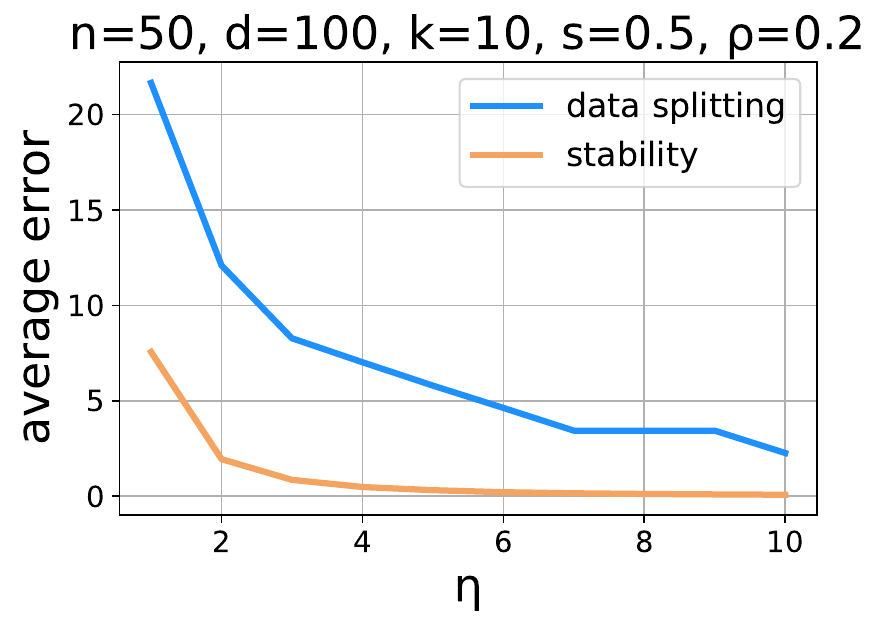}
\includegraphics[width=0.25\textwidth]{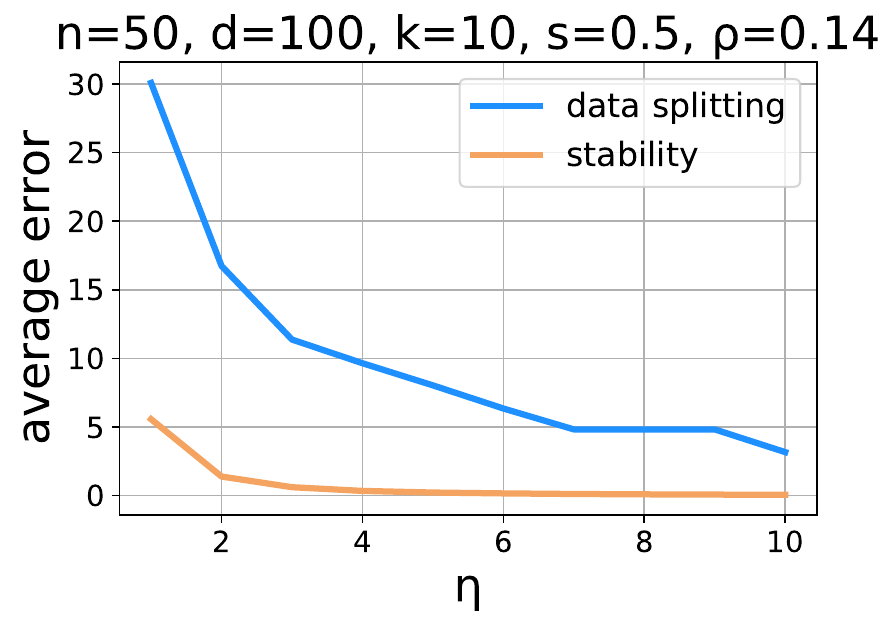}
\includegraphics[width=0.25\textwidth]{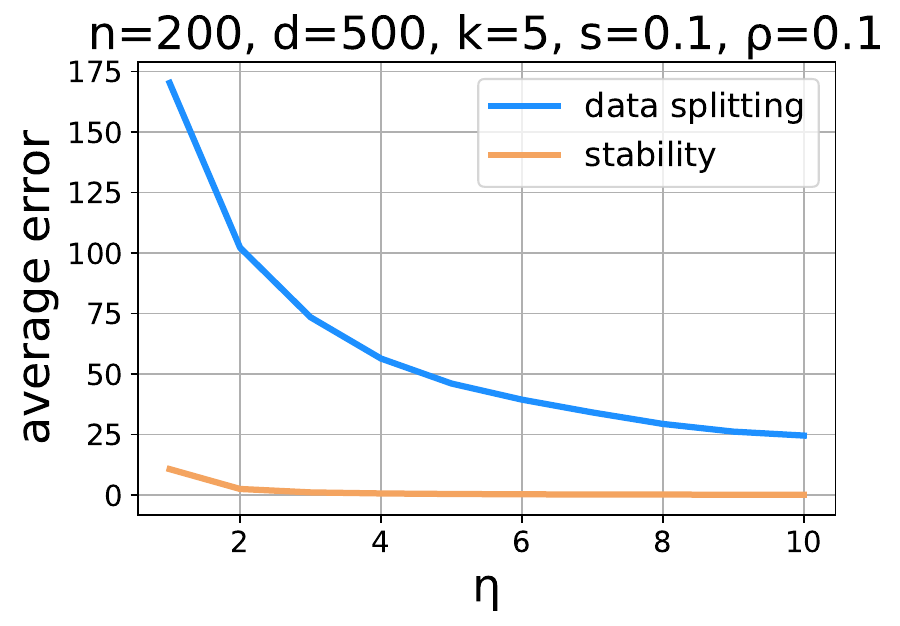}
}
\centerline{\includegraphics[width=0.25\textwidth]{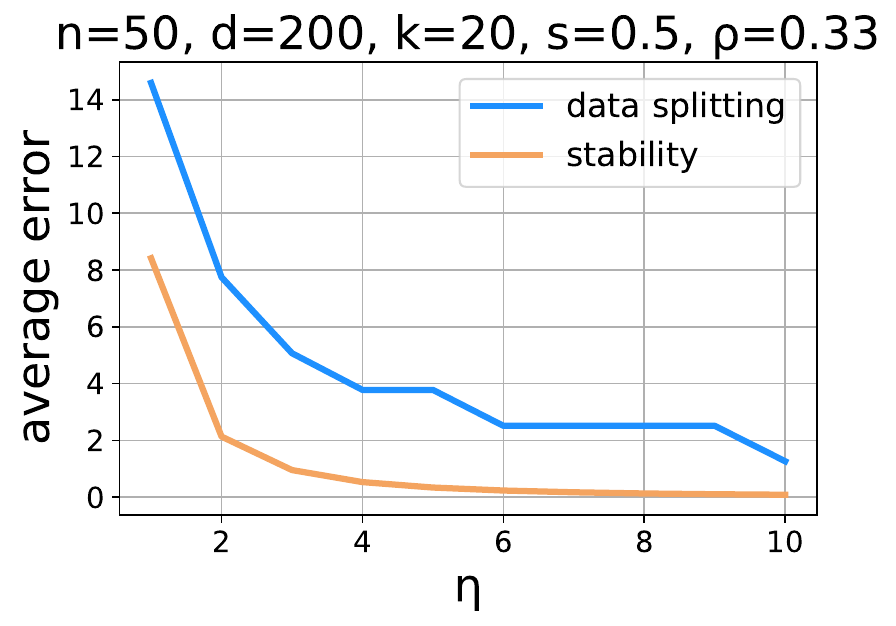}
\includegraphics[width=0.25\textwidth]{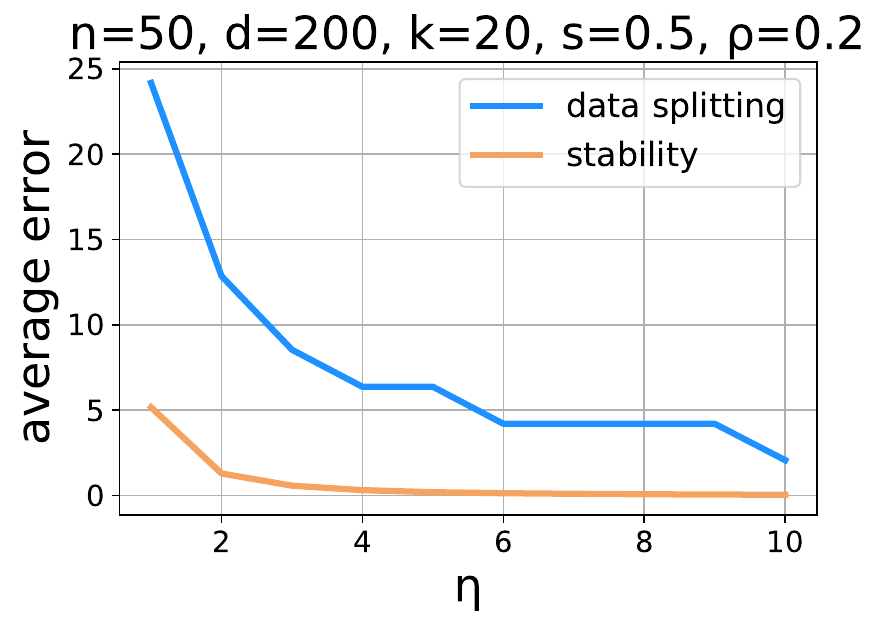}
\includegraphics[width=0.25\textwidth]{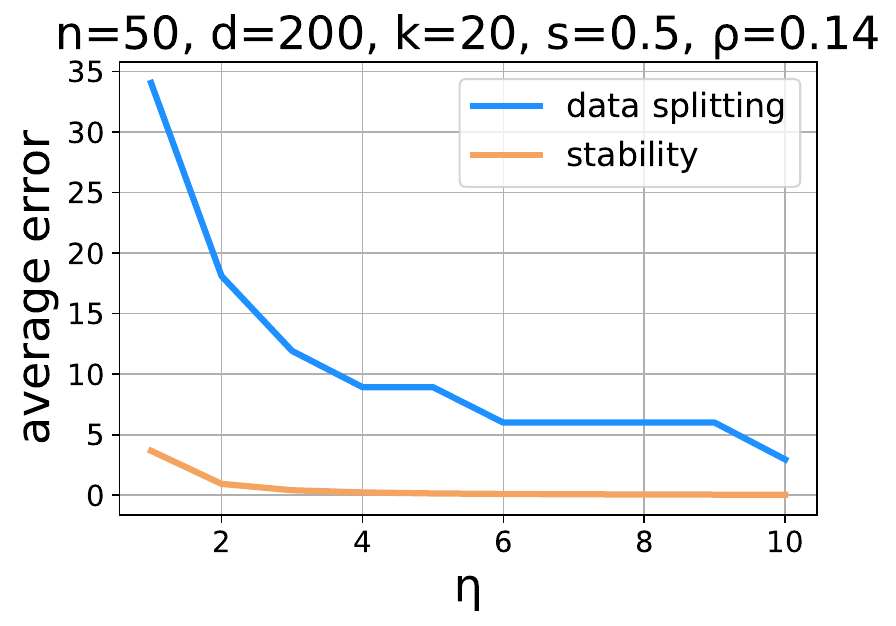}
\includegraphics[width=0.25\textwidth]{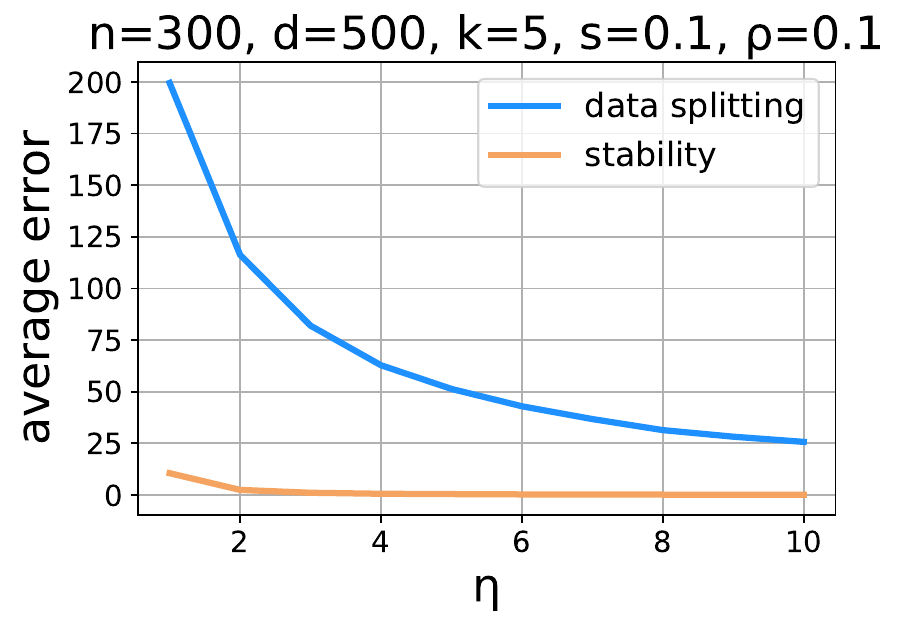}
}
\caption{Comparison of FDR after stable marginal screening and marginal screening with data splitting, with varying dimension and signal strength (first three columns) and sample size (last column), in the Bernoulli design case. The errors are sampled from a Laplace distribution.}
\label{fig:screening_comparison_bern_experr}
\end{figure}

\paragraph{Simulation details.}
All plots of FDR and average error are averaged over $10000$ trials. Width plots are averaged over $500$ trials because the stability correction implies interval widths of low variability. In each trial we generate a new triplet $(X,\beta,y)$, and run stable LASSO (respectively, marginal screening) on this data for all $\eta$, as well as LASSO (respectively, marginal screening) with data splitting. For both stability and data splitting we take a Bonferroni correction to provide simultaneous coverage over all selected variables. Moreover, we note that $\eta$ in the plots refers to the noise parameter in Algorithm~\ref{alg:lasso} and Algorithm~\ref{alg:marginal}; the actual stability parameter used to compute the corresponding splitting fraction for data splitting is multiplied by the number of steps of composition, i.e. we use $f(k\eta)$ as the splitting fraction. Finally, the $\delta$ parameter in Algorithm~\ref{alg:lasso} and Algorithm~\ref{alg:marginal} is set to $\delta=0.5$ throughout.

\newpage

\section{Generalizations of stable algorithms}
\label{sec:extensions}

We show how our stable algorithms can be generalized beyond the setting of Gaussianity. Our proofs exploited Gaussianity of the outcome vector only in terms of the decay of its tails. In general we only need to know how tightly $y$ concentrates around $\mu$ in order to reproduce the stable versions of the LASSO and marginal screening. To show this, we generalize our approach to all outcome vectors with bounded Orlicz norm. This includes other important cases such as general subgaussian and subexponential vectors~$y$.

\begin{definition}[Orlicz norm]
\label{def:orlicz1d}
A function $\psi:\R_\geq 0\rightarrow \R_\geq 0$ is an Orlicz function if $\psi$ is convex, non-decreasing, and satisfies $\psi(0)=0$, $\psi(x)\rightarrow \infty$ as $x\rightarrow \infty$. For an Orlicz function $\psi$, the Orlicz norm or $\psi$-norm of a random variable $W$ is defined as
$$\|W\|_\psi = \inf\left\{s>0: \E\left[\psi\left(\frac{|W|}{s}\right)\right]\leq 1\right\}.$$
\end{definition}
This definition immediately implies a tail bound by Markov's inequality:
\begin{equation}
\label{eqn:orlicz_bound}
	\PP{|W|\geq s\|W\|_\psi} \leq \PP{\psi\left(\frac{|W|}{\|W\|_\psi}\right)\geq \psi(s)} \leq \frac{\E\left[\psi\left(\frac{|W|}{\|W\|_\psi}\right)\right]}{\psi(s)} \leq \frac{1}{\psi(s)}.
\end{equation}

A natural extension of Definition \ref{def:orlicz1d} to random vectors is to consider all one-dimensional projections.

\begin{definition}[Orlicz norm in $\R^n$]
	For a random vector $W\in\R^n$ and Orlicz function $\psi$, we define the Orlicz norm of $W$ as
	$$\|W\|_\psi = \inf\left\{s>0: \sup_{v\in\R^n: \|v\|_2\leq 1}\|W^\top v\|_\psi \leq s\right\}.$$
\end{definition}

We state a corollary of Theorem \ref{thm:conf-ints-general}, where we construct confidence intervals for stable selection methods as long as the outcomes have bounded $\psi$-norm, for some Orlicz function~$\psi$.

\begin{corollary}
Suppose $\|y-\mu\|_\psi \leq G$, for some known $G>0$, and fix $\delta\in(0,1)$. Let $\hat M$ be an $(\eta,\tau,\nu)$-stable model selection algorithm. For all $j\in\hat M$, let
$$\ci_{j\cdot \hat M} = \left(\hat \beta_{j\cdot \hat M} \pm \psi^{-1}\left(|\hat M| e^\eta/ \delta\right)G \sqrt{((X_{\hat M}^\top X_{\hat M})^{-1})_{jj}}\right).$$
Then
$$\PP{\exists j\in \hat M~:~\beta_{j\cdot M} \not\in \ci_{j\cdot \hat M}} \leq \delta + \tau + \nu.$$	
\end{corollary}

\begin{proof}
Fix a model $M$. We only need to argue that
$$\PP{\max_{j\in M}\left|\frac{\hat \beta_{j\cdot M} - \beta_{j\cdot M}}{G\sqrt{((X_{\hat M}^\top X_{\hat M})^{-1})_{jj}}}\right| \geq \psi^{-1}\left(|M| e^\eta/\delta\right)} \leq \delta e^{-\eta}.$$
Invoking Theorem \ref{thm:conf-ints-general} then completes the proof.

Denote by $X_{j\cdot M}$ the residual vector when $X_j$ is regressed onto all other variables in $M$; that is, $X_{j\cdot M} = P^\perp_{X_{M\setminus j}}X_j$, where $P^\perp_{X_{M\setminus j}}$ denotes the projection matrix onto the orthocomplement of $X_{M\setminus j}$. With this notation, we can express the least-squares solution as
$$\hat \beta_{j\cdot M} = \frac{X_{j\cdot M}^\top y}{\|X_{j\cdot M}\|_2^2}, ~~ \beta_{j\cdot M} = \frac{X_{j\cdot M}^\top \mu}{\|X_{j\cdot M}\|_2^2}.$$
Moreover, $(X_M^\top X_M)^{-1}_{jj} = \frac{1}{\|X_{j\cdot M}\|_2^2}$. Using this fact, we have
\begin{align*}
&\lefteqn{\PP{\max_{j\in M}\left|\frac{\hat \beta_{j\cdot M} - \beta_{j\cdot M}}{G\sqrt{((X_{\hat M}^\top X_{\hat M})^{-1})_{jj}}}\right| \geq \psi^{-1}\left(|M| e^\eta/ \delta\right)}}\\
&= \PP{\max_{j\in M} |v_{j\cdot M}^\top (y-\mu)| \geq G \psi^{-1}\left(| M| e^\eta/\delta\right)},
\end{align*}
where we define $v_{j\cdot M}$ to be the random unit vector $\frac{X_{j\cdot M}}{\|X_{j\cdot M}\|_2}$. By applying the tail bound given by
Eq.~\eqref{eqn:orlicz_bound} together with a union bound, we get
\begin{align*}
	\PP{\max_{j\in M} |v_{j\cdot M}^\top (y-\mu)| \geq G \psi^{-1}\left(|M| e^\eta/\delta\right)} &\leq \sum_{j\in M}\PP{|v_{j\cdot M}^\top (y-\mu)| \geq \psi^{-1}\left(| M| e^\eta/ \delta\right)}\\
	&\leq |M| \left(\psi\left(\psi^{-1}\left(|M|e^\eta/\delta\right)\right)\right)^{-1}\\
	&= \delta e^{-\eta}.
\end{align*}
\end{proof}

Now we can generalize the stable versions of the LASSO and marginal screening. We state counterparts of Algorithm \ref{alg:lasso} and Algorithm \ref{alg:marginal} which ensure stability as long as we know a bound $G$ on $\|y-\mu\|_\psi$.

\begin{algorithm}[H]
\SetAlgoLined
\begin{flushleft}
\textbf{input: }design matrix $X\in\R^{n\times d}$, outcome vector $y\in\R^{n}$, $\ell_1$-constraint $C_1$, number of optimization steps $k$, typical stability parameters $\delta\in(0,1), \eta>0$\\
\textbf{output: } LASSO solution $\thetalasso \in\R^d$\newline
Initialize $\theta_1 = 0$\newline
\For{$t=1,2,\dots,k$}{
\ $\forall\phi \in C_1 \cdot \{\pm e_i\}_{i=1}^d$, sample $\xi_{t,\phi} \stackrel{\text{i.i.d.}}{\sim} \text{Lap}\left(\frac{4\psi^{-1}(1/\delta)C_1 \|X\|_{2,\infty}G}{n\eta}\right)$\newline
  $\forall\phi \in C_1 \cdot \{\pm e_i\}_{i=1}^d$, let $\alpha_\phi = \frac{2}{n}\phi^\top X^\top(y - X\theta_t) + \xi_{t,\phi}$\newline
Set $\phi_t = \argmin_{\phi\in C_1 \cdot\{\pm e_i\}_{i=1}^d} \alpha_\phi$\newline
Set $\theta_{t+1} = (1-\Delta_t)\theta_t + \Delta_t \phi_t$, where $\Delta_t = \frac{2}{t+2}$
}
 Return $\thetalasso = \theta_{k+1}$
 \end{flushleft}
\caption{Stable LASSO algorithm under general Orlicz norm $\psi$}
\label{alg:orlicz-lasso}
\end{algorithm}

\begin{algorithm}[H]
\SetAlgoLined
\begin{flushleft}
\textbf{input: } design matrix $X\in\R^{n\times d}$, outcome vector $y\in\R^{n}$, model size $k$\\
\textbf{output: } $\hat M = \{i_1,\dots,i_k\}$\newline
Compute $(c_1,\dots,c_d) = \frac{1}{n}X^\top y\in \R^d$\newline
$\text{res}_1 = [d]$\newline
\For{$t=1,2,\dots,k$}{
\ $\forall i \in \text{res}_i$, sample $\xi_{t,i} \stackrel{\text{i.i.d.}}{\sim} \text{Lap}\left(\frac{2\psi^{-1}(1/\delta)\|X\|_{2,\infty}G}{n\eta}\right)$\newline
$i_t = \argmax_{i\in\text{res}_t} |c_i + \xi_{t,i}|$\newline
$\text{res}_{t+1} = \text{res}_t \setminus i_t$
}
 Return $\hat M = \{i_1,\dots,i_k\}$
 \end{flushleft}
\caption{Stable marginal screening algorithm under general Orlicz norm $\psi$}
\label{alg:orlicz-marginal}
\end{algorithm}

\end{document}